\documentclass[11pt]{article}
\usepackage{amsthm}
\usepackage{amssymb}
\usepackage{epsfig}
\usepackage{epstopdf}
\usepackage[dvipsnames,usenames]{color}
\usepackage{paralist}
\usepackage{amsmath}
\usepackage{amsfonts}
\usepackage{float}
\usepackage{xcolor}
\usepackage{lpic}
\usepackage[all]{xy}
\usepackage[hidelinks]{hyperref}
\usepackage{mathabx} 
\usepackage{graphicx}
\usepackage{caption}
\usepackage{subcaption}
\usepackage{tikz}

\usetikzlibrary{cd}

\newcounter{contador}
\newcounter{teoA}

\newtheorem{teoa}[teoA]{Theorem}

\newtheorem{propo}[contador]{Proposition}
\newtheorem{teo}[contador]{Theorem}
\newtheorem{lem}[contador]{Lemma}
\newtheorem{defi}[contador]{Definition}
\newtheorem{corol}[contador]{Corollary}

\newtheorem{nota}[contador]{Remark}

\newcounter{ex}

\newcommand{\rec}{\noindent}    

\newcommand{\R}{{\mathbb R}}

\newcommand{\N}{{\mathbb N}}
\newcommand{\Q}{{\mathbb Q}}
\newcommand{\Z}{{\mathbb Z}}

\newcommand{\A}{{\cal{A}}}

\title{Invariant graphs and dynamics of a family of continuous piecewise linear planar maps}

\author{Anna Cima$^{(1)}$, Armengol Gasull$^{(1)}$, V\'{\i}ctor Ma\~{n}osa$^{(2)}$ and Francesc Ma\~{n}osas$^{(1)}$
    \\*[.1truecm]
    {\small \textsl{$^{(1)}$ Departament de Matem\`{a}tiques, Facultat
            de Ci\`{e}ncies,}}
    \\*[-.25truecm] {\small \textsl{Universitat Aut\`{o}noma de Barcelona,}}
    \\*[-.25truecm] {\small \textsl{08193 Bellaterra, Barcelona,
    Spain}}
    \\*[-.25truecm] {\small \textsl{anna.cima@uab.cat,
            armengol.gasull@uab.cat, francesc.manosas@uab.cat}}\\
    \\*[-.25truecm] {\small \textsl{$^{(2)}$ Departament de Matem\`{a}tiques,}}
     \\*[-.25truecm] {\small \textsl{Institut de Matem\`{a}tiques de la UPC-BarcelonaTech (IMTech),}}           
    \\*[-.25truecm] {\small \textsl{Universitat Polit\`{e}cnica de Catalunya}}
    \\*[-.25truecm] {\small \textsl{Colom 11, 08222 Terrassa, Spain}}
    \\*[-.25truecm] {\small \textsl{victor.manosa@upc.edu}}}

\date{\today}

\begin{document}

\maketitle
\begin{abstract}   
We consider the family of piecewise linear maps
$$F_{a,b}(x,y)=\left(|x| - y + a, x - |y| + b\right),$$
where $(a,b)\in \R^2$. This family belongs to a wider one that has deserved some interest in the recent years as it provides a framework for generalized Lozi-type maps. 
Among our results, we prove that for  $a\ge 0$ all the orbits are eventually periodic and moreover that there are at most three  different periodic behaviors formed by at most seven points. For $a<0$ we prove that for each $b\in\R$ there exists a compact graph $\Gamma,$ which  is invariant under the map $F$, such that for each $(x,y)\in \R^2$ there exists $n\in\mathbb{N}$ (that may depend on $x$) such that $F_{a,b}^n(x,y)\in \Gamma.$ We give explicitly all these invariant graphs and we characterize the dynamics of the map restricted to the corresponding graph for all $(a,b)\in\mathbb{R}^2$ obtaining, among other results, a full characterization of when $F_{a,b}|_{\Gamma}$  has positive or zero entropy. 
\end{abstract}

%
%



\newpage

\tableofcontents

\newpage
 
\section{Introduction and main results}
In this paper we consider the family of piecewise linear maps of the form
\begin{equation}\label{e:F}
F_{a,b}(x,y)=\left(|x| - y + a, x - |y| + b\right),
\end{equation}
where $(a,b)\in \R^2$. In all the paper, when
no confusion is possible we will write $F$ instead of~$F_{a,b}.$

Piecewise linear maps appear in the study of power electronics, neural networks, mechanical systems with friction or in economy \cite{ban1999,B,S19,ZM}, and nowadays they are a subject of intense research both from the theoretical point of view and from the point of view of applications, see \cite{GSM18,GQ,L23,MS18} and other references cited below.

A particular case of piecewise linear maps are the continuous ones defined by absolute values. Their interest, due to its dynamic richness, is clear to the scientific community, at least, since 1978 when R.~Lozi introduced the celebrated map that bears his name \cite{L78}, and when a particular case of which, known as the Gingerbread map, was studied by R.~Devaney in \cite{D84}. This type of maps holding some absolute value function also have been studied by M.~R.~Herman \cite[Chap. VIII]{H} in 1986 and by  J.-M.~Strelcyn \cite[Sec. 3]{St} in 1991 and are still the object of much interest, both for themselves \cite{GT20,TG20}, and for the fact that they are conjugated with other types of maps, such the as max-type ones \cite{GL,LN22}.

We remark that many of the maps treated in the above papers are conservative, while our map~\eqref{e:F} is strongly dissipative. This property provokes that the tools used in our work are different from the ones utilized in the quoted papers. 

In \cite{GL}, Grove and Ladas introduced the family of maps $G(x,y)=(|x|+\alpha y+\beta,x+\gamma|y|+\delta)$ with $\alpha,\beta,\gamma\in\mathbb{R}$ and $\delta\in\{-1,0,1\}$, in the spirit of generating a broader framework for studying generalized Lozi-type maps. The family of maps $F_{a,b}$ intersects this general family. In the last fifteen years, some works have appeared that analyze different particular cases of the Grove-Ladas family, see for example \cite{ABK21,TLL13, TLS17}, just to cite those works that include subcases of the family $F_{a,b}$. Essentially, these works characterize cases in which  every orbit  converges to a fixed point, to a periodic orbit, or it is eventually periodic (that is, the points in it reach a periodic orbit in a finite number of iterations). In fact, our starting interest in the study of the family $F_{a,b}$ was I.~Bula and 
A.~S\={\i}le's a talk, at the 26th International Conference on Difference Equations and Applications that took place in Sarajevo in 2021, \cite{BS21} (see also \cite{BK}). In that talk, it was expressed the hypothesis that for $a,b<0$ all orbits are eventually periodic, and also was proved the existence of periodic orbits of different periods for some values of the parameters. One of our  motivations was to explore this hypothesis. 

As we will see, the global dynamics of the family $F_{a,b}$ is substantially richer. In fact, the dynamics generated by $F$ strongly depends on the parameters $a$ and $b$, and  there are some cases where the dynamics is extremely simple while other ones present chaotic behaviors in $1$-dimensional objects (graphs) that capture the final dynamics of the map, see Theorems~\ref{t:teoA}, \ref{t:teoB} and \ref{t:teoD}.  We notice, however, that there are reasons why numerical evidences seem to support the initial hypothesis. More concretely, the results of Theorem~\ref{t:teoC} give a partial explanation of why numerically it is difficult to see complicated dynamics although, as we will prove in this paper, they actually do exist for some values of the parameters when $a<0.$

Before starting, let us see that the family of maps $F_{a,b}$ has only one essential parameter.  Indeed, note that for any $\lambda>0,$  \begin{equation}\label{conj}\lambda F_{a,b}(x/\lambda,y/\lambda)=F_{\lambda a,\lambda b}(x,y).\end{equation} 
This equality implies that for any $(a,b)\in \R^2$ and for any $\lambda>0$ the maps $F_{\lambda a,\lambda b}$ and $F_{a,b}$ are conjugate. Hence in our proofs we can restrict our attention to three cases $a\in\{-1,0,1\}.$

As usual, for a map $G:\R^2\longrightarrow\R^2$ we  will denote by $\operatorname{Per}(G)$ the set  of all periodic points of $G.$ Given a periodic orbit, unless otherwise stated, by period is meant minimal period. Next result characterizes completely the  dynamics of $F$ when $a\ge0.$

\begin{teoa}\label{t:teoA}
	If $a\ge 0,$  for each $\mathbf{x}\in\R^2$   there exists $n\ge 0,$ that may depend on $\mathbf{x},$ such that $F^n(\mathbf{x})\in \operatorname{Per}(F).$ Moreover, the set $\operatorname{Per}(F)$ is formed by a fixed point and, depending on $a$ and $b,$ either two or none 3 periodic orbits.
	\end{teoa}

The proof of Theorem \ref{t:teoA} is done in Section 3. In fact, in Propositions \ref{aiguala1}, \ref{azerobu}, \ref{azerobmenysu} and \ref{azerobzero}, there is  a more detailed description of the dynamics in subcases that cover the whole case $a\ge 0$. 

Much more interesting dynamics appear when $a<0.$ In order to perform this study we begin by proving that the dynamics is concentrated in an one-dimensional compact set of the plane.

\begin{teoa}\label{t:teoB}
	 If $a<0,$ for each $b\in\R$  there is a compact graph $\Gamma=\Gamma_{a,b},$ which is invariant under the map $F,$ such that for every $\mathbf{x}\in \R^2$ there exists a non-negative integer $n,$ that may depend on $\mathbf{x},$ such that $F^n(\mathbf{x})\in \Gamma.$  
\end{teoa}

 As we have explained, to study the case $a<0$ it suffices to consider the case $a=-1$ by using the conjugation \eqref{conj}. Then, for  $a=-1$  and each value of $b$ all the graphs $\Gamma_{-1,b}$ are the ones given in the Appendix adding to them, if necessary, 
 a fixed point and a three periodic orbit that are given explicitly in Proposition \ref{primera reduccioamenys1}. The conjugation can be used afterwards to obtain each $\Gamma_{a,b}$ from its corresponding $\Gamma_{-1,-b/a}.$
 
We prove  that, apart from the isolated points eventually added following the results of Proposition~\ref{primera reduccioamenys1}, the graph $\Gamma$ is a connected set formed by the union of at most 23 compact segments, see Figure~\ref{f:F} for an example of graph with 23 segments. These segments have  one of the four  slopes $0$, $1$, $-1$ and $\infty$.  Moreover, we will prove that there appear exactly  37 topologically different graphs, according the  values of $a<0$ and $b$, see again the Appendix.
 
 It is also remarkable that when the initial condition $\mathbf{x}=(x,y)$ is such that $xy\ge0,$ the value $n$ in Theorem \ref{t:teoB} is at most 11, and this upper bound is reached for some $x,y,a<0$ and $b,$ see Proposition~\ref{p:lemanou}.  On the other hand for $xy<0$ and some values of $a<0$ and $b$ this number is unbounded. 
 
 From Theorem \ref{t:teoB} we get that all $\omega$-limit sets of $F$ are contained in $\Gamma.$ Thus to study the dynamics of $F$ it suffices to study the dynamics of $F\vert_{\Gamma}.$   Before giving a more detailed description of these dynamics we first prove in next theorem a  result that provides a partial explanation of the reason why only simple dynamical behaviors are mostly the ones that can be numerically detected, see also Section \ref{s:future}.

 \begin{teoa}\label{t:teoC} Set $a<0,$ $b\in\R$ and  let $\Gamma$ be  the corresponding invariant graph for $F$ given by Theorem \ref{t:teoB}. Then, for  an open and dense set of  initial conditions $\mathbf{x}\in \Gamma$    there are  at most three possible   $\omega$-limit sets.  Moreover, if  $b/a\in \Q$ these  $\omega$-limit sets are periodic orbits.

\end{teoa}

As we will see, in each $\Gamma$ there are segments that collapse by $F$ to single points in $\Gamma.$ In a few words, what we will prove is that the open and dense set of the statement is formed by the union of all the preimages by $F$ of these segments.  We will show that, depending on the values of the parameters, these three $\omega$-limit sets can be periodic orbits, Cantor sets or other much more complicated subsets of $\Gamma.$
 
 In fact, if a property in a topological space is satisfied for all elements in an open and dense subspace it is usually said that it is {\em generic}.  What we suspect is that, rather of being only generic, the property that we have proved in~Theorem~\ref{t:teoC} is satisfied by a full 
  Lebesgue measure set of initial conditions in $\Gamma.$ We have been able to prove this fact only for some values of $b$  in \cite{CGMM}, but we do not consider this question in this work.

  As we can see in the Appendix, to describe with more detail the dynamics on each~$\Gamma,$ a lot of cases arise.  Recall that there are 37 different graphs. For most of them, we define a suitable partition of $\Gamma$ and we consider the oriented graph associated to the partition, which allow us to study the \emph{topological entropy} of $F\vert_{\Gamma}$ in each case and, moreover, to elucidate its dynamics.  In other cases, for instance, when the graph is homemorphic to a circle we use another approach. 
See Section~\ref{s:prelim} for more details on the tools that are used. 
  
  Shortly, the entropy $h$ is a non-negative real number associated to a map such that when $h=0$ then the dynamics is ``simple'' and when
     $h>0$ it is ``complicated''. More specifically, when $F\vert_{\Gamma}$ has entropy $h>0$ then it  has periodic orbits with infinitely different periods and the orbits have many different combinatorial behaviors. Indeed, by using the ideas of the book of Alsed\`a, Llibre and Misiurewicz~\cite{ALM} it can be
     seen that when the map  $F|_{\Gamma}$ has positive entropy, then it is chaotic in the
     sense of Li and Yorke, see ~\cite{LY}. In particular, this means that it has periodic points with arbitrarily large periods and, there exists an uncountable set $\mathcal{S}\subset\Gamma$, called \emph{scrambled set}, so that for any $p,q\in\mathcal{S}$ and each periodic point $r$ of $F\vert_{\Gamma},$ 
     \begin{align*}
     	&\limsup_{n\to\infty} |F^n(p)-F^n(q)|>0, \quad\liminf_{n\to\infty} |F^n(p)-F^n(q)|=0\quad \mbox{and}\\
     	&\limsup_{n\to\infty} |F^n(p)-F^n(r)|>0. 
     \end{align*}
  We will give more details about the entropy and how to compute it in Section~\ref{s:prelim}. 
  
 All our results about this matter are summarized in the following theorem.

\begin{teoa}\label{t:teoD}

 Set $a<0, b\in\R$ and define $c=-b/a$. Consider the map $F$ given in~\eqref{e:F}, restricted to its corresponding invariant graph $\Gamma$ given in Theorem~\ref{t:teoB}. Then there exist $\alpha$ and $\beta$ such that $F|_{\Gamma}$  has positive entropy if and only if $c\in (\alpha,-1/36)\cup(\beta,1)\cup(1,8),$ where $\alpha\in (-112/137,-13/16)\approx(-0.8175,-0.8125),$  $\beta\in(603/874,563/816)\approx(0.6899,0.6900),$ and in these two intervals the entropy of $F\vert_{\Gamma}$ is non-decreasing in $c.$ Moreover,
	the entropy as a function of $c$ is discontinuous at $c=-1/36$.

\end{teoa}

As we will see in the proof, the transition from zero entropy to chaos, which is present in the two intervals for the parameters $\alpha, \beta$ appearing in the theorem, can be described by means of two associated  one-parameter families of unimodal maps and taking advantage of the results of \cite{BMT}. We will prove that these families are ``full families'' in the sense that all possible unimodal dynamics are present in the family. Moreover, the behavior of the entropy with respect to the parameter is monotonic.

We remark that although we have not been able to obtain explicit expressions of the values  $\alpha$ and $\beta$ of the statement, the property described in the above paragraph allows to obtaining rational upper and lower bounds for them, as sharp as desired. We compute more accurate bounds in~\cite{CGMM}.

 It is worth noting that the discontinuity of the entropy as a function of the parameters is not so rare in the context of continuous maps with zero entropy.

Moreover, in the proof of Theorem \ref{t:teoD}, for each value of $c$ we either obtain the exact value of the entropy or we give explicit lower bounds for it. 
When the entropy is zero we also detail all the dynamics of $F.$ In fact, in the majority of cases, there are only a finite number of periodic orbits, some of them repelling and some other attracting, and the dynamics is pre-periodic, that is periodic after finite number of iterates. However, when the graph is a topological circle, the dynamics will be determined by the associated rotation number of $F|_{\Gamma}$. When this number is rational, we will characterize all the possible periods. When it is irrational, we will prove that the  $\omega$-limit sets are Cantor sets. Another totally different situation with zero entropy corresponds to the cases where $c$ is either $\alpha$ or $\beta.$

Theorems \ref{t:teoB} and \ref{t:teoC} will be proved in Section~\ref{sec:4}  and  Theorem \ref{t:teoD}  in Section~\ref{sec:5}. A more detailed description of the dynamics and the values of the entropy is given in Propositions \ref{primer}--\ref{sis}, \ref{sotasota1},  \ref{sota1}, \ref{exac1},  \ref{b12},  \ref{b24}, \ref{b48} and \ref{bmes8}.

 We believe that the study of the maps $F_{a,b}$ considered in this work, and the discovery of the invariant graphs in Theorem \ref{t:teoB}, is interesting because provides a natural continuous two dimensional discrete dynamical system for which all the final dynamics is one dimensional. Moreover, this one dimensional dynamics  presents all the richness of the one dimensional setting.  Finally, we believe that it is also valuable
  to the extent that the study that we have made can be extended to other families of continuous piecewise linear maps. More concretely, for maps such that there is an open region on which the linear part of the map has rank one (and therefore there is a collapse of the dynamics to a line, ray or segment) and such that, except for a controllable set of points, the positive orbit of the rest of the points visits this region,  see Section \ref{s:future}
for some additional comments.

\section{Preliminary results}\label{s:prelim}

\subsection{Topological entropy}

The {\em topological entropy} of a continuous selfmap of a compact
space was introduced in 1965 by Adler, Konheim, and McAndrew (\cite
{AKM}). It measures  the combinatorial complexity of the map. Bowen,
Misiurewicz and Ziemian (\cite {Bo3,MZ}) generalized this
notion for non necessarily continuous maps. We do not describe here
the general definition and we will only introduce the definitions for
a particular class of one-dimensional maps and state some basic
properties that we will use along the paper. We address the reader
to the original articles and also to the chapter dedicated to the
entropy in \cite{ALM}. The results in  this reference state and prove
the results for the interval or the circle maps. However, the proofs
are easily adaptable and also valid for general one-dimensional
spaces like compact graphs. A {\em graph} is a pair $(X,V)$ where $X$ is a compact
Hausdorff space and $V\subset X$ is finite and such that $X\setminus V$ is the disjoint union of a finite number of open subsets of X, called {\em edges}, with the property that each of them is homeomorphic to an open interval of the real line.

Let $G$ be a compact graph and let $f:G\longrightarrow G$ be a map. We
will say that $f$ is {\em piecewise monotone} if there exists
$\mathcal {A},$ a finite cover of $G$ by intervals (i.e. segments of the edges), such that
for all $A\in \mathcal {A}$,
 $f(A)$
is an interval and $f$ restricted to $A$ is continuous and monotone.
 Note that in this definition $f$ is not
necessarily a continuous map. Now we will use a particular type of
covers: the ones corresponding to partitions formed by closed intervals. Note that two intervals
of a partition can only intersect at one point and, in this case, the
intersection is a common endpoint of both intervals.

Let $f:G\longrightarrow G$ be a piecewise monotone map on a compact
graph $G.$ Let $\mathcal {P}=\{I_1,\ldots,I_n\}$ be a finite
partition of $G$ by closed intervals. We say that $\mathcal {P}$  is
a {\em mono-partition} if $f(I_i)$ is homeomorphic to an interval and $f\vert_{I_i}$
is continuous and monotone for all $i\in\{1,\ldots,n\}.$ We call
{\em turning points} the endpoints of the intervals $I_i$ and we
denote by $C$ the set of all turning points. When $x\in G\setminus
C$ we define the address of $x$ as $A(x):=I_i$ if $x\in I_i.$ When
$x\in G\setminus\left( \cup_{i=0}^{m-1}f^{-i}(C)\right)$ we define
the {\em itinerary of length $m$ of $x$} as the sequence of symbols
$$I_m(x)=A(x)A(f(x))\ldots A(f^{m-1}(x)).$$ Let $N(f,\mathcal {P},m)$ be the
number of different itineraries of length $m.$ Note that $N(f,
\mathcal {P},m)\le n^m$. Then, the following holds.

\begin{lem}\label{grownumber} Let $f:G\longrightarrow G$ be a piecewise monotone map on a compact
	graph $G.$ Let $\mathcal {P}$  be a mono-partition. Then 
	$ \lim\limits_{m} \sqrt[m]{N(f,\mathcal {P},m)}$ exists. Moreover, this limit is independent of the
	choice of the mono-partition $\mathcal P$, and $$h(f):=\ln\left(\lim\limits_{m}
	\sqrt[m]{N(f,\mathcal {P},m)}\right)$$ is the topological entropy of
	$f.$\end{lem}

We call the number $s(f):=\lim\limits_{m} \sqrt[m]{N(f,\mathcal {P},m)}$ the {\em growth
	number of $f.$}

It is well-known that the entropy is an invariant for conjugation. In the case of interval maps, it is also invariant for a more general notion. Let $f,g$ be a piecewise monotone self maps on the interval $I$. We will say that $f$ and $g$ are {\em semiconjugated} if there exists a non-decreasing map $s:I\longrightarrow I$ such that $s(I)=I$ and $g\circ s=s\circ f.$ In this case it is also well-known that $h(f)=h(g).$

\begin{nota} \label{colapses} For  $f:G\longrightarrow G$, a piecewise monotone map on a compact
	graph $G$, if there is an interval $I\subset G$ such that $f$
	restricted to $I$ is constant, then we can collapse the interval to a
	point $p$ obtaining a new graph $\tilde G.$ We can also define
	$\tilde f$ on $\tilde G$  by $\tilde f(x)= p$ if $f(x)\in I.$ In
	this way we obtain a piecewise monotone map $\tilde f$ on $\tilde
	G.$ Clearly $s(f)=s(\tilde f),$ so this operation does not affect
	the computation of entropies.\end{nota}

Let $f:G\longrightarrow G$ be a piecewise monotone map on a compact
graph $G.$  Let $\mathcal {P}$  be a mono-partition. We will say
that $\mathcal {P}$ is {\em Markov partition} if for all $I\in
\mathcal {P},\,\,f(I)$ is the union of some elements of $\mathcal
{P}.$ Clearly, in this situation, the set of turning points is an
invariant set.

Once again, let $f:G\longrightarrow G$ be a piecewise monotone map on a compact
graph $G$ and let $\mathcal {P}=\{I_1,\ldots,I_n\}$ be a
mono-partition. The {\em associated matrix to $\mathcal P$} is the
$(n\times n)$-matrix defined by
$$m_{i,j}=\left\{
\begin{array}{ll}
	1, & \hbox{if $I_j\subset f(I_i)$;} \\
	0, & \hbox{otherwise.}
\end{array}
\right.$$

We denote it by $M(f,\mathcal {P})$ and also denote by  $r(\mathcal
{P})$ its  spectral radius (i.e., the maximum of the modulus of its eigenvalues). From the Perron-Frobenius Theorem (see
\cite {Gant}) we know that the spectral radius of $M(f,\mathcal
{P})$ is reached in a non-negative real eigenvalue.

We have the following result:

\begin{lem}\label{Markov} Let $f:G\longrightarrow G$ be a piecewise monotone map on a compact
	graph $G$ and let $\mathcal {P}=\{I_1,\ldots,I_n\}$ be a
	mono-partition. Then $r(\mathcal {P})\le s(f).$ Moreover if
	$\mathcal {P}$ is Markov, then $r(\mathcal {P})= s(f).$
\end{lem}

\begin{nota} \label{grafcota} There is an alternative matrix, $\bar M(f,\mathcal {P})$ that we can associate to a
	mono-partition $\mathcal {P}=\{P_1,\ldots,P_n\}$ by the rule
	$$\bar m_{i,j}=\left\{
	\begin{array}{ll}
		1, & \hbox{if $f(I_i)$ intersects the interior of $I_j$;} \\
		0, & \hbox{otherwise.}
	\end{array}
	\right.$$
	
	Notice that when $\mathcal {P}$ is a Markov partition we get $\bar
	M(f,\mathcal {P})= M(f,\mathcal {P}).$ Also notice that $\bar
	M(f,\mathcal {P})$ can be thought as the Markov matrix of a map $g$
	(maybe discontinuous) that has $\mathcal {P}$ as a Markov partition.
	Clearly $N(f,\mathcal {P},m)\le N(g,\mathcal {P},m)$ and then $s(f)\le s(g).$ Therefore,
	if we denote by $\bar r(\mathcal {P})$ the spectral radius of $\bar
	M(f,\mathcal {P})$ we will get $s(f)\le \bar r(\mathcal {P}).$
\end{nota}

There is a nice method to compute spectral radius of a square
$(n\times n)$-matrix $M$ with entries $m_{i,j}\in \{0,1\}$, see \cite{BGMY}. Since we will use it extensively in this paper, we briefly recall it. We
construct an abstract oriented graph whose vertices are $I_1,\ldots
,I_n$ and there is an oriented arrow from $P_i$ to $P_j$ if and and
only if $I_j\subset f(I_i).$ Next we introduce the notion of \emph{rome}.

\begin{defi}\label{d:roma}
	Let $M=(m_{ij})_{i,j=1}^n$ be an $n\times n$ matrix  with $m_{i,j}\in\{0,1\}$. For a sequence $p=(p_j)_{j=0}^k$ of elements of $\{1,2,\ldots n\}$ its width $w(p)$
	is defined by $w(p)=\prod_{j=1}^k  m_{p_{j-1}p_j}.$ And $p$ is called a path if $w(p)\ne 0.$
	In this case, $k=\mathit{l}(p)$ is the length of the path $p.$ A subset $R\subset \{1,2\ldots n\}$ is called a rome if there is no
	loop outside $R,$ i.e., there is no path $(p_j)_{j=0}^k$ such that $p_0=p_k$ and $(p_j)_{j=0}^k$ is disjoint from $R.$ For a rome $R$ we call a path
	$(p_j)_{j=0}^k$ simple if $\{p_0,p_k\}\subset R$ and $\{p_1,\ldots ,p_{k-1}\}$ is disjoint from $R.$
\end{defi}

Roughly speaking, a rome $R$ is a collection of vertices of the oriented graphs such that any loop in the graph must pass through an element of $R$.  Clearly, the choice  of the name rome  in the above definition is motivated by the ancient proverb ``all roads lead to Rome''. Note that a path in the matrix associated with the oriented graph corresponds to a path in the graph.

For a rome $R=\{r_1,\ldots ,r_k\},$ where $r_i\ne r_j$ for $i\ne j,$
it is defined a matrix function $A_R$ by $A_R=(a_{ij})_{i,j=1}^k$
where $a_{ij}(x)=\sum_{p} w(p)\,\lambda^{-\mathit{l}(p)}$ where the
summation is over all simple paths originating at $r_i$ and
terminating at $r_j.$ By $E$ we denote the unit matrix (of an
appropriate size).
\begin{teo} \label{rome} (See \cite{BGMY})
	Let $R=\{r_1,\ldots,r_k\}$ (with $r_i\ne r_j$ for $i\ne j$) be a rome.
	Then the characteristic polynomial of $M$ is equal to $(-1)^{n-k}\,\lambda^n\,\mathrm{det}(A_R(\lambda)-E).$
\end{teo}

We emphasize that the logarithm of the largest root of $\mathrm{det}(A_R(\lambda)-E)$ is just the topological entropy. This fact will be used later when computing entropy of certains graphs. 

From now we will speak indiscriminately about the entropy of a matrix $M$ or the entropy of the associated oriented graph as the logarithm of its spectral radius. In the next remark we collect some easy observations about the oriented graph associated to a matrix $M$ that allows us to know if the  matrix has or not has positive entropy. We will say that two loops in the oriented graph associated to $M$ are conected if there is a path that begins in one element of the first loop and ends in one element of the second one and viceversa.

\begin{nota}\label{romaentr} 	Let $M=(m_{ij})_{i,j=1}^n$ be an $n\times n$ matrix  with $m_{i,j}\in\{0,1\}$ and consider
its associated oriented graph. Assume that  $R=\{r_1,\ldots ,r_k\},$ is a rome. Then the following assertions hold.
\begin{enumerate} \item[(i)] If for some $1\le j\le k$ there are two different loops passing trough $r_j$ then the entropy of $M$ is positive.\item [(ii)] If for some $1\le i<j \le k,$ $r_i$ and $r_j$ are connected then the entropy of $M$ is positive.
	\item[(iii)] If for all $1\le i\le k,$ there are one and only one loop passing through $r_i,$ then  $M$ has zero entropy. Note that in this case for all $1\le i<j \le k,$ $r_i$ and $r_j$ are not connected.
\end{enumerate}
\end{nota}

Lastly we collect in the next lemma some results about topological entropy $h$ that we will use in the rest of the paper.

\begin{lem}\label{entroo} The following statements
hold: \begin{enumerate}[(i)]\item If $f:X\longrightarrow X$ is a continuous map in a compact
space, then $h(f^n)=nh(f).$
\item Let $\mathcal M$ be the space of continuous maps of a compact interval with the topology of uniform convergence
and consider the map $h:\mathcal M\longrightarrow \R.$ Then $h$ is lower semi-continuous, that is, for any $f\in \mathcal M$ 
$$\liminf_{g\to f} h(g)\ge h(f).$$
\end{enumerate}
\end{lem}

In our setting, from Theorem \ref{t:teoB} we know that the dynamics generated by $F$ when $a=-1$ reduces to study the action of $F$ on some one-dimensional invariant graphs. 

For each one of the cases we want to determine the entropy of $F$ restricted to the corresponding graph, say $\Gamma.$ Once we fix a mono-partition we observe that there are some intervals $I$ such that $F$ restricted to $I$ reduces to a single point, say $p.$ We denote this behavior by $I\twoheadrightarrow p.$ Then by using Remark \ref{colapses} we eliminate the intervals such that after a finite number of iterates collapse to a point in the corresponding oriented graph. This procedure will simplify our computations when computing the entropy of $F|_\Gamma$   or some bounds of it.

\subsection{Rotation numbers and set of periods of a rotation interval}

For some values of the parameter $b,$ the associated invariant graph $\Gamma$ is a topological circle and $F\vert_{\Gamma}$ is a degree one circle map such that its liftings are non-decreasing. We denote by $\cal {L}$ this class of maps. For a degree one homeomorphisms of the circle, Poincaré introduced the notion of {\em rotation number } that can be easily generalized for our class of maps (see \cite{ALM}).

\begin{propo} Let $g\in \cal {L}$ and  let $G:\R\longrightarrow\R$ be a lifting of $g.$ Then for all $x \in \R$ the limit $$\lim\limits_{n} \frac{G^n(x)-x}{n}$$ exists and it is independent of $x.$ This limit (mod $\Z$) is also independent of the choice of the lifting.\end{propo}

The above limit (mod $\Z$) will be called the {\em rotation number of $g$} and denoted by $\rho(g).$

In the next proposition we summarize the basic properties of the rotation number that we will use in this work. For more details and proofs, see \cite{ALM,Ni}.

\begin{propo} \label{rot} Let $g\in \cal {L}.$ The following assertions hold:
	\begin{enumerate}[(i)] \item The map $g$ has periodic orbits if and only if $\rho(g)\in\Q.$
		\item If $\rho(g)={p}/{q}$ with $(p,q)=1$ then all the periodic orbits of $g$ have period $q.$
		\item If $\rho(g)\in\R\setminus\Q$ then the $\omega$-limit  is the same for all $x\in{\mathbb S}^1$ and either it is ${\mathbb S}^1$ or a closed subset of ${\mathbb S}^1$ without isolated points and empty interior (Cantor set). The first alternative is not possible if $g$ is constant on some interval.
		\item The map $\phi:\cal {L}\longrightarrow \R/\Z$ defined by $\phi(g)=\rho(g)$ is continuous in the $\mathcal{C}^0$-topology. 
		\item All the maps in $\cal{L}$ have zero entropy.
		\end{enumerate} \end{propo} 

 The following results allow us to obtain a simple method to find, constructively, the set of periods that arise in a continuous parametric family of circle maps with a prescribed rotation interval, see the proof of Proposition \ref{segon}.

\begin{lem}\label{l:divisorfunction}
Consider the interval $[a_1,a_2]$ with $0\leq a_1<a_2$ and $a_1,a_2\in\mathbb{R}$. Set $n\in\mathbb{N}$, and let $d(n)$ denote 
\emph{the divisor function}  which gives the number of divisors of $n$. If 
\begin{equation}\label{e:novaeqdn}
d(n)<\lfloor n\,(a_2-a_1)\rfloor-1
\end{equation}
(where $\lfloor\, \rfloor$ is the floor function), then there is at least an irreducible fraction $\ell/n$ with $\ell\in\mathbb{N}\cup\{0\}$ in the interval $[a_1,a_2]$.
\end{lem}

\begin{proof} Note that there are exactly $\lfloor n\,(a_2-a_1)\rfloor+1$ integers in the interval $[0, \lfloor n\,(a_2-a_1)\rfloor] $ (two of them at the ends of the interval). Therefore, there are at least $ \lfloor n\,(a_2-a_1)\rfloor-1$ integers in the interval $[n\,a_1 ,n\,a_2 ]$, since
$$
[n\,a_1 ,n\,\lfloor (a_2-a_1)\rfloor +n\,a_1  ]\subseteq [n\,a_1 ,n\,a_2 ],
$$ and, at most, two integers can be lost at the ends of the first interval.

Hence, if \eqref{e:novaeqdn} is true, there are more integer numbers in $[n\,a_1 ,n\,a_2 ]$ than divisors of $n$. In consequence, there are more fractions in $[a_1,a_2]$ of the type $\ell/n$ with $\ell\in\mathbb{N}\cup\{0\}$, than divisors of $n.$ As a consequence, one of these fractions must be irreducible.
\end{proof}

\begin{corol}\label{c:divisorfunction}
Let $D(n)$ be any upper bound function for $d(n)$, for which there exists $n_0$ such that for all $n\geq n_0$, it holds $D(n)<\lfloor n\,(a_2-a_1)\rfloor-1$. Then, for each $n\ge n_0,$ there exists an irreducible fraction $\ell/n\in[a_1,a_2].$ 
\end{corol}

The proof of the Corollary is very easy by using Lemma~\ref{l:divisorfunction} since by construction, for all $n\geq n_0$ it holds $d(n)<D(n)<\lfloor n\,(a_2-a_1)\rfloor-1$.

There are some different explicit upper bound functions for $d(n)$, see \cite{Nic88,NicRob}. In the proof of Proposition \ref{segon}, we will use the naive one 
\begin{equation}\label{e:cotaded}
D(n)=2\sqrt{n}.
\end{equation}
Indeed, observe that the divisors of $n$ appear in pairs of the form $d$ and $n/d$, plus $\sqrt{n}$ if $n$ is a square number. Hence the largest possible divisor that $n$ could have is $\sqrt{n}$  and, therefore, an upper bound of $d(n)$ is $2\sqrt{n}$.

\subsection{Additional notation}

 To end this section of preliminary results, we introduce the following notation:
for $i=1,2,3,4,$ denote by $F_i$ the expression of the affine map $F$ restricted to each one of the quadrants $Q_1=\{(x,y): x\geq 0, y\geq 0\},Q_2=\{(x,y): x\leq 0, y\geq 0\},Q_3=\{(x,y): x\leq 0, y\leq 0\}$ and $Q_4=\{(x,y): x\geq 0, y\leq 0\}.$ 
Note that since the expressions of $F_1,F_2,F_3,F_4$ are
\begin{equation}\label{e:Fis}
\begin{array}{lcl}
	F_1(x,y)&=&(x-y+a,x-y+b),\\
	F_2(x,y)&=&(-x-y+a,x-y+b),\\
	F_3(x,y)&=&(-x-y+a,x+y+b),\\
	F_4(x,y)&=&(x-y+a,x+y+b),
	\end{array}
\end{equation}
we get that the straight lines of slope $1$ contained in $Q_1$ and also the straight lines of slope $-1$ contained in $Q_3$ collapse to a point. Hence,  when calculating the entropy of the maps $F|_\Gamma$, where $\Gamma$ is the graph that appears in Theorems \ref{t:teoB} and \ref{t:teoD}, the intervals of the associated abstract oriented graph which are the preimages of these points will not be considered.

\section{The case $a\ge 0$}\label{sec:3}

In this section we will prove Theorem \ref{t:teoA}. Recall that, from  \eqref{conj}, we know that to study the case $a\ge 0$ it suffices to consider the cases $a=1$ and $a=0.$

\subsection{The case $a=1$}

From the expressions of $F_i$ for $i=1,2,3,4$ given in \eqref{e:Fis}, we see that the affine maps $F_2,F_4$ are non-degenerate, in the sense that their associate matrix has full-rank $2$, while $F_1,F_3$ are degenerate since its associate matrix has rank $1$. In fact,  $F(Q_1)$ reduces to the straight line $y=x+b-1$ while $F(Q_3)$ reduces to $y=-x+b+1$, $x\geq 1.$

\begin{propo}\label{aiguala1} When $a=1,$ the following statements hold.
	
	\begin{enumerate}\item[(i)] For $b\le 2,$ $F$ has the fixed
		point $p=(2-b,1)\in Q_1.$ For $b\in [-1/2,2]$ and for all
		$(x,y)\in\R^2,$ $F^5(x,y)=p,$ while for
		$b<-1/2,$ $F^6(x,y)=p.$
		
		\item[(ii)] For $b>2,$ $F$ has the fixed
		point $q=\left(\frac{2-b}{5},\frac{1+2b}{5}\right)\in Q_2.$  Also it has two
		3-periodic orbits, namely
		$\mathcal{P}=\{(b-2,1),(b-2,2b-3),(2-b,1)\}$ and
		$$\mathcal{Q}=\left\{\left(\frac{b-2}{3},\frac{2b-1}{3}\right),\left(\frac{2-b}{3},\frac{2b-1}{3}\right),\left(\frac{2-b}{3},1\right)\right\}.$$
		Moreover, for each $(x,y)\in \R^2\setminus\{q\}$ there exists $n\in\N,$
		that depends on $(x,y),$ such that $F^n(x,y)\in \mathcal{P}\cup \mathcal{Q}.$
	\end{enumerate}
\end{propo}

Statement $(i)$ for the cases $b=0$ and $b=1$ can also be found in \cite{TLL13}.

\begin{proof}[Proof of Proposition \ref{aiguala1}] 
$(i)$ A direct computation shows that $F(p)=p$. Moreover, straightforward computations prove that for $b\in [-1/2,2]:$  
	\begin{itemize}
		\item $F^3(Q_1)\subset
		Q_1$ and $F^5(Q_1)=p.$
		\item $F^3(Q_2)\subset Q_1$ and
		$F^5(Q_2)=p.$
		\item $F^2(Q_3)\subset Q_1$ and $F^4(Q_3)=p.$

		\item $F(Q_4)\subset Q_1\cup Q_4.$ If $(x,y)\in Q_4$  and $F(x,y)\in Q_1$ then $F^3(x,y)=p.$ If $(x,y)\in Q_4$  and
		$F(x,y)\in Q_4$ then $F^4(x,y)=p.$
		\end{itemize}
	
	From these facts we obtain statement $(i)$ for $b\in [-1/2,2].$ For $b<-1/2$, again some computations give
	\begin{itemize}\item $F^4(Q_1)\subset
		Q_1$ and $F^6(Q_1)=p.$
		\item $F^4(Q_3)\subset Q_1$ and $F^6(Q_3)=p.$
		\item $F(Q_2)\subset Q_3\cup Q_4,\,\,F^4(Q_2)\subset Q_1$ and	$F^6(Q_2)=p.$
		\item $F(Q_4)\subset Q_1\cup Q_4.$ If $(x,y)\in Q_4$  and $F(x,y)\in Q_1$ then $F^3(x,y)=p.$ If $(x,y)\in Q_4$  and
		$F(x,y)\in Q_4$ then $F^4(x,y)\in Q_1$ and
		$F^6(x,y)=p.$
	\end{itemize}
This ends the proof of statement $(i).$
	
\bigskip	
	
	$(ii)$ When $b=2$ the fixed point is $(0,1)$ and when $b>2$ this fixed
	point bifurcates and the $3$-periodic orbits $\mathcal{P},$ $\mathcal{Q}$ and the fixed point $q$ appear. Again routine computations when $b>2$  allow to prove the facts described in points $(a), (b)$ and $(c)$ below. 
	
	 $(a)$ It holds that
	 $F^3(Q_1)\subset Q_1$ and
		$F^3(Q_1)$ is the polygonal that joints the point $(b-2,1)$
		with $(0,b-1),$ the point $(0,b-1)$ with $(b+2,2b+1)$ and ends with
		the horizontal line $(x,2b+1), x\ge b+2.$ We denote by $L$ this
		polygonal and by $A_1,A_2,B,C$ and $D$ the partition described in Figure~\ref{dinL} (a).  It holds that
		$F^3(A_1)=(b-2,1),\,\,F^3(A_2)=A_1\cup A_2\cup B,$ 
		$F^3(B)=F^3(C)=(b-2,2b-3)$ and $F^3(D)=A_1\cup A_2\cup B$, see Figure \ref{dinL} $(b).$
		
	\begin{figure}[H]
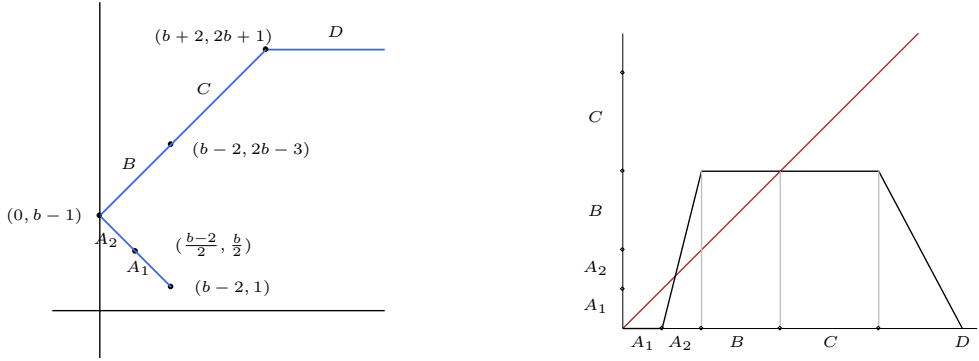


			\qquad\qquad	\begin{subfigure}{0.50\textwidth}
				\begin{lpic}[l(2mm),r(2mm),t(2mm),b(2mm)]{Pol(0.25)}
					\lbl[t]{105,45; {\tiny $(b-2,1)$}}
					\lbl[t]{95,70; {\tiny $(\frac{b-2}{2},\frac {b}{2})$}}
					\lbl[t]{55,55; {\tiny ${A_1}$}}
					\lbl[t]{38,70; {\tiny ${A_2}$}}
					\lbl[b]{5,75; {\tiny $(0,b-1)$}}
					\lbl[t]{50,110; {\tiny ${B}$}}
					\lbl[b]{115,110; {\tiny $(b-2,2b-3)$}}
						\lbl[t]{90,150; {\tiny ${C}$}}
					\lbl[b]{95,170; {\tiny $(b+2,2b+1)$}}
					\lbl[t]{160,180; {\tiny ${D}$}}
				\end{lpic} 
			\end{subfigure} 
			\begin{subfigure}{0.50\textwidth}
				\begin{lpic}[l(2mm),r(2mm),t(2mm),b(2mm)]{GF3(0.25)}
					\lbl[t]{15,15; {\tiny ${A_1}$}}
					\lbl[t]{35,15; {\tiny ${A_2}$}}
					\lbl[t]{65,15; {\tiny ${B}$}}
					\lbl[t]{115,15; {\tiny ${C}$}}
					\lbl[t]{185,15; {\tiny ${D}$}}
						\lbl[t]{-10,35; {\tiny ${A_1}$}}
					\lbl[t]{-10,55; {\tiny ${A_2}$}}
						\lbl[t]{-10,85; {\tiny ${B}$}}
					\lbl[t]{-10,135; {\tiny ${C}$}}
				\end{lpic}
			\end{subfigure}

			\caption{Dynamics of $F^3$ on $Q_1$ when $b>2.$ $(a)$ The polygonal $L=F^3(Q_1)$, left; $(b)$ The graphic of $F^3$ restricted to $L$, right.}\label{dinL}			
		\end{figure}
			
	From these facts, it follows that the map $F^3$ restricted to
	$L$ has two attracting fixed points $(b-2,1),(b-2,2b-3),$ a
	repelling fixed point $r=({(b-2)}/{3},{(2b-1)}/{3})$ and 
	for each $(x,y)\in L\setminus \{r\}$ there exists $n$ that depends
	on $(x,y)$ such that $F^{3n}(x,y)\in \mathcal{P}.$ Observe that the two fixed attracting fixed points correspond to the points of $\mathcal{P}$ which belong to $Q_1$ and that the repelling one is the point of $\mathcal{Q}$ which belongs to $Q_1.$	Thus we get that
	for each $(x,y)\in Q_1$ there exists $n,$ that depends on $(x,y),$ such
	that $F^{3n}(x,y)\in \mathcal{P}\cup \mathcal{Q}.$ So statement~$(ii)$ holds for every
	$(x,y)\in Q_1.$

		$(b)$  If $x\in Q_3\cup Q_4$ then some computations give: either $F(x)\in Q_1$ or $F^2(x)\in Q_1.$ Hence the result $(ii)$ follows from the previous case $(a).$
		
		$(c)$  Lastly, the affine map restricted to $Q_2$ has complex
		eigenvalues with modulus $\sqrt {2}.$ Hence, if $(x,y)\in Q_2\setminus \{q\},$ from $(3)$ it follows
	that there exists $m\in \N$ such that $F^{m}(x,y)\in Q_1\cup
	Q_3\cup Q_4.$ This ends the proof of the proposition.
	\end{proof}
	
	\subsection{The case $a=0.$ Proof of  Theorem \ref{t:teoA}}
	As we noted in Section 1, when $a=0$ it is enough to consider either $b=1$, $b=-1$ or $b=0.$ We begin by considering $b=1.$ The result that we get is the following:
	\begin{propo}\label{azerobu}
		Assume that $a=0$, $b=1.$ Then $F$ has the fixed point $p=\left(-1/5,2/5\right)\in Q_2,$ the two $3$-periodic orbits
		$\mathcal{P}=\{(\pm 1,0),(1,2)\}\}$
		and
		$\mathcal{Q}=\{\left(-1/3,0), (\pm 1/3,2/3) \right)\},$
		and for each $(x,y)\in\R^2\setminus\{p\}$ there exists $n\in\N,$ 	that depends on $(x,y),$ such that $F^n(x,y)\in \mathcal{P}\cup \mathcal{Q}.$
		\end{propo}
	\begin{proof}	We divide the study in three cases:

		$(a)$ It can be seen that $F^3(Q_1)\subset Q_1$ and
			$F^3(Q_1)$ is the polygonal that joints the point $(1,0)$
			with $(0,1),$ the point $(0,1)$ with $(1,2)$ and ends with
			the horizontal line $(x,2), x\ge 2.$ We denote by $K$ this
			polygonal, see Figure \ref{dinK} $(a).$ Let $A_1,A_2,B$ and $C$  be the segments described in Figure \ref{dinK} $(a).$ Then, 
			$F^3(A_1)=(1,0),\,\,F^3(A_2)=A_1\cup A_2\cup B,$ 
			$F^3(B)=(1,2)$ and $F^3(C)=A_1\cup A_2\cup B,$ see Figure \ref{dinK} $(b).$ 
			\begin{figure}[H]
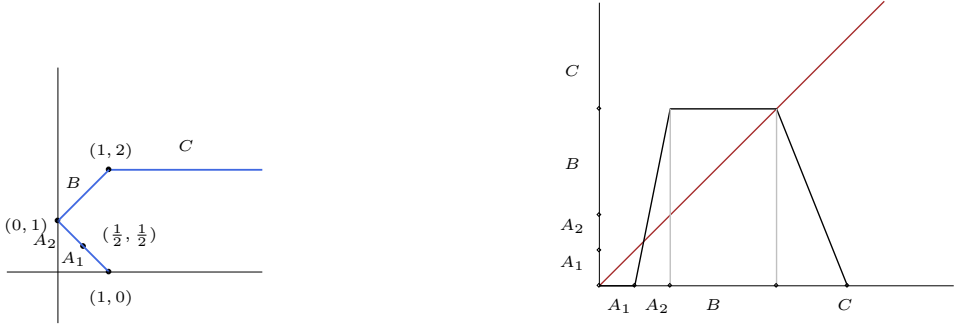

				\centering
				\begin{subfigure}{0.55\textwidth}
					
					\begin{lpic}[l(2mm),r(2mm),t(2mm),b(2mm)]{Pol2(0.25)}
						
						\lbl[t]{85,20; {\tiny $(1,0)$}}
						\lbl[t]{95,55; {\tiny $(\frac{1}{2},\frac {1}{2})$}}
						\lbl[t]{65,40; {\tiny $  {A_1}$}}
						\lbl[t]{50,50; {\tiny $  {A_2}$}}
						\lbl[b]{40,50; {\tiny $(0,1)$}}
						\lbl[t]{65,80; {\tiny $  {B}$}}
						\lbl[b]{85,90; {\tiny $(1,2)$}}
						\lbl[t]{125,100; {\tiny $  {C}$}}
					\end{lpic}
				\end{subfigure} 
				\begin{subfigure}{0.40\textwidth}
					
					\begin{lpic}[l(2mm),r(2mm),t(2mm),b(2mm)]{GF2(0.25)}
						\lbl[t]{15,15; {\tiny $  {A_1}$}}
						\lbl[t]{35,15; {\tiny $  {A_2}$}}
						\lbl[t]{65,15; {\tiny $  {B}$}}
						\lbl[t]{135,15; {\tiny $  {C}$}}
						\lbl[t]{-10,38; {\tiny $  {A_1}$}}
						\lbl[t]{-10,58; {\tiny $  {A_2}$}}
						\lbl[t]{-10,90; {\tiny $  {B}$}}
						\lbl[t]{-10,140; {\tiny $  {C}$}}
					\end{lpic}
				\end{subfigure}
				
				\caption{Dynamics of $F^3$ on $Q_1.$ $(a)$ The polygonal $K=F^3(Q_1)$, left. $(b)$ Graphic of $F^3$ restricted to $K$, right.}\label{dinK}
				\end{figure}
			
Now the result follows by using the same arguments as the ones used in the proof of  $(ii)$ of Proposition \ref{aiguala1}.

		$(b)$  If $(x,y)\in Q_3\cup Q_4$ then either $F(x,y)\in Q_1$ or $F^2(x,y)\in Q_1$, hence the result follows from~$(a).$
			
		$(c)$  The affine map restricted to $Q_2$ has complex
			eigenvalues with modulus $\sqrt {2}.$ Again, the result follows by the same argument given in the proof of case~$(c)$ of Proposition~\ref{aiguala1}.
	\end{proof}

\begin{propo}\label{azerobmenysu}
	Assume that	 $a=0$ and $b=-1.$ Then $(1,0)\in Q_1$ is the fixed point of $F,$ $F^4(\R^2)\subset Q_1$ and for all $(x,y)\in\R^2,$ $F^6(x,y)=(1,0).$
\end{propo}
	\begin{proof}
		Following the orbits of the points in each one of the quadrants $Q_i$ separately, we see that $F^4(Q_i)\subset Q_1$ and $F^6(Q_i)=\{(1,0)\}$ for $i=1,2,3,4.$ 
	\end{proof}

	\begin{propo}\label{azerobzero}
		Assume that	 $a=0$ and $b=0.$ Then $(0,0)$ is the fixed point of $F$ and $F^5(\R^2)=(0,0).$ More precisely:
		\begin{itemize}
			\item [(i)] If $(x,y)\in (Q_1\cup Q_2)\setminus \{y=0\}$ then $F^5(x,y)=(0,0).$
				\item [(ii)] If $(x,y)\in (Q_3\cup Q_4)\setminus \{y=0\}$ then $F^4(x,y)=(0,0).$
					\item [(iii)] For all $x>0, F^2(x,0)=(0,0).$
					\item [(iv)] For all $x>0, F^4(-x,0)=(0,0).$
			
			\end{itemize}
			\end{propo}
		\begin{proof}
			Following the orbits of the points in each one of the quadrants $Q_i$ separately, we easily check that $(i)$ and $(ii)$ are satisfied. To prove $(iii)$  simply notice that  for all $x>0,$ $(x,0)\rightarrow (x,x) \rightarrow (0,0).$ Similarly, to prove  $(iv)$ observe that it holds that $(-x,0)\rightarrow (x,-x) \rightarrow (2x,0)\rightarrow (2x,2x)\rightarrow (0,0).$
			\end{proof}

\begin{proof}[Proof of Theorem A]
It follows from Propositions \ref{aiguala1} to \ref{azerobzero}.
\end{proof}

\section{The case $a<0$ (I). Proofs of  Theorems \ref{t:teoB} and \ref{t:teoC}}\label{sec:4}

By using once more \eqref{conj}, when $a<0$ we know that it is not restrictive to assume in all the section that $a=-1.$

\subsection{Proof of Theorem \ref{t:teoB}}

 By using the expressions of $F_i$ for $i=1,2,3,4$ given in \eqref{e:Fis} it can be easily seen that $F_1(Q_1)$ reduces to the straight line $y=x+b+1$ while $F_3(Q_3)$ reduces to $y=-x+b-1$ for $x\geq -1.$ See Figure \ref{f:imatgesQ13}.

	\begin{figure}[H]
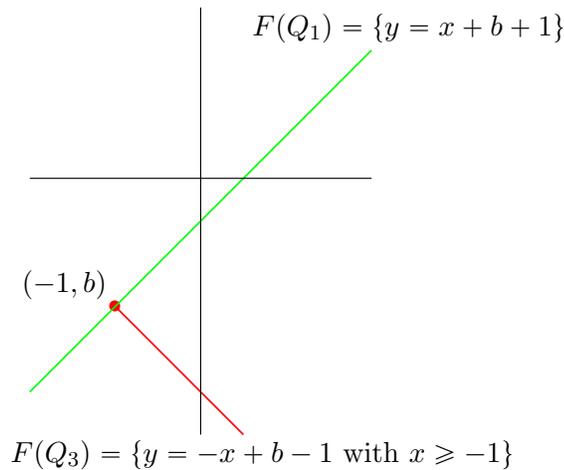

		\centering
		\begin{lpic}[l(2mm),r(2mm),t(2mm),b(2mm)]{bmespetitque-1(0.30),draft}
			\lbl[r]{56,70; $(-1,b)$}
			\lbl[t]{190,190; $F(Q_1)=\{y=x+b+1\}$}
			\lbl[b]{125,-12; $F(Q_3)=\{y=-x+b-1 \mbox{ with } x\geq -1\}$}

		\end{lpic}
		\caption{Images of $F(Q_1)$ and $F(Q_3)$ (in the figure $b<0$).}\label{f:imatgesQ13}
	\end{figure}

 First we will see that except for certain fixed points and a 3-periodic orbit (that only exists for certain values of $b$), the rest of the dynamics is captured by the images of the line and the half line mentioned above. Notice that the fixed point, 
$(-(b+2)/5,(2b-1)/5)$, of the affine map
 $F_2$ has $-1\pm \rm{i}$ as eigenvalues of its linear part  and hence it is an unstable focus. This implies that the positive orbit by $F$ of any point in $Q_2$ different from the fixed point of $F_2$ can not be entirely contained in $Q_2.$ 
 Something similar happens in $Q_4$: the fixed point of $F_4$, $(-b,-1)$, has associated eigenvalues $1\pm \rm{i}$, hence the positive orbit by $F$ of any point in $Q_4$ different of this fixed point  can not be entirely contained in $Q_4.$

\begin{propo}\label{primera reduccioamenys1} For $a=-1$ and for all $b\in\R,$ the orbit of every point in $Q_2\cup Q_4$ meets $Q_1\cup Q_3$ except: the fixed point of $F$ in $Q_2$ (when $b>1/2$);  the fixed points of $F$ in $Q_4,$ (when $b<0$); and a three periodic orbit located at $Q_2\cup Q_4$ (when $3/4<b<2$).

\end{propo}
\begin{proof}
 We are going to characterize the points in $Q_4$ such that its orbit never meets $Q_1\cup Q_3.$ 

Since the points in $Q_4$ can not be always in $Q_4$ (see the argument above), it is enough to prove the result for the points in $Q_4$ such that its image by $F=F_4$ is in $Q_2.$ Let us denote this set by
$
K:=\{(x,y)\in Q_4\mbox{ such that } F(x,y)\in Q_2\}.$
Observe that 
\begin{equation}\label{DefiniciodeK} 
K=\{(x,y)\mbox{ such that } x\geq 0,\, y\leq 0;\,x-y-1\leq 0;\, x+y+b\geq 0 \}.
\end{equation}

We will consider 5 different cases according the values of $b.$ 

\smallskip

$(i)$ Assume first that $b\leq -1.$ 
From the inequalities defining $K$ in \eqref{DefiniciodeK} we get
$ -(b+1)/2\leq b\leq 0$. So if $b<-1$ we get a contradiction and $K$ is empty.  If $b=-1$, then $K$ reduces to the point $(1,0)$ whose second iterate is in $Q_3$. In both cases, since the points in $Q_4$ can not be always in $Q_4,$ their orbits have to reach $Q_1$ or $Q_3,$ as desired.

	\begin{figure}[H]
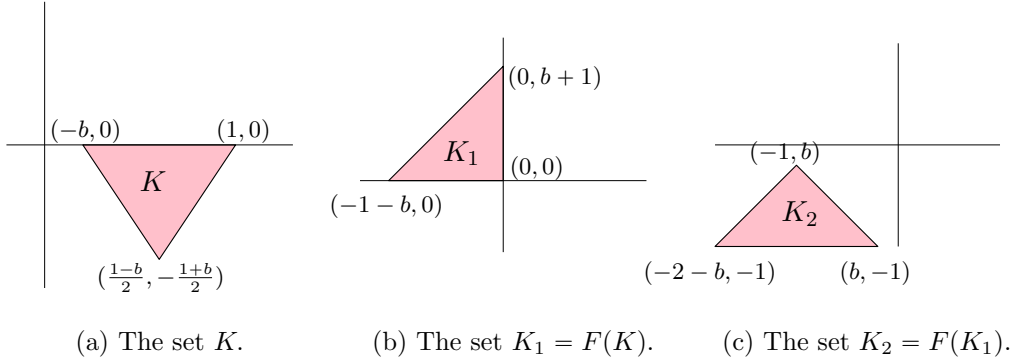

		\centering
		\begin{subfigure}{0.30\textwidth}
			\begin{lpic}[l(2mm),r(2mm),t(2mm),b(2mm)]{seq11(0.20)}
				\lbl[t]{100,80; ${K}$}
				\lbl[t]{55,115; {\footnotesize $(-b,0)$}}\lbl[t]{160,115; {\footnotesize $(1,0)$}}\lbl[b]{105,0; {\footnotesize $(\frac{1-b}{2},-\frac{1+b}{2})$}}
			\end{lpic}
				\caption{The set $K.$}
		\end{subfigure} 
	\begin{subfigure}{0.30\textwidth}
				\begin{lpic}[l(2mm),r(2mm),t(2mm),b(2mm)]{seq12(0.20)}
				\lbl[t]{70,100; {${K_1}$}}
				\lbl[t]{20,65; {\footnotesize $(-1-b,0)$}}
				\lbl[t]{120,90; {\footnotesize $(0,0)$}}
				\lbl[t]{132,150; {\footnotesize $(0,b+1)$}}
			\end{lpic}
			\caption{The set $K_1=F(K).$}
		\end{subfigure}
\begin{subfigure}{0.30\textwidth}
				\begin{lpic}[l(2mm),r(2mm),t(2mm),b(2mm)]{seq13(0.20)}
			\lbl[t]{60,60; {${K_2}$}}
				\lbl[t]{50,100; {\footnotesize $(-1,b)$}}
			\lbl[t]{110,20; {\footnotesize $(b,-1)$}}
			\lbl[t]{0,20; {\footnotesize $(-2-b,-1)$}}
		\end{lpic}
		\caption{The set $K_2=F(K_1).$}
	\end{subfigure}       
		\caption{The evolution of $K$ under the action of $F$ when $-1<b\leq 0.$}\label{f:evolK}
		\end{figure}
	
\smallskip

$(ii)$  Assume next that $-1<b\leq 0.$ From the inequalities \eqref{DefiniciodeK} we easily obtain that $K$ is the triangle defined in Figure \ref{f:evolK}, where also are shown its two first iterates $K_1=F(K)$ and $K_2=F(K_1)$, obtaining that, after two iterates, every point in $K$ arrives to $Q_3,$ as we wanted to prove.

\begin{figure}[H]
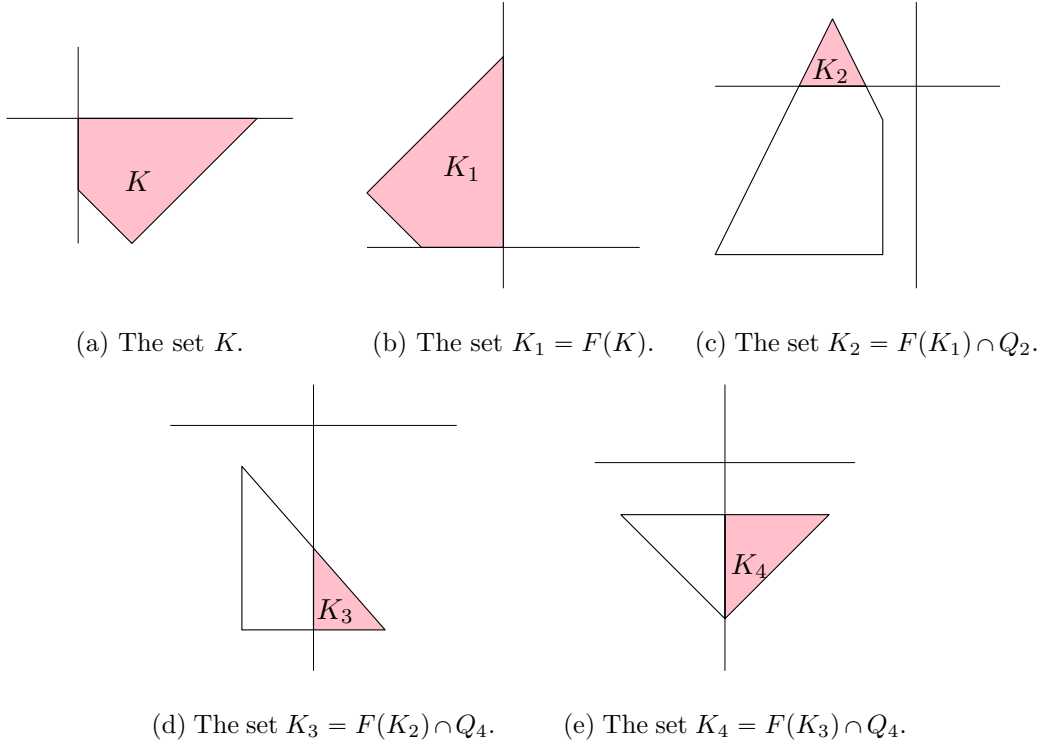

	\centering
	\begin{subfigure}{0.30\textwidth}
		
		\begin{lpic}[l(2mm),r(2mm),t(2mm),b(2mm)]{seq21(0.20)}
			\lbl[t]{90,80; {${K}$}}
		\end{lpic}
		\caption{The set $K.$}
	\end{subfigure} 
	\begin{subfigure}{0.30\textwidth}
		
		\begin{lpic}[l(2mm),r(2mm),t(2mm),b(2mm)]{seq22(0.20)}
			\lbl[t]{70,90; {${K_1}$}}
		\end{lpic}
		\caption{The set $K_1=F(K).$}
	\end{subfigure}
	\begin{subfigure}{0.30\textwidth}
		
		\begin{lpic}[l(2mm),r(2mm),t(2mm),b(2mm)]{seq23(0.20)}
			\lbl[t]{80,155; {${K_2}$}}
		\end{lpic}
		\caption{The set $K_2=F(K_1)\cap Q_2.$}
	\end{subfigure}       
	\newline
	\begin{subfigure}{0.30\textwidth}
		
		\begin{lpic}[l(2mm),r(2mm),t(2mm),b(2mm)]{seq24(0.20)}
			\lbl[t]{112,50; {${K_3}$}}
		\end{lpic}
		\caption{The set $K_3=F(K_2)\cap Q_4.$}
	\end{subfigure}
	\qquad
	\begin{subfigure}{0.30\textwidth}
		\begin{lpic}[l(2mm),r(2mm),t(2mm),b(2mm)]{seq25(0.20)}
			\lbl[t]{114,80; {${K_4}$}}
		\end{lpic}
		\caption{The set $K_4=F(K_3)\cap Q_4.$}
	\end{subfigure}
	\caption{The evolution of $K$ under the action of $F$ when $0<b\leq {1}/{2}.$ }\label{f:P13iii}
	
\end{figure}
	
\smallskip	

$(iii)$  Assume that $0<b<{1}/{2}$. The quadrilateral $K$ is the quadrilateral defined by its vertices
$
	K:=\left\langle (0,0), (0,-b), (1,0), \left({(1-b)}/{2},-{(b+1)}/{2}\right)\right\rangle ,
$
	 see Figure \ref{f:P13iii}. Then $K_1=F_4(K)\subset Q_2$ is the quadrilateral 
$K_1=\left\langle (0,0), (0,b+1), (-1,b), (b-1,0)\right\rangle.$ Since $F_2(K_1)\subset Q_2\cup Q_3$, we define $K_2=F_2(K_1)\cap Q_2$. It is the triangle 
$
K_2=\left\langle (-b-1,0), (-1,b), (b-1,0)\right\rangle.
$
Analogously, we compute the sets $K_3=F(K_2)\cap Q_4$ and $K_4=F(K_3)\cap Q_4$,	which are the triangles 
$
K_{3}=\left\langle (0,b-1), (0,-1), (b,-1)\right\rangle$ and
$K_{4}=\left\langle (0,2b-1), (0,b-1), (b,2b-1)\right\rangle.$
 The image of $K_4$ is a triangle, whose position depends on whether  $b\le{1}/{4}$ or $b>{1}/{4}$, see Figure~\ref{f:P13iiib}. In the first case, it is easy to see that $F_4(K_4)\subset Q_3$ and the propositions follows.

	In the second case, if ${1}/{4}<b\leq {1}/{3}$ we consider the triangle
\begin{equation}\label{e:K5nou}
K_5=F_4(K_4)\cap Q_2=\left\langle (-5b+1,0), (-b,0), (-b,4b-1)\right\rangle .
\end{equation}
If ${1}/{3}<b\leq {1}/{2}$ we consider the quadrilateral (that collapses to a triangle for $b=1/2$):
$$
K_5=F_4(K_4)\cap Q_2=\left\langle (b - 1, 0), (-b, 0), (-b, 4b - 1), (-2b, 3b-1)\right\rangle .
$$

	\begin{figure}[H]
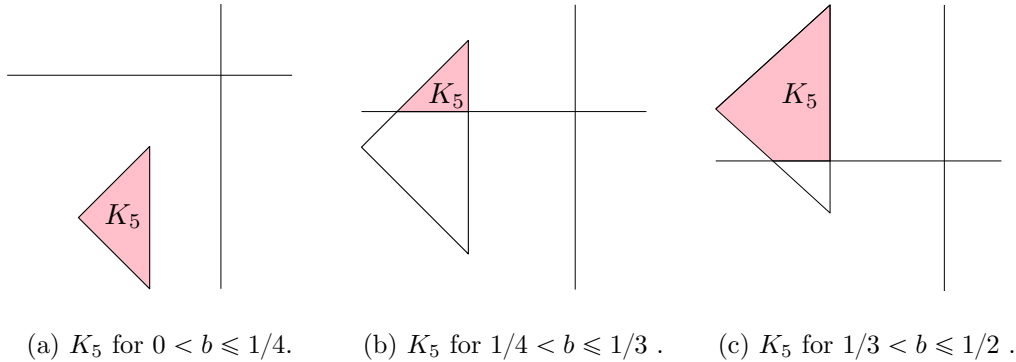

		\centering
		\begin{subfigure}{0.30\textwidth}
			
			\begin{lpic}[l(2mm),r(2mm),t(2mm),b(2mm)]{seq26(0.20)}
				\lbl[t]{80,60; {${K_5}$}}
			\end{lpic}
			\caption{$K_5$ for $0<b\leq {1}/{4}.$}
		\end{subfigure} 
		\begin{subfigure}{0.30\textwidth}
			
			\begin{lpic}[l(2mm),r(2mm),t(2mm),b(2mm)]{seq27(0.20)}
				\lbl[t]{60,140; {${K_5}$}}
			\end{lpic}
			\caption{$K_5$ for ${1}/{4}<b\leq {1}/{3}$ .}
		\end{subfigure}
	\begin{subfigure}{0.30\textwidth}
			
			\begin{lpic}[l(2mm),r(2mm),t(2mm),b(2mm)]{seq27bis(0.20)}
				\lbl[t]{60,140; {${K_5}$}}
			\end{lpic}
			\caption{$K_5$ for ${1}/{3}<b\leq {1}/{2}$ .}
		\end{subfigure}
		\caption{Relative position of $F(K_4)$ depending on $b.$}\label{f:P13iiib}
			\end{figure}

In these last two cases $K_5\subset K_1\subset Q_2.$ Hence, by definition of $K$,  this means that the points in $K_5$ come from points in $K_4$ that are also in $K$. 	This fact implies that if the orbit of a point in $K,$ during its five first iterates, never meets $Q_1\cup Q_3$ then $F^4(p)\in K_4$ and $F^5(p)\in K_5\subset K_1\subset Q_2.$ Therefore, 
its itinerary must be $(4224)^{\infty}$ (This notation means  that the point must repeat indefinitely a trajectory in $Q_4\to Q_2\to Q_2\to Q_4$). Consider
	$$  \widetilde{K}=\{p\in K: F^i(p)\notin Q_1\cup Q_3 \mbox{ for all } i\in \N\}.$$
	Then $  \widetilde{K}$ is a bounded set and the map $G:=F_4\circ F_2\circ  F_2\circ F_4$ leaves invariant the set $\widetilde{K}.$ But this  is a contradiction with the fact that $G$ is the expansive map
	$G(x,y)=(4x+b-2,4y+2b+1).$
	It only remains to avoid the possibility that the fixed point of G gives rise to a 4-periodic orbit. The fixed point of $G$ is $r=\left({(2-b)}/{3},{-(2b+1)}/{3}\right)\in Q_4,$ but $F_4\left(r\right)=\left({b}/{3},{1}/{3}\right)\in Q_1.$ Hence it does not follow the prescribed itinerary. In conclusion,  the set
	$  \widetilde{K}$ is empty.

\smallskip

		\begin{figure}[H]
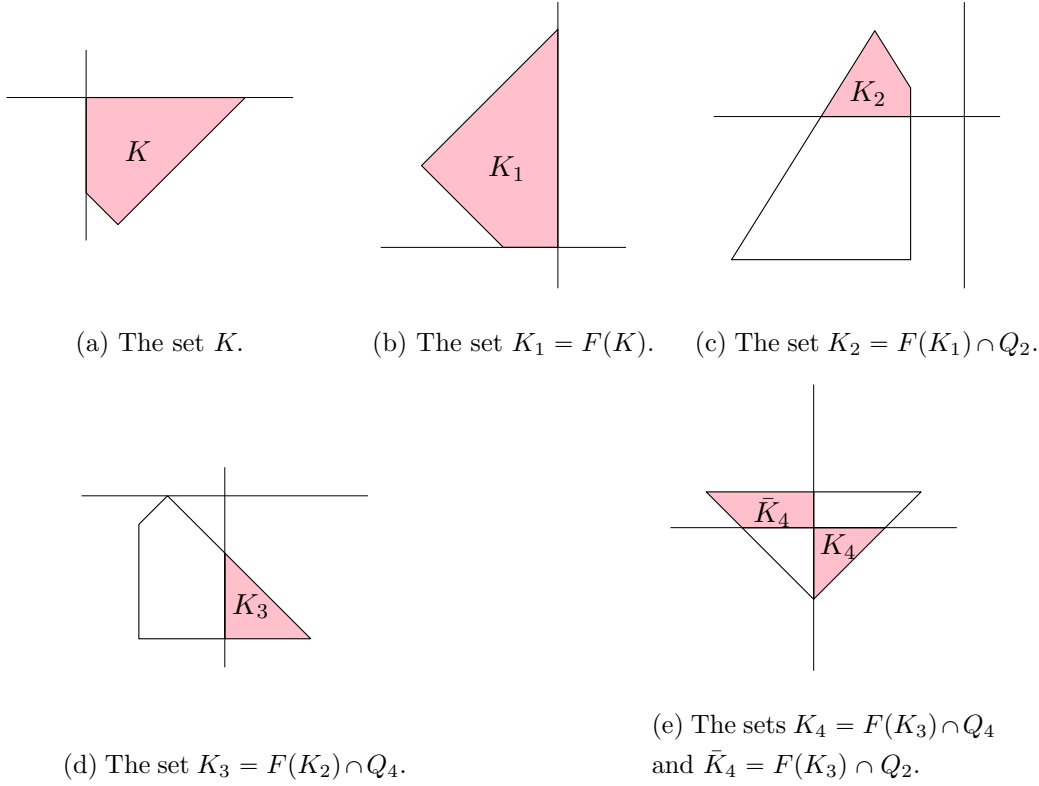

	\centering
	\begin{subfigure}{0.30\textwidth}
		
		\begin{lpic}[l(2mm),r(2mm),t(2mm),b(2mm)]{seq31(0.20)}
			\lbl[t]{90,100; {${K}$}}
		\end{lpic}
		\caption{The set $K.$}
	\end{subfigure} 
	\begin{subfigure}{0.30\textwidth}
		
		\begin{lpic}[l(2mm),r(2mm),t(2mm),b(2mm)]{seq32(0.20)}
			\lbl[t]{100,90; {${K_1}$}}
		\end{lpic}
		\caption{The set $K_1=F(K).$}
	\end{subfigure}
	\begin{subfigure}{0.30\textwidth}
		
		\begin{lpic}[l(2mm),r(2mm),t(2mm),b(2mm)]{seq33(0.20)}
			\lbl[t]{105,140; {${K_2}$}}
		\end{lpic}
		\caption{The set $K_2=F(K_1)\cap Q_2.$}
	\end{subfigure}       
	\newline
	\begin{subfigure}{0.30\textwidth}
		
		\begin{lpic}[l(2mm),r(2mm),t(2mm),b(2mm)]{seq34(0.20)}
			\lbl[t]{115,80; {${K_3}$}}
		\end{lpic}
		\caption{The set $K_3=F(K_2)\cap Q_4.$}
	\end{subfigure}
	\hspace{3cm}
	\begin{subfigure}{0.30\textwidth}
		
		\begin{lpic}[l(2mm),r(2mm),t(2mm),b(2mm)]{seq35(0.20)}
			\lbl[t]{114,93; {${K_4}$}}
			\lbl[t]{70,117; {${\bar{K}_4}$}}
		\end{lpic}
		\caption{The sets $K_4=F(K_3)\cap Q_4$ and $\bar{K}_4=F(K_3)\cap Q_2.$}
	\end{subfigure}
	\caption{The evolution of $K$ under the action of $F$ when ${1}/{2}<b<1.$}\label{f:P13iv}
	
\end{figure}

$(iv)$ Consider  now ${1}/{2}<b\leq 1.$ In this case, $K$ and $K_1$  are the quadrilaterals defined by the same vertices than in the preceding case $(iii).$	We will follow the same procedure as before, see Figure \ref{f:P13iv}. In this case, $K_2=F_2(K_1)\cap Q_2$, $K_3=F_2(K_2)\cap Q_4$, and $K_4=\bar{K}_4\cup K_4,$ where
$\bar{K}_4=F_4(K_3)\cap Q_2$ and $K_4=F_4(K_3)\cap Q_4$. 
In particular,
\begin{align*}
K_2:=&F_2(K_1)\cap Q_2=\left\langle (-b-1,0), (-b,0), (-b,2b-1), (-1,b)\right\rangle ,\\
K_3:=&F_2(K_2)\cap Q_4=\left\langle (0,b-1), (0,-1), (b,-1)\right\rangle ,\\
K_4:=&F_4(K_3)\cap Q_4=\left\langle (0,0), (0,b-1), (-b+1,0)\right\rangle ,\\
\bar{K}_4:=&F_4(K_3)\cap Q_2=\left\langle (0,0), (0,2b-1), (-b,2b-1), (b-1,0)\right\rangle .\\
\end{align*}

Let us follow first $K_4.$ We get that
$	K_5:=F_4(K_4)=\left\langle (-1,b), (-b,2b-1), (-b,1)\right\rangle$ and 
	$K_6:=F_2(K_5)=\left\langle (b-2,-1), (-b,-2b-1), (-b,-1) \right\rangle.$	
Since $K_6$ lies in the third quadrant, we are done in this case. See Figure \ref{f:P13ivb}.
	
	\begin{figure}[H]
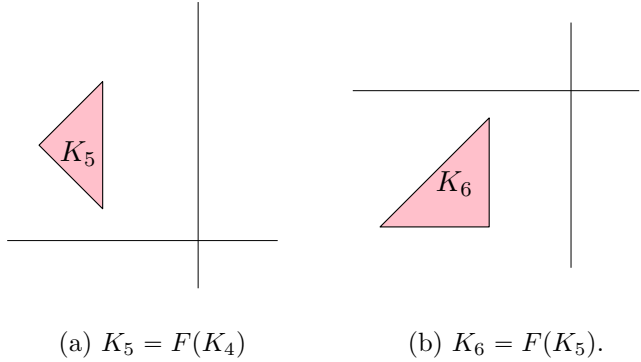

		\centering
		\begin{subfigure}{0.30\textwidth}
			
			\begin{lpic}[l(2mm),r(2mm),t(2mm),b(2mm)]{seq36(0.20)}
				\lbl[t]{55,100; {${K_5}$}}
			\end{lpic}
			\caption{$K_5=F(K_4).$}
		\end{subfigure} 
		\begin{subfigure}{0.30\textwidth}
			
			\begin{lpic}[l(2mm),r(2mm),t(2mm),b(2mm)]{seq37(0.20)}
				\lbl[t]{70,80; {${K_6}$}}
			\end{lpic}
			\caption{$K_6=F(K_5).$}
		\end{subfigure}
		
		\caption{The sets $K_5,K_6$ when ${1}/{2}<b\leq 1.$}\label{f:P13ivb}
	\end{figure}

	Regarding $\bar{K}_4$ we observe that the points of $K_3$ whose image is in $\bar{K}_4,$ are points in $Q_4$ whose image is in $Q_2$ and hence are points in $K.$ In consequence, if the orbit of a point in $K$ never meets $Q_1\cup Q_3,$ the itinerary of this point must be $(422)^{\infty}.$ The map $F_2\circ F_2\circ F_4$ is the expansive map:
	$F_2\circ F_2\circ F_4(x,y)=(2x+2y+b,-2x+2y+1)$
	and its fixed point is $r:=\left({(2-b)}/{5},-{(2b+1)}/{5}\right)\in Q_4.$ When $3/4\leq b \leq 1$ it gives rise to the $3$-periodic orbit of~$F$:
	\begin{equation}\label{orbita422}
		r\rightarrow \left(\frac{b-2}{5},\frac{2b+1}{5}\right)\rightarrow \left(\frac{-3b-4}{5},\frac{4b-3}{5}\right)=:s \rightarrow r.
	\end{equation}
	But when $1/2<b<3/4,$  $s\in Q_3$, so it does not follow the prescribed itinerary.
	
\smallskip	
	
$(v)$	 Finally, we take $b>1$. We follow the same procedure as above, see Figure \ref{f:P13v}. We get  
$
K:=\left\langle (0,0), (1,0), (0,-1)\right\rangle,$ $
K_1:=F_4(K)\cap Q_2=\left\langle (-1,b), (0,b-1), (0,b+1)\right\rangle$ and 
$K_2:=F_2(K_1)\cap Q_2=\left\langle (-b-1,0), (-b,0), (-b,1)\right\rangle.$	
	
		\begin{figure}[H]
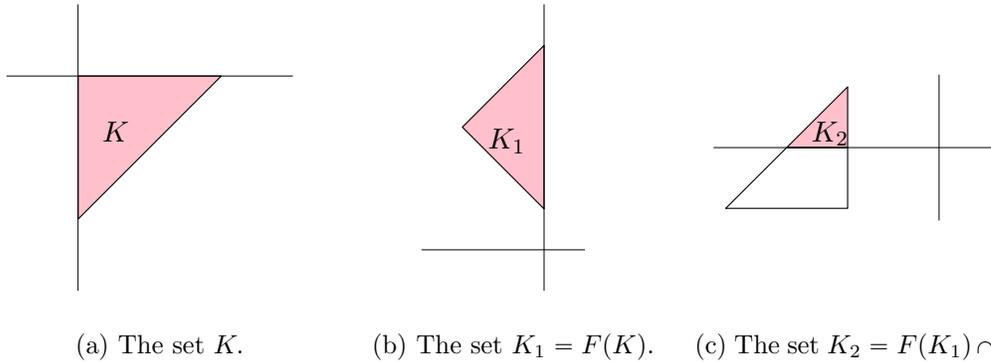

		\centering
		\begin{subfigure}{0.30\textwidth}
			
			\begin{lpic}[l(2mm),r(2mm),t(2mm),b(2mm)]{seq41(0.20)}
				\lbl[t]{75,115; ${K}$}
			\end{lpic}
			\caption{The set $K.$}
		\end{subfigure} 
		\begin{subfigure}{0.30\textwidth}
			
			\begin{lpic}[l(2mm),r(2mm),t(2mm),b(2mm)]{seq42(0.20)}
				\lbl[t]{100,110; {${K_1}$}}
			\end{lpic}
			\caption{The set $K_1=F(K).$}
		\end{subfigure}
		\begin{subfigure}{0.30\textwidth}
			
			\begin{lpic}[l(2mm),r(2mm),t(2mm),b(2mm)]{seq43(0.20)}
				\lbl[t]{80,115; {${K_2}$}}
			\end{lpic}
			\caption{The set $K_2=F(K_1)\cap Q_2.$}
		\end{subfigure}

		\caption{The evolution of $K$ under the action of $F$ when $b>1.$}\label{f:P13v}
	\end{figure}

	Now, we distinguish two cases: $b>2$ and $1<b<2$. When $b>2$, we have
	\begin{align*}
K_3:=&F_2(K_2)=\left\langle (b-1,0), (b-2,-1), (b,-1)\right\rangle \subset Q_4,\\
K_4:=&F_4(K_3)=\left\langle (b-2,2b-1),  (b-2,2b-3), (b,2b-1)\right\rangle \subset Q_1,
\end{align*}	
and, as we can see, the points in $K$ either after two iterates they are in  $Q_3$, or after four they are in $Q_1$, see Figure \ref{f:P13v2}.

	\begin{figure}[H]
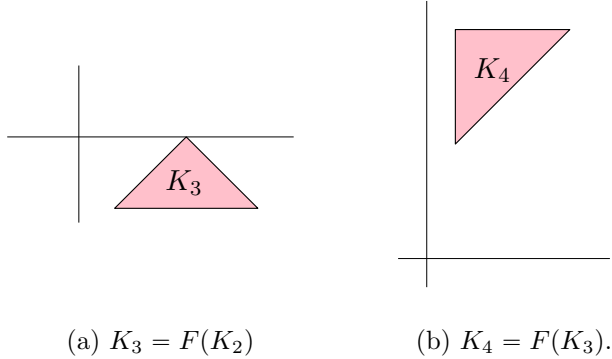

		\centering
		\begin{subfigure}{0.30\textwidth}
			
			\begin{lpic}[l(2mm),r(2mm),t(2mm),b(2mm)]{seq46(0.20)}
				\lbl[t]{120,80; {${K_3}$}}
			\end{lpic}
			\caption{$K_3=F(K_2).$}
		\end{subfigure} 
		\begin{subfigure}{0.30\textwidth}
			
			\begin{lpic}[l(2mm),r(2mm),t(2mm),b(2mm)]{seq47(0.20)}
				\lbl[t]{90,155; {${K_4}$}}
			\end{lpic}
			\caption{$K_4=F(K_3).$}
		\end{subfigure}
		
		\caption{The sets $K_3$ and $K_4$ when $b>2.$}\label{f:P13v2}
	\end{figure}

When $1<b<2$, we obtain:
	\begin{align*}
K_3:=&F_2(K_2)\cap Q_4=\left\langle (0,-b+1), (0,-1), (b,-1), (b-1,0)\right\rangle,\\
K_4:=&F_4(K_3)\cap Q_2=\left\langle  (0,1), (b-2,1), (b-2,2b-1), (0,b-1) \right\rangle,
\end{align*}	
see Figure~\ref{f:P13v3}.
	\begin{figure}[H]
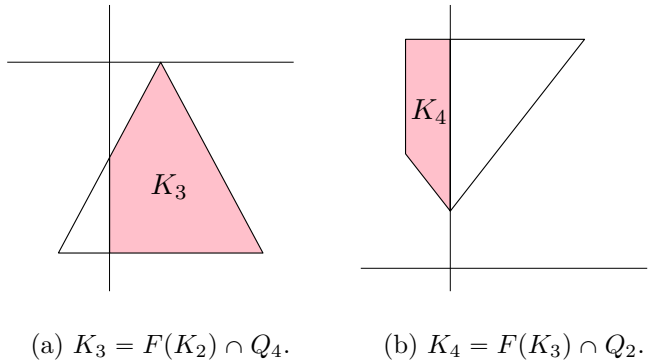

		\centering
		\begin{subfigure}{0.30\textwidth}
			
			\begin{lpic}[l(2mm),r(2mm),t(2mm),b(2mm)]{seq44(0.20)}
				\lbl[t]{110,80; {${K_3}$}}
			\end{lpic}
			\caption{$K_3=F(K_2)\cap Q_4.$}
		\end{subfigure} 
		\begin{subfigure}{0.30\textwidth}
			
			\begin{lpic}[l(2mm),r(2mm),t(2mm),b(2mm)]{seq45(0.20)}
				\lbl[t]{48,130; {${K_4}$}}
			\end{lpic}
			\caption{$K_4=F(K_3)\cap Q_2.$}
		\end{subfigure}
		
		\caption{The sets $K_3$ and $K_4$ when $1<b<2.$}\label{f:P13v3}
	\end{figure}
We observe that the points of $K_3$ whose images belong to $K_4,$ are points of $K.$ 	As in the above case, if some point of $K$ never meets $Q_1\cup Q_3,$ its itinerary must be $(422)^{\infty}$ and it must be a fixed point of $F_2\circ F_2\circ F_4$ which gives the $3$-periodic orbit (\ref{orbita422}).

In short, we have seen that the orbits of the points in $Q_4$ always visit 
$Q_1\cup Q_3,$ except the following particular cases: when the fixed point of $F$, $ (-b,-1),$ belongs to $Q_4$ (this happens only when $b\leq 0$) and when there is a periodic orbit living in $Q_2\cup Q_4$ (which only happens when $3/4\leq b\leq 2$). In this last case, the periodic orbit is the one given in~\eqref{orbita422}.

 Now consider points in $Q_2.$ For any point in $Q_2$ different from a fixed one, its orbit have to leave $Q_2,$ so its iterates arrive to $Q_1\cup Q_3\cup Q_4$ (in fact, in the proof of Proposition \ref{p:P15nova}, below, we will prove that any point in $Q_2$ leaves $Q_2$ in two iterates). Then, the only points in $Q_2$ which do not meet  $Q_1\cup Q_3$ are the fixed point of $F$ in $Q_2, \left(-{(2+b)}/{5},{(2b-1)}/{5}\right),$ which exists when $b \geq 1/2,$ and the two points of the $3$-periodic orbit (\ref{orbita422}) that are in $Q_2.$
\end{proof}

\begin{propo}\label{p:lemanou}
Set $a=-1$. The following statements hold:
\begin{enumerate}[(a)]
\item For all $b\in\R$ there exists a compact graph $\Gamma$ which is invariant under the map $F$. The graphs for each value of $b$ are given in the Figures \ref{ff:1}--\ref{f:23} of the Appendix. 
\item For all $b\in\R$, $F^{11}(x,y)\in \Gamma$ for all $(x,y)\in Q_1\cup Q_3$.
\end{enumerate}
\end{propo}

\begin{proof}
(a) As we mentioned at the beginning of this section, we recall that $F_1(Q_1)$ reduces to the straight line $y=x+b+1$ and $F_3(Q_3)$ reduces to the half-line $y=-x+b-1$ for $x\geq -1.$ See Figure \ref{f:imatgesQ13}.  To obtain the graphs in Figures \ref{ff:1}--\ref{f:23} of the Appendix, we have followed some iterations of these lines until we have found a graph which is invariant by $F$, by joining some of these images. It is a very tedious job, but not at all difficult: it is simply a matter of iterating a polygonal curve. By reasons of space we prefer do not   reproduce here all the cases but,  as an example, we expose the details in a case with intermediate difficulty.

Set $2/3<b\leq 5/7$. Consider the points that are listed in the caption of Figure \ref{f:A}, and the additional points $M_1=(-1, b)$, $M_2=(-b,-1)$, $M_3=(b, -1)$, $M_4=(b,2b-1)$, $M_5=(-b,1)$, $M_6=(b-2,-1)$ and $M_7=(-b+2,2b-3)$.
We will denote $\Gamma_i=F^i(Q_3)$. The points $M_i$ are some endpoints of these graphs. Then:

\begin{description}
\item[$\Gamma_1$:]   With the above notation, $F(Q_3)$ is the half-line 
$\Gamma_1=\cup_{i=1}^3 I_{1,i}$, where $I_{1,1}=\overline{M_1T_2}$, $I_{1,2}=\overline{T_2X_2}$, and $I_{1,3}=F(Q_3)\cap Q_4$ is the half-line whose border is the point $X_2$, see Figure \ref{f:gl}~$(a).$

\item[$\Gamma_2$:] A computation gives that $F(I_{1,1})=I_{2,1}\cup I_{2,2}$, $F(I_{1,2})=X_3$ and $F(I_{1,3})=I_{2,3}\cup I_{2,4}$, where $I_{2,1}=\overline{M_2T_1}$, $I_{2,2}=\overline{T_1X_3}$, $I_{2,3}=\overline{X_3R_1}$, and $I_{2,4}=F^2(Q_3)\cap Q_1$ is the half-line whose border is the point $R_1$. With this notation $\Gamma_2=\cup_{i=1}^4 I_{2,i}$, see Figure~\ref{f:gl}~$(b).$

\item[$\Gamma_3$:] We have that $F(I_{2,1})=I_{3,1}\cup I_{3,2}$, 
$F(I_{2,2})=I_{3,3}$, $F(I_{2,3})=I_{3,4}\cup I_{3,5}$ and $F(I_{2,4})=I_{3,6}\cup I_{3,7}$, where $I_{3,1}=\overline{M_3X_2}$, $I_{3,2}=\overline{X_2T_2}$, $I_{3,3}=\overline{T_2X_4}$,  $I_{3,4}=\overline{X_4W}$, $I_{3,5}=\overline{WR_2}$, $ I_{3,6}=\overline{R_2S}$ and 
$I_{3,7}=F^3(Q_3)\cap Q_1$ is the half-line whose border is the point $S$. With this notation $\Gamma_3=\cup_{i=1}^7 I_{3,i}$, see Figure~\ref{f:gl}~$(c).$

\item[$\Gamma_4$:] In this case $F(I_{3,1})=I_{4,1}\cup I_{4,2}$, 
$F(I_{3,2})=X_3$, $F(I_{3,3})=I_{4,3}\cup I_{4,4}\cup I_{4,5}$, $F(I_{3,4})=X_5$,
$F(I_{3,5})=I_{4,6}$, and $F(I_{3,6})=I_{4,7}\cup I_{4,8}$, where 
$I_{4,1}=\overline{M_4R_1}$, $I_{4,2}=\overline{R_1X_3}$, 
$I_{4,3}=\overline{X_3T_2}$,  $I_{4,4}=\overline{T_2X_2}$, $I_{4,5}=\overline{X_2X_5}$, $ I_{4,6}=\overline{X_5R_3}$, $ I_{4,7}=\overline{R_3X_1}$ and
$ I_{4,8}=\overline{X_1P_1}$. 
With this notation $\Gamma_4=\cup_{i=1}^8 I_{4,i}$, see Figure \ref{f:gl}~$(d).$
It is important to notice that $\Gamma_4\subset\Gamma$, see Figure~\ref{f:gl4gamma}.

\item[$\Gamma_5$:] In this case $F(I_{4,1})=I_{5,1}$, 
$F(I_{4,2})=I_{5,2}\cup I_{5,3}$, $F(I_{4,3})=I_{5,4}\cup I_{5,5}$, $F(I_{4,4})=X_5$,
$F(I_{4,5})=I_{5,6}$, $F(I_{4,6})=I_{5,7}\cup I_{5,8}\cup I_{5,9}$, 
$F(I_{4,7})=I_{5,10}$ and $F(I_{4,8})=I_{5,11}$
where $I_{5,1}=\overline{M_5R_2}$, $I_{5,2}=\overline{R_2W}$, 
$I_{5,3}=\overline{WX_4}$,  $I_{5,4}=\overline{X_4T_1}$, $I_{5,5}=\overline{T_1X_3}$, $ I_{5,6}=\overline{X_3X_6}$, $ I_{5,7}=\overline{X_6Z_1}$, 
$ I_{5,8}=\overline{Z_1Q}$, $ I_{5,9}=\overline{QR_4}$,
$ I_{5,10}=\overline{R_4X_2}$ and $ I_{5,11}=\overline{X_2P_2}$.
With this notation $\Gamma_5=\cup_{i=1}^{11} I_{5,i}$, see Figure~\ref{f:gl}~$(e).$

\item[$\Gamma_6$:] In this case $F(I_{5,1})=I_{6,1}\cup I_{6,2}$, 
$F(I_{5,2})=I_{6,3}$, $F(I_{5,3})=X_5$, $F(I_{5,4})=I_{6,4}\cup I_{6,5}$,
$F(I_{5,5})=I_{6,6}$, $F(I_{5,6})=I_{6,7}$, 
$F(I_{5,7})=I_{6,8}\cup I_{6,9}$, $F(I_{5,8})=Z_2$,  
$F(I_{5,9})=I_{6,10}$, $F(I_{5,10})=I_{6,11}$ and $F(I_{5,11})=I_{6,12}\cup I_{6,13}$, 
where $I_{6,1}=\overline{M_6X_1}$, $I_{6,2}=\overline{X_1R_3}$, 
$I_{6,3}=\overline{R_3X_5}$,  $I_{6,4}=\overline{X_5X_2}$, 
$I_{6,5}=\overline{X_2T_2}$, $ I_{6,6}=\overline{T_2X_4}$, 
$ I_{6,7}=\overline{X_4X_7}$, 
$ I_{6,8}=\overline{X_7Y_1}$, 
$ I_{6,9}=\overline{Y_1Z_2}$,
$ I_{6,10}=\overline{Z_2R_5}$,
$ I_{6,11}=\overline{R_5X_3}$,
$ I_{6,12}=\overline{X_3R_1}$,
and $ I_{6,13}=\overline{R_1P_3}$.
With this notation $\Gamma_6=\cup_{i=1}^{13} I_{6,i}$, see Figure~\ref{f:gl}~$(f).$

\item[$\Gamma_7$:] Finally, we have $F(I_{6,1})=I_{7,1}$, 
$F(I_{6,2})=I_{7,2}$, $F(I_{6,3})=I_{7,3}\cup I_{7,4}\cup I_{7,5}$, 
$F(I_{6,4})=I_{7,6}$,
$F(I_{6,5})=X_3$, $F(I_{6,6})=I_{7,7}\cup I_{7,8}\cup I_{7,9} \cup I_{7,10}$, 
$F(I_{6,7})=X_5$, $F(I_{6,8})=I_{7,7}\cup I_{7,10}$,  
$F(I_{6,9})=I_{7,11}$, $F(I_{6,10})=I_{7,12}$,
$F(I_{7,11})=I_{7,13}$, $F(I_{6,12})=I_{7,14}\cup I_{7,15}$ and
$F(I_{6,13})=I_{7,16}\cup I_{7,17}$
where 
$I_{7,1}=\overline{M_7X_2}$, $I_{7,2}=\overline{X_2R_4}$, 
$I_{7,3}=\overline{R_4Q}$,  $I_{7,4}=\overline{QZ_1}$, 
$I_{7,5}=\overline{Z_1X_6}$, $ I_{7,6}=\overline{X_6X_3}$, 
$ I_{7,7}=\overline{X_3T_2}$, 
$ I_{7,8}=\overline{T_2Y_2}$, 
$ I_{7,9}=\overline{Y_2X_2}$,
$ I_{7,10}=\overline{X_2X_5}$,
$ I_{7,11}=\overline{Y_2Z_3}$,
$ I_{7,12}=\overline{Z_3R_6}$,
$ I_{7,13}=\overline{R_6X_4}$,
$ I_{7,14}=\overline{X_4W}$,
$ I_{7,15}=\overline{WR_2}$,
$ I_{7,16}=\overline{R_2S}$, and
$ I_{7,17}=\overline{SP_4}$.
With this notation $\Gamma_7=\cup_{i=1}^{17} I_{7,i}$, see Figure~\ref{f:gl}~$(g).$
\end{description}

Then, we define the graph $\Gamma:=\Gamma_5\cup\Gamma_6\cup\Gamma_7$.

 This procedure must be done also for the images of $Q_1$ and for all the cases displayed in Appendix. The cases corresponding to the values of $b$ on the border of each interval are obtained either by the collapse of edges to a point, or because some existing edge crosses one of the coordinate axes entering or disappearing from a quadrant.

\vfill
\newpage  

\begin{figure}[H]
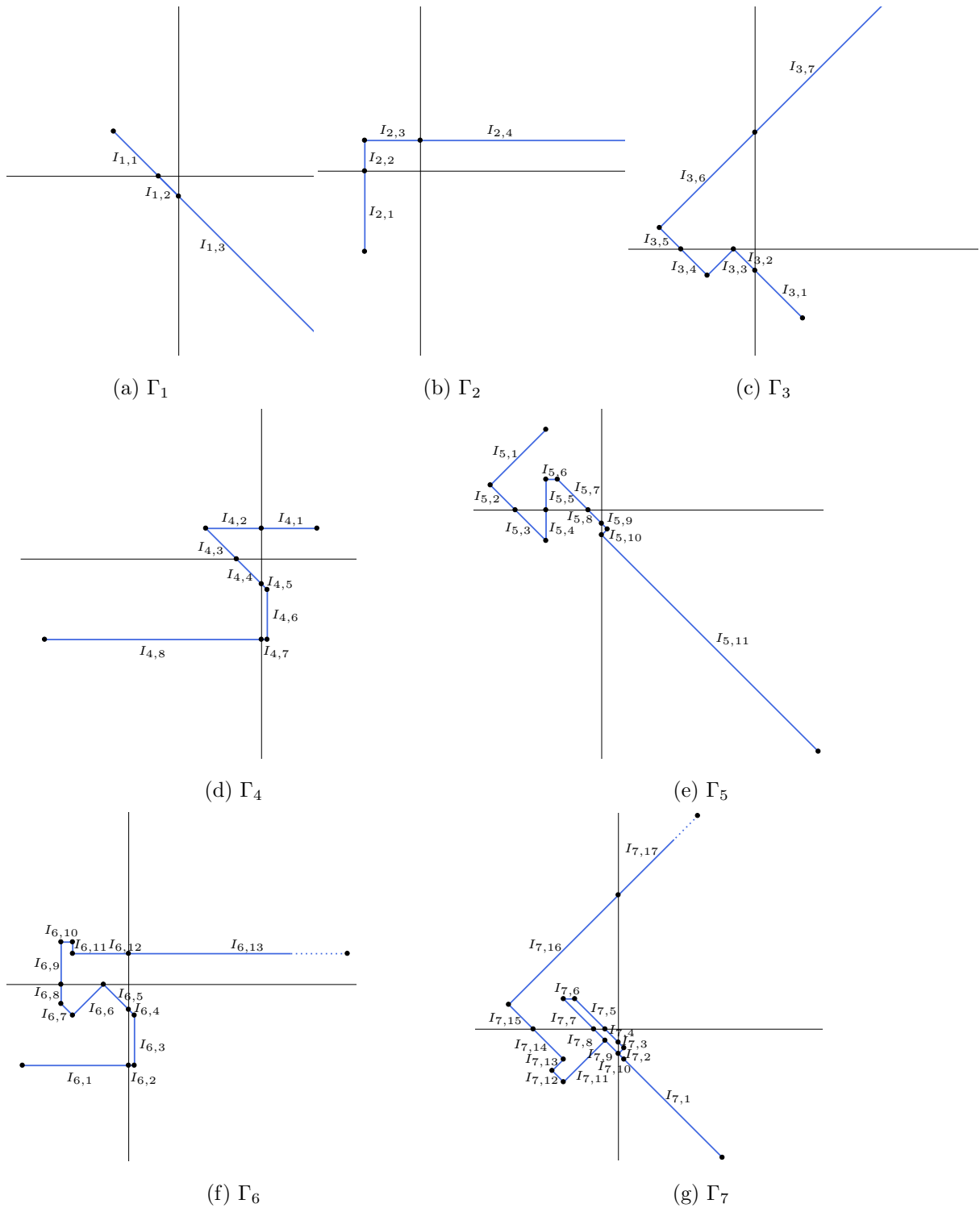

  \begin{subfigure}{0.3\linewidth}
    \begin{lpic}[l(0mm),r(0mm),b(0mm),t(0mm)]{gl1(0.3)}
      \lbl[l]{60,110; {\tiny $I_{1,1}$}}
      \lbl[l]{78,91; {\tiny $I_{1,2}$}}
      \lbl[l]{108,62; {\tiny $I_{1,3}$}}
    \end{lpic}
    \caption{ $\Gamma_1$}
  \end{subfigure}
  \hspace{0.3cm}
  \begin{subfigure}{0.3\linewidth}
    \begin{lpic}[l(0mm),r(0mm),b(0mm),t(0mm)]{gl2(0.3)}
      \lbl[l]{30,80; {\tiny $I_{2,1}$}}
      \lbl[l]{30,110; {\tiny $I_{2,2}$}}
      \lbl[l]{36,125; {\tiny $I_{2,3}$}}
      \lbl[l]{95,125; {\tiny $I_{2,4}$}}
    \end{lpic}
    \caption{$\Gamma_2$}
  \end{subfigure}
 \hspace{0.3cm}
  \begin{subfigure}{0.3\linewidth}
    \begin{lpic}[l(0mm),r(0mm),b(0mm),t(0mm)]{gl3(0.3)}
        \lbl[l]{86,37; {\tiny $I_{3,1}$}}
      \lbl[l]{67,55; {\tiny $I_{3,2}$}}
      \lbl[l]{53,50; {\tiny $I_{3,3}$}}
      \lbl[l]{25,50; {\tiny $I_{3,4}$}} 
      \lbl[l]{10,65; {\tiny $I_{3,5}$}}
      \lbl[l]{30,100; {\tiny $I_{3,6}$}}
      \lbl[l]{90,160; {\tiny $I_{3,7}$}} 
    \end{lpic}
    \caption{$\Gamma_3$}   
  \end{subfigure}
  
  \begin{subfigure}{0.5\linewidth}
    \begin{lpic}[l(0mm),r(0mm),b(0mm),t(0mm)]{gl4(0.3)}
     \lbl[l]{151,133; {\tiny $I_{4,1}$}}
      \lbl[l]{120,133; {\tiny $I_{4,2}$}}
      \lbl[l]{107,117; {\tiny $I_{4,3}$}}
      \lbl[l]{123,102; {\tiny $I_{4,4}$}} 
      \lbl[l]{145,98; {\tiny $I_{4,5}$}}
      \lbl[l]{148,80; {\tiny $I_{4,6}$}}
      \lbl[l]{143,60; {\tiny $I_{4,7}$}}
      \lbl[l]{75,60; {\tiny $I_{4,8}$}}      
    \end{lpic}
    \caption{$\Gamma_4$}
  \end{subfigure}
  \begin{subfigure}{0.5\linewidth}
    \begin{lpic}[l(0mm),r(0mm),b(0mm),t(0mm)]{gl5(0.3)}
      \lbl[l]{12,170; {\tiny $I_{5,1}$}}
      \lbl[l]{2,145; {\tiny $I_{5,2}$}}
      \lbl[l]{19,128; {\tiny $I_{5,3}$}}
      \lbl[l]{43,128; {\tiny $I_{5,4}$}} 
      \lbl[l]{43,145; {\tiny $I_{5,5}$}}
      \lbl[l]{39,160; {\tiny $I_{5,6}$}}
      \lbl[l]{57,150; {\tiny $I_{5,7}$}}
      \lbl[l]{53,134; {\tiny $I_{5,8}$}} 
      \lbl[l]{75,133; {\tiny $I_{5,9}$}}
      \lbl[l]{76,124; {\tiny $I_{5,10}$}}
      \lbl[l]{135,67; {\tiny $I_{5,11}$}}       
    \end{lpic}
    \caption{$\Gamma_5$}
  \end{subfigure}

  \begin{subfigure}{0.5\linewidth}
    \begin{lpic}[l(0mm),r(0mm),b(0mm),t(0mm)]{gl6(0.3)}
     \lbl[l]{35,47; {\tiny $I_{6,1}$}}
      \lbl[l]{70,47; {\tiny $I_{6,2}$}}
      \lbl[l]{75,65; {\tiny $I_{6,3}$}}
      \lbl[l]{72,85; {\tiny $I_{6,4}$}} 
      \lbl[l]{63,93; {\tiny $I_{6,5}$}}
      \lbl[l]{47,85; {\tiny $I_{6,6}$}}
      \lbl[l]{21,82; {\tiny $I_{6,7}$}}
      \lbl[l]{17,94; {\tiny $I_{6,8}$}} 
      \lbl[l]{17,110; {\tiny $I_{6,9}$}}
      \lbl[l]{23,128; {\tiny $I_{6,10}$}}
      \lbl[l]{39,120; {\tiny $I_{6,11}$}} 
      \lbl[l]{58,120; {\tiny $I_{6,12}$}}
      \lbl[l]{125,120; {\tiny $I_{6,13}$}}    
    \end{lpic}
    \caption{$\Gamma_6$}
  \end{subfigure}
   \begin{subfigure}{0.5\linewidth}
    \begin{lpic}[l(0mm),r(0mm),b(0mm),t(0mm)]{gl7(0.3)}
     \lbl[l]{107,37; {\tiny $I_{7,1}$}}
      \lbl[l]{85,60; {\tiny $I_{7,2}$}}
      \lbl[l]{83,67; {\tiny $I_{7,3}$}}
      \lbl[l]{76,73; {\tiny $I_{7,4}$}} 
      \lbl[l]{66,85; {\tiny $I_{7,5}$}}
      \lbl[l]{46,96; {\tiny $I_{7,6}$}}
      \lbl[l]{44,80; {\tiny $I_{7,7}$}}
      \lbl[l]{53,68; {\tiny $I_{7,8}$}} 
      \lbl[l]{64,60; {\tiny $I_{7,9}$}}
      \lbl[l]{70,53; {\tiny $I_{7,10}$}}
      \lbl[l]{58,48; {\tiny $I_{7,11}$}} 
      \lbl[l]{30,47; {\tiny $I_{7,12}$}}
      \lbl[l]{30,57; {\tiny $I_{7,13}$}} 
      \lbl[l]{24,65; {\tiny $I_{7,14}$}}
      \lbl[l]{10,80; {\tiny $I_{7,15}$}} 
      \lbl[l]{32,120; {\tiny $I_{7,16}$}}
      \lbl[l]{85,173; {\tiny $I_{7,17}$}}        
    \end{lpic}
    \caption{$\Gamma_7$}
  \end{subfigure}
  \caption{First iterates of the quadrant $Q_3$ for $a=-1$ and $2/3<b\leq 5/7$. In this case $\Gamma:=\Gamma_5\cup\Gamma_6\cup\Gamma_7$.}\label{f:gl}
\end{figure}

\newpage

{\small
\begin{table}[H]
\centering
\begin{tabular}{|c|c|c|c|c|c|c|c|c}
\hline
 & $b\leq -2$ & $-2<b\leq -1/4$ & $-1/4<b<0$ & $0\leq b\leq 3/16$ \\
\hline
\hline
$N_1$ & 8 & 6 &  5&  6 \\
\hline
$N_3$ & 5 & 5 & 4 &   4 \\
\hline
\hline
 & $3/16<b< 4/15$ &$4/15\leq b\leq 2/3$ &$2/3<b\leq 7/4$  & $b> 7/4$ \\
\hline
\hline
$N_1$ & 11 &6 & 5& 5\\
\hline
$N_3$ & 9& 4& 4&5\\
\hline
\end{tabular}
\caption{Arrival times of points in $Q_1\cup Q_3$ to $\Gamma$.}\label{tb:tempsarribada}
\end{table}
}

The verification of the invariance of the graphs $\Gamma$ for each value of $b$ is again a lot of routine work. To do it, for each of the different graphs, it must be verified that the image of each edge entirely contained in a single quadrant, remains in the graph. To this end, in each graph we have indicated the points that characterize the part of each edge that is entirely contained in each quadrant. In consequence, we only have to check that the images of these points belong to graph. 

$(b)$
For a fixed value of $b$, and once we have proved the invariance of the graph $\Gamma$, we only have to check that there exist  values $N_1$ and $N_3$ such that $F^{N_1}(Q_1)$ and 
$F^{N_3}(Q_3)$ are contained in $\Gamma$. For instance in the above case, for $2/3<b\leq 5/7$,
one can see that $F^4(Q_3)=\Gamma_4\subset \Gamma$, hence $N_3=4$. See Figure \ref{f:gl4gamma}.
In each case, we get the arrival times to $\Gamma$ written in Table \ref{tb:tempsarribada}.
In summary, from the  case-by-case study, for all $b\in\mathbb{R}$ we have $F^{11}(Q_1)\in \Gamma$ and 
$F^{9}(Q_3)\in \Gamma$. Therefore $F^{11}(Q_1\cup Q_3)\in \Gamma$.
\end{proof}

\begin{figure}[H]
 \centerline{ \includegraphics[width=0.5\textwidth]{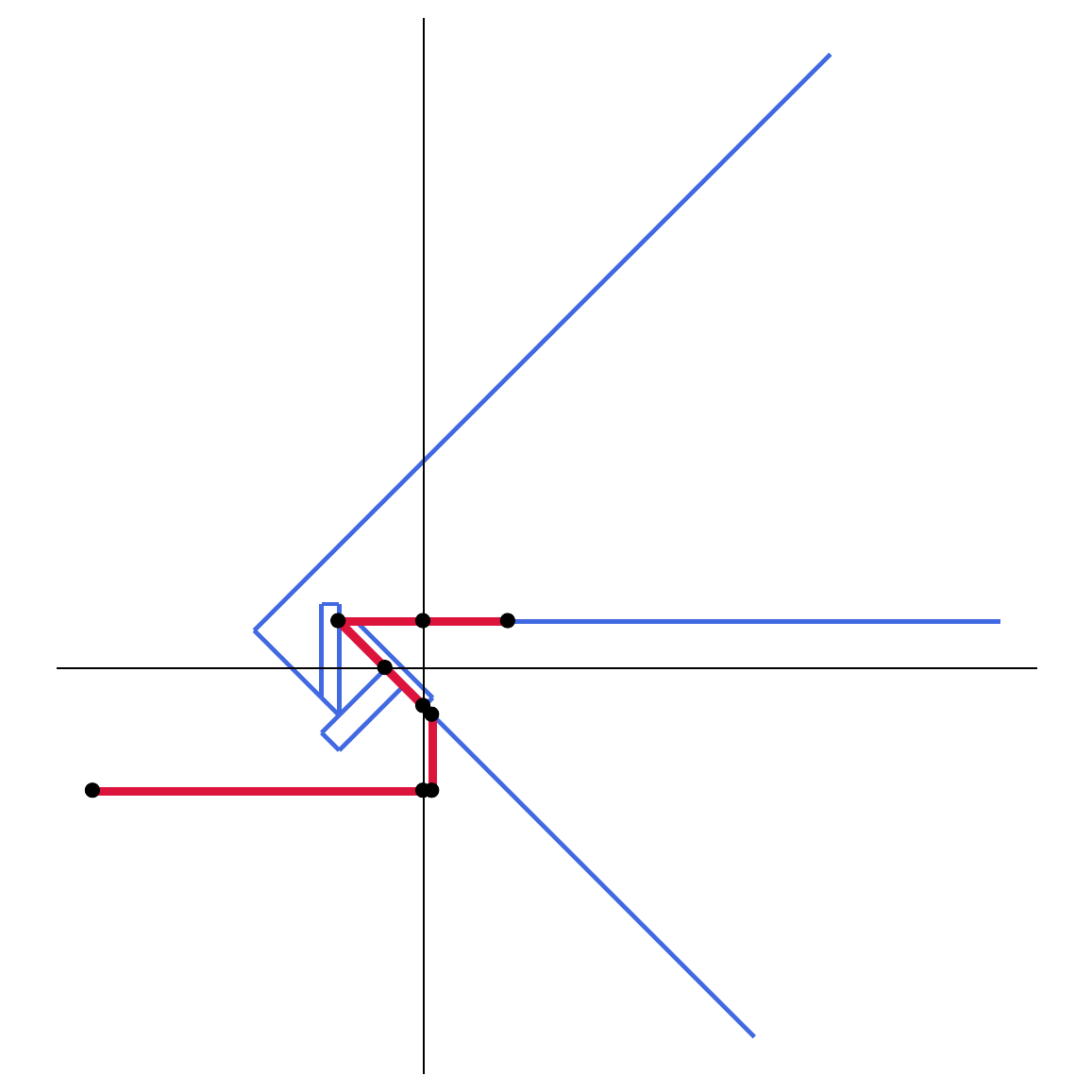}}
 \caption{The graph $\Gamma_4\subset \Gamma$ for $a=-1$ and $2/3<b\leq 5/7$.}\label{f:gl4gamma}
\end{figure}

From the above result, when $3/16\leq b\leq 4/15$,    $(-b,2b-1)\in\Gamma$ is the unique fixed point of $F$ in $\mathbb{R}^2.$ The next result states that for this range of the parameter $b$,  each point in  $Q_2\cup Q_4$ also leaves this region in an uniformly bounded number of iterates. As we  have already comment, this is not the case for other values of $b.$

\begin{propo}\label{p:P15nova}
Assume $a=-1.$ If $3/16\leq b\leq 4/15$, then for any $(x,y)\in Q_2\cup Q_4$ there exists $n\leq 11$ such that $F^n(x,y)\in Q_1\cup Q_3$.  
\end{propo}
\begin{proof}
We start stating the following claims:

\noindent \emph{Claim 1:} If  $(x,y)\in Q_4$ is such that $F(x,y)\in Q_2$, then there exists $n\leq 6$ such that $F^n(x,y)\in Q_1\cup Q_3$.

\noindent \emph{Claim 2:} If  $(x,y)\in Q_4$ is such that $F(x,y)\in Q_4$, then there exists $n\leq 4$ such that $F^n(x,y)\in Q_1\cup Q_2 \cup Q_3$.

\noindent \emph{Claim 3:} If  $(x,y)\in Q_2$ is such that $F(x,y)\in Q_2$, then there exists $n\leq 2$ such that $F^n(x,y)\in Q_3\cup Q_4$.

If the claims are true, by collecting them, a point in $Q_4$ reaches $Q_1\cup Q_3$ in, at most, $9$ iterates; and a point in $Q_2$ reaches $Q_1\cup Q_3$ in, at most, $11$ iterates, so the result follows. Now we prove the claims:

In the proof of Proposition \ref{primera reduccioamenys1} we have seen that if $3/16\leq b\leq 1/4$, then any point $(x,y)\in Q_4$ with $F(x,y)\in Q_2$, that is the points in the set $K$ given in \eqref{DefiniciodeK}, reach $Q_1\cup Q_3$ in at most $5$ iterates. Next, let us take 
 $1/4\leq b\leq 4/15.$ As mentioned in the proof of  Proposition~\ref{primera reduccioamenys1}, the set $K_5$ is the triangle given in~\eqref{e:K5nou}. A computation shows that
 $$
 K_6=F(K_5)=\langle (5b-2,-4b+1), (b-1,0), (-3b,-4b+1)\rangle \subset Q_3.
 $$
Hence every point in $K$ reaches $Q_1\cup Q_3$ in, at most, $6$ iterates. Hence Claim 1 is proved.

Now we define $T_1:=F(Q_4)\cap Q_4$. Some computations show that $T_1$ is the unbounded region defined by  the half-line $\{y=0,\,x\geq 0\}$, the segment $\overline{(0,0),(0,b-1)}$ and the half-line $\{y=-x+b-1,\,x\geq 0\}$.

The set $F(T_1)$ is the unbounded region delimited by $\{y=x+b+1,\,x\geq -1\},$ the segment
$\overline{(-1,b),(-b,2b-1)}$  and $\{y=2b-1,\,x\geq -b\}$. Taking into account Claim 1, we only have to keep track of the points in $T_2=F(T_1)\cap Q_4$, which is the unbounded region defined by $\{y=0,\,x\geq 0\}$, the segment
$\overline{(0,0),(0,2b-1)}$  and $\{y=2b-1,\,x\geq 0\}$.

The set $F(T_2)$ is the unbounded region delimited by $\{y=x+b+1,\,x\geq -1\},$ the segment
$\overline{(-1,b),(-2b,3b-1)}\subset \{y=-x+b-1\}$  and $\{y=x+5b-1,\,x\geq -2b\}$. If $b\geq 1/5$ then $F(T_2)\in Q_1\cup Q_2$ and we are done. Conversely, if $b<1/5$ then we define
$T_3=F(T_2)\cap Q_4$, which is the triangle
$
T_3=\langle (0,0), (-5b+1,0), (0,5b-1)\rangle.
$

Finally, the set $T_4=F(T_3)$ is the triangle
$
T_4=\langle (-1,b), (-5b,-4b+1), (-5b,6b-1)\rangle\subset Q_2.
$
So Claim 2 is proved.

It is easy to observe that 
$
S_1=F(Q_2)\cap Q_2=\langle (-1,b), (-b-1,0), (b-1,0)\rangle.
$
A computation shows that 
$F(S_1)=\langle (-b,-1), (b,-1), (-b,2b-1)\rangle \subset Q_3\cup Q_4,$
hence Claim 3 is proved, and therefore, the result follows.
\end{proof}

As a consequence of the above result, we stress that when $3/16\leq b\leq 4/15$,  for every point in $\mathbb{R}^2$ the arrival time to the unique fixed point is uniformly bounded.

\begin{corol}\label{c:c16nou}
Assume $a=-1.$ If $3/16\leq b\leq 4/15$, then $F^{22}(\mathbb{R}^2)=\{(-b,2b-1)\}=\Gamma$.
\end{corol}
\begin{proof} 
The result is a consequence of the Proposition \ref{p:P15nova}, which states that each point in $Q_2\cup Q_4$ reaches $Q_1\cup Q_3$ in at most $11$ iterates, and of the  proof of Proposition \ref{p:lemanou} (see the Table \ref{tb:tempsarribada}) which shows that, in this case, each point in $Q_1\cup Q_3$ reaches $\Gamma=\{(-b,2b-1)\}$ in at most~$11$ iterates.
\end{proof}

\begin{proof}[Proof of Theorem \ref{t:teoB}]  
From Proposition  \ref{primera reduccioamenys1}, for all $b\in\R$, the orbit of every point in $Q_2$ or in $Q_4$ meets $Q_1$ or $Q_3,$ except the fixed points of $F$ in $Q_2$ and $Q_4$, that exist when $b>1/2$ and when $b<0$, respectively; and a three periodic orbit located at $Q_2\cup Q_4$ that exists when $3/4<b<2$. On the other hand, from Proposition \ref{p:lemanou}, each point in $Q_1\cup Q_3$ reaches $\Gamma$ in finite time. From  these results the theorem is proven.
\end{proof}

\subsection{Structure of the $\omega$-limits and proof of Theorem \ref{t:teoC}}

Our next objective is to show that almost all  periodic orbits that appear in the dynamics of $F$ are repulsive. To do this we need to compute the action of the  linear parts of $F$ on some specific set of directions. Set $V=\{v_1,v_2,v_3,v_4\}$ where $v_1=(1,0),\,v_2=(0,1),\,v_3=(1,1)$ and $v_4=(1,-1).$ Also for $i=1,\ldots,4$ let $A_i$ be the linear part of $F_i.$ With this notation we have:

\begin{lem}\label {eigen} Direct computations give
\begin{enumerate}\item[(a)] $A_1(v_1)=v_3,\,A_1(v_2)=-v_3,\,A_1(v_3)=(0,0)$ and $A_1(v_4)=2v_3.$
\item[(b)] $A_2(v_1)=-v_4,\,A_2(v_2)=-v_3,\,A_2(v_3)=-2v_1$ and $A_2(v_4)=2v_2.$
\item[(c)] $A_3(v_1)=-v_4,\,A_3(v_2)=-v_4,\,A_3(v_3)=-2v_4$ and $A_3(v_4)=(0,0).$
\item[(d)] $A_4(v_1)=v_3,\,A_4(v_2)=-v_4,\,A_4(v_3)=2v_2$ and $A_4(v_4)=2v_1.$
\end{enumerate}
\end{lem}

To state our next result we need to introduce some definitions. Set $a=-1, b\in \R$ and let $\mathbf{x}=(x,y)\in \Gamma,$  where $\Gamma$ is the associated invariant graph to $F=F_{-1,b}.$ We will say that $\mathbf{x}$ is regular if $xy\ne 0$ and there exists a neighborhood of $\mathbf{x}$ which is a segment that contains $\mathbf{x}$ in its interior. We call a {\em vertex} of $\Gamma$ to  any point of $\Gamma$ that is not regular, that is, a point that belongs to the axes, or that belongs  at least to two different segments with different directions, or that is an endpoint of the graph. We denote by $W$ be the set of vertices of~$\Gamma.$ By simple inspection it follows that the cardinality of $W$ is finite. Then any connected component of $\Gamma\setminus W$ is an open segment contained in the interior of some quadrant. We call {\em edge} of $\Gamma$ the closure of any connected component of $\Gamma\setminus W.$ Clearly each edge is a segment that has associated a direction which is unique up to scaling. Any edge contained in the first (respectively third) quadrant with associated direction
$v_3$ (respectively $v_4$) will be called  a {\em plateau}. This name is motivated because for real 1-dimensional maps the intervals where these maps are constant collapse to a point and in the graphs of these maps these intervals  look like plateaus in a mountain. 

By abuse of notation we say that a subset of $\Gamma$ is an {\em open interval} (respectively {\em closed interval}) if it is homeomorphic to an open (respectively closed) interval. Given an interval $J\subset \Gamma$ we denote by $l(J)$ its length computed in the usual euclidian metric in $\R^2.$

\begin{lem}\label{edges} Set $a=-1.$ The following assertions hold 
\begin{enumerate} \item[(a)] 	The direction of any edge of the associated graph $\Gamma$ belongs to $V.$
\item[(b)] The image of any plateau is a single point. The image of a non plateau edge $J$ is a nondegenerated interval. Moreover, $l(F(J))=\sqrt 2\,l(J)$  when $J$ is contained in the second or the fourth quadrant or has the horizontal or vertical directions, while $l(F(J))=2\,l(J)$ when it has the direction $v_4$ and it is contained in $Q_1$ or it has the direction $v_3$ and it is contained in $Q_3.$
\item[(c)] Any periodic orbit of $F$ that does not visit any plateau is repulsive.

\end{enumerate}
\end{lem}
\begin{proof} $(a)$ follows by direct inspection of $\Gamma.$ It is due to the fact that 
$F(Q_1)$ is a straight line with direction $v_3$ and 
$F(Q_3)$ is a straight line with direction $v_4.$ Since from Lemma~\ref{eigen} the set of directions $V$ is invariant by $F$ and the graph $\Gamma$ is obtained iterating $F$ over $Q_1$ and $Q_3,$ the result follows.

$(b)$ Follows directly from Lemma \ref{eigen}.

$(c)$ Let $\mathbf{x}_0$ be a point of the periodic orbit and  for $i=1\ldots n-1$ set $\mathbf{x}_i=F^i(\mathbf{x}_0).$ First we consider the case when all the points of the orbit are regular. In this case, for all $i=0,\ldots n-1,$ $\mathbf{x}_i$ belongs to the interior on an edge, namely $L_i.$  Each of these edges $L_i$ is contained in some quadrant $Q_{j_i}$ and has associated the direction $v_{j_i}\in V.$ In this case, since $F$ acts linearly in a little neighborhood of each  $\mathbf{x}_i,$  there is $U$ a little neighborhood of $\mathbf{x}_0$ contained in $L_0$ satisfying $F^n(U) \subset L_0,$ and $F^n\vert_U$ is affine. 
From Lemma \ref{eigen} we have that $A_{j_i}(v_{j_i})=k_iv_{j_{i+1}}$ where $|k_i|\in \{1,2\}.$ Note that $k_i\ne 0,$ because otherwise  the edge $L_i$ is a plateau contradicting the hypothesis. So after $n$ iterates the direction $v_0$ is maped to $k_1k_2\ldots k_nv_0$ and $|k_1k_2\ldots k_n|=2^m$ with $m\le n.$ Note also that from Lemma \ref{eigen}  any  cycle in the directions at some moment has the factor 2 or $-2.$ This is direct for the directions $v_3$ and $v_4.$ And this occurs in the second step for directions $v_1$ and $v_2,$ because they are always sent to directions $v_3$ and $v_4.$ So $m>0.$ This implies that the absolute value of the slope of the map restricted to the subinterval of $L_0$ which is send to $L_0$ is $2^m$ with $m>0.$ So the fixed point is repulsive. This ends the proof of $(c)$ in this case. 

Now we consider the case when some of the points of the orbit is a vertex. Let $\mathbf{x}$ be a vertex belonging to the periodic orbit, and let $k$ be the number of edges containing $\mathbf{x}.$  Then a small neighborhood of $\mathbf{x}$ is like a $k$-star. If $F^n$ is a local homeomorphism at $\mathbf{x}$ then it permutes the edges. Therefore, some power $m$ of $F^n$ maps the beginning of each edge at $\mathbf{x}$ to itself. Then by the same arguments of the regular case we have that  each lateral slope of $F^{nm}$ at the point is $2^{\ell}$ with $0<\ell\le mn.$ If $F^n$ is not a local homeomorphism at $\mathbf{x},$ since the orbit does not visit any plateau, necessarily the image of the beginning of some of the edges by $F^n$ coincide. Therefore, $F^n$ permutes a subset $W$ of $j<k$ edges and it sends the beginning of the remaining $k-j$ edges to some edges in $W.$ As before, some power $m$ of $F^n$ maps the beginning of each edge in $W,$ to itself with slope $2^{\ell}$ with $0<\ell<mn.$ This ends the proof of the lemma.
\end{proof}

Our next objective is to prove Theorem C. Let $\mathcal W$ be the union of the interiors of the plateaus of $\Gamma$ and let $\mathcal U=\cup_{i=0}^{\infty}F^{-i}(\mathcal W)$ be the union of all its preimages.  Set $\mathcal G=\Gamma\,\setminus\,\mathcal U.$ Then, the following holds.

\begin{propo}\label {interior}
Assume $a=-1.$ For all $b\in\R,$ the set $\mathcal U$ is dense in $\Gamma.$ 
\end{propo}
 To prove this proposition we will prove an auxiliary technical lemma. We will say that a graph is a {\em tree} if it does not contain any circuit. That is if any closed path is homotopic to a constant path.

 \begin{lem}\label{tree} For all graphs introduced in Theorem~\ref{t:teoB} there exists a finite collection $\{\mathcal T_i\}_{i=1}^n$ of closed and connected trees satisfying $ {\mathcal T}_i\subset \Gamma$ for all $i,\,\,{\mathcal T}_i\cap{\mathcal T}_j=\emptyset $ if $i\ne j$ and $\mathcal G\subset \cup _{i=1}^n\mathcal T_i.$ \end{lem} 
 
 \begin{proof} To prove the lemma we will need to study
 each of the different topological situations of the graph $\Gamma.$ For brevity we skip here the details and only explain some cases. When $b\in [0,1/2]$ (Figures \ref {f:10}-\ref {f:19}) the graph $\Gamma$ is itself a tree and there is nothing to prove. Moreover, when $b\in (-\infty,-1/5]$ (Figures \ref {ff:1}-\ref {f:6}), or when $b\in [2,\infty)$ (Figures \ref {ff:22}-\ref {f:23}) removing  the interior of one or two plateaus from $\Gamma$ we already obtain a connected tree. The other cases are more complicated but one can easily check that the result holds in each particular case. We explain here two of these cases and skip the others for sake of brevity.  Consider now
 $b\in (-1/5,-1/8)$ (Figure~\ref{f:7}). Removing from $\Gamma$ the interior of the plateaus $\overline{R_2Q}, \overline{X_3Z_2},$ and the interior of $\overline{X_2Z_1}$, which is a preimage of the interior of the plateau $\overline{X_3Z_2},$ we obtain a connected tree and the result follows. The last case we study whith detail is when $b\in (1,3/2]$ (Figure \ref{f:21a}). Here we first remove from $\Gamma$ the interior of the plateau
 $\overline{R_2Y_3}.$   We also remove the interior of $\overline {R_1Y_2}$ which is a preimage of the interior of this plateau. Set $R_0=(b/2,(b-2)/2)\in \overline{Y_1Z_2}$ and note that the interior of $\overline {R_0Y_1}$ is a preimage of the interior of  $\overline {R_1Y_2}.$ Set $X_5=F(X_4)=(2-b,2b-3)\in \overline {Y_1Z_2}.$ In fact, when $b\le  4/3,\,X_5\in \overline{R_0Z_2}$ and therefore, there is an open subinterval  $J\subset \overline{X_4T_1}$ such that $F(J)$ is the interior of $\overline {R_0Y_1}.$ In this case, removing also from $\Gamma \,\,$ the interior $\overline {R_0Y_1}$ and $J,$ we obtain a connected tree ending the proof in this case (see Figure \ref{f:tree}). When $b\in (4/3,3/2),\,X_5\in \overline {R_0Y_1}$  and the corresponding preimage $J$ of the interior of $\overline {R_0Y_1}$ is a tree that  contains $X_4$ and intersects $\overline{T_1X_4},\,\, \overline{X_4Y_3}$ and $\overline{X_1X_4}.$ Thus, removing also $J$ and the interior of $\overline {R_0Y_1}$ from $\Gamma,$ we obtain two disjoint connected trees ending the proof in this case. We finish here the study of the particular cases and the proof of the lemma.
   \end{proof}
   
 \begin{figure}[H]
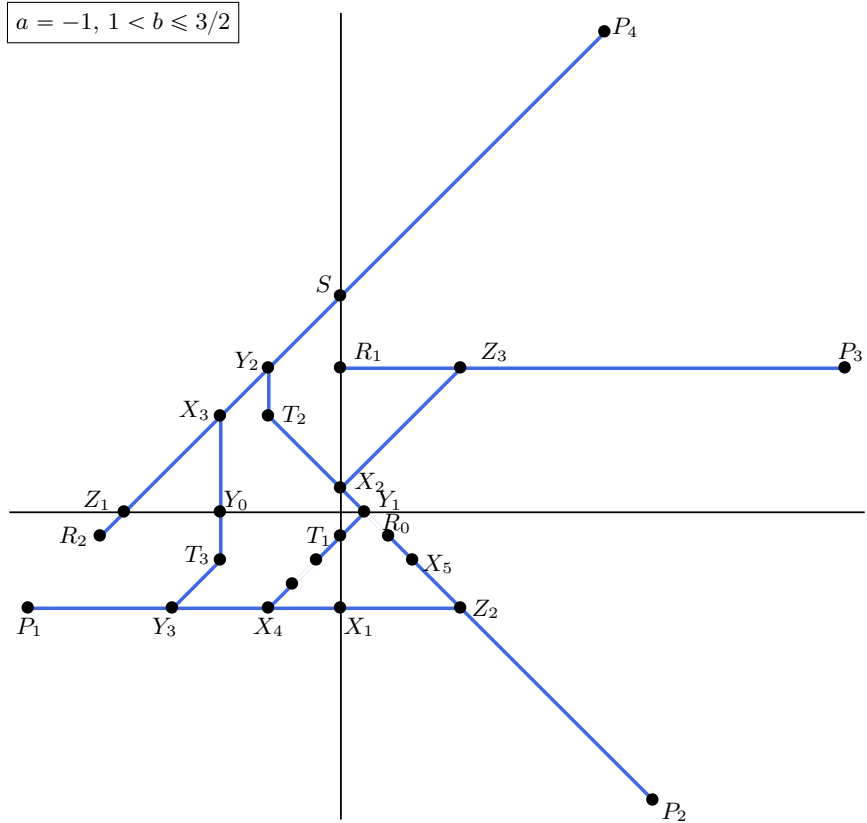

 	\footnotesize
 	\centering
 	
 	\begin{lpic}[l(2mm),r(2mm),t(2mm),b(2mm)]{tree(0.48)}
 		\lbl[l]{3,185; $\boxed{a=-1,\,1< b\leq 3/2}$}
 		
 		\lbl[l]{137,184; $P_{4}$}
 		\lbl[r]{75,127; $S$}
 		\lbl[r]{59,110; $Y_2$}
 		\lbl[r]{48,99; $X_3$}
 		
 		\lbl[r]{26,79; $Z_1$}
 		\lbl[r]{21,71; $R_2$}
 		
 		\lbl[c]{8,51; $P_{1}$}
 		\lbl[c]{38,51; $Y_3$}
 		\lbl[c]{61,51; $X_4$}
 		\lbl[c]{81,51; $X_1$}
 		\lbl[l]{106,55; $Z_2$}
 		
 		\lbl[l]{64,98; $T_{2}$}
 		
 		\lbl[l]{148,10; $P_{2}$}
 		
 		\lbl[l]{80,83; $X_{2}$}
 		
 		\lbl[l]{86,74; $R_0$}
 		
 		\lbl[l]{95,65; $X_5$}
 		
 		\lbl[l]{85,79; $Y_1$}
 		
 		\lbl[l]{51,79; $Y_0$}
 		
 		\lbl[c]{190,112;  $P_{3}$}
 		\lbl[c]{111,112; $Z_3$}
 		\lbl[l]{78,112; { $R_1$}}
 		
 		\lbl[r]{75,71; { $T_1$}}
 		
 		\lbl[r]{48,67; { $T_3$}}
 		
 	\end{lpic}
 	
 	\caption{The tree in the proof of Lemma \ref{tree},  obtained from $\Gamma$ for $a=-1$ and $1/< b\leq 4/3$, when removing some preimages of the plateaus. When $4/3< b\leq 3/2,\,\, X_5\in \overline{R_0Y_1}$ and the preimage of
 	the interior of $\overline {R_0X_1} $ is a subtree of the tree with endpoints $T_1,Y_3$ and $X_1$ that contains $X_4.$ So removing this last piece from $\Gamma$ we obtain two connected and disjoint  trees.}\label{f:tree} 
 
\end{figure}

\begin{proof}[Proof of Proposition  \ref{interior}.]
Since $\mathcal G=\Gamma\,\setminus\,\mathcal U,$ it is equivalent to show that the interior of $\mathcal G$ is empty. Suppose to arrive a contradiction that the interior of $\mathcal G$ is not empty. Then $\mathcal G$ must contain some open interval and hence also must contain a nondegenerate closed interval that we denote by $J.$ First of all we claim that if $F\vert_J$ is not injective then there is a point $\mathbf{z}\in J$ in which $F\vert_J$ is not locally injective. This is clear when $F(J)$ is also an interval. If not, by Lemma \ref {tree} we have that $F(J)\subset \mathcal T_i$ for some $\mathcal T_i\subset \Gamma$ and hence it is also a subtree of $\Gamma.$ Since it is not an interval it has at least three endpoints and therefore, there exists $\mathbf{z}$ belonging to the interior of $J$ such that $F(\mathbf{z})$ is an endpoint of $F(J).$ Clearly $F\vert_J$ is not localy injective at $\mathbf{z}.$ This ends the proof of the claim.

There are only three type of points in which the map $F\vert_J,$ being $J\subset \Gamma$ an interval, is not locally injective. Either the point belongs to a plateau or it belongs to $Q_1$ and it is the intersection of two edges with some prescribed slopes or it belons to $Q_3$ and also is the intersection of two edges with prescribed slopes. We summarize these possible situations in Figure \ref {f:notinjective}.

\begin{figure}[H]
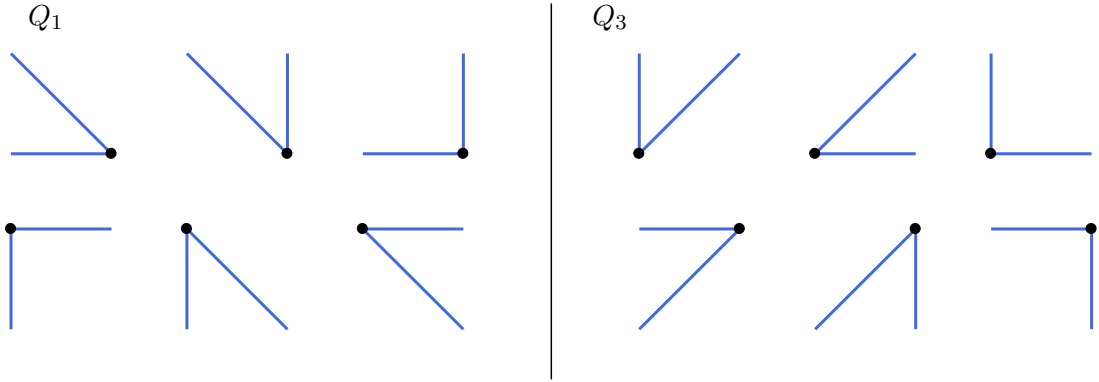

	
	\begin{lpic}[l(2mm),r(2mm),t(2mm),b(2mm)]{notinjective(0.75)}
		
		\lbl[l]{5,67; $Q_1$}
		\lbl[l]{105,67; $Q_3$}

	\end{lpic}
	\caption{Points $\mathbf{z}\in J$, where $J$ is a subinterval of $\Gamma$, for which $F\vert_{J}$ is not locally injective.}\label{f:notinjective}
\end{figure}

Let $\delta:=\sup \{l(J),\,\, J\subset\mathcal G\}$ that is positive because we are assuming that $\mathcal G$ contains open intervals. Choose $I,$ a subinterval of $\mathcal G$ satisfyng that $\sqrt 2l(I)> \delta.$ If $F\vert_I$ is injective it follows that $F(I)$ is an interval and from Lemma \ref{edges} $(b)$ we get that $l(F(I))\ge \sqrt 2l(I)>\delta;$ a contradiction because $\mathcal G$ is positively invariant and hence $F(I)$ would be an interval contained in $\mathcal G$ with length greater than $\delta.$ Therefore, $F\vert_I$ is not injective and from the previous claim  $I$ must contain a point $z$ in which $F\vert_I$ is not locally injective. Then we are in one of the  situations described in Figure \ref {f:notinjective}.

 We will prove that in any case 
some iterated image of $I$ is an interval and its length is greater than $\delta$ obtaining the desired contradiction. We only study the most complicated case because in the first iterate the length of the interval is strongly reduced. This occurs when the cardinality of the set of preimages in $I$ of a point of $F(I)$ is greater than two. This will imply that there are at least two vertices in which $F\vert_I$ is not locally injective belonging to $I.$ This can occur for example when $b\in (-1/5,-1/9)$ (Figures \ref{f:7} and \ref{f:8}) assuming that the interval $I$ begins at a point $\mathbf{x}\in \overline{P_1Y_2},$ joins this point with $Y_2,$ joins $Y_2$ with $T_3$ and also joins $T_3$ with a point $\mathbf{y}\in \overline{T_3R_2}.$ Direct computations show that $F(I)\subset \overline{R_3P_2}$ and $$l(F(I))\ge \frac{4\sqrt 2-2}{7}l(I).$$ Thus we have, assuming that $I\subset \mathcal G,$ that  $F^2(I)\subset \overline{R_4P_3},\,\,F^3(I)\subset \overline{R_5S}$ and $F^4(I)\subset \overline{R_6P_5}$ and $l(F^4(I))\ge 2\sqrt 2 \frac{4\sqrt 2-2}{7}l(I)\ge \sqrt 2 l(I)>\delta$; a contradiction. A similar situation occurs when $b\in (3/4,10/13]$ (Figures \ref{f:E}-\ref{f:J}) when the vertices involved are $X_4$ and $Z_3$ or when $b\in (10/13,11/14]$ (Figure \ref{f:K}) with the vertices $Y_6$ and $Z_3.$ In all these cases we arrive at a contradiction by computing the length of $F^4(I).$ In all the remaining cases where the cardinality of the preimages in $I$ of a point of $F(I)$ is at most two we have $l(F(I))\ge {l(I)}/{\sqrt2 }$ an it suffices to consider  $l(F^3(I))$ to reach a contradiction. This ends the proof of the proposition.
\end{proof}

\begin{proof}[Proof of Theorem C] Clearly $\mathcal U$ is open and by Proposition \ref{interior} it is also dense in $\Gamma.$ Moreover the $\omega$-limit of any point of $\mathcal{U}$ coincides with the $\omega$-limit of some plateau. A simple inspection of all the possible situations for the graph $\Gamma$ shows that there are at most three different possible $\omega$-limits for the plateaus. Although in some cases there are more than three plateaus, there are some of them sharing their $\omega$-limits in such a way we obtain at most three different behaviors. For example in the case $3/4<b\le 154/205$ (see Figure \ref{f:E}) there are six plateaus, namely the segments $\overline{SSX_4},$ $\overline{SQW_6},$ $\overline{QQW_4},$ $\overline{\Pi_2X_2},$ $\overline{QR_4}$ and $\overline{SP_4}.$ However since  $F(\overline{SQW_6})=X_{20},$ $F(\overline{SSX_4})=X_5,$ $F(\overline{\Pi_2X_2})=X_3,$ $F(\overline{QQW_4})=W_5$ and $F^2(W_5)=X_{20}$ it follows that these four plateaus share the same $\omega$-limit, the $18$-periodic orbit of $X_5.$ A similar situation holds in all the cases having more than three plateaus. Then the first assertion of the theorem follows. 
	
Assume now that $b=p/q$ with $(p,q)=1.$ Denote by $$\Z_q:=\{x\in\Q \mbox{ such that } x=r/q \mbox{ for some } r\in \Z \}.$$ Since $\Gamma$ is compact it follows that the cardinality of $\Gamma\cap \Z_q\times \Z_q$ is finite. On the other hand a simple inspection shows that all the vertices of $\Gamma$ belong to $\Z_q\times \Z_q.$ In particular, the image of any plateau belongs to $\Z_q\times \Z_q.$ Since $F(\Z_q\times \Z_q)\subset \Z_q\times \Z_q,$ and this set is finite, we obtain that any $\omega$-limit in $\Z_q\times \Z_q$ is a periodic orbit. This ends the proof of the theorem.
\end{proof}

\section{The case $a<0$ (II). Case analysis and proof of Theorem \ref{t:teoD}}\label{sec:5}

In the following, we use again the conjugation \eqref{conj}, and we will work with the normalized map $F_{a,b}$ with $a=-1$.

\subsection{The case $a=-1$ and $b\leq -2$}\label{ss:primera}

In this section we freely use, without citing it explicitly,  the results stated in Proposition~\ref{rot}.

\begin{propo}\label{primer} Assume that $a=-1$ and $b\leq -2.$ Then the following holds.
	\begin{enumerate}
		\item [(a)] The graph $\Gamma$ is a topological circle and the map $F\vert_{\Gamma}\in \cal{L}.$  In particular, it has zero entropy. Moreover, its rotation number is $1/7.$
		\item [(b)]The map $F$ has the fixed point $p=(-b,-1)\in Q_4$ and two 7-periodic orbits, 	
		\begin{align*}\mathcal{P}=\left\{(-b-2,-1), (-b-2,-3),(-b,-5),
		(-b+4,-5), \right.\\  	\left.(-b+8,-1), (-b+8,7), (-b,1)\right\}\end{align*}
	which is the orbit of the only plateau of $\Gamma,$ and
	 $${\mathcal{Q}=\left\{\left(-b-\frac{16}{15},-\frac{1}{15}\right), \left(-b-2,-\frac{17}{15}\right), \left(-b-\frac{28}{15},-\frac{47}{15}\right), \left(-b+\frac{4}{15},-5\right),\right.}$$
	 $${ \qquad\qquad\qquad\qquad\qquad \left. \left(-b+\frac{64}{15},-\frac{71}{15}\right),\left(-b+8,-\frac{7}{15}\right), \left(-b+\frac{112}{15},\frac{113}{15}\right)\right\}}.$$

\rec Furthermore, for any $(x,y)\in\Gamma\setminus \mathcal{Q}$ there exists some $n$ such that $F^n(x,y)\in \mathcal{P}.$
	\end{enumerate}
\end{propo}
\begin{proof}
	
From Proposition \ref{p:lemanou}, we know that the dynamics of $F$ is concentrated in the graph $\Gamma$ of Figure \ref{ff:1}. Note that, in this case, $\Gamma$ is a topological circle and $F\vert_{\Gamma}$ is non-decreasing. Here  we have $F(P_i)=P_{i+1}$ for $i=1\ldots 6,$  $F(R_1)=R_2$ and $F(P_7)=F(R_2)=F(S)=P_1.$

We can see that the interval $\overline{R_2S}\twoheadrightarrow P_1$ and also 
$\overline{P_6R_1}\rightarrow \overline{P_7R_2}\twoheadrightarrow P_1,$ where $\twoheadrightarrow$ means ``collapses to''. Since these intervals collapse, we can neglect them. We denote $A:=\overline{P_{1}P_{2}}$, $B:=\overline{P_{2}P_{3}}$, $C:=\overline{P_{3}P_{4}}$, $D:=\overline{P_{4}P_{5}}$, $E:=\overline{P_{5}R_1}$, $G:=\overline{P_{6}R_2}$ and $H:=\overline{P_{1}S}.$ 

The oriented graph corresponding to the covering $\mathcal{A}=\{A,B,C,D,E,G,H\}$ is the following:

\begin{center}
\begin{tikzcd}
	A \arrow[r] & B \arrow[r] & C \arrow[r] & D \arrow[r] & E \arrow[r] & G \arrow[r] & H \arrow[llllll, bend right]
\end{tikzcd}
\end{center}

We stress the fact that we do not include in this graph the coverings of the plateau and its preimage.

Note that $F^7$ leaves each of these intervals  invariant and in particular, the graphic of $F^7$ restricted to $A$ looks like Figure \ref{Grafic0}. Also note that the subinterval of $A$ where $F^7$ is constant appears because $D$ also covers $\overline{R_1P_6}$ that collapses after two iterates. From these facts we directly obtain items $(a)$ and $(b).$ \end{proof}

	\begin{figure}[H]
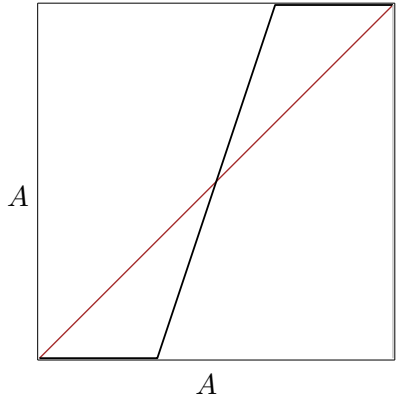

		\centering
		\begin{lpic}[l(2mm),r(2mm),t(2mm),b(2mm)]{G1(0.25)}
			\lbl[c]{90,-10; { $A$}}
			
			\lbl[c]{-10,90; { $A$}}
			
		\end{lpic}
		\caption{Graphic of $F^7$ on $A$ when $ b\leq -2.$}
		\label{Grafic0}
	\end{figure}

\subsection{The case $a=-1$ and $-2< b\leq -1$}

As before, the map $F$ has the fixed point $p=(-b,-1)\in Q_4.$ Consider the graph $\Gamma$ which appears in Figure \ref{ff:2}. It is still a topological circle and $F\vert_{\Gamma}$ is still non-decreasing. Moreover for $i=1,\ldots 6,\,F(P_i)=P_{i+1},$ for $i=1,2,3,\,F(R_i)=R_{i+1},\,F(S_1)=S_2,\,F(Q)=P_3$ and $F(P_7)=F(S_2)=F(T)=R_1.$  It remains to study the image of $R_4=(-b+4,4b+3)$ which is $(-5b,4b+7).$ Then the orbit of $R_4$ depends on the value of $b.$

The following Proposition states all results concerning this range of parameters.

\begin{propo}\label{segon}
	Assume that $a=-1$ and $-2< b\leq -1.$ Then the graph $\Gamma$ is a topological circle and the map $F\vert_{\Gamma}\in \cal{L}.$ In particular, it has zero entropy. Moreover

\begin{enumerate}
	\item [(a)] When $-2<b<-15/8$ the rotation number of $F\vert_{\Gamma}$ is $1/7.$ Furthermore, $F\vert_{\Gamma}$ has  two-periodic orbits of period 7: $\mathcal{P}$ which is the orbit of $(-b-2,-1)$ that is reached by the two plateaus of $\Gamma$, and $\mathcal{Q}$ which is the orbit of $\left(-b-2,-\frac{16b+15}{15}\right)$ and it is repulsive. Lastly, for every $(x,y)\in \Gamma\setminus \mathcal{Q}$ it exists $n\in\N$ such that $F^n(x,y)\in \mathcal{P}.$ When $b=-\frac{15}{8}$ the situation is essentially the same but both 7-periodic orbits coincide.
	\item [(b)] When $-7/4<b\le-1$ the rotation number of $F\vert_{\Gamma}$ is $1/6.$ Furthermore, $F\vert_{\Gamma}$ has two-periodic orbits of period 6: $\mathcal{P}$ which is the orbit of $(-b-2,-1)$ that is reached by the two plateaus of $\Gamma$ and $\mathcal{Q}$ which is the orbit of $\left(-\frac{7b+16}{15},\frac{8b-1}{15}\right)$ and it is repulsive. Also, for every $(x,y)\in \Gamma\setminus \mathcal {Q}$ it exists $n\in\N$ such that $F^n(x,y)\in \mathcal{P}.$ When $b=-7/4$ both orbits coincide.
	\item [(c)] For $-{15}/{8}\leq b \leq -{7}/{4},$ the rotation number of $F\vert_{\Gamma}$ varies continuously from $1/7$
	to $1/6.$   For the values $b$ with irrational rotation number, the only periodic point of $F$ is the fixed point $p$ and the $\omega$-limit of any point $q$ different from $p$ is a   Cantor subset of~$\Gamma,$ which is fixed and independent of $q.$ The set of periods for $F|_\Gamma,$ that only arise for the values of $b$ with rational rotation number, are all the natural numbers \emph{except} $2-5$, $8-12$, $14-17$, $18$, $21-24$, $26$, $28-30$, $35$, $36$, $38-40$, $42$, $50$, $52$, $54$, $57$, $60$, $64-66$, $78$, $96$, $100$, $102$, $138$ and $220$. Moreover,
	 if the rotation number is the rational number $r/s$ with $(r,s)=1$, then 
$s=6m+7n$ for  certain $m,n\in \N$ and  either
	\begin{itemize}
		\item $F$ has exactly two periodic orbits of period $s,$ one which we call $\mathcal{P}$ and it is the orbit of $(-b-2,-1)$, that is visited by both plateaus, and another one that we call $\mathcal{Q},$ which is a repelling periodic orbit. Furthermore, for each $(x,y)\in \Gamma\setminus \mathcal {Q}$ it exists~$n$ such that $F^n(x,y)$ belongs to $\mathcal{P},$ or
		
		\item The two $s$-periodic orbits coincide and $F$ has exactly one periodic orbit of period~$s,$ which we call $\mathcal{P},$ which is the orbit of $(-b-2,-1).$ Moreover, for each $(x,y)\in \Gamma$ it exists~$n$ such that $F^n(x,y)$ belongs to $\mathcal{P}.$
	\end{itemize}
\end{enumerate}
\end{propo}
\begin{proof}
	
The proof is essentially the same as in Proposition \ref{primer}. We only explain the main differences.
$(a)$	Consider the graph which appears in Figure  \ref{ff:2}. 
	We are going to add the successive images of the point $R_4.$ Set $R_5:=(-5b,4b+7)=F(R_4)\in \overline{P_{5}S_1},$ $R_6:=(-9b-8,7)=F(R_5)\in \overline{P_{6}S_2}$ and $R_7:=(-9b-16,-8b-15)=F(R_6)\in \overline{P_{7}T}.$ Hence $F(R_7)=R_1$ and $R_1$ is a 7-periodic point.
	
	There are some intervals on the graph which collapse to a point, namely:
	$\overline{S_2T}\twoheadrightarrow R_1\,,\,\overline{P_2Q}\twoheadrightarrow P_3.$ We have to consider also its preimages:
		$\overline{QR_2}\rightarrow \overline{P_{3}R_{3}}\rightarrow \overline{P_{4}R_{4}}\rightarrow \overline{P_{5}R_{5}}\rightarrow \overline{P_{6}R_{6}}\rightarrow \overline{P_{7}R_{7}}\twoheadrightarrow R_1,$ $\overline{P_{6}S_1}\rightarrow \overline{P_{7}S_2}\twoheadrightarrow R_1 $ and also $\overline{P_{1}R_1}$, because $\overline{P_{1}R_{1}}\rightarrow \overline{P_{2}R_{2}}=\overline{P_{2}Q}\cup \overline{QR_2}.$
		
		We name the rest of the intervals by $A:=\overline{P_{2}R_1}\,,\,B:=\overline{P_{3}R_{2}}\,,\,C:=\overline{P_{4}R_{3}}\,,\,D:=\overline{P_{5}R_{4}}\,,\,E:=\overline{S_1R_{5}}\,,G:=\overline{S_2R_6}\,,\,H:=\overline{P_{1}T}.$
		
		The corresponding oriented graph, which is of Markov type, is:
		
\begin{center}
		\begin{tikzcd}
			A \arrow[r] & B \arrow[r] & C \arrow[r] & D \arrow[r] & E \arrow[r] & G \arrow[r] & H \arrow[llllll, bend right]
		\end{tikzcd}
\end{center}

As before, $F^7$ leaves A invariant and the graph of $F^7\vert_{A}$ looks like Figure \ref{Grafic0} and statement $(a)$ follows.

$(b)$ To prove this item we use the same arguments applied in the above case. The calculations are the following. For these values of $b$ we consider the same graph as before. Now
	the point $R_5:=(-5b,4b+7)=F(R_4)\in \overline{S_1P_{6}},$ is in the first quadrant, and we denote by $R_6:=(-9b-8,-8b-7)=F(R_5)\in \overline{S_2P_7}.$ Hence, $F(R_6)=R_1.$ Therefore $R_1$ is now a 6-periodic point.
	
	The intervals such that after some iterates reduce to a point are: $\overline{S_2T}\twoheadrightarrow R_1,$ $\overline{P_{2}Q}\twoheadrightarrow P_3,$ $\overline{S_1R_5}\rightarrow \overline{S_2R_6}\twoheadrightarrow R_1,$ $\overline{P_{1}T}\rightarrow\overline{P_{2}R_1}\rightarrow\overline{P_{3}R_2}\rightarrow\overline{P_{4}R_3}\rightarrow\overline{P_{5}R_4}\rightarrow\overline{P_{6}R_5}\rightarrow\overline{P_{7}R_6}\twoheadrightarrow R_1.$ 
	
	Now the names of the intervals are: $A:=\overline{P_{1}R_1}\,,\,B:=\overline{QR_2}\,,\,C:=\overline{P_{3}R_3}\,,\,D:=\overline{P_{4}R_{4}}\,,\,E:=\overline{P_{5}S_{1}}\,,G:=\overline{P_{6}S_{2}},$ and the oriented graph is
	
	\begin{center}
		\begin{tikzcd}
		A \arrow[r] & B \arrow[r] & C \arrow[r] & D \arrow[r] & E \arrow[r] & G \arrow[lllll, bend right]
	\end{tikzcd}
	\end{center}

In this case $F^6$ leaves $A$ invariant, and the graph of $F^6\vert_{A}$ looks like the graph of $F^7\vert_{A}$ in Figure~\ref{Grafic0}. The result follows in the same way of the precedent situations.

 $(c)$ The continuity of the rotation number with respect to the parameter $b,$
follows from Proposition \ref{rot}. As a consequence, when $b$ runs from $-15/8$ to $-7/4$,  the rotation number at least takes all the values between $1/7$ and $1/6.$ Hence, for all $s$ such that $1/7<r/s<1/6$ for a certain $r\in\N$ with $(r,s)=1,$ we have values of $b$ such that $F$ has periodic orbits of period~$s.$ To obtain the  set of periods, we apply Corollary \ref{c:divisorfunction} to the interval $[a_1,a_2]=[1/7,1/6]$ with the upper bound function $D(s)=2\sqrt{s}$ given in \eqref{e:cotaded}. According to this result,  if $$2\sqrt{s}<\lfloor s/42\rfloor -1$$ for all $s\geq s_0$ then there exist periodic orbits of period $s.$ Some straightforward computation show that the above equation holds if
$s>7140.75$, hence if $s\geq s_0=7141$. Now, it is easy to check, for instance with the help of a symbolic computing software, what are the denominators $s$ of the irreducible fractions in $[1/7,1/6]$ with $s<s_0$, obtaining the ones stated in the statement.

The statement concerning the existence of a Cantor set as $\omega$-limit is a consequence of Proposition \ref{rot} and the existence of edges that colapse to a point (plateaus) that prevent the possibility that the whole $\Gamma$ is the $\omega$-limit. The  arguments to prove the statement when the rotation number is rational are very similar to the ones of the previous two cases.

As in case $(a)$ we have $R_5\in\overline{P_5S_1}.$ Consider the following partition on $\Gamma:$ $A=\overline{P_1R_{1}},\,B=\overline{R_{1}P_2},\,C=\overline{P_2Q},\,D=\overline{QR_2},\,E=\overline{R_{2}P_3},\,G=\overline{P_3R_3},\,H=\overline{R_{3}P_4},\,I=\overline{P_4R_4},\,J=\overline{R_{4}P_5},\,K=\overline{P_5S_1},\,L=\overline{S_1P_6},\,M=\overline{P_6S_2},\,N=\overline{S_2P_7},\,O=\overline{P_7T},\,$ and $U=\overline{TP_{1}}.$ Looking at the dynamics of $F$ over $\Gamma$ we see that $F(A)=C\cup D,F(B)=E,F(D)=G,F(E)=H,F(G)=I,
F(H)= J,F(I)\subset K, F(J)\subset K\cup L,F(K)=M,
F(M)=O\cup U\cup A, F(U)=B.$ Moreover $F^6(C)=F(N\cup O)=F^2(L)=R_1.$

Assume that $F$ has a $s$-periodic orbit that does not intersect the orbit which passes trough $R_1=(-b-2,-1).$ From Lemma \ref{edges} $(c)$ we know that this orbit is repulsive. Denote by $r_1,r_2,\ldots,r_s$ its points ordered counter-clockwise. Looking at the dynamics of $F$ on $\Gamma,$ it follows that such an orbit does not visit $C,L,N$,  or $O$, which are forbidden intervals for it, because, at the end, they  collapse to $R_1.$  Hence such orbit always passes through $K$ and when it does it, it follows either, the loop $KMUBEHJ$ or the loop $KMADGI$, since otherwise it should visit the forbidden intervals. Hence the itinerary of the point lying on the periodic orbit which begins in $K,$ let us say   $(-5b,\bar{y}),$ is formed by blocks of $KMUBEHJ$ and $KMADGI.$ 
Then the complete itinerary of this point has the form $$\left((KMUBEHJ)^{n_1}(KMADGI)^{m_1}\cdots (KMUBEHJ)^{n_k}(KMADGI)^{m_k} \right)^{\infty},$$ for certain integer numbers $n_2,\ldots, n_k$, $m_1,\ldots, m_{k-1}\ge 1$ and
$n_1,m_k\ge 0.$ Then we get $s=6(\sum_{i=1}^k m_i)+7(\sum_{i=1}^k n_i)=:6m+7n.$ 

Now the map $F^s\vert_{[r_{i},r_{i+1}]}$ is non decreasing, has $r_i$ and $r_{i+1}$ as fixed points and, from Lemma \ref{edges}, is piecewise linear with all the linear pieces with slope $0$ or $2^k$ for some $k>0.$ Therefore, there is one and only one more fixed point $q\in (r_{i},r_{i+1})$ such that $F^s$ is constant in a neighborhood of $q.$ The unicity is due to the fact that the plateau $\overline{P_2Q}$ is applied by $F^5$ in $P_7\in \overline{P_1T_2}.$  So if the image of the plateaus goes to a periodic orbit, it is unique. So $(-b-2,-1)$ is also a periodic orbit of period $s.$ The graphic of  $F^s$ on each interval $[r_{i},r_{i+1}]$  essentially 
(we mean, modulus some intervals of constancy, that do not cut the diagonal) looks as in Figure~\ref{Falaq}.

\begin{figure}[H]
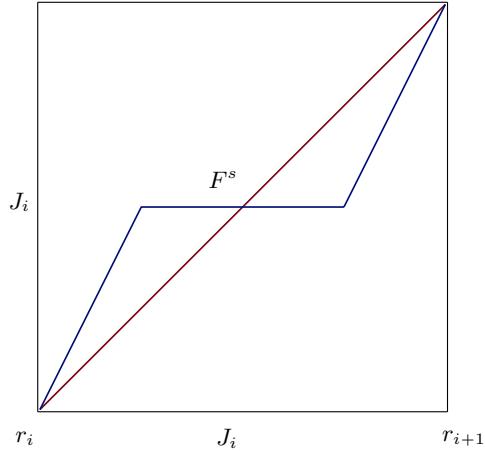

	\footnotesize
	\centering
	\begin{lpic}[l(2mm),r(2mm),t(2mm),b(2mm)]{Falaq(0.25)}
		\lbl[r]{8,100; ${J_{i}}$}
		\lbl[r]{10,-5; ${r_{i}}$}
		\lbl[r]{100,-5; ${J_{i}}$}
		\lbl[r]{210,-5; ${r_{i+1}}$}
		\lbl[r]{100,110; ${F^{s}}$}
	\end{lpic}
	\caption{The graphic of $F^s.$ }\label{Falaq}
\end{figure}

The graphic is as depicted in Figure \ref{Falaq} because the map is non-decreasing, piecewise linear with slopes either 0 or greater than one and the slope at the points $r_i$ is greater than one. Also, from that graphic we see that every point on $\Gamma$ not lying in the repulsive periodic orbit reaches the attractive periodic orbit.

In case that $F$ has a unique periodic orbit, it must be the orbit of the point $(-b-2,-1).$ Also in this case, a similar analysis about the possible periodic itineraries of the plateaus, shows also that when they are periodic points the period must be of the same form $s=6m+7n.$ Denoting now by $J_1,J_2,\ldots, J_s$ the intervals determined by the points of this periodic orbit, the graphic of $F^s$ on each $J_i$ only can cut the diagonal in the points $r_i$ and $r_{i+1}.$ Hence this orbit is semistable, not hyperbolic. Also in this case, every point on $\Gamma$ reaches  this periodic orbit. 
\end{proof}

\subsection{The case $a=-1$ and $-1< b\leq -3/4$}

Before to deal with this range of values of $b,$ we need to introduce the three parametric family of {\it trapezoidal maps} studied in \cite{BMT}. The trapezoidal maps (i.e., maps whose graph is trapezoidal) $T_{X,Y,Z}$ is the family of continuous piecewise affine self maps of $[0,1]$ having a sub-interval $J$  in which the map is constant, with absolute value of the slope on both  sides of~$J$  greater than one and which sends both endpoints to zero. Qualitatively a map of this type looks like the one of forthcoming Figure~\ref{GrafT}.
The family is described by three parameters $(X,Y,Z)\in (0,1)^3$ where $X$ is the inverse of the slope on the left of $J,$ $Y$  is minus the inverse of the slope on the right of $J$ and $Z$ is the length of $J.$  It is proved in the cited paper two basic facts. The first one is that when  we fix $X_0$ and $Y_0$ and consider the uniparametric family $T_{X_0,Y_0,Z}$ then it is a {\it full family} in the sense that all possible unimodal dynamics is represented in  $T_{X_0,Y_0,Z}$ (see Theorem 1 of \cite{BMT}). This fact is proved  using the {\it kneading theory} and showing that the itinerary of the {\it turning point} in the family covers all possible unimodal kneading itineraries. The second fact is that the entropy of $T_{X_0,Y_0,Z}$ monotonically decreases with $Z.$

\begin{figure}[H]
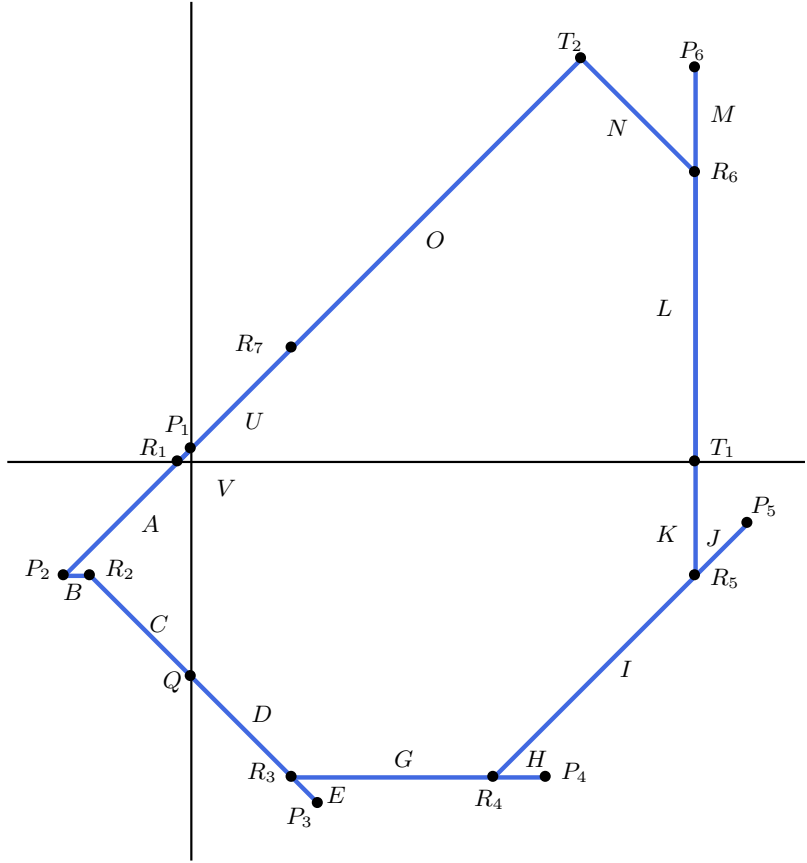

	\footnotesize
	\centering
	\begin{lpic}[l(2mm),r(2mm),t(2mm),b(2mm)]{aamm-cas-3(0.6)}
		\lbl[r]{45,94; $R_1$}
		\lbl[r]{60,85; $V$}
		\lbl[r]{43,77; $A$}
		\lbl[r]{26,62; $B$}
		\lbl[r]{45,55; $C$}
		\lbl[r]{50,99; $P_1$}
		\lbl[r]{19,67; $P_2$}
		\lbl[l]{31,67; $R_2$}
		\lbl[r]{48,42; $Q$}
		\lbl[r]{68,35; $D$}
		\lbl[r]{69,22; $R_3$}
		\lbl[c]{74,12; $P_3$}
		\lbl[l]{80,17;$E$}
		\lbl[l]{95,25;$G$}	
		\lbl[l]{124,25;$H$}
		\lbl[l]{145,45;$I$}		
		
		\lbl[l]{164,74;$J$}
		\lbl[l]{153,75;$K$}		
		\lbl[l]{153,125;$L$}
		\lbl[l]{165,168;$M$}		
		
		\lbl[l]{142,165;$N$}		
		\lbl[l]{102,140;$O$}
		\lbl[l]{62,100;$U$}

		\lbl[c]{116,16; $R_4$}
		\lbl[l]{132,22; $P_4$}
		\lbl[l]{165,65; $R_5$}
		\lbl[l]{165,94; $T_1$}
		\lbl[c]{177,81; $P_5$}
		\lbl[l]{165,155; $R_6$}
		\lbl[c]{161,182; $P_6$}
		\lbl[c]{134,184; $T_2$}
		\lbl[c]{63,117; $R_7$}

	\end{lpic}
	\caption{The graph $\Gamma$ for $a=-1$ and $-1< b\leq -3/4$. }\label{particio-1-3/4}
\end{figure}

For $-1< b\leq -3/4$ the invariant graph is given in Figure \ref {f:3} and we reproduce it here in Figure~\ref{particio-1-3/4} with some additional notation.

In Figure~\ref{particio-1-3/4}, for $i=1,\ldots,6,\,\,P_i=F^{i-1}((0,b+1)).$  Indeed, $P_2=(-b-2,-1),\,\,P_3=(b+2,-3),\,\,P_4=(b+4,2b-1),
\,\,P_5=(-b+4,4b+3),\,\,P_6=(-5b,4b+7).$ Also for $i=1,\ldots,7,\,R_i=F^{i-1}((-b-1,0)).$ We have $R_2=(b,-1),\,\,R_3=(-b,2b-1),\,\,R_4=(-3b,2b-1),\,\,R_5=(-5b,-1),\,\,R_6=(-5b,-4b-1),\,\,R_7=(-b,1).$ Moreover
$Q=(0,b-1),\,\,T_1=(-5b,0)$ and $T_2=F(T_1)=(-5b-1,-4b).$ We have $F(R_7)=F(T_2)=P_2$ and $F(S)=R_3.$

Here the plateau edges are $O\cup U$ and $C.$ Moreover $F(L)=O$ so $F^2(L)=P_2$ and the intervals $O,U,C$ and $L$ are not relevant for finding the entropy of the oriented graph associated to $\Gamma.$ Also for all $b\in (-1,3/4]$ all the coverings between these intervals are fixed, except for~$M.$ The coverings of $M$ depend on the location of $(-9b-8,-8b-7)=P_7=F(P_6).$ Then for $-1<b\leq -3/4$ we have at least the following arrows in the directed graph:

\begin{center}
	{\scriptsize\begin{tikzcd}
			A \arrow[rrrrrrrr, bend left] \arrow[r] & D \arrow[r] & G \arrow[r] & I \arrow[r] & K \arrow[r] & N \arrow[lllll, bend left] \arrow[r] & V \arrow[r] & B \arrow[r] & E \arrow[r] & H \arrow[r] & J \arrow[r] & M
	\end{tikzcd}}
\end{center}

\begin{propo}\label{menys1menys34} Assume $a=-1$ and $b\in(-1,-3/4].$ Then the following holds:
	\begin{itemize} \item[(a)] If $b\in(-1,-8/9],$ $h(F\vert_{\Gamma})=0.$ Furthermore, $F\vert_{\Gamma}$ has two periodic orbits of period $6:$ $\mathcal{P}$ which is the orbit of $P_2=(-b-2,-1)$ that is reached by the two plateaus of $\Gamma,$ and $\mathcal{Q},$ which is the orbit of $\left(-{(7b+16)}/{15},{(8b-1)}/{15}\right)\in A$ that is repulsive.
		\item[(b)]   If $b\in(-8/9,-112/137],$ $h(F\vert_{\Gamma})=0.$ Furthermore, $F\vert_{\Gamma}$ has one periodic orbit of period~$12$ which is the orbit of $P_2$ and it is reached by the two plateaus of $\Gamma,$ and two $6-$periodic repulsive orbits: the previous $\mathcal{Q}$ and  the orbit of $\left(-{(9b+8)}/{9},{1}/{9}\right)\in V=\overline{P_1R_1}.$
		\item[(c)] When $b\in [-13/16,-3/4],\,\,h(F\vert_{\Gamma})>0.$ The two plateaus meet the orbit of $P_2$ which can follow very different itineraries.
		\item[(d)] When $b\in [-112/137,-13/16]$ there exists a subinterval $\Pi\subset \Gamma$ which is invariant by $F^6$ and such that it is visited for all elements of $\Gamma$ except for the points of the repulsive orbit $\mathcal{Q}$ that still is $6-$periodic; the map $F^6\vert_{\Pi}$ is semiconjugated to $T_{1/16,1/8,Z}$ with $Z=\frac{55 b}{16(3b-1)}$ and  there exists $\alpha\in (-112/137,-13/16)$ such that $h(F\vert_{\Gamma})=0$  for $b\in[-112/137,\alpha]$ while $h(F\vert_{\Gamma})>0$ and non-decreasing when $b\in(\alpha,-13/16].$  Moreover the orbit of $P_7=F^5(P_2)\in \Pi$ under $F^6$ runs through all the dynamic situations offered by the maximum of a unimodal application.
		
\end{itemize}\end{propo}

\begin{proof}
	We are going to prove that the entropy is zero for $b\in(-1,-112/137]$. $(a)$ We begin with $b\in(-1,-7/8].$  Then $F(M)\subset U\cup V,$ and therefore,  the entropy of $F$ is less or equal than the entropy that we would achieve assuming that the image of $M$ covers exactly $U$ and $V$ (this is the exact situation when $b=-7/8).$ The graph we obtain with this last assumption is

	\begin{center}
		{\scriptsize\begin{tikzcd}
				A \arrow[rrrrrrrr, bend left] \arrow[r] & D \arrow[r] & G \arrow[r] & I \arrow[r] & K \arrow[r] & N \arrow[lllll, bend left] \arrow[r] & V \arrow[r] & B \arrow[r] & E \arrow[r] & H \arrow[r] & J \arrow[r] & M\arrow[lllll, bend left]
		\end{tikzcd}}
	\end{center}
	
	It turns out that the entropy of this graph  is zero because it does not have linked loops,  see Remark \ref{romaentr} (iii). In consequence $h(F\vert_{\Gamma})=0$ for this range of parameters.
	
	$(b)$ When $b\in (-7/8,-14/17]$ the situation is essentially the same because now $M$ covers also a little portion of $A$ but this portion collapses
	after two iterates. And the same holds for $b\in (-14/17,-112/137]$ while the portion of $A$ covered by $M$ collapses after seven iterates.
	
	Trivial calculations prove the assertions on the behaviors of the plateaus.
	
	$(c)$ For $-13/16\le b\leq -3/4$, we need to refine the partition. Let $P_7=F(-5b,4b+7)=(-9b-8,-8b-7)\in A$ and $P_8=F(P_7)=(17b+14,-16b-15).$ When $b\ge -7/9,\,
	P_8\in E.$ We add $P_7$ to the partition and we denote by $A_1$ the interval $\overline{R_1P_7}\subset A.$ Now we have the following oriented graph of coverings: 
	
	\begin{center}
		{\scriptsize\begin{tikzcd}
				A_1\arrow[r] & D \arrow[r] & G \arrow[r] & I \arrow[r] & K \arrow[r] & N \arrow[lllll, bend left] \arrow[r] & V \arrow[r] & B \arrow[r] & E \arrow[r] & H \arrow[r] & J \arrow[r] & M\arrow[lllll, bend left]\arrow[lllllllllll,bend left]
		\end{tikzcd}}
	\end{center}
	
Notice that the above graph does have connected loops. From Remark \ref{romaentr} (ii) we have that $F|_{\Gamma}$ has positive entropy, that moreover could be easily computed. As an example we detail these computations later, in other similar situations.

	When $b< -7/9,\,\,
	P_8\in D$ and we need to refine again the partition. Let $P_9=F(P_8)=(33b+28,2b-1)\in G,\,\,P_{10}=F(P_9)=(31b+28,36b+27)\in I$ and $P_{11}=F(P_{10})= (-5b, 68b+55).$ When $b\ge -55/68$ we get that $P_{11}\in L$ while $P_{11}\in K$ when $b\in (-111/136, -55/68).$ Adding $P_8,P_9,P_{10} $ to the partition, the intervals $D,G,I$ split in two subintervals. We denote by $D_1=\overline{(0,b-1)P_8},\,\,G_1=\overline{R_3P_9},\,\,I_1=\overline{R_4P_{10}}$ and by $A_2,D_2,G_2$ and $I_2$ the corresponding remaining subintervals of $A,D,G,I.$ With this notation and assuming that 
	$b\ge -55/68$ we get the following oriented graph of coverings that also gives positive entropy.
	
	\begin{center}
		{\scriptsize\begin{tikzcd}
				A_1\arrow[r] & D_1 \arrow[r] & G_1 \arrow[r] & I_1 \arrow[r] & K \arrow[r] & N \arrow[lllll, bend left] \arrow[r] & V \arrow[r] & B \arrow[r] & E \arrow[r] & H \arrow[r] &J \arrow[r] & M\arrow[lllll, bend left]\arrow[lllllllllll,bend left]
		\end{tikzcd}}
	\end{center}
	
	Lastly, when $b\in[-13/16,-55/68),\,\,P_{11}\in K.$ Now
	we need to compute two more iterates. We denote by $P_{12}=F(P_{11})=(-73b-56,64b+55)\in N$ and $ P_{13}=F(P_{12})= (-137b-112,-136b-11)\in A.$ We also denote by $K_1= \overline {P_{11}R_5},\, N_1=\overline{P_{12}R_6}$ and by $K_2,N_2$  the corresponding remaining subintervals of $K,N.$ With this notation we obtain the following oriented graph of coverings that still gives positive entropy.
	
	\begin{center}
		{\scriptsize\begin{tikzcd}
				A_1\arrow[r] & D_1 \arrow[r] & G_1 \arrow[r] & I_1 \arrow[r] & K_1 \arrow[r] & N_1  \arrow[r] & V \arrow[r] & B \arrow[r] & E \arrow[r] & H \arrow[r] & J \arrow[r] & M\arrow[lllll, bend left]\arrow[lllllllllll,bend left]
		\end{tikzcd}}
	\end{center}
	This ends the proof of $(c).$

	$(d)$ For these values of $b$ consider the interval $\Pi=\overline {P_7R_7}=U\cup V\cup A_1$ (recall that $P_7\in A$). We have  $F^6(A_1)=F^5(D_1)=F^4(G_1)=F^3(I_1)=F^2(K_1)
	=F(N_1)\subset \Pi$ and $F^5(B)=F^4(E)=F^3(H)
	=F^2(J)=F(M)\subset \Pi.$ Lastly we have $F(A_2)=D_2\cup E$ and $F^4(D_2)=F^3(G_2)=F^2(I_2)
	=F(K_2)=F(N_2)\subset A_2\cup \Pi.$ Collecting these facts it follows that if $F^n(x)\notin \Pi$ for all $n\ge 0$ it must follow infinitely many times the following loop. 
	
	\begin{center}
		{\scriptsize\begin{tikzcd}
				A_2\arrow[r] & D_2 \arrow[r] & G_2 \arrow[r] & I_2 \arrow[r] & K_2 \arrow[r] & N_2  \arrow[lllll,bend left] 
		\end{tikzcd}}
	\end{center}
	
	The map $F_1\circ F_4^4\circ F_3$ has a repelling fixed point, namely
	$\left(\frac{-7b-16}{15},\frac{8b-1}{15}\right)$ which gives the stated 6 periodic orbit for $F$ and the only points in $\Gamma$ that do not visit $\Pi.$

	The map $F^6$ from $\Pi$ to $\Pi$ is the following.
	Since the points of $\Pi$ write as $(x,x+b+1)$ where $x\in [-9b-8,-b],$  $F^6(x,x+b+1)=(\bar g(x),\bar g(x)+b+1)$ where $$\bar g(x)=\begin{cases} 7b+16(x+1) & \mbox{if $x\in [-9b-8,-b/2-1]$},\\
		-b & \mbox{if $x\in [-b/2-1,-b-1]$},\\
		-9b-8x-8 & \mbox{if $x\in [-b-1,0]$},\\
		-9b-8 & \mbox{if $x\in [0,-b]$}.
		
	\end{cases}
	$$
	
	This map linearly conjugates to the map $\tilde g:[0,1]\longrightarrow[0,1]$  defined by $$\tilde g(x)=\begin{cases} 16x-\frac{16b+13}{b+1} & \mbox{if $x\in [0,\frac{17b+14}{16(b+1)}]$},\\
		1 & \mbox{if $x\in [\frac{17b+14}{16(b+1)},\frac{8b+7}{8(b+1)}]$},\\
		-8x+\frac{9b+8}{b+1} & \mbox{if $x\in [\frac{8b+7}{8(b+1)},\frac{9b+8}{8(b+1)}]$},\\
		0 & \mbox{if $x\in [\frac{9b+8}{8(b+1)},1]$}.
		
	\end{cases}
	$$
	
	Since the last piece of the map is constant we can collapse this last interval and we get the map $g^*:[0,\frac{9b+8}{8(b+1)}]\longrightarrow [0,\frac{9b+8}{8(b+1)}]$ defined by $$g^*(x)=\begin{cases} 16x-\frac{16b+13}{b+1} & \mbox{if $x\in [0,u]$},\\
		\frac{9b+8}{8(b+1)} & \mbox{if $x\in [u,v]$},\\
		-8x+\frac{9b+8}{b+1} & \mbox{if $x\in [v,\frac{9b+8}{8(b+1)}]$},
	\end{cases}
	$$
	where $u=\frac{137b+112}{128(b+1)}$ and $v=\frac{7(9b+8)}{64(b+1)}.$
	We note that because of the collapsing, $\tilde g$ and $g^*$ are only semiconjugated.
	
		\begin{figure}[H]
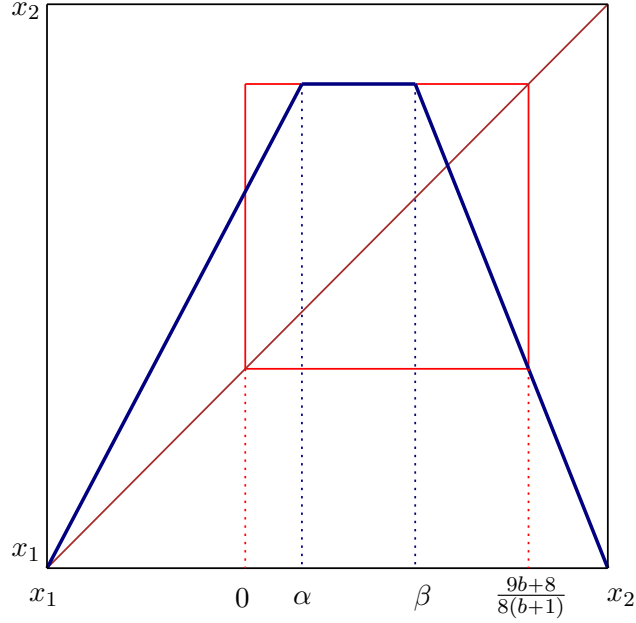

		\centering
		\begin{lpic}[l(2mm),r(2mm),t(2mm),b(2mm)]{grtb(0.40)}
			\lbl[r]{3,10; {$x_1$}}
			\lbl[r]{3,190; {$x_2$}}
			\lbl[c]{4,-5; {$x_1$}}
			\lbl[c]{196,-5; {$x_2$}}
			\lbl[c]{70,-5; {$0$}}
			\lbl[c]{90,-5; {$u$}}
			\lbl[c]{130,-5; {$v$}}
			\lbl[c]{166,-5; {$\frac{9b+8}{8(b+1)}$}}
				\end{lpic}
		\caption{Sketch of the graphic of $\hat {g}(x)$ in blue and the graphic of $g^*(x)$
			inside the red box. The graphic is not to scale.}
		\label{GrafT}
	\end{figure}
	
	Now we extend this map on a certain interval $[x_1,x_2]\supset \left[0,\frac{9b+8}{8(b+1)}\right]$ to get a trapezoidal map $\hat{g}(x):$
	
	$$\hat{g}(x)=\begin{cases} 16x-\frac{(16b+13)}{b+1} & \mbox{if $x\in [x_1,u]$},\\
		\frac{9b+8}{8(b+1)} & \mbox{if $x\in [u,v]$},\\
		-8x+\frac{9b+8}{b+1} & \mbox{if $x\in [v,x_2]$}.
		
	\end{cases}
	$$
	Here $x_1=\frac{16b+13}{15(b+1)}<0$ is the repulsive fixed point of $16x-\frac{(16b+13)}{b+1}$ and $x_2=\frac{119b+107}{120(b+1)}>\frac{9b+8}{8(b+1)}$ satisfies that $\hat{g}(x_2)=x_1.$ See Figure~\ref{GrafT}. We note that since $x_1$ is repulsive, for each $x\in(x_1,x_2)$ there exists  $n$ such that $\hat{g}^n(x)\in (0,\frac{9b+8}{8(b+1)}).$ So the dynamics of $g^*$ can be studied analyzing $\hat{g}.$

	Lastly we scale $\hat{g}$ to get a map from $[0,1]$ to $[0,1],$ getting $g(x):$

	$$g(x)=\begin{cases} 16x, & \mbox{if $x\in [0,\frac{7b+16}{48(1-3b)}]$},\\
		\frac{7b+16}{3(1-3b)} & \mbox{if $x\in [\frac{7b+16}{48(1-3b)},\frac{8-79b}{24(1-3b)}]$},\\
		-8x+8 & \mbox{if $x\in [\frac{8-79b}{24(1-3b)},1]$}.
		
	\end{cases}
	$$
	
	Note that $g=T_{1/16,1/8,Z}$ with $Z=\frac{55b}{16(3b-1)}.$   Hence using item $(i)$ of Lemma \ref{entroo} we know that for  $b\in [-112/137,-13/16],$ 
	$h(F\vert_{\Gamma})=h(T_{1/16,1/8,Z})/6,$
	where $Z=\frac{55b}{16(3b-1)}.$ Since the entropy of
	$T_{1/16,1/8,Z}$  is non increasing in $Z$, and  $Z$ is decreasing with $b,$ we get that $h(F\vert_{\Gamma})$ is nondecreasing in $b.$
	
	Now set $$\alpha =\sup \{b\in [-112/137,-13/16], \mbox{ such that } h(F\vert_{\Gamma})=0\}.$$ It is not difficult to see that $\alpha\in (-112/137,-13/16).$ Moreover from item $(ii)$ of Lemma~\ref{entroo}  it follows that when $b=\alpha,$ $h(F\vert_{\Gamma})=0.$ 
\end{proof}

\subsection{The case $a=-1$ and $-3/4<b<0$}

For this range of values of $b$ there are six different invariant graphs displayed in Figures~\ref{f:4}--\ref{f:9} of the Appendix. To illustrate how to use the approach introduced in Section~\ref{s:prelim} to get the exact entropy of a map, in this section we compute it for some values of $b.$ 
 Denote by $h_1\approx 0.19463$ the logarithm of the positive real root of the polynomial $\lambda^6-\lambda -2$ and by $h_2\approx 0.12639$ the logarithm of the positive real root of the polynomial $\lambda^6-\lambda -1.$ Next proposition summarizes the behavior of the entropy when $b\in (-3/4,0).$ We stress the fact that
the entropy is discontinuous at $b=-1/36.$

 \begin{propo}\label{tercer} Assume $a=-1$ and $b\in (-3/4,0).$ Then the orbit of $(-b-2,-1)$ is $5$-periodic and we denote it by $\mathcal P.$ Also when $b\ge -1/8$ the orbit of $(b,-1)$ is $7$-periodic and we denote it by $\mathcal R.$  Then
 	\begin{itemize}
 	\item[(a)] When $b\in (-3/4,-1/5],$ $h(F\vert_{\Gamma})=h_1.$ Moreover both plateaus go to $\mathcal P.$
 	 \item[(b)] When $b\in[-1/5,-1/9],$ $h(F\vert_{\Gamma})\ge h_2.$  Moreover when $b\in[-1/5,-5/28],$ all plateaus go to $\mathcal P,$ while when $b\in [-1/8,-1/9]$ all plateaus go either to  $\mathcal P,$ or to $\mathcal R.$ When $b\in (-5/28,-1/8)$ the plateau $\overline {SP_4}$ is mapped to to $\mathcal P,$ while the other two plateaus can have different $\omega$-limits.
 	\item[(c)] When $b\in[-1/9,-1/16],$ $h(F\vert_{\Gamma})= h_1.$  Moreover all plateaus go either to $\mathcal P$ or to $\mathcal R.$
 		\item[(d)] When $b\in[-1/16,-1/36),$ $h(F\vert_{\Gamma})\ge  \frac{\ln 2}{6}.$ Moreover all plateaus go either to $\mathcal P$ or to~$\mathcal R.$
 	\item[(e)] When $b\in[-1/36,0),$ $h(F\vert_{\Gamma})=0.$ Moreover all plateaus go either to $\mathcal P$ or to $\mathcal R.$ Also appears two repulsive orbits with periods 5 and 7.
 	\item[(f)] The entropy is discontinuous at $b=-1/36.$
 	\end{itemize}
 \end{propo}

\begin{proof} $(a)$ First we consider the case $b\in (-3/4,-1/4]$ that corresponds to Figure \ref {f:4} in the Appendix. Comparing the graphs of Figures \ref{particio-1-3/4} and \ref{f:4}  we see that the only difference is that the interval which links the points $R_5=(-5b,-1)$ and $P_{5}=(-b+4,4b+3)$ in Figure~\ref{particio-1-3/4} is contained in the fourth quadrant (and named $J$ in Figure~\ref{particio-1-3/4}) while in Figure~\ref{f:4} of the Appendix the interval who links these two points has a part in the fourth quadrant (the interval $\overline{R_5S}$ which we name $J$ again) and the other part which is a plateau contained in the first one, which we name $W$ ($W=\overline{SP_5}$).  So we use the same notation introduced in the previous case adding the points $S$ and the interval $W.$ As before, the intervals $C,O,U$ are plateau edges as the new interval $W.$ We also note that since the interval $L\to O$ and $O$ collapses, it is not rellevant for computing the entropy. The main difference with the previous case is that now $P_7=F(P_6)=P_2$ and therefore, $P_2$ is a $5$-periodic point.  What we get, then, is the following Markov oriented graph:

\begin{center}
{\scriptsize\begin{tikzcd}
	A \arrow[rrrrrrrr, bend left] \arrow[r] & D \arrow[r] & G \arrow[r] & I \arrow[r] & K\arrow[r] & N \arrow[lllll, bend left] \arrow[r] & V \arrow[r] & B \arrow[r] & E \arrow[r] & H \arrow[r] & J \arrow[r] & M\arrow[lllll, bend left]\arrow[lllllllllll,bend left]
\end{tikzcd}}
\end{center}

 According to Definition \ref{d:roma}, $R=\{A,E\}$ is a rome of the above directed graph. Observe that there are only paths of length $6$ connecting $A$ and $E$ with themselves; two paths of length 1 and 8 connecting $A$ with $E$; and a path of length $4$ connecting $E$ with $A$; hence, the associated matrix function $A_R(\lambda)$ that appears in Theorem \ref{rome} is given by
$$
A_R(\lambda)=\begin{pmatrix}
\lambda^{-6} & \lambda^{-1}+\lambda^{-8}\\
\lambda^{-4} & \lambda^{-6}
\end{pmatrix}
$$and, therefore, the characteristic polynomial of the matrix associated with the directed graph is equal to $(-1)^{10}\,\lambda^{12}\,\mathrm{det}(A_R(\lambda)-E)=\lambda^6(\lambda^6-\lambda-2),$ and the entropy is $h_1.$ 

Now consider $b\in [-1/4,-1/5].$ For these values of  $b$ we have to consider the graphs of Figures \ref{f:5} ($b\le -2/9$) and \ref{f:6} $(b\ge -2/9)$ of the Appendix. We only explain the computations for the first case. The second one follows with the obvious adaptations.
We have for $i=1,\ldots,5,\,\, F(P_i)=P_{i+1},\,f(P_6)=P_2, F(T_1)=T_2,\,\,F(X_1)=X_2$ and for $i=1,\ldots,6,\, F(R_i)=R_{i+1}.$ Therefore $P_2$ belongs to a $5$-periodic orbit. Following the orbit of $R_7$ we have $R_8:=(7b,8b+1)=F(R_7) \in \overline{R_1P_2},\,\,R_9:=
(-15b-2,16b+1)=F(R_8)\in\overline{R_3P_{3}},\,\, R_{10}:=(-31b-4,2b-1)=F(R_9)\in\overline{R_4P_{4}}$, $R_{11}:=(-33b-4, -28b-5)=F(R_{10})\in\overline{SP_5}.$ Hence $F(R_{11})=P_6.$

Concerning the plateaus we have the following  collapses : $$\overline{SP_{5}}\twoheadrightarrow P_{6},\,\overline{P_{1}T_2}\twoheadrightarrow P_{2}, \,\overline{R_{2}Q}\twoheadrightarrow R_3 \mbox { and } F^9(R_3)=P_6.$$
Furthermore, $\overline{R_6X_1}\rightarrow\overline{R_{7}X_{2}}\rightarrow\overline{R_{8}P_{2}}\rightarrow\overline{R_{9}P_{3}}\rightarrow\overline{R_{10}P_{4}}\rightarrow\overline{R_{11}P_5}\twoheadrightarrow P_{6},$ and also $\overline{R_6T_1}\rightarrow\overline{R_7T_2}\rightarrow\overline{R_{8}P_{2}}\rightarrow \cdots \rightarrow\overline{R_{11}P_5} \twoheadrightarrow P_{6}.$

We name the rest of the intervals: $A:=\overline{R_{1}R_{8}}$, $ B:=\overline{P_{2}R_{2}}$, $ C:=\overline{QR_3}$, $ D:=\overline{R_{3}R_{9}}$, $ E:=\overline{R_3R_4}$, $ G:=\overline{R_4R_{10}}$, $ H:=\overline{R_4R_{5}}$, $ I:=\overline{R_5S}$, $ J:=\overline{R_5R_6}$, $ K:=\overline{X_1R_7}$, $ L:=\overline{T_1P_{6}}$, $ M:=\overline{P_{1}R_1}.$

The oriented graph for these intervals, which is of Markov type, now is 

\begin{center}
{\scriptsize\begin{tikzcd}
		A \arrow[rrrrrrrr, bend left] \arrow[r] & C \arrow[r] & E \arrow[r] & H \arrow[r] & J \arrow[r] & K \arrow[lllll, bend left] \arrow[r] & M \arrow[r] & B \arrow[r] & D \arrow[r] & G \arrow[r] & I \arrow[r] & L\arrow[lllll, bend left]\arrow[lllllllllll,bend left]
\end{tikzcd}}
\end{center}
and hence its entropy is, as before, $h_1$.

$(b)$ Now we have to consider the graphs of Figures \ref{f:7} ($b\le -1/8$) and \ref{f:8} $(b\ge -1/8)$ of the Appendix. As before we only explain with detail the first case. The second one follows with the natural adaptations.

As usual for $i=1,\ldots,4,\,\, F(P_{i})=P_{i+1},\,\,$  for $i=1,\ldots 6,\,F(R_{i})=R_{i+1}$, for $i=1,2,\, F(X_{i})=X_{i+1},\,\,F(T_i)=T_{i+1}$ and $F(Y_1)=Y_2,\,F(Z_1)=Z_2.$ Moreover $F(P_5)=P_1.$ 

We need to take into account some more points of the orbit of $T_1.$ Namely $T_4=(-9b,10b-1)\in\overline{R_3P_{2}},\,
 T_5:=(-19b,2b-1)\in\overline{R_{4}P_{3}}$ and $T_6:=(-21b,-16b-1)\in\overline{SP_{4}}.$ Note that for $i=1,\ldots,5$ still $F(T_i)=T_{i+1}.$

We consider the intervals $A=\overline{R_1P_{1}}$, $B=\overline{T_3R_{2}}$, $C=\overline{R_{3}T_4}$, $D=\overline{R_4T_{5}}$, $E=\overline{R_5S}$, $G=\overline{T_1P_5}$ and $H=\overline{T_2R_1}.$ With these notations we have that the oriented graph associated to $\Gamma$ has at least the following  coverings:

\begin{center}
{\begin{tikzcd}
		A  \arrow[r] & C \arrow[r] & D \arrow[r] & E \arrow[r] & G \arrow[r] \arrow[llll,bend left] & H \arrow[r] & B\arrow[lllll,bend left].
\end{tikzcd}}
\end{center}

This oriented graph has $\{C\}$ as a rome having two loops of lengths $5$ and $6$ respectively. Then its entropy is the logarithm of the positive root of $\lambda^{-6}+\lambda^{-5}-1.$ That is $h_2.$ This ends the proof of the first part of this item.

Concerning the plateaus, note that when $b\in [-1/5,-5/28],$ $F^{10}(\overline{R_2Q})=P_5\in\mathcal P$ and $F^{4}(\overline{Z_2X_3})=P_5\in\mathcal P.$ When $b\in [-1/8,-1/9],$ we have that $R_8=(7b,-8b-1)$
and $F(R_8)=R_2.$ Therefore $R_2$ is a $7$-periodic point as claimed. Moreover, $F(\overline{Z_2Q})=R_3\in\mathcal R.$

For items $(c),(d)$ and  $(e)$ we have to consider the graph in Figure \ref {f:9}. Note that $F(R_8)=R_2$ and then $R_2$ is a $7$-periodic point. Therefore, in all three cases the orbit of the plateaus goes to the $5$-periodic orbit of $P_1$ or to the $7$-periodic orbit of $R_2.$ There are five different situations 
depending of the orbit of the point $T_3=(9b,-1).$ In general we denote $T_i=F^{i-1}(-5b,0).$

$(c)$ We have $T_4=(-9b,10b-1)\in\overline{R_3P_{2}},$ $T_5=(-19b,2b-1)\in\overline{R_4P_{3}},$
$T_6=(-21b,-16b-1)\in\overline{SP_{4}}.$ Denoting $A=\overline{R_1P_{1}},$ $B=\overline{T_2R_1}$ $C=\overline{R_3T_4},$ $D=\overline{R_4T_5},$ $E=\overline{R_5S},$ $G=\overline{R_6T_1},$ $H=\overline{T_1P_5},$ $I=\overline{YT_2},$ $J=\overline{X_2T_3},$  and $K=\overline{R_2T_3}$ we obtain the following graph:

$$\begin{tikzcd}
	H \arrow[r] \arrow[rd]& A \arrow[r]  & C \arrow[r] & D \arrow[r] & E \arrow[r] \arrow[llll, bend right] & G \arrow[r] & I \arrow[r] & J  \arrow[lllll, bend right]                  \\
	                       & B \arrow[r] & K  \arrow[u]              & & & & &
\end{tikzcd}$$

An accurate analysis of all other covers in the oriented graph, shows that all the entropy is concentrated in the  above graph. As in the previous case $\{C\}$ is a rome having two loops of lenght $6$ and one of length $5.$  So its entropy is the logarithm of the positive root of $2\lambda^{-6}+\lambda^{-5}-1.$ This shows that $h(F\vert_{\Gamma})= h_1$ in this range of parameters.

$(d)$ Using the notation introduced in the proof of statement $(c),$ now we have that  $T_6\in \overline{R_5S}=E$ and we need to compute some more points of the orbit. We have $T_7=(-5b,-36b-1)\in H$ and $T_8=(31b,32b+1).$
When $b\le -1/32$ we have that $T_8\in A.$ Now we define
 $E_1=\overline{R_5T_6},$ 
 $H_1=\overline{T_1T_7},$  and we obtain the following graph of coverings: 

$$\begin{tikzcd}
 E_1 \arrow[r] \arrow[rd]& G \arrow[r]  & I \arrow[r] & J \arrow[r]  & C \arrow[r] &D \arrow[lllll, bend right] \\
	& H_1 \arrow[r] & B  \arrow[r]              &K \arrow[ur]& &
\end{tikzcd}$$ 

Here $\{C\}$ is still a rome having two loops of lenght $6.$
Then the entropy of this subsystem is ${(\ln 2)}/{6}.$

When $b> -1/32$ it is necessary to compute more points of the orbit. Now $T_8\in B,$ $T_9= (-63b-2,-1)\in K,$ $T_{10}=(63b+2,-62b-3)\in C,$ $T_{11}=(125b+4,2b-1)\in D,$ $T_{12}=(123b+4,128b+3)\in E_1,$ $T_{13}=(-5b,252b+7)\in G,$ and
$T_{14}=(-257b-8,248b+7).$ 

When $b\in (-1/32,-7/248]$ we have that $T_{14}\in \overline{R_6Y}.$ In this situation, denoting 
	$C_1=\overline{T_{10}T_{4}},$ $D_1=\overline{T_{11}T_{5}},$ $E_0=\overline{T_{12}T_6},
	$ $G_1=\overline{T_{13}T_{1}},$ $H_1=\overline{T_7T_{1}},$ $I=\overline{YT_2},$ $J=\overline{X_2T_3},$ $B_1=\overline{T_2T_8}$ and $K_1=\overline{T_3T_9}$ we get the following graph 
	
	$$\begin{tikzcd}
		E_0 \arrow[r] \arrow[rd]& G_{1} \arrow[r]  & I \arrow[r] & J \arrow[r]  & C_1 \arrow[r] &D_1 \arrow[lllll, bend right] \\
		& H_1 \arrow[r] & B_1 \arrow[r]              &K_1 \arrow[ur]& &
	\end{tikzcd}$$
that has again entropy ${(\ln 2)}/6.$ This ends the proof of $(d)$ when $b\in (-1/32,-7/248].$

Lastly when $b\in (-7/248,-1/36)$ we have that $T_{14}\in I$ and we need to compute two more points of the orbit of $T_1.$ Now $T_{15}=(9b,-504b-15)\in J,$ $T_{16}=(495b+14,-494b-15)\in\overline {R_3T_{10}}.$
Denoting by $I_1=\overline{T_{2}T_{14}}$ and $J_1=\overline{T_3T_{15}}$ we obtain the graph, with the same entropy as the preceding ones

$$	\begin{tikzcd}
	E_0 \arrow[r] \arrow[rd]& G_1 \arrow[r]  & I_1 \arrow[r] & J_1 \arrow[r]  & C_1 \arrow[r] &D_1 \arrow[lllll, bend right] \\
	& H_1 \arrow[r] & B_1 \arrow[r]              &K_1 \arrow[ur]& &
\end{tikzcd}$$
This ends the proof of $(d).$

$(e)$-$(f)$ For $b\in [-1/36,0)$ we have that $T_7\in G.$ Remember that $F(R_8)=R_2.$ Now we consider the  partition of $\Gamma$ given by all its vertices and also $T_4,T_5$ and $T_6.$  The intervals $\overline{SP_4},\,\overline{X_1Y}$ and $\overline{Z_2Q}$ are plateaus. Note that $F(\overline{SP_4})= P_5\in \mathcal P,\,F(\overline {Z_2Q})= R_3\in \mathcal R$ and $F(\overline{X_1Y})= X_2\rightarrow R_3\in\mathcal R.$

Using the previous notation the resultant intervals are $A,B,G,I,H,J,K$ and $C_0=\overline{QR_3},$ $C_1=\overline{R_3T_4},$ $C_2=\overline {T_4P_{2}},$ $D_0=\overline{R_3R_4},$ $D_1=\overline{R_4T_5},$   $D_2=\overline{T_5P_3},$ $E_0=\overline {R_4R_5},$ $E_1=\overline {R_5T_6},$ $E_2=\overline {T_6S},$  $G_0=\overline {R_5R_6},$
$I_{0}=\overline {R_6R_7},$ $I_1=\overline {R_7X_1},$  $L_0=\overline {R_7Z_1},$ $L_1=\overline {X_2R_8},$ $M=\overline {R_8Z_2}.$ We do not consider
the remaining intervals because they collapse after two iterations. We note that this partition is not a Markov partition. However we can apply the Remark \ref{grafcota} to obtain and upper bound of the entropy of $F.$ Thus we obtain that the entropy of the following graph is an upper bound for $h(F\vert_{\Gamma}):$

\begin{tikzcd}
	 C_2 \arrow[d]& & & C_0 \arrow[r] & D_0 \arrow[r]  & E_0 \arrow[rd] & I_0 \arrow[r]  & L_0 \arrow[r]&M \arrow[lllll, bend right] \\
	 D_2\arrow[d] &  H\arrow[r]\arrow[d] & A\arrow[llu]\arrow[rd] \arrow[ru]       & &L_1\arrow[lu] &I_1\arrow[l] &G_0\arrow[u] & & \\
	  E_2\arrow[ru]\arrow[rrrrrr,bend right,dashed] &B\arrow[r] & K\arrow[r] &C_1\arrow[r] & D_1\arrow[r]&E_1\arrow[r,dashed]&G\arrow[lu]\arrow[r] &I\arrow[r] & J\arrow[lllll, bend left] &&
\end{tikzcd}

All the covers are fully realized by $F$ except $E_2\longrightarrow G$ and $E_1\longrightarrow G,$ and for this reason we plot dashed arrows in these two cases. From Remark~\ref{grafcota}, this graph (taking into account all type of arrows) has zero entropy and hence $h(F\vert_{\Gamma})=0.$ 
Observe that, in particular,  we have encountered a discontinuity in the entropy at $b=-1/36,$ proving item~$(f).$ 

We also note that in this oriented graph there are two loops with lengths $7$ and $5$ that force the periodic points $({113b}/{15},-{(112b+15)}/{15})\in C_0$ and $(-5b,{(-36b+1)}/{7})\in H.$ From Lemma \ref{edges} both periodic orbits are repulsive. 
\end{proof}

\subsection{The case $a=-1$ and $0\leq b\leq 1$}

We separate the analysis in different subcases.
\subsubsection{The case $a=-1$ and $0\leq b\leq 1/2.$}

Next proposition shows that when $b$ belongs to the interval $[0,1/2]$ the entropy of the map is always zero, and describes the dynamics of all points. 

\begin{propo}\label{quart} Assume that $a=-1,\,b\in[0,1/2]$ and let $p=(-b,2b-1)\in Q_3$ be the fixed point of $F.$ 
	\begin{itemize}
		\item [(a)] If $0\leq b < 3/16$ then  $p$ is the image of one plateau and $F$ has the two 5-periodic orbits:
		 $$\!\!\!\!\mathcal{P}=\{(-5b,-4b+1),(9b-2,-1),(-9b+2,10b-3),(-19b+4,2b-1),(-21b+4,-16b+3)\}$$ that contains the image of the other plateau and $\mathcal{Q}$ given by
		 {\footnotesize $$ \left\{\left(\frac{b+2}{7},\frac{6b-9}{7}\right),\left(\frac{4-5b}{7},2b-1\right),\left(\frac{4-19b}{7},\frac{16b-3}{7}\right),
		 	\left(-5b,\frac{4b+1}{7}\right),\left(\frac{31b-8}{7},\frac{-32b-1}{7}\right)\right\}.
		$$} When $b=3/16$ both orbits coincide.
Moreover for $0\leq b \le 3/16$ and each $(x,y)\in\Gamma\setminus \{\mathcal Q\}$ there exists $n\in\N$ such that $F^n(x,y)\in \mathcal{P}\cup \{p\}.$

		\item [(b)] If $3/16< b < 4/15$ then for all $(x,y)\in\R^2,$  $F^{22}(x,y)=p.$
		\item [(c)] If $4/15\le b\leq 1/2$ then  $p$ is the image of one plateau and $F$ has the two 4-periodic orbits:
$$\mathcal{P}=\{(-5b,-4b+1),(9b-2,-8b+1),(17b-4,2b-1),(15b-4,20b-5)\}$$ that contains the image of the other plateau, and 
$$\mathcal{Q}=\left\{\left(\frac{15b-8}{3},-4b+1\right),\left(\frac{2-3b}{3},\frac{6b-5}{3}\right),\left(\frac{4-9b}{3},2b-1\right),\left(\frac{4-15b}{3},\frac{1}{3}\right)\right\}$$
or 
$$\mathcal{Q}=\left\{\left(-b,\frac{6b-5}{3}\right),\left(\frac{2-3b}{3},\frac{6b-5}{3}\right),\left(\frac{4-9b}{3},2b-1\right),\left(\frac{3b-4}{3},\frac{1}{3}\right)\right\}$$
depending on whether $4/15<b<4/9$ or $4/9\leq b \leq 1/2,$ respectively. Furthermore, for each $(x,y)\in\Gamma\setminus \{\mathcal Q\}$ there exists $n\in\N$ such that $F^n(x,y)\in \mathcal{P}\cup \{p\}.$
\end{itemize}
\end{propo}

\begin{proof}
	To prove the proposition we need the graphs given in the Appendix. 
Concerning item $(a)$ we have the four different graphs given in Figures \ref{f:10}--\ref{f:13}. We are going to prove the result in one case; for the other ones the same argument works. Consider the graph given in Figure \ref{f:13}, for $1/6<b\leq 3/16.$ For $i=1,\ldots,5$ we denote by $A_i$ the path in $\Gamma$ joining $p$ with $P_i=F^{i-1}(P_1),$ where $P_1=(-5b,-4b+1).$
Direct computations show that for $i=1,\ldots,4,\,\, F(A_i)=A_{i+1}$ and $F(A_5)=A_1.$ Now we consider the action of $F^5$ on $A_1.$ During the transition we have some plateaus. In particular, $F$ sends the beginning of $A_1$ and the end of $A_5$ to $p$ and $P_1$ respectively. Hence, $F^5$ restricted to $A_1$ is constant at the beginning and also at the end of $A_1,$ 
and the graphic of $F^5$ on $A_1$ is as in Figure~\ref{Grafic1}.
\begin{figure}[H]
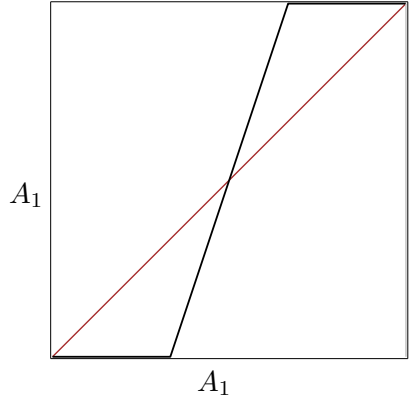

 	\centering
	\begin{lpic}[l(2mm),r(2mm),t(2mm),b(2mm)]{G1(0.25)}
			\lbl[c]{90,-10; {$A_1$}}
			
			\lbl[c]{-10,90; {$A_1$}}
				
		\end{lpic}
\caption{Graphic of $F^5$ on $A_1$ when $1/6 < b\leq 3/16$.}
\label{Grafic1}
 \end{figure}
The three points on the diagonal are $p$ (the fixed point of $F$); the point $P_1,$ which is in~$\mathcal{P},$ and  is an attractive periodic orbit; and it appears another fixed point $q$  that corresponds with a $5g$-periodic orbit  $\mathcal{Q},$ which is repulsive. Clearly, every point $x\ne q$ in $A_1,$ tends either to $p$ or to $P_1,$ with a finite number of iterations.

Statement $(b)$ is a direct consequence of Corollary \ref{c:c16nou}. Notice that this case corresponds precisely to the situation considered in Figure~\ref{f:14} where the graph $\Gamma$ reduces to be the point~$p.$

$(c)$ Lastly for the remaining values of $b$, $4/15<b\leq 1/2,$ we proceed in a similar way. It can be verified that now the announced  orbit $\mathcal{P}$ is, in fact, $4$-periodic.
We choose the range $1/3<b\leq 1/2$ to illustrate one of the cases. Looking at Figure \ref{f:19} of the Appendix we call $A_1,A_2,A_3$ and $A_4$ the paths in $\Gamma$ joining the fixed point $p$ with the points of the periodic orbit as before. Then we get that the graphic of $F^4$ restricted to $A_1$ is exactly the one of Figure \ref{Grafic1}. Now, for $F,$ we have an attractor fixed point, an attractive $4$-periodic orbit and a repulsive $4$-periodic orbit.
\end{proof}

\subsubsection{The case $a=-1$ and $1/2<b\leq 2/3.$}
When $b>1/2$ the point $(-b,2b-1)$ is no longer a fixed point and the new  fixed point of $F$ is $p:=\left(-(2+b)/5,(2b-1)/{5}\right)$ and it is located in $Q_2$ and does not belong to $\Gamma.$

When $1/2<b\leq 2/3$ the invariant graph which contains all the dynamics of $F$ is the one given in Figure \ref{f:20} of the Appendix. With the notation of this figure we have that for $i=1,2,3,\,\,F(P_{i})=P_{i+1}$ and $F(P_4)=P_1;$ for $i=1,2,\,\,  F(R_{i})=R_{i+1},\,F(X_i)=X_{i+1}$ and $F(R_{3})=F(R_1); F(Y_1)=Y_2,\,\,F(Y_2)=F(Q)=R_2,\,\,F(Z)=R_1$ and $F(S)=P_{1}.$ Concerning the point $X_{3}$ we see that $F(X_{3})=(-3b+2,4b-3)\in\overline{QP_2}.$

 The plateaus are $\overline{R_3Z}\twoheadrightarrow R_1$, $\overline{Y_2Q}\twoheadrightarrow R_2$, $\overline{P_4S}\twoheadrightarrow P_1.$ Also $\overline{R_3Y_1}\rightarrow \overline{R_1Y_2}\twoheadrightarrow R_2.$
 
 Let  $A=\overline{SX_2}\cup \overline{X_2Z},$ $B=\overline{P_{1}X_3}\cup \overline{X_3R_1},$ $C=\overline{P_2Q}$ and $D=\overline{P_3R_2}.$ Furthermore, $M=\overline{R_2Y_2},$  $N=\overline{R_2Y_1}$ and $P=\overline{R_3Y_2}.$ The oriented graph associated to this partition is of Markov type and it is given by:

$$\begin{tikzcd}
	A \arrow[r] & B \arrow[r] \arrow[rd] & C \arrow[r] & D \arrow[lll, bend right] &                          \\
	&                        & M \arrow[r] & N \arrow[r]               & P \arrow[ll, bend right]
\end{tikzcd}$$

From Remark \ref{romaentr} (iii) this graph has zero entropy and we get that the map itself has zero entropy.
Next proposition describes the dynamics of $F.$
\begin{propo}\label{cinc}
	Assume that $a=-1$ and $1/2<b\leq 2/3.$ Then $h(F\vert_{\Gamma})=0.$  Let $p$ be the fixed point of $F.$ Then $F$ has two $4$-periodic orbits
	$$\mathcal{P}=\left\{\left(-b-2,-1\right),\left(b+2,-3\right),\left(b+4,2b-1\right),\left(-b+4,5\right)\right\}$$ that captures one plateau,
	and 
	$$\mathcal{Q}_1=\left\{\left(\frac{b-4}{3},\frac{4b-1}{3}\right),\left(\frac{25b}{3},-1\right),\left(\frac{5b-2}{3},\frac{-2b-1}{3}\right),\left(\frac{7b-4}{3},2b-1\right)\right\}$$
	when $b\ge 4/7$ or
	$$\mathcal{Q}_2=\left\{\left(\frac{-13b+4}{3},\frac{4b-1}{3}\right),\left(3b-2,\frac{-14b+5}{3}\right),\left(\frac{5b-2}{3},\frac{-2b-1}{3}\right),\left(\frac{7b-4}{3},2b-1 \right)\right\}$$
	when $b\leq 4/7,$ and two $3$-periodic orbits
	$$\mathcal{R}=\left\{\left(-b,2b-1\right),\left(-b,-2b+1\right),\left(3b-2,-2b+1\right)\right\},$$ that captures the other two plateaus,
	and 
	$$\mathcal{S}=\left\{\left(\frac{b-2}{3},\frac{2b-1}{3}\right),\left(-b,\frac{2b-1}{3}\right),\left(\frac{b-2}{3},\frac{-2b+1}{3}\right)\right\}.$$
	Moreover, when $b\ge 4/7$ (respectively $b\leq 4/7$) for each $(x,y)\in\Gamma$  there exists $n\in\N$ such that $F^n(x,y)\in \mathcal{P}\cup  \mathcal{R}\cup\mathcal Q_1\cup\mathcal S$ (respectively $F^n(x,y)\in \mathcal{P}\cup  \mathcal{R}\cup\mathcal Q_2\cup\mathcal S$). 
\end{propo}
\begin{proof}
	
The fact that $h(F\vert_{\Gamma})=0$ is proved in the comments before the statement of the proposition. Also the statements about the beaviour of the plateaus are simple computations. To see the rest of the statement we have to consider the graph given in Figure \ref{f:20} and the partition introduced in the above comments. We notice that the periodic orbits $\mathcal{P}$ and $\mathcal{R}$  are formed for some vertices of the graph: $\mathcal{P}=\{P_{1},P_{2},P_{3},P_{4}\}$ and $\mathcal{R}=\{R_1,R_2,R_3\}.$  

Looking at the oriented graph, it can be seen that if a point in $\Gamma$ satisfies that its iterates never meet $\mathcal{P}\cup \mathcal{R},$ then  after some iterates it has to follow one of the two loops $ABCD$ or $MNP.$ Since from Proposition \ref{edges} these loops have associated only one periodic orbit (which must be $\mathcal Q_1$ or $\mathcal Q_2$ and $\mathcal S$) that is repulsive, we obtain the desired result. 
\end{proof}

\subsubsection{The case $a=-1$ and $2/3<b\leq 5/7.$}

For this range of values of $b$ we have to consider the graph  $\Gamma$ given in Figure \ref{f:A}. 

We notice that the periodic orbit $\mathcal{P}$ introduced in the above proposition still survives when $2/3<b\leq 5/7$, while the orbit $\mathcal{R}$ is transformed in  $$\mathcal{X}=\big\{X_5=(3b-2,1-2b),\, X_6=(5b-4,2b-1),\, X_7=(-7b+4,4b-3)\big\}.$$

As usual the notation of the points in Figure \ref{f:A} has dynamical meaning. That is, for $i=1,2,3,$ $F(P_i)=P_{i+1}$ and $F(P_4)=P_1.$ For $i=1,\ldots,6,\,\,F(X_i)=X_{i+1}$ and $F(X_7)=X_5=F(W).$ For $i=1,\ldots,7,\,\,F(R_i)=R_{i+1}.$ For $i=1,2,\,\,F(Z_i)=Z_{i+1}$ and $F(Z_3)=R_7.$ $F(Y_1)=Y_2, F(T_1)=T_2$ and $F(T_2)=F(Y_2)=X_3.$ Lastly $F(Q)=Z_2$ and $F(S)=P_1.$ Note that $R_{8}=(29b-20,2b-1),$ and its position depends on the sign of $29b-20.$

There are five plateaus. The first one, $\overline{SP_4},$ is captured by the orbit $\mathcal P.$ The second one, $\overline{WX_4},$ is captured by the orbit~$\mathcal X.$ The third one, $\overline{T_2X_2},$ is also  captured by the orbit~$\mathcal X.$ The two remaining plateaus $\overline {R_6Z_3}$ and $\overline {QZ_1}$ are captured by the orbit of $R_8$ and their behaviors will depend on $b.$ We also observe that there are other intervals in the partition that collapse to a point after some iterates. Namely  $$\overline{R_4Q}\rightarrow \overline{R_5Z_2}\rightarrow \overline{Z_3R_6}\twoheadrightarrow R_7\,,\,\overline{X_2X_5}\rightarrow \overline{X_3X_6}\rightarrow \overline{X_4X_7}\twoheadrightarrow X_5$$ and   $\overline{X_7Y_1}\rightarrow \overline{X_2X_5}\cup \overline{X_2Y_2}$ and $\overline{X_4T_1}
\rightarrow \overline{X_5X_2}\cup \overline{X_2T_2}.$

	\begin{propo}\label{sis} Consider $a=-1$ and $2/3<b\leq 5/7,$ and let $p$ be the fixed point of $F.$ Then the following holds:
		
		\begin{itemize}
			\item [(a)] 	For $b\in (2/3,603/874)$ the map $F$ has zero entropy. Moreover the plateaus are absorbed by $\mathcal P,\mathcal X$ and in certain cases by an additional  periodic orbit.
			\item [(b)] For $b\in [563/816,5/7]$ the map has positive entropy.
			
			\item[(c)] For $b\in [603/874,563/816]$ there exists a subinterval $X\subset \Gamma$ invariant by $F^7$ such that any point of $\Gamma$ except a 3-periodic orbit, a 7- periodic orbit, a 4-periodic orbit and the preimages of these orbits, visit $X;$ the map  $F^7\vert_X$ is semi-conjugated to the trapezoidal map $T_{1/16,1/16,Z}$ with $Z= \frac{45-60b}{48b-29}$ and there exists $\beta\in(603/874,563/816)$ such that
			$h(F\vert_{\Gamma})=0$ for $b\in(603/874,\beta]$ while  
			 $h(F\vert_{\Gamma})>0$ and non-decreasing when $b\in(\beta,563/816].$  
		\end{itemize}
		
	\end{propo}
	\begin{proof}
		To prove item $(a)$ first we consider the subcase $b\le 20/29.$ We consider the Figure \ref{f:A} with the partition given by the marked points. As usual we do not consider in the analysis of the oriented graph the intervals of the partition that collapse after some iterations. These intervals are described in the comments before the statement of the Proposition.

		We denote the rest of the intervals as follows: $A=\overline{SR_2},$  $G=\overline{R_2W},$  $B_2=\overline{P_{1}X_1},$  $B_1=\overline{X_1R_3},$  $H=\overline{R_3X_5},$ $C_2=\overline{P_2R_7},$  $C_1=\overline{R_7X_5},$    $D=\overline{X_2R_4},$ $E=\overline{Z_1X_6},$   $I=\overline{Z_2Y_1},$  $K=\overline{Z_3Y_2},$ $L=\overline{T_2X_3},$ $M=\overline{X_3T_1},$  $O=\overline{X_4T_2},$  $V_1=\overline{R_1X_6},$ $V_2=\overline{R_1P_3},$ $U=\overline{X_3R_5}$ and  $N=\overline{X_4R_6}.$

		We have that $R_8=(29b-20,2b-1)\in V_1,$  and then $F(C_2)$ contains $V_2$ and a subset of $V_1$ and $F(C_1)$ is a subset of $V_1.$ Hence $C_1,C_2$ do not cover $V_1.$ We are going to assume that $C_1,C_2$ cover $V_1$ and we will see that the corresponding oriented graph has zero entropy. This graph looks like:
		
		{\footnotesize
			\begin{tikzcd}
				A \arrow[r] \arrow[d] & B_2 \arrow[r] \arrow[rd] & C_2 \arrow[r] \arrow[rd, dashed] & V_2 \arrow[lll, bend right] &               &               &             &             &                                       &             &             &                          \\
				B_1 \arrow[d]           &                          & C_1 \arrow[r, dashed]            & V_1 \arrow[r]               & G \arrow[r] & H \arrow[r] & E \arrow[r] & I \arrow[r] & K \arrow[llllll, bend left] \arrow[d] &  &  &  \\
				D \arrow[d] & &  &  &  & &  & & L\arrow[d] &  &  \\
				U \arrow[r]             & N \arrow[ruu]            &                                  &                             &               &               &             &             &                                 M\arrow[r]      & O\arrow[lu]            &             &                         
		\end{tikzcd}}
		
		\bigskip
		Taking in consideration the rome $\{A,C_1,L\}$ and applying Theorem \ref{rome} we obtain that this graph has zero entropy. So $F$ has zero entropy. 
		
		From the comments before this proposition we know that the orbit of $R_8$ determines the behavior of the plateaus. Moreover, since $R_8\in V_1$ we see from the graph that  its orbit either is captured for some plateau which forces that: either $R_8$ is periodic; or it is absorbed by the orbit $\mathcal X$; or it is absorbed by  the 3-periodic orbit forced by the loop $LMO.$ In any case either $R_8$ is periodic or it is absorbed by $\mathcal X.$ This ends the proof of the first statement in this case.

		Assume now that $b\in (20/29,603/874].$ In this situation, $R_8\in V_2$ and we need to consider some points of its orbit. Some computations give $R_9=(27b-20,28b-19)\in A ,\,\,
		R_{10}=(38-55b,-1)\in B_1,\,\,R_{11}=(38-55b,37-54b)\in D,\,\,R_{12}=(-b,75-108b)\in U ,\,\,R_{13}=(109b-76,108b-75)\in N,\,\,R_{14}=(150-217b,218b-151)\in C_1 ,\,\,R_{15}=(300-435b,2b-1)\in V_1
		,\,\,R_{16}=(433b-300,301-436)\in G
		,\,\,R_{17}=(3b-2,870b-601)\in H
		,\,\,R_{18}=(-867b+598,874b-603)\in\overline{Z_1R_4}.$ Note that  $\overline{Z_1R_4}$ is a collapsing interval and $F^3(\overline{Z_1R_4})=R_7.$ In particular, $R_8$ belongs to a $14$-periodic orbit.
		
		We consider the associated partition and rename the intervals in the following way 
		$A_0=\overline{R_9R_2},$ 
		$A_1=\overline{R_9S},$ 
		$G_0=\overline{R_2R_{16}},$ 
		$G_1=\overline{R_{16}W},$  $B_2=\overline{P_1X_1},$  $B_1=\overline{X_1R_{10}},$ $B_0=\overline{R_{10}R_{3}},$ 
		$H_0=\overline{R_{3}R_{17}},$
		$H_1=\overline{R_{17}X_5},$
		$C_2=\overline{P_{2}R_{7}},$  $C_1=\overline{R_{7}R_{14}},$ 
		$C_0=\overline{R_{14}X_{5}},$   $D_0=\overline{R_{4}R_{11}},$
		$D_1=\overline{R_{11}X_{2}},$ $E=\overline{Z_{1}X_{6}}$   
		$I=\overline{Z_{2}Y_{1}},$  
		$K=\overline{Z_3Y_2},$ 
		$L=\overline{X_3T_2},$ 
		$M=\overline{X_3T_1},$  
		$O=\overline{X_4T_2},$ 
		$V_0=\overline{X_{6}R_{15}},$ $V_1=\overline{R_{15}R_{1}},$ $V_2=\overline{R_{1}R_{8}},$
		$V_3=\overline{R_8P_{3}},$ 
		$U_0=\overline{R_5R_{12}},$
		$U_1=\overline{R_{12}X_{3}},$
		$N_0=\overline{R_{6}R_{13}}$
		and  $N_1=\overline{R_{13}X_{4}}.$ 
		
		In our situation, $F^3(V_1)=F^2(G_0)=F(H_0)\subset \overline {Z_1R_4}$ that collapses. So we collapse $V_1,G_0$ and $H_0.$ Thus, after these collapses, the map is Markov and the associated oriented graph is 
		
		\hspace{-0.5cm}{\scriptsize \begin{tikzcd}
				A_0 \arrow[r] & B_0  \arrow[r]                 & D_0 \arrow[r]        & U_0  \arrow[r]             & N_0 \arrow[r]                & C_1 \arrow[r]                 &  V_2 \arrow[llllll, bend right]           &             &                                        &             &             &                          \\
				A_{1} \arrow[r] \arrow[d] & B_{1} \arrow[r] & D_{1} \arrow[r] & U_{1} \arrow[r]  & N_{1} \arrow[r] & C_{0} \arrow[r] & V_{0} \arrow[r] & G_1 \arrow[r] & H_1 \arrow[r] & E \arrow[r]  & I \arrow[r] & K \arrow[d]\arrow[llllll, bend right] \arrow[llllllu]\\
				B_2 \arrow[d] \arrow [rrrrru]             &                  &                           &                 &                 &                 &             &             &                                        &             &             &    L\arrow[d]                      \\
				C_2 \arrow[d]                &                  &                           &                 &                 &                 &             &             &                                        &             &             &  M\arrow[d]                        \\
				V_3 \arrow[uuu,bend left]                &    &                           &                 &                 &                 &             &             &                                        &             &             &       O\arrow [uu, bend left]             
		\end{tikzcd}}
		
A rome for this graph is $\{C_0,C_1,A_1,L\}$ and all the elements of the rome has one and only one loop associated. From Remark \ref{romaentr} (iii) it follows that it has zero entropy. This ends the proof of $(a).$

		$(b)$ First we consider the case $b\in [563/816,1201/1740].$ We keep the notation used in the previous item. Now $R_{18}\in E$ and we need to compute some  more iterates. Let $R_{19}=(-7b+4,-1740+1201)\in I,\,\,R_{20}=(1747b-1206,1734b-1197)\in K$ and $R_{21}=(-3481b+2402,3482b-2403)\in \overline {X_{3}R_{14}}.$ We also need to split $E$ as $E_0=\overline {R_{18}Z_{1}}$ and $E_1=\overline {X_6R_{18}}$ ; $I$ as $I_0=\overline {Z_2R_{19}}$ and $I_1=\overline {R_{19}Y_1}.$ Lastly $K$ splits as $K_0=\overline {Z_3R_{20}}$ and $K_1=\overline {R_{20}Y_2}.$ With this notation we have the following graph of coverings
		
		\hspace{-0.5cm}{\scriptsize \begin{tikzcd}
				A_0 \arrow[r] & B_0  \arrow[r]                 & D_0 \arrow[r]        & U_0  \arrow[r]             & N_0 \arrow[r]                & C_1 \arrow[r]  \arrow[d]              &  V_1 \arrow[r]           &   G_0    \arrow[r]      &  H_0\arrow[r]                                       & E_0\arrow[r] &  I_0\arrow[r]      &    K_0\arrow[llllll,bend right]                       \\
				&  &  &   &  & V_2 \arrow[lllllu, bend left] &  &  & &  && 
		\end{tikzcd}}
A rome for this graph is $\{C_1\}$ and has two different loops. From Remark \ref{romaentr} (i) it follows that this graph has positive entropy.

		When $b\in (1201/1740,301/436],\,\,R_{19}\in\overline {X_7Y_1}$ and we do not need to consider $R_{20}$ and successive images.  Moreover, we do not split the intervals $I$ and $K.$ Thus we obtain the same graph that in the previous case changing $I_0$ and $K_0$ by $I$ and $K.$ Moreover, in this case, the orbit $\mathcal X$ captures $R_8.$
		
		When $b\in (301/436,38/55],\,\, R_{16}\in \overline{WX_7}$ and we  do not need to consider the orbit after $R_{16}.$  Then we do not split the intervals $G,H,E,I,K.$ Moreover the last graph of coverings, also works putting $G,H,E,I,K$ instead $G_0,H_0,E_0,I_0,K_0.$
		Moreover, in this case, the orbit $\mathcal X$ also captures $R_8.$
		
		Lastly when $b\in (38/55,5/7],\,\, R_{10}\in B_2.$ Then we do not need to consider $R_i$ for $i\ge 11.$ Then rename the intervals as:
		$A_0=\overline{R_9R_2},$ 
		$B_1=\overline{R_3X_1},$  
		$H=\overline{R_3X_5},$
		$C_1=\overline{X_5P_7},$ 
		$D=\overline{X_2R_4},$
		$E=\overline{Z_1X_6}$   $I=\overline{Y_1Z_2},$  $K=\overline{Z_3Y_2},$  $V_1=\overline{X_6R_1},$ $V_2=\overline{R_1R_{8}},$
		$U=\overline{X_3R_5}$
		$N=\overline{X_4R_6}$
		and we get the following oriented graph of coverings
		
		{\scriptsize \begin{tikzcd}
				A_0 \arrow[r] & B_1  \arrow[r]                 & D \arrow[r]        & U \arrow[r]             & N \arrow[r]                & C_1 \arrow[r]  \arrow[d]              &  V_1 \arrow[r]           &   G    \arrow[r]      &  H\arrow[r]                                       & E\arrow[r] &  I\arrow[r]      &    K\arrow[llllll,bend right]                       \\
				&  &  &   &  & V_2 \arrow[lllllu, bend left] &  &  & &  && 
		\end{tikzcd}}\newline
		It also gives positive entropy. In this case the orbit of $R_8$ depends on $b.$

						$(c)$ Assume $b\in [603/874,563/816]$ and let $X=\overline{R_{15}R_{8}}.$ The points of $X$ can be written as $(x,2b-1)$ where $x\in [300-435b,29b-20]$.  It occurs that $F^7(X)\subset X.$ More precisely,
						 we have $F^7(x,2b-1)=(g(x),2b-1)$ where $ g:[300-435b,29b-20]\longrightarrow [300-435b,29b-20]$ is defined by
						
						$$g(x)=\begin{cases} 4-3b+16x &\mbox{if $x\in [300-435b,2b-3/2] $,}\\29b-20 & \mbox{if $x\in [2b-3/2,0] $,}\\ 29b-16x-20 & \mbox{if $x\in [0,29b-20] $.}	
						\end{cases}	$$

An inspection of the dynamics out $X$ shows that the only points in $\Gamma$ that do not visit $X$ are the  4-periodic orbit that is given by the fixed point of $F_1\circ F_4\circ F_3 \circ F_2,$ the 3-periodic orbit given by the fixed point of $F_3\circ F_2^2,$  and the 7-periodic orbit given by the fixed point of $F_3\circ F_2^2\circ F_4 \circ F_2^2\circ F_4,$ all of them repulsive, and its preimages.
							
The rest of the proof of item $(c)$ follows the same steps as the proof of item $(d)$ of Proposition \ref{menys1menys34} and we omit it.\end{proof}

\subsubsection{The case $a=-1$ and $5/7<b\leq 1$.}
We are going to prove that for these values of $b$ the map always has positive entropy (except for $b=1$). From the Appendix we see that we have twelve different graphs, and among them, there are some ones which are really complex.  We resume all results in the next three  Propositions.

\begin{propo}\label{sotasota1}
	When $5/7<b\leq 4/5$ the map has positive entropy.
\end{propo}
\begin{proof}
We take into account two different cases.

	 Consider first $5/7<b\leq 3/4.$ For these values of $b$ we have three different invariant graphs, the ones given in Figures \ref{f:B}--\ref{f:D}. We take into account three more points of the orbit of $R_6,$	
 $R_7=(15b-10,-14b+9)\in \overline{X_5P_2},R_8=(29b-20,2b-1)\in
 \overline{R_1P_{3}}$ and $R_9=(27b-20,28b-19)\in\overline{Y_{2}P_{4}}.$ 	
	
	Consider the intervals $A=\overline{Y_{2}R_{2}},$ $B=\overline{Y_3X_1},$ $C=\overline{X_1R_3},$ $D=\overline{X_5R_7},$ $E=\overline{R_1R_8},$ $I=\overline{X_2R_4},$ $G=\overline{X_3R_5}$ and  $H=\overline{X_4R_6}.$ The coverings between these sets are illustrated in the following diagram:
	
	\begin{equation}\label{e:diagramap22}
	\begin{tikzcd}
		A \arrow[r] \arrow[rd] & B \arrow[r] & D \arrow[r] & E  \arrow[lll, bend right] &               \\
		& C \arrow[r] & I \arrow[r] & G \arrow[r]                & H \arrow[llu]
	\end{tikzcd}
\end{equation}

This particular graph has two loops of length $4$ and $7$ respectively with a rome consisting of a unique interval, say $\{A\}.$ Hence, by Theorem \ref{rome}, its entropy is the logarithm of the greatest root of $\lambda^{7}-\lambda^{3}-1,$ that is $h_3\approx 0.12943.$  
If we made the oriented graph of the entire partition explicit in the corresponding figures, these two loops would appear, and perhaps others. Therefore, the entropy of F is greater than or equal than $h_3.$

\begin{figure}[H]
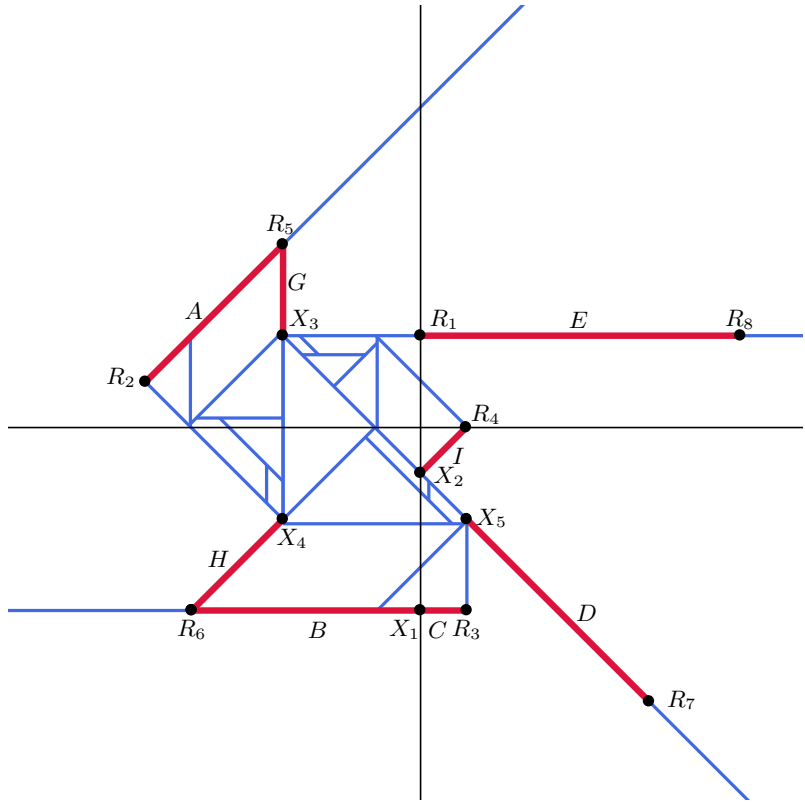

	\footnotesize
	\centering
	\begin{lpic}[l(2mm),r(2mm),t(2mm),b(2mm)]{grafpropo22b(0.55)}
		\lbl[c]{68,141; $R_{5}$}
		\lbl[r]{33,104; $R_{2}$}
		\lbl[l]{43,43; $R_{6}$}
		\lbl[c]{98,43; $X_1$}
		\lbl[c]{113,43; $R_{3}$}
		\lbl[c]{119,70; $X_5$}
		\lbl[c]{165,26; $R_{7}$}
		\lbl[l]{104,118; $R_{1}$}
		\lbl[c]{179,118; $R_{8}$}
		\lbl[l]{105,80; $X_{2}$}
		\lbl[l]{114,95; $R_{4}$}
		\lbl[l]{70,118; $X_{3}$}
		
		\lbl[l]{67,65; $X_{4}$}
		
		\lbl[c]{47,120; $A$}
		\lbl[c]{77,43; $B$}
		\lbl[c]{106,43; $C$}
		\lbl[c]{142,47; $D$}
		\lbl[c]{140,118; $E$}
		\lbl[c]{111,85; $I$}
		\lbl[c]{72,127; $G$}
		\lbl[c]{53,60; $H$}
	\end{lpic}\caption{Partial view of the graph  $\Gamma$ for $a=-1$ and $3/4<  b\leq 154/205$, given in Figure~\ref{f:E}, whose dynamics includes the one described in the diagram 
		\eqref{e:diagramap22}. The same dynamics occurs for the cases and graphs in Figures~\ref{f:F}--\ref{f:L}, with the analogous points $R_i, X_i$ and sets $A$--$H$. 
	}\label{f:figPropo22}
	
\end{figure}

Next, consider $3/4<b\leq 4/5.$ We proceed in the same manner, getting a lower bound of the entropy of $F.$  This case comprehends the $8$ different graphs given in Figures \ref{f:E}--\ref{f:L}. In all the cases we consider the following intervals: $A=\overline{R_5R_2},$ $B=\overline{R_6X_1},$ $C=\overline{X_1R_3},$ $D=\overline{X_5R_7},$
$E=\overline{R_{1}R_{8}},$ $I=\overline{X_2R_{4}},$ $G=\overline{X_3R_{5}},$ $H=\overline{X_{4}R_{6}}$, see Figure \ref{f:figPropo22}. With this notation we get that the dynamics of $F\vert_{\Gamma}$ contains  the oriented graph \eqref{e:diagramap22}, which has positive entropy.\end{proof}

\begin{propo}\label{sota1}
When  $a=-1$ and $4/5<b< 1$  the map $F\vert_{\Gamma}$ has positive entropy.

\end{propo}
\begin{proof}
	We consider the graph given in Figure \ref{f:M} and we will follow the orbit of $P_{11}.$ To prove the statement we consider three different cases.

		Assume first that $4/5<b\le 6/7.$  We add some new points of the $R_1$ and $X_1$ orbits namely, $X_7=(3b-4,4b-3)\in \overline{R_2R_5},$ $X_8=(-7b+6,-1)\in \overline{X_1R_3},\,\,R_7=(2-b,2b-3)\in\overline{X_5P_2},\,\,R_8=(4-3b,2b-1)\in \overline {X_6P_3}$ and $R_9=(2-b,2b-3)\in \overline{R_5S}.$  
		
		Now we define: $G=\overline{R_{2}T},$ $A_0=\overline{R_5X_7},$ $B=\overline{R_6X_1},$ $D=\overline{X_2X_5},$ $V=\overline{X_5R_7},$ $H=\overline{R_3X_5},$ $U=\overline{R_4X_6},$ $L=\overline{X_3R_1},$
		$E=\overline{X_6R_8}.$ The covers between these sets are:
		
		\begin{center}
			\begin{tikzcd}
				A_0 \arrow[r] & B \arrow[r] \arrow[rd] & D \arrow[r] & L\arrow[r]              & G \arrow[r] & H \arrow[r] & U\arrow[llllll, bend right]\\
				&                        & V \arrow[r] & E \arrow[lllu, bend left=49] &               &             &                       
			\end{tikzcd}
		\end{center}
This graph has $\{B\}$ as a rome with two loops of length $7$ and $4.$ Therefore $h(F\vert_{\Gamma})\ge h_3\approx 0.12943$ where recall that $h_3$ is the logarithm of the positive real root of $\lambda^7-\lambda^3-1.$
	
	Assume now that $6/7< b \le 12/13.$ We shall follow the same steps as the previous case. Now $X_8\in B$ and we have to consider  $X_9=(7b-6,-6b+5)\in D$ and  $X_{10}=(13b-12,2b-1)\in L.$ 
	
The new considered sets are: $A=\overline{R_{2}X_7},$
	 $A_0=\overline{R_5X_7},$  $B_0=\overline{R_6X_8},$ 	 
	  $C=\overline{X_1R_{3}},$ $D_0=\overline{X_5X_9},$ $V=\overline{X_5R_7},$ $O=\overline{X_2Q},$ $N=\overline{R_5X_3},$ 	  
	   $J=\overline{R_1X_6},$ $E=\overline{X_6R_8},$ $I=\overline{X_4R_6}$. With this notation we obtain the following oriented graph of covers for these intervals
	 \begin{center}
	 	\begin{tikzcd}
	 		A_0 \arrow[r] & B_0 \arrow[r] \arrow[rd] & D_0 \arrow[r] & J\arrow[r]              & A \arrow[r] & C \arrow[r] & O\arrow[r]&N \arrow[r]&I \arrow[lllllld] \\
	 		&                        & V \arrow[r] & E \arrow[lllu, bend left=49] &               &             &      &
	 		&                       
	 	\end{tikzcd}
	 \end{center}
 
Now $\{B_0\}$ is a rome having two loops of lengths $11$ and $4.$
	 Therefore the entropy of $F$ is greater than or equal to $h_4,$ the logarithm of the positive root of $\lambda^{11}-\lambda^7-1.$  

Lastly assume that $12/13<b<1.$ Now, the main difference with the previous case is that $X_{10}\in J.$ Then we name $J_0=\overline {X_{10}X_6},$ and we rename $J=\overline {R_1X_{10}},\,\,B=\overline {X_1X_8},\,\,D=\overline {X_2X_9},\,\,V=\overline {X_5R_7}$  and we maintain the previous notation for the rest of the intervals.
Note  that $F(J)\subset A, F(B)=D,$ while $F(A)=C \cup B$ and $F(D)= J\cup L.$ We note that the positive orbit of $X_6$ at some point leaves $A\cup B\cup J\cup D.$ If this does not happen we would have that the orbit by $F^4= F_4\circ F_3\circ F_2\circ F_1$ of $X_6$ is entirely contained in $J$ which has no sense because $\left(F_4\circ F_3\circ F_2\circ F_1\right)\vert_{J}$ is linear, expansive and its fixed point does not belong to~$J.$ Therefore there exists $n>5$ such that $F^n(X_6)\notin A\cup B\cup J\cup D.$ Clearly, either $n=4k+2$ and $F^n(X_6)\in C$ for some $k\ge 1$ or 
$n=4k$ and $F^n(X_6)\in L$ for some $k\ge 2.$ 
Now consider the first case. We add to the graph the points of the orbit of $X_6$ until $F^{4k+1}(X_6).$
This gives a natural partition of the intervals $J,B$ and $D$ into $k$ subintervals and $A$ into $k+1$ subintervals. We write for $i\in \{1,\ldots,k-1\},$ $J_i=\overline{F^{4i}(X_6)
	F^{4(i+1)}(X_6)},$ $A_i=\overline{F^{4i+1}(X_6)
	F^{4(i-1)+1}(X_6)},$ $B_i=\overline{F^{4i+2}(X_6)F^{4(i-1)+2}(X_6)},$ $D_i=\overline{F^{4i+3}(X_6)F^{4(i-1)+3}(X_6)}$ and $J_k=\overline{F^{4k}(X_6)R_1},$ $A_k=\overline{F^{4k+1}(X_6)F^{4k-3}(X_6)},$ $B_k=\overline{F^{4k-2}(X_6)X_1},$  $D_k=\overline{F^{4k-1}(X_6)X_2}$ and $A_{k+1}=\overline{F^{4k+1}(X_6)R_2}.$ With this  notation we have

	\begin{center}
	\begin{tikzcd}
		 &                        & V \arrow[r] & E \arrow[llld, bend right=49] &               &             &  \\       
		A_0 \arrow[r] & B_0 \arrow[r] \arrow[ru] & D_0 \arrow[r] & J_0\arrow[llld]& & & &\\	    
		A_1 \arrow[r] & B_1 \arrow[r] & D_1 \arrow[r] & J_1\arrow[llld]& & & &\\
		
			\vdots \arrow[r] &\vdots \arrow[r]  & \vdots \arrow[r] & \vdots \arrow[llld]& & & &\\	    
		A_k \arrow[r] & B_k \arrow[r] & D_k \arrow[r] & L\arrow[r]              & G \arrow[r] & H \arrow[r] & U\arrow[lllllluuu]\\             
	\end{tikzcd}
\end{center}

This graph has a rome $\{A_0\}$ with two loops of lengths $4$ and $4k+7,$ that gives positive entropy. This ends the proof for the first case.

In the second case we add to the graph the points of the orbit of $X_{6}$ until $F^{4k-1}(X_{6}).$
This gives a natural partition of the intervals $A,B$ and $D$  into $k$ subintervals and a partition of $J$ into $k-1$ subintervals. We write for 
$i\in \{1,\ldots,k-1\},$ $A_i=\overline{F^{4i+1}(X_6)F^{4(i-1)+1}(X_6)},\, B_i=\overline{F^{4i+2}(X_6)F^{4(i-1)+2}(X_6)},$ $D_i=
  \overline{F^{4i+3}(X_6)F^{4(i-1)+3}(X_6)}$ and $ A_k=\overline{R_2F^{4k-3}(X_6)},$ $B_k=\overline{F^{4k-2}(X_6)X_1}$ and $D_k=\overline{F^{4k-1}(X_6)X_2}.$ On the other hand, the interval $J$ splits into $k-1$ subintervals that we will write for $i\in \{1,\ldots,k-2\},$ $J_i=\overline{F^{4i}(X_6)F^{4(i+1)}(X_6)}$ and $
J_{k-1}=\overline {R_1F^{4(k-1)}(X_6)}.$ With this notation we obtain the following oriented graph:
	\begin{center}
	\begin{tikzcd}
		&                        & V \arrow[r] & E \arrow[llld, bend right=49] &               &             &  \\       
		A_0 \arrow[r] & B_0 \arrow[r] \arrow[ru] & D_0 \arrow[r] & J_0\arrow[llld]& & & &\\	    
		A_1 \arrow[r] & B_1 \arrow[r] & D_1 \arrow[r] & J_1\arrow[llld]& & & &\\
		
		\vdots \arrow[r] &\vdots \arrow[r]  & \vdots \arrow[r] & \vdots \arrow[llld]& & & &\\	    
		A_k \arrow[r] & C \arrow[r] & O \arrow[r] & N\arrow[r]              & I \arrow[lluuuu] &  & \\             
	\end{tikzcd}
\end{center}

This graph has also a rome $\{A_0\}$ with two loops of lengths $4$ and $4k+7,$ that gives positive entropy. This ends the proof for the second case and ends also the proof of the proposition.~\end{proof}

\begin{nota}\label{cont1} In the proof of the last item of the above proposition, it is not difficult to see that $k\to \infty$ when $b\to 1^-.$ This fact, together with accurate computations, and the use of Remark \ref{grafcota} allow us to prove $\lim\limits_{b\to 1^-}h(F\vert_{\Gamma})=0.$  We do not give the details here, to do  not unnecessarily prolong the work.  In fact, a similar result can be proved when $b$ tends to 1 from the other side and when $b$ tends to $8.$ See Remark \ref{cont2} and Remark \ref{cont3}.
\end{nota}

\begin{propo}\label{exac1} Assume that $a=-1$ and $b=1.$ Then:
\begin{itemize}
	\item[(a)] The map has zero entropy.
	\item[(b)] The map $F$ has two $4$-periodic orbits
	$$\mathcal{P}=\left\{(-3,-1),(3,-3),(5,1),(3,5)\right\}\,\text{and}\,\,\mathcal{X}=\left\{(-1,1),(-1,-1),(1,-1),(1,1)\right\}$$ that absorb the three plateaus 
	and one $3$-periodic orbit
$$	\mathcal{S}=\left\{\left(-\frac{1}{3},-\frac{1}{3}\right),\left(-\frac{1}{3},\frac{1}{3}\right),\left(-1,\frac{1}{3}\right)\right\}.$$
Furthermore, for each $(x,y)\in\Gamma$ there exists $n\in\N$ such that $F^n(x,y)\in \mathcal{P}\cup \mathcal{X}\cup \mathcal{S}.$
\end{itemize}
\end{propo}

\begin{proof}
	Consider the graph in Figure~\ref{f:22}, which is the one corresponding to $b=1$ with the defined partition. We begin by listing the intervals that collapse under the action of $F,$ that is $\overline{SP_4}\twoheadrightarrow P_1\in \mathcal{P},$ $\overline{R_2X_4}\twoheadrightarrow X_5\in \mathcal{X},$ $\overline{R_1X_3}\rightarrow \overline{R_2X_4}\twoheadrightarrow X_5\in \mathcal{X}$ and $\overline{X_1X_5}\rightarrow \overline{X_2X_6}\twoheadrightarrow X_3\in \mathcal{X},$ where the points $R_i$'s and $X_i$'s are defined below.

\begin{figure}[H]
	\centering
	
	\begin{lpic}[l(2mm),r(2mm),t(2mm),b(2mm)]{biguala1-v2(0.50)}	
		\lbl[l]{140,178; $P_4$}
		\lbl[r]{75,122; $S$}
		\lbl[r]{65,113; ${A}$}
		\lbl[r]{55,103; $X_3$}
		\lbl[r]{44,90; ${B}$}
		\lbl[r]{34,82; $R_2$}
			\lbl[c]{14,49; $P_1$}
	\lbl[r]{38,57; ${C}$}
		\lbl[c]{58,49; $X_4$}
		\lbl[r]{71,57; ${D}$}
		\lbl[c]{81,49; $X_1$}
		\lbl[l]{101,58; $X_5$}
		\lbl[r]{122,40; ${G}$}
		\lbl[l]{48,80; $Q$}
			\lbl[r]{62,86; ${K}$}
			\lbl[r]{60,69; ${L}$}
			\lbl[r]{70,92; ${M}$}
			\lbl[c]{78,84; $X_2$}
			\lbl[r]{94,69; ${H}$}
			\lbl[r]{73,65; ${N}$}
			\lbl[r]{76,102; $R_{1}$}
				\lbl[r]{102,102; $X_6$}
		\lbl[r]{138,102; ${I}$}
			\lbl[r]{88,102; ${J}$}
		\lbl[l]{143,14; $P_2$}
				\lbl[c]{184,104;  $P_{3}$}
	\end{lpic}
	\caption{The graph $\Gamma$ for $a=-1$ and $b=1$. Here $S=(0, 2)$, $X_3=(-1,1)$, $R_{2}=(-2, 0)$, $P_1=(-3, -1)$, $X_4=(-1, -1)$, $X_1=(
		0, -1)$, $X_5=(1, -1)$, $Q=(-1,0)$, $X_2=(0,0)$, $R_{1}=(0,1)$, $X_6=(1, 1)$, $P_{2}=(3,-3)$, $P_{3}=(5,1)$, $P_{4}=(3,5).$}\label{f:22} 
\end{figure}
Studying the action of $F$ on these intervals we get the following coverings:

\begin{center}
\begin{tikzcd}
	A \arrow[r] & C \arrow[r] & G \arrow[r] & I \arrow[lll, bend right]
\end{tikzcd}\qquad
\begin{tikzcd}
	L \arrow[r] & H \arrow[r]           & J \arrow[r]                           & B \arrow[r]              & D \arrow[lll, bend right] \\
	& N \arrow[u] \arrow[r] & M \arrow[r] \arrow[llu, bend left=49] & K \arrow[ll, bend right] &                          
\end{tikzcd}
\end{center}

Now the set $\{A,H,N\}$ is a rome and each of its elements has one and only one loop associated. Therefore from Remark \ref{romaentr} (iii)  it follows that  this oriented graph has zero entropy.

Easy computations show that the loop $ACGIA$ gives two fourth periodic orbits $\mathcal P$ and $\mathcal X,$ the first one attractive and the second one repulsive. Notice that the second one coincides with the orbit associated to the loop $HJBDH.$ 

Then  if a orbit does not intersect $\mathcal P\cup \mathcal X,$ it has to follow the loop $NMK$ infinitely times.  From  Lemma \ref{edges} this gives the announced three periodic orbit which is repulsive. 
\end{proof}

\subsection{Case $a=-1$ and $1<b< 2$}

From the Appendix we see that we have four different graphs (Figures \ref{f:21a}--\ref{f:21d}) for different values of $b$ beteween 1 and 2. We observe that all of them have a subgraph which is invariant by $F,$ the one given in the next  Figure \ref{vaja}. 
Apart from the natural vertices, we added in this figure some more points that we need for the computations. Namely, $R_3\in \overline{Z_2P_2},\,\,R_4\in \overline {Z_3P_3}$ from the orbit of $R_1$, $R_0\in\overline{Y_1Z_2}$ a preimage of $R_1,$ and $X_0\in\overline{X_3Z_1}$ a preimage of $X_1.$ Note also that $X_5=(2-b,2b-3)\in \overline{X_2Z_2}.$

We have some intervals which after some iterates reduce to a single point, among them:
$$\overline{P_{4}S}\twoheadrightarrow P_{1},\,\overline{Z_1X_0} \rightarrow  \overline{Z_2X_1} \rightarrow \overline{Z_3X_2}\twoheadrightarrow X_3,\, \overline{Y_1R_0}\rightarrow \overline{Y_2R_1}\rightarrow \overline{Y_3R_2}\twoheadrightarrow R_3.$$

 \begin{figure}[H]
 	\footnotesize
 	\centering
 	
 	\begin{lpic}[l(2mm),r(2mm),t(2mm),b(2mm)]{figesp(0.40)}
 		\lbl[l]{138,184; $P_4$}
 		\lbl[r]{75,127; $S$}
 		\lbl[r]{59,110; $Y_2$}
 		\lbl[r]{48,99; $X_3$}
 		\lbl[r]{27,79; $Z_1$}
 		\lbl[r]{21,71; $R_2$}
 		\lbl[c]{8,51; $P_{1}$}
 		\lbl[c]{38,51; $Y_3$}
 		\lbl[c]{61,51; $X_4$}
 		\lbl[c]{81,51; $X_1$}
 		\lbl[l]{106,55; $Z_2$}
 		\lbl[l]{148,10; $P_{2}$}
 		\lbl[l]{80,82; $X_2$}
 		\lbl[l]{86,74; $Y_1$}
 		\lbl[c]{190,112;  $P_{3}$}
 		\lbl[c]{108,112; $Z_3$}
 		\lbl[l]{78,112; { $R_1$}}
 		\lbl[l]{117,45; $R_3$}
 		\lbl[l]{126,112; $R_4$}
 		\lbl[l]{92,69; $R_0$}
 		\lbl[l]{32,91; $X_0$}
 	\end{lpic}
 	
 	\caption{External part of the graph $\Gamma$ for $a=-1$ and $1< b\leq 2$. Here $S=(0, b +1)$, $Y_{2}=(b -2, 2 b -1)$, $X_{3}=(-b, 1)$, $Z_1=(-b -1, 0)$, $R_2=(-2 b, -b +1)$, $Y_3=(
 		-3 b +2, -1)$, $X_4=(b -2, -1)$, $X_1=(0, -1)$, $Z_2=(b, -1)$, $Y_1=(b -1, 0)$, $X_2=(0, b -1)$, $Z_3=(b, 
 		2 b -1)$, $R_1=(0, 2 b -1)$,
 		$P_{1}=(-b -2, -1)$, $P_{2}=(b +2, -3)$, $P_{3}=(b +4, 2 b -1)$, $P_{4}=(-b +4, 5)$,
 		$R_3=(3b-2,-2b+1)$, $R_4=(5b-4,2b-1)$, $R_{0}=(b/2, b/2 - 1)$ and $X_{0}=(-b/2 - 1, b/2)$.}\label{vaja}
 \end{figure}

We introduce some notation. Set $A=\overline{X_3X_0},$ $B=\overline{X_4X_1},$ $C=\overline{Z_2R_0},$ $D=\overline{Z_3R_1},$ $A_0=\overline{X_3Y_2},$ $B_0=\overline{Y_3X_4},$ $C_0=\overline{Z_2R_3},$ $D_0=\overline{Z_3R_4},$ $E=\overline{Z_1R_2},$ $G=\overline{Y_1X_2}.$

Notice that $F(A)=B$ and $F(C)=D$ while $F(B)\subset C\cup\overline{X_2R_0}$ and $F(D)=A\cup\overline {R_2X_0}.$ 

\begin{propo} \label{b12} Assume that $a=-1$ and $b\in (1,2).$ Then $h(F\vert_{\Gamma})>0.$ \end{propo}

\begin{proof}
First we consider the case when $b\in (1,4/3).$ In this situation, $X_5\in C.$ The proof follows the same ideas used to prove the third case in Proposition \ref{sota1}. We will pursue the orbit of $X_3.$ Note that this orbit cannot be contained in $A\cup B\cup C \cup D,$ because in this situation the orbit by $F^4=F_1\circ F_4\circ F_3\circ F_2$ of $X_3$ is entirely contained in $A.$ This is not possible because the map $F^4$ restricted to $A$ is linear, expansive and its fixed point does not belong to $A.$ So there exists $n$ be such that 
$F^n(X_3)\notin A\cup B\cup C \cup D.$ Note that there are only two possibilities: either $n=4k$  or $n=4k+2$ for some $k\ge 1.$ In the first case, $F^{4k}(X_3)\in\overline {R_2X_0}.$ In the second one, $F^{4k+2}(X_3)\in\overline {R_0X_2}.$

We investigate the first case. We introduce in the partition the orbit of $X_3$ until $F^{4k-1}(X_3).$ This gives the points $F^4(X_3),\ldots ,F^{4k-4}(X_3)\in A,$ $
F^5(X_3),\ldots ,F^{4k-3}(X_3)\in B,$ $F^2(X_3),\ldots ,F^{4k-2}(X_3)\in C$ and $F^3(X_3),\ldots ,F^{4k-1}(X_3)\in D.$ In this way we obtain $k$ subintervals of $A$ that we denote for $0<i<k$ as $A_i=\overline {F^{4(i-1)}(X_3)F^{4i}(X_3)}$ and $A_k=\overline{F^{4(k-1)}(X_3)X_0}.$ With the same spirit we denote for $0<i<k$ the following intervals
$B_i=\overline {F^{4(i-1)+1}(X)F^{4i+1}(X_3)}$ and $B_k=\overline{F^{4(k-1)+1}(X_3)X_1}.$ In the interval $C$ we will consider the intervals $C_1=\overline{Z_2F^2(X_3)}$ and for $i=2,\ldots, k,\,\,C_i=\overline{F^{4(i-2)+2}(X_3)F^{4(i-1)+2}(X_3)}$. Lastly, in the interval $D$ we will consider the subintervals
$D_1=\overline{Z_3F^3(X_3)}$ and for $i=2,\ldots, k,\,\,D_i=\overline{F^{4(i-2)+3}(X_3)F^{4(i-1)+3}(X_3)}$. With this notation we obtain the following graph:

\begin{tikzcd}	
	C_1\arrow[r]&  D_1\arrow[r]& A_1\arrow[r]         & B_1\arrow[llld]&  & & & & & & & \\
	C_2\arrow[r]&  D_2\arrow[r]& A_2\arrow[r]         &B_2\arrow[r]& \ldots\arrow[r] &C_k\arrow[r] &D_k\arrow[r] &A_k\arrow[r] &B_k \arrow[lllld]& & &  	\\
	  & & & &G\arrow[r] &    A_0\arrow[r]&    B_0\arrow[r] \arrow[lllllluu]  & C_0\arrow[r]& D_0\arrow[lll, bend left]    & &                                                                               
\end{tikzcd}

This graph has $\{B_0\}$ as a rome and two loops, one of length $4k+3$ and the other with length 4. From Remark \ref{romaentr} (i) this gives positive entropy  and ends the proof in this case.

In the second case, we add the orbit of $X_3$ until $F^{4k+1}(X_3).$ In this way, each of the intervals
$A,B,C,D$ is subdivided in $k+1$ subintervals that we number in the same spirit of the previous case. Thus we obtain the following graph:

\hspace{-0.7cm}\begin{tikzcd}	
	C_1\arrow[r]&  D_1\arrow[r]& A_1\arrow[r]         & B_1\arrow[llld]&  & & & & & & & \\
	C_2\arrow[r]&  D_2\arrow[r]& A_2\arrow[r]         &B_2\arrow[r]& \ldots\arrow[r] &B_k\arrow[r]&C_{k+1}\arrow[r] &D_{k+1}\arrow[r] &E \arrow[llld] && &   	\\
	& & & & &    C_0\arrow[r]&    D_0\arrow[r]   & A_0\arrow[r]& B_0\arrow[lll, bend left]\arrow[lllllllluu]    & &                                                                               
\end{tikzcd} 

Also from Remark \ref{romaentr} (i) this is also a graph with positive entropy because $\{B_0\}$ is a rome and it has two cycles with lengths $4(k+1)+3$ and $4.$ This ends the proof of this case.

When $b\in (4/3,2)$ we have that $X_5\in\overline {X_2R_0}$ and we obtain the following oriented graph of coverings 

\hspace{3cm}\begin{tikzcd}	
	A_0\arrow[r]&  B_0\arrow[r]\arrow[d]& C_0\arrow[r]         & D_0\arrow[lll,bend right ]  \\ &
	C\arrow[r]&  D\arrow[r]& E\arrow[lu]                                                                                      
\end{tikzcd}

\noindent which, also from Remark \ref{romaentr} (i), gives positive entropy.
\end{proof}

\begin{nota}\label{cont2} Similarly that  in Remark \ref{cont1}, in this case, we can conclude that  $ \lim\limits_{b\to 1^+} h(F\vert_{\Gamma})=0$. \end{nota}

\subsection{The case $a=-1$ and $2\leq b \leq 4$}
From the Appendix we see that we have two different graphs depending on whether $2\leq b<3$ (Figure \ref{ff:22}) or $3\leq b \leq 4$ (Figure \ref{f:23}). The only difference between these two cases is that in the first one the point $P_{4}=(-b+4,5)$ is in the right hand side of $Y_2$ (and hence it generates the interval $\overline{P_{4}Y_{2}}$) while in the second one $P_{4}\in \overline{SY_2}.$ Since the part of the straight line $y=x+b+1$ contained in the first quadrant collapses to the point $P_1,$ this difference will not influence the computation of the entropy. 

In addition, we also note that $F^2(\overline{Z_2X_1})=F(\overline{X_2Z_3})=X_3$ and, since $X_4\in\overline{Z_2X_1}$ the interval $\overline{X_3Z_1}$ also collapses.

The rest of the graph is covered with: $A=\overline{Z_3P_3},$ $B=\overline{Y_2Z_3},$ $C=\overline{SX_3},$ $D=\overline{Z_1P_1},$ $E=\overline{P_{1}X_1},$ $I=\overline{Y_1Z_2},$ $G=\overline{Z_2P_2},$ $H=\overline{Y_1X_2}.$ From this we get the following oriented graph:

\begin{center}
\begin{tikzcd}
	&             &                                   & G \arrow[llld, bend right] &             &                           \\
	A \arrow[r] & C \arrow[r] & E \arrow[r] \arrow[ru] \arrow[rd] & I \arrow[r]                & B \arrow[r] & D \arrow[llu, bend right] \\
	&             &                                   & H \arrow[llu, bend left]   &             &                          
\end{tikzcd}
\end{center}

It turns out that this graph is of Markov type. We have that the set $\{C\}$ is a rome. Three periodic orbits, of periods $3,4,7$ pass through $C.$  Hence we have proved the following:
\begin{propo}\label{b24}
	Let $h_5\approx 0.25344$ be the logarithm of the greatest real root of $\lambda^7-\lambda^4-\lambda^3-1.$ When $a=-1$ and $2\leq b\leq 4,\,\, h(F\vert_{\Gamma})=h_5.$
\end{propo}

\subsection{The case $a=-1$ and $4< b< 8$}

\begin{propo}\label{b48}
Assume that $a=-1$ and $b\in (4,8).$ Then $h(F\vert_{\Gamma})>0.$
\end{propo}
\begin{proof} We follow with the same graph, the one of Figure \ref{f:23}. To perform the analysis we have to add more points to the partition. We add these points in the Figure \ref{f:espe3}.

	 \begin{figure}[H]
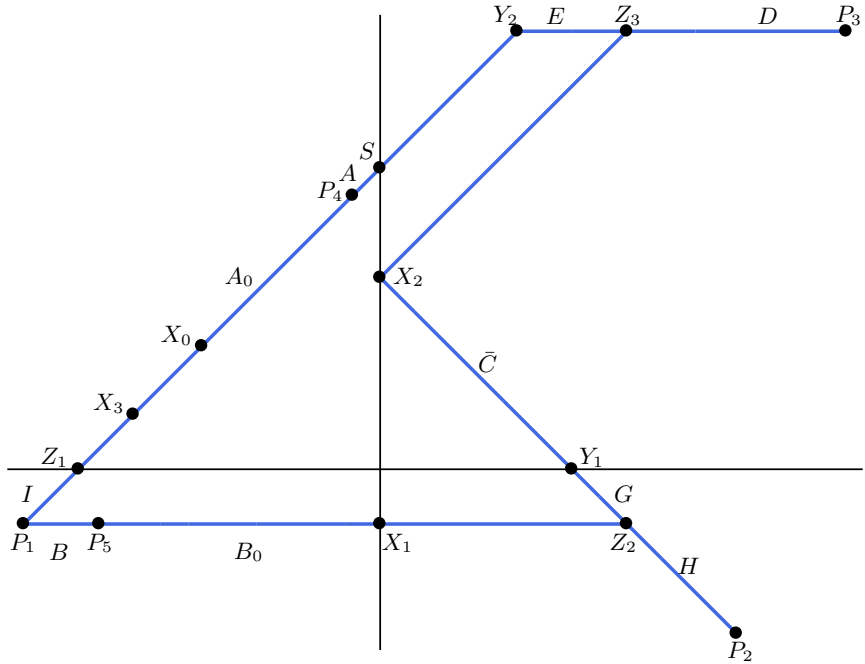
\label{graficespecial3}
	 	\footnotesize
	 	\centering
	 	
	 	\begin{lpic}[l(2mm),r(2mm),t(2mm),b(2mm)]{am-especial-3d(0.60)}
	 		\lbl[r]{85,138; $S$}
	 		\lbl[r]{81,133; $A$}
	 		\lbl[r]{17,70; $Z_1$}
	 		\lbl[c]{8,62; $I$}
	 		\lbl[c]{7,51; $P_1$}
	 		\lbl[c]{90,51; $X_1$}
	 	\lbl[c]{15,49; $B$}
	 		\lbl[c]{24,51; $P_5$}
	 		\lbl[c]{140,51; $Z_2$}
	 		\lbl[l]{163,27; $P_2$}
	 		\lbl[l]{89,110; $X_2$}
	 		\lbl[c]{133,70; $Y_1$}
	 		\lbl[c]{190,168;  $P_3$}
	 		\lbl[c]{141,168; $Z_3$}
	 		\lbl[c]{114,168;  $Y_2$}
	 		\lbl[c]{125,168;  $E$}
	 		\lbl[r]{78,129; $P_{4}$}
	 		\lbl[c]{55,110; $A_0$}
	 		\lbl[c]{26,82; $X_3$}
	 		\lbl[c]{57,49; $B_{0}$}
	 		\lbl[c]{41,97;  $X_0$}
	 		\lbl[c]{172,168;  $D$}
	 		\lbl[c]{110,91;  $\overline{C}$}
	 		\lbl[l]{138,62;  $G$}
	 		\lbl[l]{152,46;  ${H}$}
	 	\end{lpic}
	 	
	 	\caption{The graph $\Gamma$ for $a=-1$ and $4<b<8$. Here $S=(0, b +1)$, $Z_1=(-b -1, 0)$, $P_{1}=(-b -2, -1)$, $X_{1}=(0, -1)$, $Z_2=(b, -1)$, $P_{2}=(b +2, -3)$, $Y_1=(
	 		b -1, 0)$, $X_2=(0, b -1)$, $P_{3}=(b +4, 2 b -1)$, $Z_3=(b, 2 b -1)$, $Y_2=(b -2, 2 b -1)$, $P_{4}=(
	 		-b +4, 5)$, $X_3=(-b, 1)$,  $X_0=\left(
	 		-\frac{b}{2}-1, \frac{b}{2}\right)$. }\label{f:espe3}
	 \end{figure} 

 Note that $F^3(\overline{X_0Z_1})=F^2(\overline{X_1Z_2})=F(\overline{X_2Z_3})=X_3.$ So these three intervals collapse.
 
 We will follow the orbit of $P_4.$ We have $P_{5}=(b-10,-1)\in \overline{X_1P_1}$ and $P_6=(10-b,2b-11).$ We name the intervals in the following way: $A=\overline{SP_4},$
 $A_0=\overline{P_4X_0},$ $B=\overline{P_{1}P_{5}},$ $B_0=\overline{P_{5}X_1},\,\,\overline{C}=\overline{X_2Y_1},$ $D=\overline{Z_3P_3},$ $E=\overline{Y_2Z_3},$ $I=\overline{Z_1P_1},$ $G=\overline{Y_1Z_2}$ and $H=\overline{Z_2P_2}.$
 
 When $b\le 11/2$ then $P_{6}\in G\cup H.$ If $P_{6}\in G$ we have the following graph of covers:
 
 \begin{center}
 	\begin{tikzcd}
 			A_0\arrow[r]& B_0\arrow[r]
 	         & \overline{C}\arrow[r] \arrow[ll,bend left]       &  A\arrow[r]           & B\arrow[r] & H \arrow[r]                                                &D \arrow[llllll, bend right] \\
 \end{tikzcd}
 \end{center}

while when $P_{6}\in H$ we have:

\begin{center}
 	\begin{tikzcd}
	A_0\arrow[r]& B_0\arrow[r]\arrow[rd]
	& G\arrow[r]        &  E\arrow[r]           & I\arrow[r] & H \arrow[r]                                                &D \arrow[llllll,bend  right] \\
	& & \overline{C}\arrow[llu, bend left] & & &\\
\end{tikzcd}
\end{center}

Both oriented graphs give positive entropy and this ends the proof when $b\le 11/2.$
When $b>11/2,$ then $P_{6}\in \overline{C}$ and we name $C=\overline{P_{6}Y_1}$ and $C_0=\overline{P_{6}X_2}.$

Note that  $F(A)=B,$ $F(B)=C\cup G\cup H$ and, since $P_7\in A\cup \overline{SY_2},$ $F(C)\subset A\cup \overline{SY_2}.$

Note also that the orbit of $P_{4}$ cannot be contained in $A\cup B\cup C$ because, in this case, the orbit by $F^3=F_1\circ F_3\circ F_2$ of $P_{4}$ would be  entirely contained in $A$ in contradiction with the fact that this map restricted to $A$ is linear, expansive and its fixed point does not belong to $A.$ So there exists a first $n\ge 3$ such that $F^n(P_{4})\notin A\cup B\cup C.$

Clearly there are two possibilities. Either $n=3k$ for some $k\ge 1$ and $F^{n}(P_{4})\in \overline{SY_2}$ or
$n=3k+2$ for some $k\ge 1$ and $F^{n}(P_{4})\in G\cup H.$
We begin by studying the first situation. In this case, we add to the graph the points of the orbit of $P_{4}$  until $F^{3k-1}(P_{4}). $ Therefore, each of the intervals $A,B,C$ splits into $k$ subintervals, namely for $i=1,\ldots,k-1,\,\,
A_i=\overline{F^{3i}(P_{4})F^{3(i-1)}(P_{4})},$ $B_i=\overline{F^{3i+1}(P_{4})F^{3(i-1)+1}(P_{4})},$ $C_i=\overline{F^{3i+2}(P_{4})F^{3(i-1)+2}(P_{4})},$ $A_k=\overline{F^{3(k-1)}(P_{4})S},$ $B_k=\overline{F^{3(k-1)+1}(P_{4})P_1}$ and $C_k=\overline{F^{3(k-1)+2}(P_{4})Y_1}.$ With this notation we obtain the following oriented graph of coverings:

\hspace{1cm}	\begin{tikzcd}
	A_0\arrow[r]& B_0\arrow[r]
	& C_0\arrow[lld] \arrow [ll,bend right]                              & A_k\arrow[r]&B_k\arrow[r]\arrow[d]& H\arrow[r]& D\arrow[llllll, bend right]&\\A_1\arrow[r]           & B_1\arrow[r]&C_1\arrow[lld] & &G\arrow[r] &E\arrow[r]&I\arrow[lu] &&&\\
	\ldots\arrow[r]&\ldots\arrow[r]&\ldots\arrow[lld]&&&&&&\\
	A_{k-1}\arrow[r]           & B_{k-1}\arrow[r]&C_{k-1}\arrow[ruuu]  & & &&&
\end{tikzcd}

\noindent which, from Remark \ref{romaentr} (ii), has positive entropy.

In the second case, we add to the graph the points of the orbit of $P_{4}$  until $F^{3k+1}(P_{4}).$ Now the intervals $A$ and $B$ split in $k+1$ subintervals:  $
A_i=\overline{F^{3i}(P_{4})F^{3(i-1)}(P_{4})},$ $B_i=\overline{F^{3i+1}(P_{4})F^{3(i-1)+1}(P_{4})}$ for $i=1,\ldots,k,$ and $A_{k+1}=\overline{SF^{3k}(P_{4})},$ $B_{k+1}=\overline{P_1F^{3k+1}(P_{4})}.$ On the other hand, the interval $C$ splits  into $k$ subintervals, that we name as in the previous case, that is, for $i=1,\ldots,k-1,$ $
C_i=\overline{F^{3i+2}(P_{4})F^{3(i-1)+2}(P_{4})}$ and $C_k=\overline{Y_1F^{3k-1}(P_{4})}.$ With this notation we obtain the two different oriented graphs depending on the position of $F^{3k+2}(P_{4}).$ 

More concretely, if $F^{3k+2}(P_{4})\in G$ we obtain 

\hspace{2cm}	\begin{tikzcd}
	A_0\arrow[r]& B_0\arrow[r]
	& C_0\arrow[lld]\arrow[ll,bend right]                                & A_{k+1}\arrow[r]&B_{k+1}\arrow[r]& H\arrow[r]& D\arrow[llllll, bend right]&\\A_1\arrow[r]           & B_1\arrow[r]&C_1\arrow[lld] & & && &&&\\
	\ldots\arrow[r]&\ldots\arrow[r]&\ldots\arrow[lld]&&&&&&\\
	A_{k}\arrow[r]           & B_{k}\arrow[r]&C_{k}\arrow[ruuu]  & & &&&
\end{tikzcd}

\bigskip

On the other hand,  if $F^{3k+2}(P_{4})\in H,$ we get

\hspace{2cm} \begin{tikzcd}
 	A_0\arrow[r]& B_0\arrow[r]
 	& C_0\arrow[lld] \arrow [ll,bend right]                             & E\arrow[r]&I\arrow[r]&H\arrow[r] &D\arrow[llllll, bend right]&\\A_1\arrow[r]           & B_1\arrow[r]&C_1\arrow[lld] & & && &&&&\\
 	\ldots\arrow[r]&\ldots\arrow[r]&\ldots\arrow[lld]&&&&&&&\\
 	A_{k}\arrow[r]           & B_{k}\arrow[r]&G\arrow[ruuu]  & & &&&&
 \end{tikzcd}g

In both cases, from Remark \ref{romaentr} (i), the obtained oriented graphs have positive entropy. This ends the proof of the proposition.
 \end{proof}

\begin{nota}\label{cont3} The same considerations stated in Remark \ref{cont1} are valid in the above situation and hence we get  that $ \lim\limits_{b\to 8^-} h(F\vert_{\Gamma})=0.$
The detailed proof of this fact is written in \cite{CGMM}.\end{nota}

\subsection{The case $a=-1$ and $b\geq 8$}\label{ss:ultima}

We begin by observing that, in this case, we always have two $3-$periodic orbits, namely  $\mathcal{X}=\left\{(-b,1),(b-2,-1),(b-2,2b-3)\right\}$  and $\mathcal{Q}=\left\{\left(\frac{2-b}{3},-1\right),\left(\frac{b-2}{3},\frac{2b-1}{3}\right),\left(-\frac{b+4}{3},\frac{2b-1}{3}\right)\right\}.$ Also there is a fixed point $p\in Q_2.$

Now consider the same graph of the above case, see Figure \ref{f:espe3}. As before we consider the partition given by the intervals $A_0,A,B_0,B,C_0,C,D,E,I,G,H.$ We have the following result:

\begin{propo}\label{bmes8}
	Assume that $a=-1$ and $b\geq 8.$ Then:

	\begin{itemize}
		\item [(a)] The map $F\vert_{\Gamma}$ has zero entropy.
		\item [(b)]  The map $F$ has the two $3$-periodic orbits defined above: $\mathcal{X},$ which absorb both plateaus, and $\mathcal{Q}.$
			Furthermore, for each $(x,y)\in\Gamma $ there exists some $n\in\N$ such that $F^n(x,y)\in \mathcal{X}\cup \mathcal{Q}.$
	\end{itemize}
\end{propo}
\begin{proof}

First we consider the case $b=8.$ In this case,  $P_{4}\in\mathcal Q$ and the map after collapses is Markov. The associated oriented  graph  is:

\begin{center}
 \begin{tikzcd}
	A\arrow[r]& B\arrow[r]\arrow[d]\arrow[rd]
	& H\arrow[r] & D\arrow[r]&A_0\arrow[r]&B_0\arrow[r]& C_0\arrow[ll, bend right]\\ & C\arrow[lu]&G\arrow[r] &E\arrow[r] &I\arrow[llu] && \\
	\end{tikzcd}
\end{center}

From Remark \ref{romaentr} (iii), this oriented graph has zero entropy and the result follows. 

Now we consider  the case $b\in(8,10).$ We work with the same partition of the previous case.  Simple computations shows that $F^3(P_{4})\in A_0,$ and the partition is not Markov, but we can apply Remark \ref{grafcota} getting the following graph:

\begin{center}
 \begin{tikzcd}
	A\arrow[r]& B\arrow[r]\arrow[d]\arrow[rd]
	& H\arrow[r] & D\arrow[r]&A_0\arrow[r]&B_0\arrow[r]& C_0\arrow[ll, dashed,bend right]\\ & C\arrow[lu]\arrow[rrru,dashed]&G\arrow[r] &E\arrow[r] &I\arrow[llu] && \\
\end{tikzcd}
\end{center}

As usual,  the dashed arrow means that the images of $C_0$ and $C$ only intersect but not necessarily cover $A_0.$ This oriented graph still has zero entropy and from Remark \ref{grafcota} we obtain that $h(F\vert_{\Gamma})=0.$ Concerning the plateaus we see that $F(\overline{X_2Z_3})=X_3\in \mathcal X$ and $F(\overline{SY_2})=P_1.$ We are going to prove that the orbit of $P_1$ always meets $\mathcal X.$ Recall that in our situation, $P_7\in A_0$ and $F(A_0)=B_0,\,\,F(B_0)=C_0$ and $F(C_0)\subset A_0\cup \overline{X_0X_3}.$ Since $F^3(\overline {X_0X_3})=X_3$ we have that the orbit of $P_1$ either meets $\mathcal X$ or alternatively it always follow the loop $A_0B_0C_0.$ But since the map $F_1F_3F_2$ from $A_0$ into itself is expansive the only point in $A_0$ which could always travel following the loop $A_0B_0C_0$ is the fixed point of $F_1F_3F_2,$ namely $(-(b+4)/3,(2b-1)),$ which does not belong to $A_0.$ So the last alternative is not possible and both plateaus are absorbed by $\mathcal X$ in this case. From this analysis and looking at the above  graph it follows that for each $(x,y)\in \Gamma\setminus\mathcal Q$ there exists $n$ such that $F^n(x,y)\in \mathcal X.$

Lastly we consider the case $b\ge 10.$ Now $F(P_{4})\in \overline {X_1Z_2},$ and the intervals $B_0$ and $C_0$ disappear. We rename $B=\overline{P_1X_1}$ and $C=\overline{X_2Y_1}.$  Moreover, the intervals $A_0,D,E,I,G$ and $H$ collapse after some iterations. Indeed all of them reach the orbit $\mathcal X.$ Now the resulting partition is Markov and the associated oriented graph is:

\begin{center}
 \begin{tikzcd}
	A\arrow[r]& B\arrow[r]
	& C\arrow[ll,bend right] \\
\end{tikzcd}
\end{center}
This graph has zero entropy. Thus $h(F\vert_{\Gamma})=0.$ Also note that now $F^3(P_4)=X_3\in \mathcal X$ and the only points that do not meet $\mathcal X$ are the points of $\mathcal Q,$ the $3$-periodic orbit forced by the loop $ABC$ that is repulsive. This ends the proof of the proposition.
\end{proof}

\subsection{Proof of Theorem D}
The proof of Theorem D follows from collecting all the results in Propositions \ref{primer}--\ref{sis}, \ref{sotasota1},  \ref{sota1}, \ref{exac1},  \ref{b12},  \ref{b24}, \ref{b48}, \ref{bmes8}, in the preceding Sections.

\section{Future research directions}\label{s:future}

Throughout our research we have found that, when $a < 0$, and for all values of $b$, the dynamics of the family of maps $F_{a,b}$  across the entire plane is captured by a compact graph with a finite number of edges. We believe that this situation may occur in other piecewise continuous maps, for example, those with open regions to which all orbits access except, perhaps, a finite set, and whose image collapses into a $1$-dimensional set due to the fact that the Jacobian matrix has a rank~$1$ on the full region. This is the mechanism for the formation of invariant graphs in the case of our maps, as it is explained in Section~\ref{sec:4}. However, it remains to have a better understanding of the mechanism that provokes that the final graph contains only a finite number of edges. We believe this issue warrants further investigation, which in this work  we have simply addressed by a case by case study.

Finally, Theorem \ref{t:teoC} establishes that in each of the invariant graphs,  which capture the dynamics in the plane, there exists an open and dense set of initial conditions with at most three $\omega$-limits, which moreover, when the parameter ratio is rational, are periodic orbits. As we have suggested, this could partially explain why numerical simulations only exhibit periodic behavior. However, to demonstrate this fact conclusively, it would be necessary to show that the set of initial conditions $\mathcal{V} \subset \mathbb{R}^2$ that converge to these three $\omega$-limits has full Lebesgue measure. An intermediate result would be to prove that this set $\mathcal{V}$ is open and dense in the entire plane. To elucidate whether each one of these facts holds  remains as an interesting topic for future research.

\section*{Acknowledgments} 
\addcontentsline{toc}{section}{Acknowledgments}

We sincerely thank the anonymous reviewers for their thorough reading of our manuscript and their insightful suggestions, which have improved the quality of this article.

This work is supported by
Ministry of Science and Innovation--State Research Agency of the
Spanish Government through grants PID2022-136613NB-I00   and PID2023-146424NB-I00. It is also supported by the grants 2021-SGR-00113 
and 
2021-SGR-01039 from AGAUR of Generalitat de Catalunya. 

\vfill
\newpage

\bibliographystyle{plain}

\newpage
\appendix


\section*{Appendix: invariant graphs for the case $a=-1$}\label{app:A}
\addcontentsline{toc}{section}{\appendixname}

\begin{figure}[H]
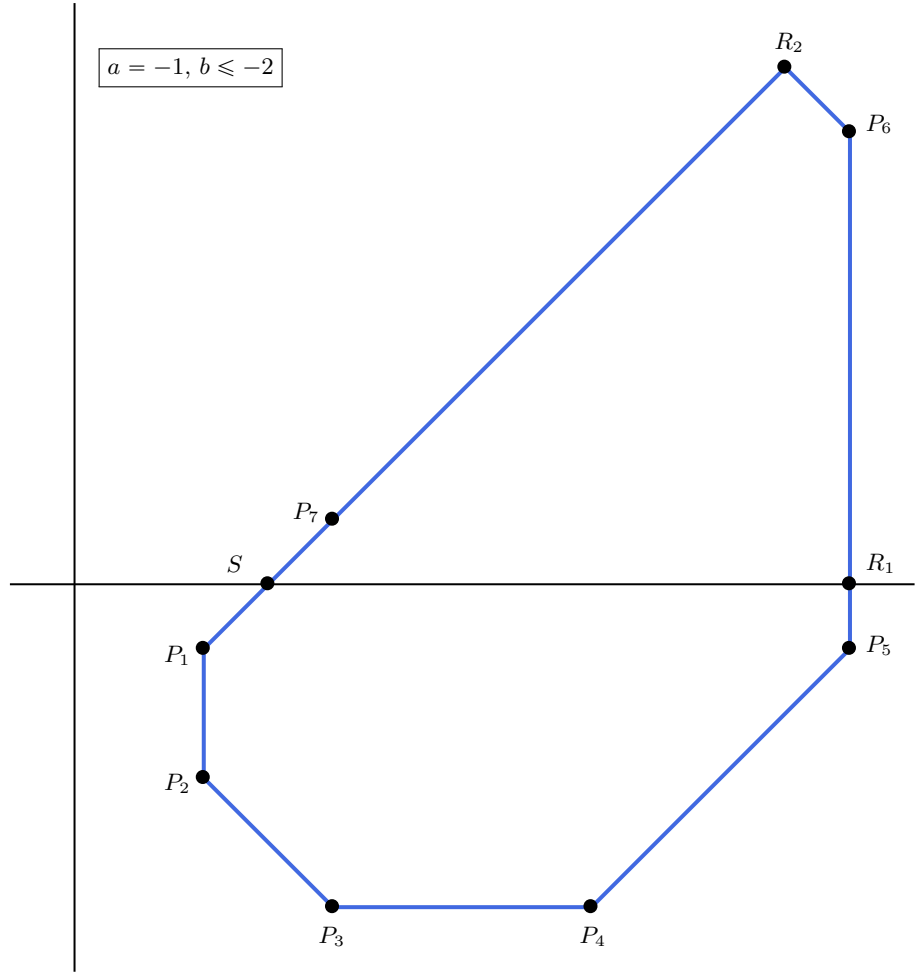

\footnotesize
\centering
\begin{lpic}[l(2mm),r(2mm),t(2mm),b(2mm)]{am-cas-1-especial(0.68)}
\lbl[l]{27,180; $\boxed{a=-1,\,b\leq -2}$}
\lbl[l]{52,83; $S$}
\lbl[r]{45,65; $P_1$}
\lbl[r]{45,40; $P_2$}
\lbl[l]{70,10; $P_3$}
\lbl[l]{121,10; $P_4$}
\lbl[l]{177,67; $P_5$}
\lbl[l]{177,83; $R_1$}
\lbl[l]{177,169; $P_6$}
\lbl[c]{162,185; $R_2$}
\lbl[l]{65,93; $P_7$}
\end{lpic}
\caption{The graph $\Gamma$ for $a=-1$ and $b\leq -2$. Here $P_{1}=(-b -2, -1)$, $P_{2}=(-b -2, -3)$, $P_{3}=(-b, -5)$, $P_{4}=(-b +4, -5)$, $P_{5}=(-b +8, -1)$, $P_{6}=(-b +8, 7)$, $P_{7}=(-b, 1)$, $R_{1}=(-b +8, 0)$, $R_{2}=(-b +7, 8)$, and $S=(-b -1, 0)$. For $b=-2$ the points $P_1=(-b -2, -1)$ and $P_2=(-b -2, -3)$ are located on the $y$-axis.}\label{ff:1}
\end{figure}

\newpage

\begin{figure}[H]
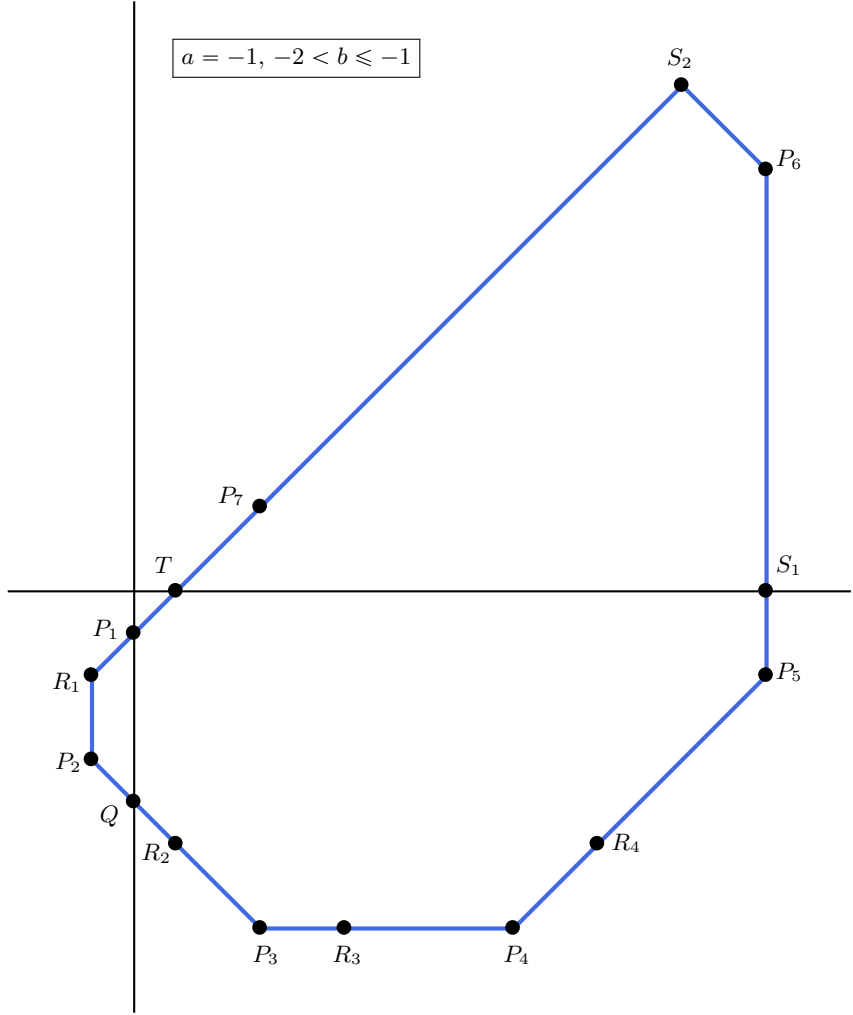

\footnotesize
\centering
\begin{lpic}[l(2mm),r(2mm),t(2mm),b(2mm)]{am-cas-2-especial(0.71)}
\lbl[l]{50,182; $\boxed{a=-1,\,-2< b\leq -1}$}
\lbl[r]{50,87; $T$}
\lbl[r]{40,75; $P_1$}
\lbl[r]{33,65; $R_1$}
\lbl[r]{33,50; $P_2$}
\lbl[r]{40,40; $Q$}
\lbl[c]{47,33; $R_2$}
\lbl[l]{65,14; $P_3$}
\lbl[l]{80,14; $R_3$}
\lbl[l]{112,14; $P_4$}
\lbl[c]{135,35; $R_4$}
\lbl[l]{163,67; $P_{5}$}
\lbl[l]{163,87; $S_{1}$}
\lbl[l]{163,163; $P_{6}$}
\lbl[c]{145,182; $S_{2}$}
\lbl[c]{61,100; $P_{7}$}

\end{lpic}
\caption{The graph $\Gamma$ for $a=-1$ and $2< b\leq -1$.
Here $P_{1}=(0, b +1)$, $P_{2}=(-b -2, 1+2 b)$, $P_{3}=(-b, -1+2 b)$, $P_{4}=(-3 b, -1+2 b)$, $P_{5}=(-5 b, -1)$, $P_{6}=(-5 b, -4 b -1)$,
$P_{7}=(-b, 1)$, $Q=(0, b -1)$, $R_{1}=(-b -2, -1)$, $R_2=(b +2, -3)$,  $R_3=(b +4, -1+2 b)$,  $R_{4}=(-b +4, 4 b +3)$, 
$S_{1}=(-5 b, 0)$,  $S_{2}=(-5 b -1, -4 b)$  and $T=(-b -1, 0)$. When $b=-1$ the points $P_{1}=(0,b+1)$ and $T=(-b -1, 0)$, as well as the points
$R_1=(-b -2, -1)$ and $P_{2}=(-b -2, 2b+1)$, collide.}\label{ff:2}
\end{figure}

\begin{figure}[H]
\footnotesize
\centering
\begin{lpic}[l(2mm),r(2mm),t(2mm),b(2mm)]{am-cas-3(0.71)}
\lbl[l]{-2,185; $\boxed{a=-1,\,-1< b\leq -3/4}$}
\lbl[r]{45,94; $R_1$}
\lbl[r]{50,99; $P_1$}
\lbl[r]{21,67; $P_2$}
\lbl[l]{33,67; $R_2$}
\lbl[r]{48,42; $Q$}
\lbl[r]{69,22; $R_3$}
\lbl[c]{78,12; $P_3$}
\lbl[c]{116,16; $R_4$}
\lbl[l]{132,22; $P_4$}
\lbl[l]{165,67; $R_5$}
\lbl[l]{165,94; $T_1$}
\lbl[c]{178,82; $P_5$}
\lbl[l]{165,155; $R_6$}
\lbl[c]{161,182; $P_6$}
\lbl[c]{134,184; $T_2$}
\end{lpic}
\caption{The graph $\Gamma$ for $a=-1$ and $-1< b\leq -3/4$. Here $P_{1}=(0, b +1)$, $P_2=(-b -2, -1)$,  $P_3=(b +2, -3)$, $P_4=(b +4, -1+2 b)$,  $P_5=(-b +4, 4 b +3)$, $P_{6}=(-5b, 4b+7)$,
$Q=(0, b -1)$, $R_1=(-b-1,0)$,  $R_2=(b,-1)$,  $R_{3}=(-b,-1+2 b)$, $R_{4}=(-3 b, -1+2 b)$,  $R_{5}=(-5 b, -1)$,  $R_{6}=(-5 b, -4 b -1)$, $T_{1}=(-5 b, 0)$  and $T_{2}=(-5 b -1, -4 b)$.}\label{f:3}
\end{figure}

\begin{figure}[H]
\footnotesize
\centering
\begin{lpic}[l(2mm),r(2mm),t(2mm),b(2mm)]{am-cas-4(0.71)}
\lbl[l]{1,185; $\boxed{a=-1,\,-3/4< b\leq -1/4}$}
\lbl[r]{55,112; $P_1$}
\lbl[r]{42,102; $R_1$}
\lbl[r]{23,65; $P_2$}
\lbl[l]{38,65; $R_2$}
\lbl[r]{55,55; $Q$}
\lbl[r]{73,38; $R_3$}
\lbl[c]{97,12; $P_3$}
\lbl[c]{95,38; $R_4$}
\lbl[c]{152,38; $P_4$}
\lbl[l]{130,70; $R_5$}
\lbl[l]{130,102; $T_1$}
\lbl[l]{152,102; $S$}
\lbl[c]{180,130; $P_5$}
\lbl[l]{130,125; $R_6$}
\lbl[l]{130,180; $P_6$}
\lbl[c]{100,157; $T_2$}
\end{lpic}
\caption{The graph $\Gamma$ for $a=-1$ and $-3/4< b\leq -1/4$.  Here $P_{1}=(0, b +1)$, $P_2=(-b -2, -1)$, $P_3=(b +2, -3)$, $P_4=(b +4, -1+2 b)$,
$P_5=(-b +4, 4 b +3)$, $P_{6}=(-5b, 1-4b)$,
$Q=(0, b -1)$, 
$R_1=(-b-1,0)$,  $R_2=(b,-1)$,  $R_{3}=(-b, 
	-1+2 b)$,  $R_{4}=(-3 b, -1+2 b)$,  $R_{5}=(-5 b, -1)$, $R_{6}=(-5 b, -4 b -1)$,  $S=(1-5 b, 0)$,  $T_{1}=(-5 b, 0)$  and $T_{2}=(-5 b -1, -4 b)$.  }\label{f:4}
\end{figure}

\vspace{-0.3cm}
\begin{figure}[H]
\footnotesize
\centering
\vspace{-1cm}
\begin{lpic}[l(2mm),r(2mm),t(2mm),b(2mm)]{am-cas-5(0.68)}
\lbl[l]{3,175; $\boxed{a=-1,\,-1/4< b\leq -2/9}$}
\lbl[l]{35,114; $R_1$}
\lbl[r]{61,133; $P_1$}
\lbl[r]{13,77; $P_2$}
\lbl[r]{58,77; $R_2$}
\lbl[l]{65,77; $Q$}
\lbl[r]{70,65; $R_3$}
\lbl[c]{115,21; $P_3$}
\lbl[l]{81,65; $R_4$}
\lbl[c]{172,65; $P_4$}
\lbl[l]{94,77; $R_5$}
\lbl[l]{97,114; $T_1$}
\lbl[r]{89,114; $X_1$}
\lbl[l]{130,114; $S$}
\lbl[c]{184,175; $P_5$}
\lbl[l]{97,107; $R_6$}
\lbl[l]{97,167; $P_6$}
\lbl[l]{63,142; $T_2$}
\lbl[l]{72,138; $R_7$}
\lbl[r]{62,138; $X_2$}
\end{lpic}

\begin{lpic}[l(2mm),r(2mm),t(2mm),b(2mm),figframe(0.20mm)]{am-cas-5-zoom(0.3)}
\lbl[c]{36,176; $T_2$}
\lbl[r]{14,150; $P_1$}
\lbl[r]{177,26; $R_6$}
\lbl[l]{181,44; $T_1$}
\lbl[r]{161,44; $X_1$}
\lbl[l]{52,160; $R_7$}
\lbl[r]{24,160; $X_2$}
\end{lpic}
\caption{The graph $\Gamma$ for $a=-1$ and $-1/4< b\leq -2/9$. Below, a detail. Here $P_1=(0, b +1)$,  $P_2=(-b -2, -1)$, $P_3=(b +2, -3)$, $P_4=(b +4, 2 b -1)$, $P_5=(-b +4, 4 b +3)$, $P_6=(-5 b, -4 b +1)$, $Q=(0, b -1)$, $R_{1}=(-b -1, 0)$, $R_{2}=(b, -1)$,  $R_{3}=(-b, 2 b -1)$,  $R_{4}=(-3 b, 2 b -1)$,  $R_{5}=(-5 b, -1)$, $R_{6}=(-5 b, -4 b -1)$, $R_{7}=(
-b, -8 b -1)$, $S=(-5 b +1, 0)$, $T_1=(-5 b, 0)$,   $T_2=(-5 b -1, -4 b)$,  $X_1=(-9 b -1, 0)$ and $X_2=(-9 b -2, -8 b -1)$.}\label{f:5}
\end{figure}

\newpage
\begin{figure}[H]
\footnotesize
\centering

\begin{lpic}[l(2mm),r(2mm),t(2mm),b(2mm)]{am-cas-6(0.71)}
\lbl[l]{3,175; $\boxed{a=-1,\,-2/9< b\leq -1/5}$}
\lbl[r]{42,113; $R_1$}
\lbl[r]{62,135; $P_1$}
\lbl[r]{15,77; $P_2$}
\lbl[l]{16,77; $X_3$}
\lbl[r]{13,87; $Z_2$}
\lbl[r]{60,77; $R_2$}
\lbl[l]{67,77; $Q$}
\lbl[r]{70,65; $R_3$}
\lbl[c]{112,21; $P_3$}
\lbl[l]{81,65; $R_4$}
\lbl[c]{170,65; $P_4$}
\lbl[l]{97,77; $R_5$}
\lbl[l]{97,113; $T_1$}
\lbl[r]{88,107; $X_1$}
\lbl[l]{130,113; $S$}
\lbl[c]{182,176; $P_5$}
\lbl[l]{97,105; $R_6$}
\lbl[c]{92,167; $P_6$}
\lbl[c]{68,138; $T_2$}
\lbl[l]{72,129; $R_7$}
\lbl[r]{58,129; $X_2$}
\lbl[c]{67,125; $Z_1$}
\end{lpic}
\caption{The graph $\Gamma$ for $a=-1$ and $-2/9< b\leq -1/5$. Here $P_1=(0, b +1)$, $P_2=(-b -2, -1)$, $P_3=(b +2, -3)$, $P_4=(b +4, 2 b -1)$,
$P_5=(-b +4, 4 b +3)$, $P_6=(-5 b, -4 b +1)$,
$Q=(0, b -1)$, $R_{1}=(-b -1, 0)$,  $R_{2}=(b, -1)$, $R_{3}=(-b, 2 b -1)$, $R_{4}=(-3 b, 2 b -1)$, $R_{5}=(-5 b, -1)$,  $R_{6}=(-5 b, -4 b -1)$, $R_{7}=(-b, -8 b -1)$, $S=(-5 b +1, 0)$,  $T_1=(-5 b, 0)$, $T_2=(-5 b -1, -4 b)$, $X_1=(-9 b -1, 0)$, $X_2=(-9 b -2, -8 b -1)$,  $X_3=(17b+2,-1)$, $Z_1=(0,-8b-1)$ and $Z_2=(8b,9b+1).$}\label{f:6}
\end{figure}

\begin{figure}[H]
\footnotesize
\centering
\begin{lpic}[l(2mm),r(2mm),t(2mm),b(2mm)]{am-cas-7(0.71)}
\lbl[l]{3,175; $\boxed{a=-1,\,-1/5< b\leq -1/8}$}
\lbl[r]{38,110; $R_1$}
\lbl[r]{13,75; $P_1$}
\lbl[r]{45,75; $X_3$}
\lbl[l]{20,95; $Z_2$}
\lbl[r]{62,75; $R_2$}
\lbl[l]{59,70; $Q$}
\lbl[r]{70,65; $R_3$}
\lbl[c]{118,17; $P_{2}$}
\lbl[l]{75,65; $R_{4}$}
\lbl[c]{175,65; $P_{3}$}
\lbl[l]{90,76; $R_{5}$}
\lbl[l]{90,110; $T_{1}$}
\lbl[r]{85,110; $X_{1}$}
\lbl[l]{122,110; $S$}
\lbl[c]{184,179; $P_{4}$}
\lbl[l]{90,97; $R_{6}$}
\lbl[c]{88,159; $P_{5}$}
\lbl[c]{58,130; $T_{2}$}
\lbl[l]{71,117; $R_{7}$}
\lbl[l]{40,117; $X_{2}$}
\lbl[c]{61,112; $Z_{1}$}
\lbl[r]{20,90; $Y_2$}
\lbl[l]{66,122; $Y_{1}$}
\lbl[c]{23,75; $T_3$}
\end{lpic}
\caption{The graph $\Gamma$ for $a=-1$ and $-1/5< b\leq -1/8$. Here $P_{1}=(-b -2, -1)$, $P_{2}=(b +2, -3)$, $P_{3}=(b +4, 2 b -1)$, $P_{4}=(-b +4, 4 b +3)$, $P_{5}=(-5 b, -4 b +1)$, $Q=(0, b -1)$,
$R_{1}=(-b -1, 0)$, $R_{2}=(b, -1)$, $R_{3}=(-b, 2 b -1)$,  $R_{4}=(-3 b, 2 b -1)$,  $R_{5}=(-5 b, -1)$, $R_{6}=(-5 b, -4 b -1)$, $R_{7}=(-b, -8 b -1)$, $S=(-5 b +1, 0)$,  $T_{1}=(-5 b, 0)$, $T_{2}=(-5 b -1, -4 b)$, $T_{3}=(9 b, -1)$, $X_{1}=(-9 b -1, 0)$, $X_{2}=(-9 b -2, -8 b -1)$, $X_{3}=(17 b +2, -1)$, $Y_{1}=(0, -9 b -1)$, $Y_{2}=(9 b, 10 b +1)$, $Z_{1}=(0, -8 b -1)$ and $Z_{2}=(8 b, 9 b +1)$.}\label{f:7}
\end{figure}

\newpage

\begin{figure}[H]
\footnotesize
\centering
\vspace{-3cm}
\begin{lpic}[l(2mm),r(2mm),t(2mm),b(2mm)]{am-cas-8(0.71)}
\lbl[l]{3,175; $\boxed{a=-1,\,-1/8< b\leq -1/9}$}

\lbl[r]{42,108; $R_1$}
\lbl[r]{20,72; $P_1$}

\lbl[r]{65,72; $R_2$}
\lbl[l]{67,75; $Q$}
\lbl[r]{70,65; $R_3$}
\lbl[c]{117,16; $P_2$}
\lbl[l]{75,65; $R_4$}
\lbl[c]{173,65; $P_3$}
\lbl[l]{85,76; $R_5$}
\lbl[l]{85,108; $T_1$}

\lbl[l]{71,107;  $X_1$}

\lbl[r]{71,110; { $Y_1$}}

\lbl[l]{117,106; $S$}
\lbl[c]{182,181; $P_4$}
\lbl[l]{85,90; $R_6$}
\lbl[c]{85,150; $P_5$}
\lbl[c]{52,122; $T_2$}

\lbl[r]{37,102; {\tiny $X_2$}}
\lbl[r]{46,99; $R_8$}
\lbl[l]{61,99; $Z_1$}
\lbl[l]{67,99; $R_7$}

\lbl[l]{37,95; $Z_2$}
\lbl[r]{32,97; $Y_2$}

\lbl[c]{35,72; $T_3$}
\end{lpic}
\begin{lpic}[l(2mm),r(2mm),t(2mm),b(2mm),figframe(0.2mm)]{am-cas-8-zoom(0.33)}
\lbl[c]{90,152; $T_2$}
\lbl[l]{140,108; $Y_1$}
\lbl[l]{150,97; $X_1$}
\lbl[r]{35,97; $R_1$}
\lbl[r]{11,70; $Y_2$}
\lbl[l]{35,69; $Z_2$}
\lbl[r]{27,83; $X_2$}
\lbl[l]{43,78; $R_8$}
\lbl[l]{128,78; $Z_1$}
\lbl[l]{146,78; $R_7$}

\end{lpic}
\caption{The graph $\Gamma$ for $a=-1$ and $-1/8< b\leq -1/9$. Here $P_1=(-b -2, -1),$ $P_2=(b+2,-3),$ $P_3=(b+4,2b-1),$ $P_4=(-b+4,4b+3),$ $P_5=(-5b, -4b+1)$, $Q=(0,b-1)$,
$R_1=(-b -1, 0)$, $R_2=(b,-1)$, $R_3=(-b,2b-1)$,
$R_4=(-3b,2b-1)$,  $R_5=(-5b, -1)$, $R_6=(-5b, -4b-1)$, $R_7=(-b,-8b-1)$, $R_8=(7b,-8b-1)$, 
$S=(-5b+1, 0),$  $T_1=(-5b, 0),$ $T_2=(-5b-1, -4b),$ $T_3=(9b,-1)$,  $X_1=(-9b-1,0)$,
$X_2=(-9b-2,-8b-1)$,  $Y_1=(0,-9b-1)$, $Y_2=(9b,10b+1),$     $Z_1=(0,-8b-1)$ and   $Z_2=(8b,-7b-1)$.  Below, a detailed view.}\label{f:8}
\end{figure}

\newpage

\begin{figure}[H]
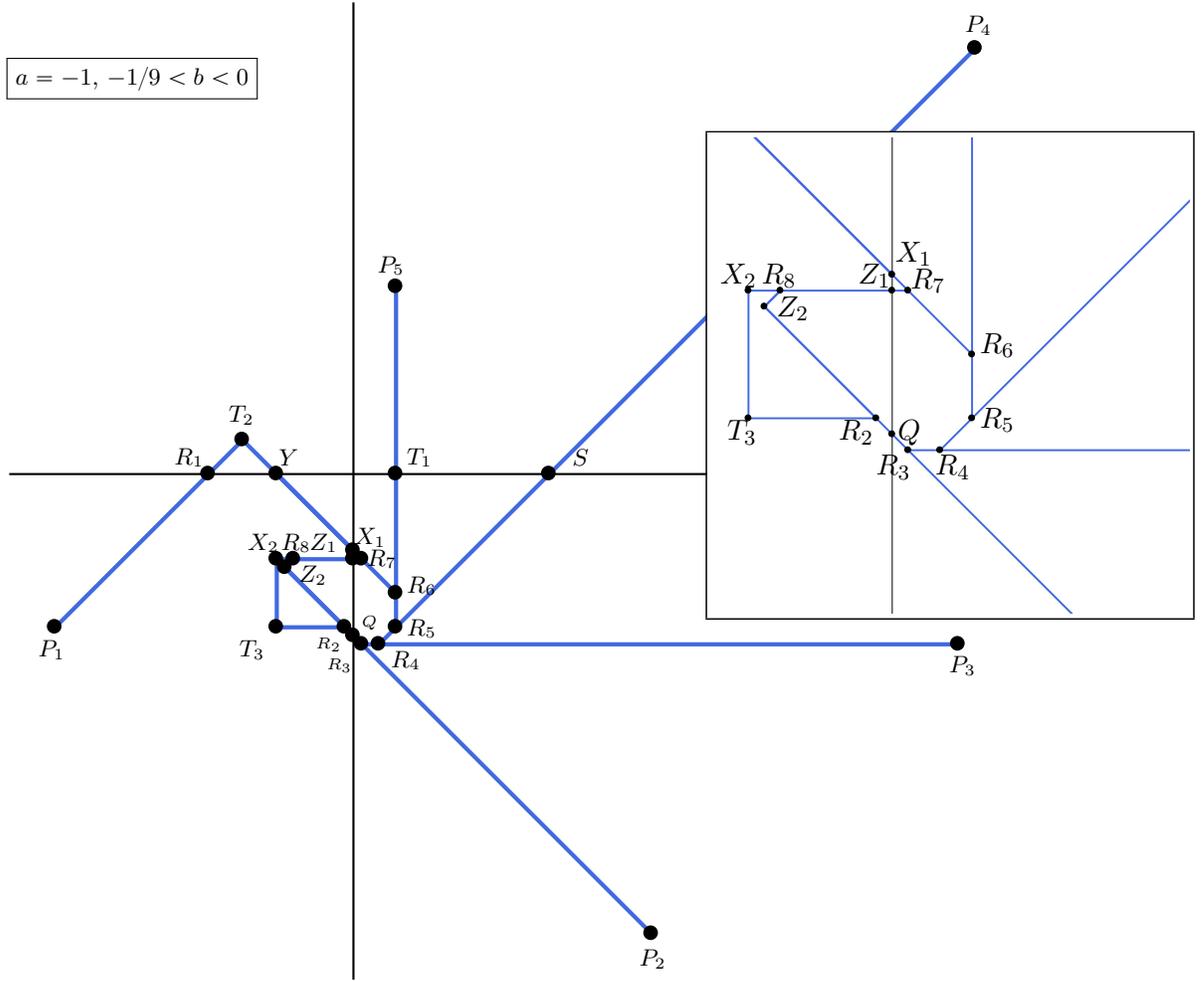

\footnotesize
\centering

\begin{lpic}[l(2mm),r(2mm),t(2mm),b(2mm)]{am-cas-9(0.71)}
\lbl[l]{3,175; $\boxed{a=-1,\,-1/9< b< 0}$}

\lbl[r]{40,104; $R_1$}
\lbl[l]{54,104; $Y$}
\lbl[l]{78,104; $T_1$}
\lbl[l]{109,104; $S$}

\lbl[c]{47,112; $T_2$}
\lbl[c]{185,185; $P_4$}
\lbl[c]{75,140; $P_5$}

\lbl[r]{54,88; { $X_2$}}
\lbl[r]{60,88; { $R_8$}}
\lbl[r]{65,88; { $Z_1$}}
\lbl[l]{69,85; { $R_7$}}
\lbl[r]{74,89;{ $X_1$}}

\lbl[r]{63,82; { $Z_2$}}

\lbl[l]{78,80; $R_6$}
\lbl[l]{78,72; $R_5$}

\lbl[r]{14,68; $P_{1}$}
\lbl[c]{49,68; $T_3$}

\lbl[c]{124,10; $P_{2}$}
\lbl[c]{182,65; $P_{3}$}

\lbl[l]{61,69; {\tiny $R_{2}$}}
\lbl[c]{71,73; {\tiny $Q$}}
\lbl[l]{63,65; {\tiny $R_{3}$}}

\lbl[l]{75,66; $R_{4}$}
\end{lpic}
\caption{The graph $\Gamma$ for $a=-1$ and $-1/9< b< 0$. Here $P_{1}=( -b-2,-1 )$,  $P_{2}=( b+2,-3 )$,  $P_{3}=( b+4,2b-1 )$,  $P_{4}=( -b+4,4b+3)$, $P_{5}=( -5b,-4b+1 )$, $Q=( 0,b-1)$, $R_{1}=( -b-1,0 )$, $R_2=(b,-1)$,  $R_{3}=( -b,2b-1 )$,  $R_{4}=( -3b,2b-1 )$,  $R_{5}=( -5b,-1)$,  $R_{6}=( -5b,-4b-1 )$, $R_7=(-b,-8b-1)$, $R_8=(7b,-8b-1)$, $S=(-5b+1,0)$, $T_{1}=( -5b,0)$,  $T_{2}=( -5b-1, -4b)$,    $T_{3}=( 9b,-1)$,
$X_1=(0,-9b-1)$, $X_2=(9b,-8b-1)$,  $Y=(-9b-1,0 )$, $Z_1=(0,-8b-1)$ and  $Z_2=(8b,-7b-1)$. Beside, a detailed view.}\label{f:9}
\end{figure}
\vspace{-16.5cm}\hspace{9.4cm}\begin{lpic}[l(2mm),r(2mm),t(2mm),b(2mm),figframe(0.2mm)]{am-cas-9-zoom(0.33)}

\lbl[r]{20,138; $X_2$}
\lbl[c]{29,138; $R_8$}
\lbl[l]{61,138; $Z_1$}
\lbl[l]{82,136; $R_7$}

\lbl[l]{72,147;{ $X_1$}}

\lbl[l]{28,125; $Z_2$}

\lbl[r]{67,75; {$R_2$}}
\lbl[l]{73,75; { $Q$}}
\lbl[r]{82,61; { $R_3$}}

\lbl[l]{92,61; $R_4$}

\lbl[r]{20,75; $T_3$}
\lbl[l]{110,80; $R_5$}

\lbl[l]{110,110; $R_6$}

\end{lpic}

\newpage

\begin{figure}[H]
\footnotesize
\centering
\begin{lpic}[l(2mm),r(2mm),t(2mm),b(2mm)]{am-cas-10(0.71)}
\lbl[l]{3,185; $\boxed{a=-1,\,0\leq b\leq 1/10}$}

\lbl[c]{52,143; $P_1$}
\lbl[r]{50,115; $Z_1$}
\lbl[l]{57,91; $Y_2$}

\lbl[r]{45,95; $Z_2$}
\lbl[c]{20,60; $P_2$}

\lbl[r]{64,76; $X_2$}
\lbl[r]{70,70; $p$}
\lbl[r]{74,64; $Q$}

\lbl[c]{126,10; $P_3$}
\lbl[c]{175,68; $P_4$}

\lbl[l]{104,113; $S$}
\lbl[l]{76,84; $Y_1$}

\lbl[l]{76,77; $X_1$}

\lbl[c]{168,184; $P_5$}
\end{lpic}

\caption{The graph $\Gamma$ for $a=-1$ and $0\leq b\leq 1/10$. Here $p=(-b,2b-1)$, $P_1=(-5b, -4b+1)$, $P_2=(9b-2,-1)$, $P_3=(-9b+2,10b-3)$, $P_4=(-19b+4,2b-1)$,  $P_5=(-21b+4,-16b+3)$, $Q=(0,b-1)$, $S=(-5b+1,0)$, $X_1=(0,2b-1)$, $X_2=(-2b,3b-1)$,  $Y_1=(0,5b-1)$, $Y_2=(-5b, 6b-1)$, $Z_1=(-5b,0)$ and $Z_2=(5b-1, -4b)$.	
}\label{f:10}
\end{figure}

\begin{figure}[H]
\footnotesize
\centering
\begin{lpic}[l(2mm),r(2mm),t(2mm),b(2mm)]{am-cas-11(0.71)}
\lbl[l]{3,185; $\boxed{a=-1,\,1/10< b\leq 1/7}$}

\lbl[c]{42,151; $P_1$}
\lbl[l]{45,122; $Z_1$}
\lbl[l]{45,104; $Y_2$}

\lbl[r]{50,92; $Z_2$}
\lbl[c]{23,58; $P_2$}

\lbl[r]{61,82; $X_2$}
\lbl[r]{66,75; $p$}
\lbl[r]{74,69; $Q$}

\lbl[c]{126,15; $P_3$}
\lbl[c]{168,70; $P_4$}

\lbl[l]{104,122; $S$}
\lbl[l]{78,95; $Y_1$}

\lbl[l]{78,80; $X_1$}

\lbl[c]{155,181; $P_5$}
\end{lpic}

\caption{The graph $\Gamma$ for $a=-1$ and $1/10< b\leq 1/7$. Here $p=(-b,2b-1)$, $P_1=(-5b, -4b+1)$, $P_2=(9b-2,-1)$, $P_3=(-9b+2,10b-3)$, $P_4=(-19b+4,2b-1)$, $P_5=(-21b+4,-16b+3)$, 
$Q=(0,b-1)$, $S=(-5b+1,0)$, $X_1=(0,2b-1)$,  
$X_2=(-2b,3b-1)$, $Y_1=(0,5b-1)$, $Y_2=(-5b, 6b-1)$, $Z_1=(-5b,0)$ and $Z_2=(5b-1, -4b)$.	}\label{f:11}
\end{figure}

\newpage

\begin{figure}[H]
\footnotesize
\centering
\begin{lpic}[l(2mm),r(2mm),t(2mm),b(2mm)]{am-cas-12(0.71)}
\lbl[l]{3,185; $\boxed{a=-1,\,1/7< b\leq 1/6}$}

\lbl[c]{25,169; $P_1$}
\lbl[l]{28,138; $Z_1$}
\lbl[l]{28,130; $Y_2$}

\lbl[l]{72,88; $Z_2$}
\lbl[c]{38,50; $P_2$}

\lbl[r]{61,89; $X_2$}
\lbl[r]{75,75; $p$}
\lbl[r]{85,65; $Q$}

\lbl[c]{135,15; $P_3$}
\lbl[c]{171,74; $P_4$}

\lbl[l]{109,138; $S$}
\lbl[l]{91,117; $Y_1$}

\lbl[l]{91,83; $X_1$}

\lbl[c]{150,181; $P_5$}
\end{lpic}

\caption{The graph $\Gamma$ for $a=-1$ and $1/7< b\leq 1/6$.  Here $p=(-b,2b-1)$, $P_1=(-5b, -4b+1)$, $P_2=(9b-2,-1)$, $P_3=(-9b+2,10b-3)$,
$P_4=(-19b+4,2b-1)$, $P_5=(-21b+4,-16b+3)$, $Q=(0,b-1)$, $S=(-5b+1,0)$, $X_1=(0,2b-1)$, $,X_2=(-2b,3b-1)$, $Y_1=(0,5b-1)$, $Y_2=(-5b, 6b-1)$, $Z_1=(-5b,0)$ and $Z_2=(5b-1, -4b)$.
}\label{f:12}
\end{figure}

\newpage

\begin{figure}[H]
\footnotesize
\centering

\begin{lpic}[l(2mm),r(2mm),t(2mm),b(2mm)]{am-cas-13(0.71)}
\lbl[l]{3,185; $\boxed{a=-1,\,1/6< b\leq 3/16}$}

\lbl[c]{21,171; $P_1$}
\lbl[l]{23,149; $Y_2$}
\lbl[l]{30,143; $Z$}

\lbl[c]{65,43; $P_2$}

\lbl[r]{67,95; $X_2$}
\lbl[r]{82,80; $p$}
\lbl[r]{100,65; $Q$}

\lbl[c]{139,22; $P_3$}
\lbl[c]{162,75; $P_4$}

\lbl[l]{106,143; $S$}
\lbl[r]{100,130; $Y_1$}

\lbl[r]{100,85; $X_1$}

\lbl[c]{130,160; $P_5$}

\lbl[r]{82,70; $Y_3$}
\end{lpic}

\caption{The graph $\Gamma$ for $a=-1$ and $1/6< b\leq 3/16$.  Here $p=(-b,2b-1)$, $P_1=(-5b, -4b+1)$, $P_2=(9b-2,-1)$, $P_3=(-9b+2,10b-3)$,
$P_4=(-19b+4,2b-1)$, $P_5=(-21b+4,-16b+3)$, $Q=(0,b-1)$, $S=(-5b+1,0)$, $X_1=(0,2b-1)$, $X_2=(-2b,3b-1)$, $Y_1=(0,5b-1)$, $Y_2=(-5b, 6b-1)$, $Y_3=(-b,-10b+1)$ and $Z=(b-1,0)$.	
}\label{f:13}
\end{figure}

\vspace{1.3cm}

\footnotesize{$\boxed{a=-1,\,3/16< b<4/15}$}

\vspace{0.2cm}

\begin{figure}[H]
\caption{In the case $a=-1$ and $3/16< b<4/15$ the graph $\Gamma$ reduces to the fixed  point  $(-b,2b-1)$ and there is no need to plot the figure.}\label{f:14}
\end{figure}

%
%
%
%
%

\newpage

\begin{figure}[H]
\footnotesize
\centering

\begin{lpic}[l(2mm),r(2mm),t(2mm),b(2mm)]{am-cas-17-new(0.71)}
\lbl[l]{3,185; $\boxed{a=-1,\,4/15\le b\leq 2/7}$}

\lbl[r]{103,90; $p$}

\lbl[r]{124,68; $Q$}

\lbl[l]{156,28; $P_2$}

\lbl[l]{129,85; $X_1$}

\lbl[l]{171,85; $P_3$}

\lbl[l]{87,112; $X_2$}

\lbl[r]{97,130; $Y_1$}

\lbl[l]{129,155; $S$}

\lbl[l]{131,167; $P_4$}

\lbl[c]{18,113; $P_1$}
\lbl[r]{77,113; $Y_2$}

\end{lpic}

\caption{The graph $\Gamma$ for $a=-1$ and $4/15< b\leq 2/7$. Here  $p=(-b,2b-1)$, $P_1=(-5b,-4b+1)$, $P_2=(9b-2,-8b+1)$, $P_3=(17b-4,2b-1)$, $P_4=(15b-4,20b-5)$, $Q=(0,b-1)$, $S=(0,5b-1)$,
$X_1=(0,2b-1)$, $X_2=(-2b,3b-1)$, $Y_1=(-5b+1,0)$ and $Y_2=(5b-2,-4b+1)$.}\label{f:17}
\end{figure}

\newpage

\begin{figure}[H]
\footnotesize
\centering

\begin{lpic}[l(2mm),r(2mm),t(2mm),b(2mm)]{am-cas-18-new(0.71)}
\lbl[l]{3,185; $\boxed{a=-1,\,2/7< b\leq 1/3}$}

\lbl[r]{87,80; $p$}

\lbl[r]{112,66; $Q$}

\lbl[c]{154,15; $P_2$}

\lbl[l]{108,87; $X_1$}

\lbl[l]{174,87; $P_3$}

\lbl[r]{67,102; $X_2$}

\lbl[r]{72,110; $Y_1$}

\lbl[l]{110,139; $S$}

\lbl[c]{145,180; $P_4$}

\lbl[c]{15,87; $P_1$}
\lbl[r]{83,87; $Y_2$}

\end{lpic}

\caption{The graph $\Gamma$ for $a=-1$ and $2/7< b\leq 1/3$. Here  $p=(-b,2b-1)$, $P_1=(-5b,-4b+1)$, $P_2=(9b-2,-8b+1)$, $P_3=(17b-4,2b-1)$, $P_4=(15b-4,20b-5)$, $Q=(0,b-1)$, $S=(0,5b-1)$, 
$X_1=(0,2b-1)$, $X_2=(-2b,3b-1)$, $Y_1=(-5b+1,0)$ and $Y_2=(5b-2,-4b+1)$.}\label{f:18}
\end{figure}

\newpage

\begin{figure}[H]
\footnotesize
\centering

\begin{lpic}[l(2mm),r(2mm),t(2mm),b(2mm)]{am-cas-19(0.71)}
\lbl[l]{3,185; $\boxed{a=-1,\,1/3< b\leq 1/2}$}

\lbl[r]{66,76; $p$}

\lbl[l]{86,64; $Q$}

\lbl[l]{140,9; $P_2$}

\lbl[l]{86,72; $X_1$}

\lbl[l]{175,72; $P_3$}

\lbl[r]{53,90; $X_2$}

\lbl[r]{70,86; $Y$}

\lbl[l]{86,116; $S$}

\lbl[l]{155,187; $P_4$}

\lbl[c]{18,57; $P_1$}
\lbl[r]{75,57; $X_3$}

\end{lpic}

\caption{The graph $\Gamma$ for $a=-1$ and $1/3< b\leq 1/2$. Here  $p=(-b,2b-1)$, $P_1=(-5b,-4b+1)$, $P_2=(9b-2,-8b+1)$, $P_3=(17b-4,2b-1)$, $P_4=(15b-4,20b-5)$, $Q=(0,b-1)$, $S=(0,5b-1)$, 
$X_1=(0,2b-1)$, $X_2=(-2b,3b-1)$, $X_3=(-b,-4b+1)$ and $Y=(-5b+1,0)$.}\label{f:19}
\end{figure}

\newpage

\begin{figure}[H]
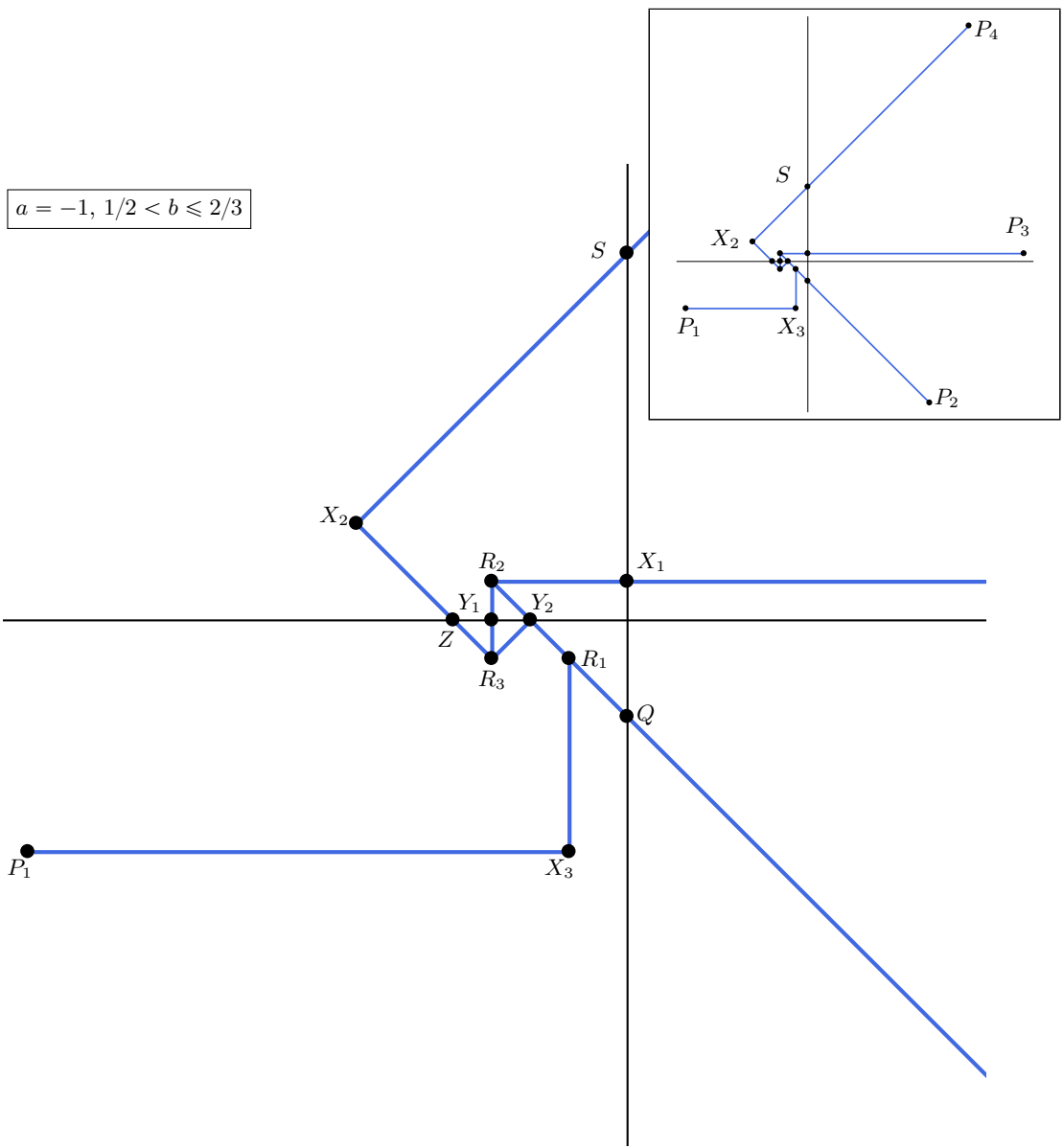

\footnotesize
\centering

\begin{lpic}[l(2mm),r(2mm),t(2mm),b(2mm)]{am-cas-20-zoom-new(0.71)}
\lbl[l]{3,185; $\boxed{a=-1,\,1/2< b\leq 2/3}$}

\lbl[l]{117,177; $S$}
\lbl[l]{64,125; $X_2$}
\lbl[l]{87,101; $Z$}
\lbl[l]{95,93; $R_3$}
\lbl[l]{95,116; $R_2$}
\lbl[l]{91,108; $Y_1$}
\lbl[r]{110,108; $Y_2$}
\lbl[l]{115,97; $R_1$}
\lbl[l]{126,86; $Q$}
\lbl[l]{126,116; $X_1$}
\lbl[l]{108,56; $X_3$}
\lbl[l]{3,56; $P_{1}$}

\end{lpic}
\caption{Detail of the the graph $\Gamma$ for $a=-1$ and $1/2< b\leq 2/3$. Here $P_{1}=(-b -2, -1)$, $P_{2}=(b +2, -3)$, $P_{3}=(b +4, 2 b -1)$, $P_{4}=(-b +4, 5)$, $Q=(0, b -1)$, $R_1=(3 b -2, -2 b +1)$, $R_2=(-b, 2 b -1)$, $R_3=(-b, -2 b +1)$, $S=(0, b +1)$, $X_1=(0, 2 b -1)$,
$X_2=(-2 b, -b +1)$,  $X_3=(3 b -2, -1)$, $Y_1=(-b,0)$,
$Y_2=(b -1, 0)$ and 
 $Z=(-3 b +1, 0)$.    Beside, larger view.}\label{f:20}
\end{figure}
\vspace{-19,5cm}\hspace{9cm}\begin{lpic}[l(2mm),r(2mm),t(2mm),b(2mm),figframe(0.2mm)]{am-cas-20-new(0.29)}

\lbl[l]{60,117; $S$}
\lbl[l]{29,86; $X_{2}$}

\lbl[r]{75,44; $X_{3}$}

\lbl[r]{25,44; $P_{1}$}
\lbl[l]{136,10; $P_{2}$}
\lbl[l]{170,92; $P_{3}$}
\lbl[l]{155,186; $P_{4}$}
\end{lpic}

\newpage

\begin{figure}[H]
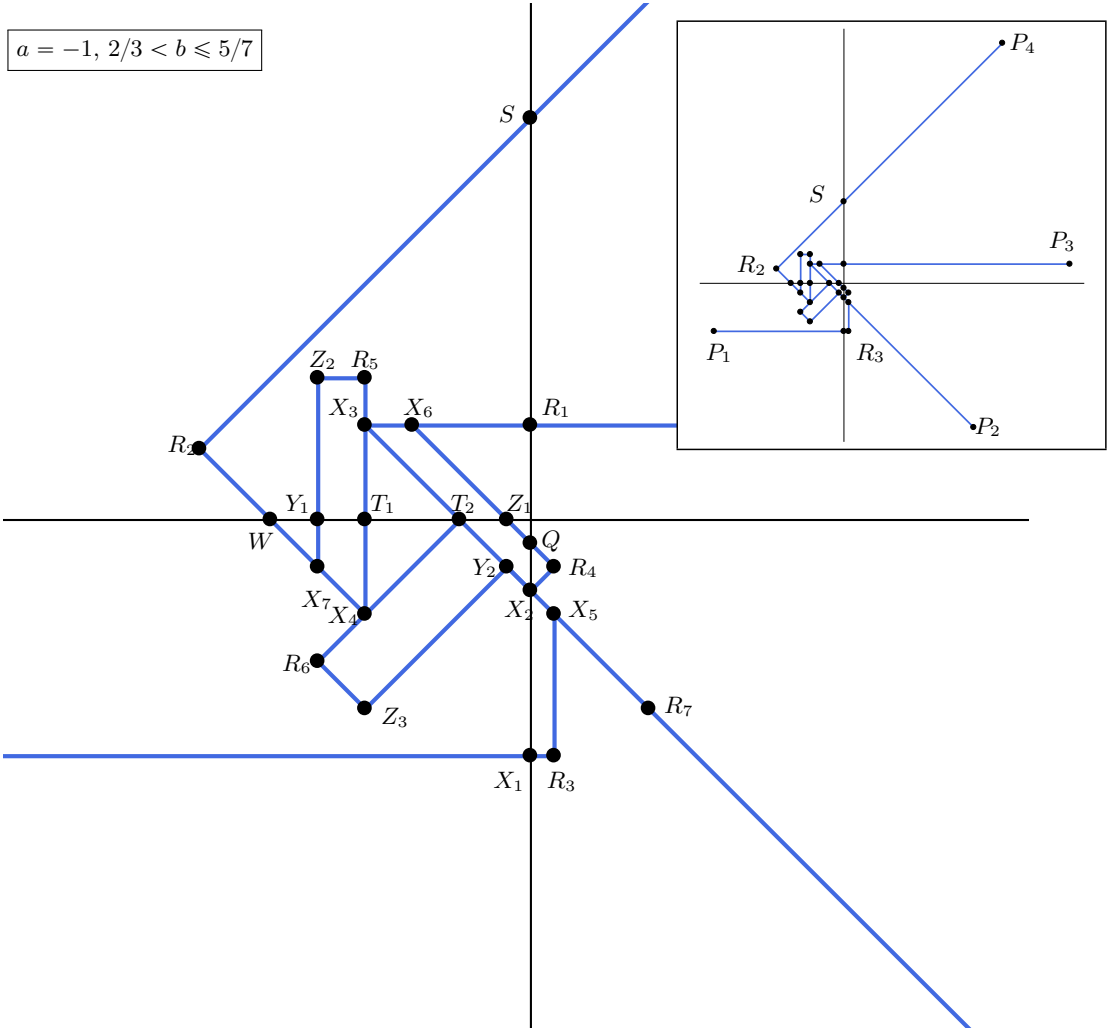

\footnotesize
\centering

\begin{lpic}[l(2mm),r(2mm),t(2mm),b(2mm)]{am-cas-A-zoom-v3b(0.71)}
\lbl[l]{3,185; $\boxed{a=-1,\,2/3< b\leq 5/7}$}

\lbl[l]{95,173; $S$}
\lbl[l]{33,111; $R_2$}
\lbl[l]{48,93; $W$}
\lbl[l]{58,82; $X_7$}
\lbl[l]{63,79; $X_4$}
\lbl[r]{60,70; $R_6$}
\lbl[l]{73,60; $Z_3$}
\lbl[l]{90,88; $Y_2$}
\lbl[l]{96,80; $X_2$}
\lbl[l]{108,80; $X_5$}
\lbl[l]{104,48; $R_3$}
\lbl[l]{94,48; $X_1$}
\lbl[l]{108,88; $R_4$}
\lbl[l]{103,93; $Q$}
\lbl[c]{99,100; $Z_1$}
\lbl[c]{80,118; $X_6$}
\lbl[c]{66,118; $X_3$}
\lbl[c]{70,127; $R_5$}
\lbl[c]{62,127; $Z_2$}
\lbl[l]{55,100; $Y_1$}
\lbl[l]{103,118; $R_1$}
\lbl[l]{71,100; $T_1$}
\lbl[l]{86,100; $T_2$}
\lbl[l]{126,62; $R_7$}

\end{lpic}

\caption{Detail of the the graph $\Gamma$ for $a=-1$ and $2/3< b\leq 5/7$. Here $P_{1}=(-b -2, -1)$, $P_{2}=(b +2, -3)$, $P_{3}=(b +4, 2 b -1
)$, $P_{4}=(-b +4, 5)$, $Q=(0, 7 b -5)$, $R_1=(0, 2 b -1)$, $R_2=(-2 b, 1-b)$, $R_3=(3 b -2, -1)$,   $R_4=(3 b -2, 4 b -3)$,  $R_5=(-b
, 8 b -5)$, $R_6=(-7 b +4, -8 b +5)$,  $R_7=(15b-10,-14b+9)$, $S=(0, b +1)$, $T_1=(-b, 0)$, $T_2=(b -1, 0)$, $W=(1-3 b, 0)$, 
$X_1=(0, -1)$,
$X_2=(0, b -1)$, $X_3=(-b, 2 b -1)$,
$X_4=(-b, 1-2 b)$, $X_5=(3 b -2
, 1-2 b)$, $X_6=(5 b -4, 2 b -1)$, $X_7=(-7 b +4, 4 b -3)$, $Y_1=(-7 b +4, 0)$, $Y_2=(7 b -5, -6 b +4)$,   $Z_1=(7 b -5, 0)$, $Z_{2}=(-7 b +4, 8 b -5)$ and
$Z_3=(-b, 9-14 b)$.  Beside, larger view. }\label{f:A}
\end{figure}
\vspace{-19,5cm}\hspace{9cm}\begin{lpic}[l(2mm),r(2mm),t(2mm),b(2mm),figframe(0.2mm)]{am-cas-A(0.29)}

\lbl[l]{60,117; $S$}
\lbl[l]{27,84; $R_2$}

\lbl[r]{95,44; $R_3$}

\lbl[r]{25,44; $P_{1}$}
\lbl[l]{136,10; $P_{2}$}
\lbl[l]{170,95; $P_{3}$}

\lbl[l]{152,186; $P_{4}$}
\end{lpic}

\newpage

\begin{figure}[H]
\footnotesize
\centering

\begin{lpic}[l(2mm),r(2mm),t(2mm),b(2mm)]{am-cas-B-zoom-new(0.71)}
\lbl[l]{3,185; $\boxed{a=-1,\,5/7< b\leq 19/26}$}

\lbl[l]{111,164; $S$}
\lbl[l]{70,125; $Y_2$}
\lbl[c]{63,115; $Z_2$}
\lbl[l]{44,99; $R_2$}
\lbl[l]{56,88; $Q$}
\lbl[r]{65,80; $X_7$}
\lbl[l]{69,72; $T_4$}
\lbl[c]{79,65; $X_4$}
\lbl[c]{62,49; $R_6$}
\lbl[l]{72,36; $Y_3$}

\lbl[l]{93,36; $Z_3$}
\lbl[l]{110,36; $X_1$}
\lbl[l]{123,36; $R_3$}

\lbl[l]{128,66; $X_5$}
\lbl[c]{126,72; $T_2$}
\lbl[c]{111,64; $W_1$}
\lbl[c]{111,74; $X_2$}
\lbl[c]{105,80; $\Pi_2$}
\lbl[c]{131,82; $R_4$}
\lbl[c]{126,89; $Y_1$}
\lbl[c]{121,93; $Z_1$}

\lbl[c]{102,110; $X_6$}
\lbl[c]{94,110; $T_3$}
\lbl[c]{80,110; $X_3$}

\lbl[c]{121,110; $R_1$}

\lbl[c]{99,97; $W_2$}
\lbl[c]{89,89; $\Pi_1$}
\lbl[c]{89,82; $W_3$}
\lbl[c]{72,89; $T_1$}
\lbl[c]{89,125; $R_5$}

\end{lpic}

\caption{Detail of the the graph $\Gamma$ for $a=-1$ and $5/7< b\leq 19/26$. Here 
$P_{1}=(-b -2, -1)$, $P_{2}=(b +2, -3)$, $P_3=(b +4, 2 b -1)$,  $P_{4}=(-b +4, 5)$, $Q=(1-3 b
, 0)$, $R_1=(0, 2 b -1)$, $R_2=(-2 b, 1-b)$, $R_3=(3 b -2, -1)$, $R_4=(3 b -2, 4 b -3)$, $R_5=(-b, 8 b -5)$, $R_6=(-7 b +4, 
-8 b +5)$, $S=(0, b +1)$, $T_1=(-7 b +4, 0)$, 
$T_2=(7 b -5, -6 b +4)$, $T_3=(13 b -10, 2 b -1)$, $T_4=(-15 b +10, 12 b -9)$, 
$W_1=(0, -13 b +9)$, 
$W_2=(13 b -10, -12 b +9)$, $W_3=(-b, 26 b -19)$, 
$X_1=(0, -1)$, $X_2=(0, b -1)$, $X_3=(-b, 2 b -1)$, $X_4=(-b, 1-2 b)$, $X_5=(3 b -2, 1-2 b)$, $X_6=(5 b -4, 2 b -1)$, 
$X_7=(-7 b +4, 4 b -3)$, 
$Y_1=(7 b -5, 0)$, 
$Y_{2}=(7 b -6, 8 b -5)$, $Y_3=(-15 b +10, -1)$, 
$Z_1=(0, 7 b -5)$, 
$Z_2=(-7 b +4, -6 b +5)$, $Z_3=(13 b -10, -1)$, 
$\Pi_1=(-b, 0)$ and $\Pi_2=(b -1, 0)$. Beside, larger view.}\label{f:B}
\end{figure}
\vspace{-19cm}\hspace{9.3cm}\begin{lpic}[l(2mm),r(2mm),t(2mm),b(2mm),figframe(0.2mm)]{am-cas-B-new(0.29)}

\lbl[l]{60,117; $S$}
\lbl[l]{27,84; $R_2$}

\lbl[r]{95,44; $R_3$}


\lbl[r]{25,44; $P_{1}$}
\lbl[l]{138,10; $P_{2}$}
\lbl[l]{170,95; $P_{3}$}

\lbl[l]{151,186; $P_{4}$}
\end{lpic}

\newpage

\begin{figure}[H]
\footnotesize
\centering

\begin{lpic}[l(2mm),r(2mm),t(2mm),b(2mm)]{am-cas-C-zoom-new(0.71)}
\lbl[l]{3,185; $\boxed{a=-1,\,19/26< b\leq 20/27}$}

\lbl[l]{114,174; $S$}
\lbl[l]{77,137; $Y_2$}
\lbl[c]{66,124; $Z_2$}
\lbl[l]{49,108; $R_2$}
\lbl[l]{60,97; $Q$}
\lbl[r]{69,92; $X_7$}
\lbl[l]{68,88; $T_4$}
\lbl[c]{84,74; $X_4$}
\lbl[c]{66,56; $R_6$}
\lbl[l]{72,46; $Y_3$}

\lbl[l]{99,46; $Z_3$}
\lbl[l]{114,46; $X_1$}
\lbl[l]{128,46; $R_3$}

\lbl[l]{132,76; $X_5$}
\lbl[c]{132,81; $T_2$}
\lbl[c]{116,69; $W_1$}
\lbl[c]{117,81; $X_2$}
\lbl[c]{111,98; $\Pi_2$}
\lbl[c]{126,92; $R_4$}
\lbl[c]{132,98; $Y_1$}
\lbl[c]{126,103; $Z_1$}

\lbl[c]{107,120; $X_6$}
\lbl[c]{99,120; $T_3$}
\lbl[c]{85,120; $X_3$}

\lbl[c]{117,120; $R_1$}

\lbl[c]{99,100; $W_2$}
\lbl[c]{93,97; $\Pi_1$}
\lbl[c]{93,104; $W_3$}
\lbl[c]{74,97; $T_1$}
\lbl[c]{93,137; $R_5$}


\end{lpic}

\caption{Detail of the the graph $\Gamma$ for $a=-1$ and $19/26< b\leq 20/27$. Here 
$P_{1}=(-b -2, -1)$, $P_{2}=(b +2, -3)$, $P_3=(b +4, 2 b -1)$, $P_{4}=(-b +4, 5)$, 
$Q=(1-3 b, 0)$, $R_1=(0, 2 b -1)$, $R_2=(-2 b, 1-b)$, $R_3=(3 b -2, -1)$, $R_4=(3 b -2, 4 b -3)$, $R_5=(-b, 8 b -5)$, $R_6=(-7 b +4,-8 b +5)$, $S=(0, b +1)$, 
$T_1=(-7 b +4, 0)$, $T_2=(7 b -5, -6 b +4)$, $T_3=(13 b -10, 2 b -1)$, $T_4=(-15 b +10, 12 b -9)$, $W_1=(0, -13 b +9)$, 	$W_2=(13 b -10, -12 b +9)$, $W_3=(-b, 26 b -19)$, 
	$X_1=(0, -1)$, $X_2=(0, b -1)$, 	$X_3=(-b, 2 b -1)$, $X_4=(-b, 1-2 b)$, $X_5=(3 b -2, 1-2 b)$, $X_6=(5 b -4, 2 b -1)$, $X_7=(-7 b +4, 4 b -3)$, 
$Y_1=(7 b -5, 0)$, $Y_{2}=(7 b -6, 8 b -5)$,
$Y_3=(-15 b +10, -1)$, 
$Z_1=(0, 7 b -5)$, $Z_2=(-7 b +4, -6 b +5)$, $Z_3=(13 b -10, -1)$, $\Pi_1=(-b, 0)$ and 
$\Pi_2=(b -1, 0)$. Beside, larger view.}\label{f:C}
\end{figure}
\vspace{-20.5cm}\hspace{9.5cm}\begin{lpic}[l(2mm),r(2mm),t(2mm),b(2mm),figframe(0.2mm)]{am-cas-C-new(0.29)}

\lbl[l]{60,117; $S$}
\lbl[l]{27,84; $R_2$}

\lbl[r]{95,44; $R_3$}


\lbl[r]{25,44; $P_{1}$}
\lbl[l]{138,10; $P_{2}$}
\lbl[l]{170,95; $P_{3}$}
\lbl[l]{151,186; $P_{4}$}
\end{lpic}

\newpage

\begin{figure}[H]
\footnotesize
\centering

\begin{lpic}[l(2mm),r(2mm),t(2mm),b(2mm)]{am-cas-D-zoom-new(0.71)}
\lbl[l]{3,185; $\boxed{a=-1,\,20/27< b\leq 3/4 }$}

\lbl[l]{111,167; $S$}
\lbl[l]{76,132; $Y_2$}
\lbl[c]{59,113; $Z_2$}
\lbl[l]{45,101; $R_2$}
\lbl[l]{55,91; $Q$}
\lbl[r]{63,85; $X_7$}
\lbl[l]{63,82; $T_4$}
\lbl[c]{79,66; $X_4$}
\lbl[c]{61,48; $R_6$}
\lbl[l]{67,40; $Y_3$}

\lbl[l]{99,40; $Z_3$}
\lbl[l]{114,40; $X_1$}
\lbl[l]{128,40; $R_3$}

\lbl[l]{132,66; $X_5$}
\lbl[c]{132,71; $T_2$}
\lbl[c]{113,59; $W_1$}
\lbl[c]{113,78; $X_2$}
\lbl[c]{111,91; $\Pi_2$}
\lbl[c]{133,85; $R_4$}
\lbl[c]{131,92; $Y_1$}
\lbl[c]{123,99; $Z_1$}

\lbl[c]{109,114; $X_6$}
\lbl[c]{102,114; $T_3$}
\lbl[c]{89,114; $X_3$}

\lbl[c]{123,114; $R_1$}

\lbl[c]{99,91; $W_2$}
\lbl[c]{90,91; $\Pi_1$}
\lbl[c]{80,105; $W_3$}
\lbl[c]{67,92; $T_1$}
\lbl[c]{91,132; $R_5$}

\end{lpic}

\caption{Detail of the the graph $\Gamma$ for $a=-1$ and $20/27< b\leq 3/4$. Here $P_{1}=(-b -2, -1)$, $P_{2}=(b +2, -3)$, $P_3=(b +4, 2 b -1)$, $P_{4}=(-b +4, 5)$, $Q=(1-3 b, 0)$, 
$R_1=(0, 2 b -1)$,
$R_2=(-2 b, 1-b)$, 
$R_3=(3 b -2, -1)$,
$R_4=(3 b -2, 4 b -3)$, 
$R_5=(-b, 8 b -5)$, 
$R_6=(-7 b +4, -8 b +5)$, 
$S=(0, b +1)$, 
$T_1=(-7 b +4, 0)$, $T_2=(7 b -5, -6 b +4)$, 
$T_3=(13 b -10, 2 b -1)$, $T_4=(-15 b +10, 12 b -9)$, $W_1=(0, -13 b +9)$, $W_2=(13 b -10, -12 b +9)$, $W_3=(-b, 26 b -19)$, $X_1=(0, -1)$, $X_2=(0, b -1)$, $X_3=(-b, 2 b -1)$, 	
$X_4=(-b, 1-2 b)$, $X_5=(3 b -2, 1-2 b)$, 	$X_6=(5 b -4, 2 b -1)$, $X_7=(-7 b +4, 4 b -3)$, $Y_1=(7 b -5, 0)$, $Y_{2}=(7 b -6, 8 b -5)$, $Y_3=(-15 b +10, -1)$, $Z_1=(0, 7 b -5)$, 
$Z_2=(-7 b +4, -6 b +5)$, $Z_3=(13 b -10, -1)$,
$\Pi_1=(-b, 0)$ and 	$\Pi_2=(b -1, 0)$.
	Beside, larger view.}\label{f:D}
\end{figure}
\vspace{-20.5cm}\hspace{9.5cm}\begin{lpic}[l(2mm),r(2mm),t(2mm),b(2mm),figframe(0.2mm)]{am-cas-D-new(0.29)}

\lbl[l]{60,117; $S$}
\lbl[l]{27,84; $R_2$}

\lbl[r]{95,44; $R_3$}


\lbl[r]{25,44; $P_{1}$}
\lbl[l]{138,10; $P_{2}$}
\lbl[l]{170,95; $P_{3}$}

\lbl[l]{151,186; $P_{4}$}
\end{lpic}

\newpage

\begin{figure}[H]
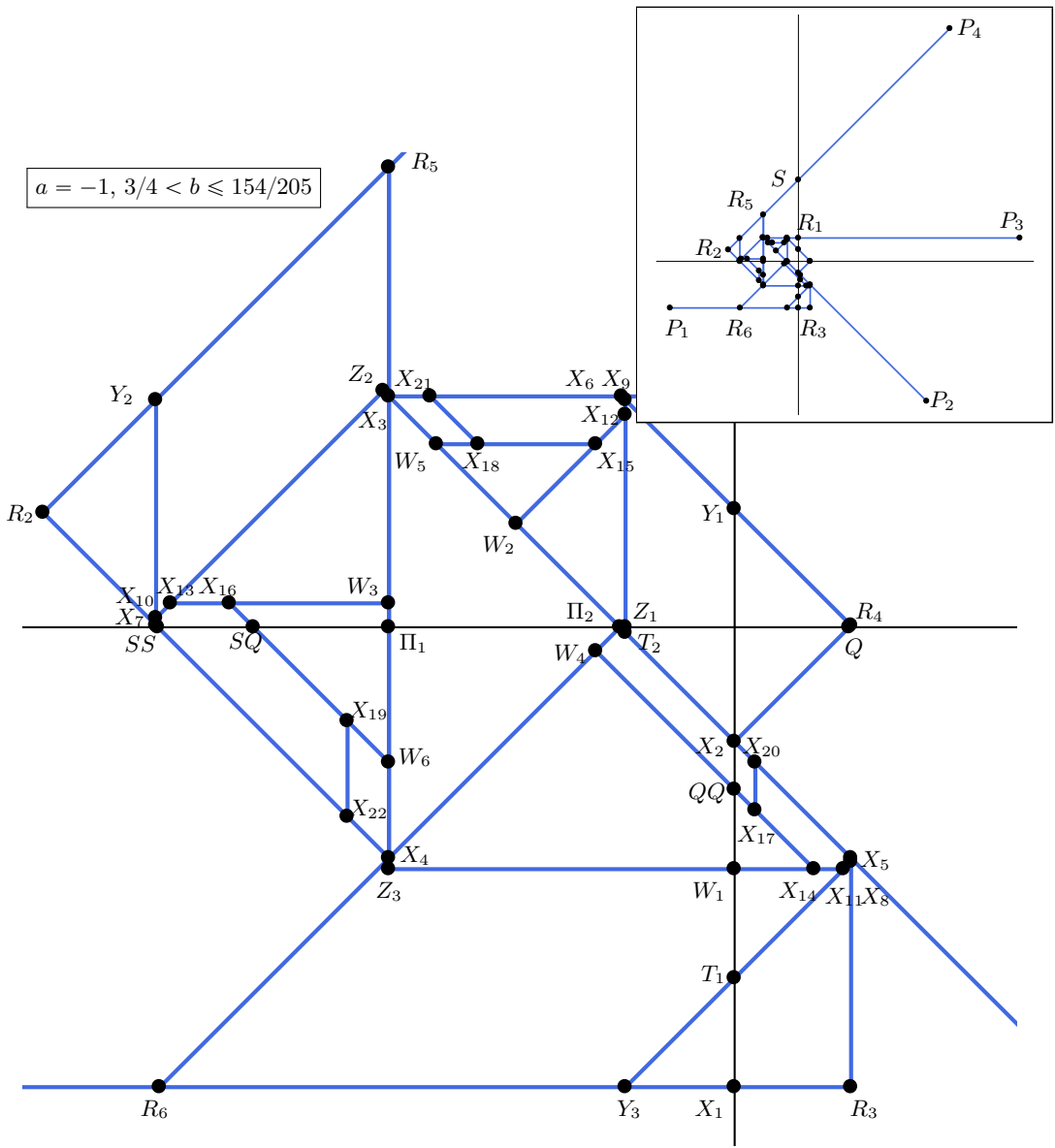

	
\footnotesize
\centering
\vspace{-1.5cm}
\begin{lpic}[l(2mm),r(2mm),t(2mm),b(2mm)]{am-cas-E0b(0.71)}
\lbl[l]{3,187; $\boxed{a=-1,\,3/4<  b\leq 154/205 }$}

\lbl[l]{-1,124; $R_{2}$}
\lbl[c]{78,58; $X_{4}$}
\lbl[l]{25,9; $R_{6}$}
\lbl[l]{162,9; $R_{3}$}
\lbl[l]{164,57; $X_{5}$}
\lbl[l]{65,151; $Z_{2}$}
\lbl[l]{117,9; $Y_{3}$}
\lbl[l]{164,50; $X_{8}$}
\lbl[c]{161,50; $X_{11}$}
\lbl[c]{73,51; $Z_{3}$}
\lbl[c]{80,192; $R_{5}$}
\lbl[c]{108,97; $W_{4}$}
\lbl[c]{152,51; $X_{14}$}
\lbl[c]{145,79; $X_{20}$}
\lbl[c]{144,62; $X_{17}$}
\lbl[c]{110,105; $\Pi_2$}
\lbl[c]{24,108; $X_{10}$}
\lbl[l]{19,147; $Y_{2}$}
\lbl[c]{23,104; $X_{7}$}

\lbl[l]{28,110; $X_{13}$}
\lbl[c]{68,110; $W_{3}$}
\lbl[c]{40,110; $X_{16}$}
\lbl[c]{78,77; $W_{6}$}
\lbl[c]{69,86; $X_{19}$}
\lbl[c]{69,67; $X_{22}$}
\lbl[l]{67,142; $X_{3}$}
\lbl[l]{130,152; $R_1$}
\lbl[l]{107,150; $X_6$}
\lbl[l]{163,105; $R_{4}$}
\lbl[5]{135,79; $X_2$}
\lbl[l]{121,99; $T_2$}
\lbl[l]{114,150; $X_9$}
\lbl[l]{110,143; $X_{12}$}
\lbl[l]{91,118; $W_2$}
\lbl[l]{113,134; $X_{15}$}
\lbl[l]{74,134; $W_{5}$}
\lbl[l]{87,134; $X_{18}$}
\lbl[l]{74,150; $X_{21}$}
\lbl[c]{25,99; $SS$}
\lbl[l]{75,99; $\Pi_{1}$}
\lbl[r]{138,35; $T_{1}$}
\lbl[r]{138,51; $W_1$}
\lbl[r]{138,124; $Y_1$}
\lbl[l]{161,98; $Q$}
\lbl[l]{120,105; $Z_{1}$}
\lbl[l]{132,9; $X_{1}$}

\lbl[r]{138,70; $QQ$}
\lbl[l]{42,99; $SQ$}
\end{lpic}
\caption{The graph $\Gamma$ for $a=-1$ and $3/4<  b\leq 154/205$. Here 
$P_{1}=(-b-2, -1)$, 
$P_{2}=(b +2, -3)$, 
$P_{3}=(b +4, 2 b -1)$, 
$P_{4}=(-b +4, 5)$, 
$Q=(-b+1, 0)$, 
$R_1=(0, 2 b -1)$, 
$R_{2}=(-2 b, -b+1)$, 
$R_{3}=(3 b -2, -1)$, 
$R_{4}=(3 b -2, 4 b -3)$, 
$R_{5}=(-b, 1)$, 
$R_{6}=(b -2, -1)$, 
$S=(0, b +1)$, 
$T_{1}=(0, -13b+9)$, 
$T_{2}=(13 b -10, -12 b+9)$, 
$W_1=(0, -26 b+19)$, 
$W_{2}=(26 b -20, -25 b +19)$,
$W_{3}=(-b, 52 b -39)$, 
$W_{4}=(-51 b +38, -52 b +39)$, 
$W_{5}=(103 b -78, -102 b +77)$, 
$W_{6}=(-b, 206 b -155)$, 
$X_1=(0,-1)$,
$X_{2}=(0, b -1)$, 
$X_{3}=(-b, 2 b -1)$, 
$X_{4}=(-b, -2 b+1)$, 
$X_5=(3 b -2, -2 b +1)$,
$X_6=(5 b -4, 2 b -1)$, 
$X_{7}=(-7 b+4, 4 b -3)$, 
$X_{8}=(3 b -2, -10b+7)$, 
$X_{9}=(13 b -10, -6 b+5)$, 
$X_{10}=(-7 b+4, 20 b -15)$, 
$X_{11}=(10-13 b, -26 b+19)$, 
$X_{12}=(13 b -10, -38 b+29)$, 
$X_{13}=(25 b -20, 52 b -39)$, 
$X_{14}=(-77 b +58, -26 b+19)$, 
$X_{15}=(-51 b +38, -102 b +77)$, 
$X_{16}=(153 b -116, 52 b -39)$
$X_{17}=(-205 b +154, 102 b -77)$, 
$X_{18}=(-307 b +230, -102 b +77)$, 
$X_{19}=(409 b -308, -204 b +153)$, 
$X_{20}=(-205 b +154, 206 b -155)$, 
$X_{21}=(-411 b +308, 2 b -1)$, 
$X_{22}=(409 b -308, -412 b +309)$, 
$Y_{1}=(0, 7 b -5)$, 
$Y_{2}=(-7 b+4, -6 b+5)$, 
$Y_{3}=(13 b -10, -1)$, 
$Z_{1}=(13b-10,0)$,
$Z_{2}=(-13b+9, 14 b -10)$, 
$Z_{3}=(-b, -26 b+19)$, 
$\Pi_{1}=(-b, 0)$, 
$\Pi_{2}=(b -1, 0)$, 
$QQ=(0,-103b+77)$,
$SQ=(205b-155,0)$ and 
$SS=(-3 b +1, 0)$.
}\label{f:E}
~\vspace{-25.8cm}

\hspace{4.2cm}$\phantom{espaiespaiespaiespaiespaiespaie}$
{\begin{lpic}[l(2mm),r(2mm),t(2mm),b(2mm),figframe(0.2mm)]{am-cas-E(0.29)}

\lbl[l]{63,115; $S$}
\lbl[l]{27,82; $R_2$}

\lbl[r]{55,44; $R_6$}
\lbl[r]{90,44; $R_3$}

\lbl[l]{42,105; $R_5$}

\lbl[l]{75,94; $R_{1}$}


\lbl[r]{25,44; $P_{1}$}
\lbl[l]{138,10; $P_{2}$}
\lbl[l]{171,95; $P_{3}$}
\lbl[l]{151,186; $P_{4}$}
\end{lpic}}

\end{figure}

\newpage

\begin{figure}[H]
\vspace{-0.5cm}
\footnotesize
\centering
\begin{lpic}[l(2mm),r(2mm),t(2mm),b(2mm)]{am-cas-F0b(0.71)}
\lbl[l]{2,194; $\boxed{a=-1,\,154/205<  b\leq 155/206 }$}

\lbl[l]{0,129; $R_{2}$}
\lbl[c]{75,58; $X_4$}
\lbl[l]{24,11; $R_6$}
\lbl[l]{158,11; $R_3$}
\lbl[l]{161,61; $X_5$}
\lbl[l]{65,153; $Z_2$}
\lbl[l]{114,11; $Y_3$}
\lbl[l]{161,58; $X_8$}
\lbl[c]{157,50; $X_{11}$}
\lbl[c]{71,50; $Z_3$}
\lbl[c]{77,191; $R_5$}
\lbl[c]{104,98; $W_4$}
\lbl[c]{147,50; $X_{14}$}
\lbl[c]{129,95; $X_{20}$}
\lbl[c]{129,77; $X_{17}$}
\lbl[c]{108,106; $\Pi_{2}$}
\lbl[c]{22,110; $X_{10}$}
\lbl[l]{20,150; $Y_2$}
\lbl[c]{18,106; $X_7$}

\lbl[l]{26,116; $X_{13}$}
\lbl[c]{66,116; $W_3$}
\lbl[c]{51,116; $X_{16}$}
\lbl[l]{72,95; $X_{19}$}

\lbl[l]{63,145; $X_{3}$}
\lbl[l]{130,152; $R_1$}
\lbl[l]{110,152; $X_6$}
\lbl[l]{160,106; $R_4$}
\lbl[5]{131,81; $X_2$}
\lbl[l]{120,100; $T_2$}
\lbl[l]{109,145; $X_9$}
\lbl[l]{107,141; $X_{12}$}
\lbl[l]{89,120; $W_2$}
\lbl[l]{109,130; $X_{15}$}
\lbl[l]{78,130; $X_{18}$}

\lbl[c]{25,100; $SS$}
\lbl[l]{72,100; $\Pi_1$}
\lbl[r]{135,36; $T_1$}
\lbl[r]{135,50; $W_1$}
\lbl[r]{135,126; $Y_1$}
\lbl[l]{159,100; $Q$}

\lbl[l]{120,106; $Z_1$}
\lbl[l]{129,11; $X_1$}

\lbl[r]{135,62; $QQ$}
\lbl[l]{56,100; $SQ$}
\end{lpic}
\caption{The graph $\Gamma$ for $a=-1$ and $154/205<  b\leq 155/206$. Here 
$P_{1}=(-b -2, -1)$, $P_{2}=(b +2, 
-3)$, $P_{3}=(b +4, 2 b -1)$, $P_{4}=(-b +4, 5)$,  
$Q=(-b +1, 0)$, 
$R_{1}=(0, 2 b -1)$, 
$R_{2}=(-2 b, -b +1)$, 
$R_3=(3 b -2, -1)$, 
$R_{4}=(3 b -2, 4 b -3)$, 
$R_5=(-b, 1)$, 
$R_6=(b -2, -1)$, 
$S=(0, b +1)$, 
$T_{1}=(0, -13 b+9)$, 
$T_{2}=(13 b -10, -12 b +9)$, 
$W_{1}=(0, -26 b +19)$, 
$W_{2}=(26 b -20, -25 b +19)$, 
$W_{3}=(-b, 52 b -39)$, 
$W_4=(-51 b +38, -52 b +39)$, 
$X_{1}=(0,-1)$, 
$X_{2}=(0, b -1)$, 
$X_{3}=(-b, 2 b -1)$, 
$X_4=(-b, -2 b +1)$, 
$X_5=(3 b -2, -2 b +1)$,
$X_{6}=(5 b -4, 2 b -1)$, 
$X_{7}=(-7b+4, 4 b -3)$, 
$X_8=(3 b -2, -10 b +7)$, 
$X_{9}=(13 b -10, -6 b +5)$, 
$X_{10}=(-7b+4, 20 b -15)$, 
$X_{11}=(10-13 b, -26 b +19)$, 
$X_{12}=(13 b -10, -38 b +29)$, 
$X_{13}=(25 b -20, 52 b -39)$, 
$X_{14}=(-77 b +58, -26 b +19)$, 
$X_{15}=(-51 b +38, -102 b +77)$, 
$X_{16}=(153 b -116, 52 b -39)$, 
$X_{17}=(-205 b +154, 102 b -77)$, 
$X_{18}=(103 b -78, -102 b +77)$, 
$X_{19}=(-b, 206 b -155)$, 
$X_{20}=(-205 b +154, 206 b -155)$, 
$Y_1=(0, 7 b -5)$, 
$Y_{2}=(-7 b+4, -6 b+5)$, 
$Y_3=(13 b -10, -1)$, 
$Z_{1}=(13b-10,0)$, 
$Z_2=(-13 b+9, 14 b -10)$, 
$Z_3=(-b, -26 b +19)$, 
$\Pi_{1}=(-b, 0)$, 
$\Pi_{2}=(b -1, 0)$, 
$QQ=(0,-103b+77)$, $SQ=(205b-155,0)$ and 
$SS=(-3 b +1, 0)$. For $b=\frac{82}{109}$, $X_{20}=T_2$.}\label{f:F}
~\vspace{-25cm}

\hspace{4.1cm}$\phantom{espaiespaiespaiespaiespaiespaie}$
{\begin{lpic}[l(2mm),r(2mm),t(2mm),b(2mm),figframe(0.2mm)]{am-cas-F(0.29)}

\lbl[l]{60,115; $S$}
\lbl[l]{27,82; $R_{2}$}

\lbl[r]{55,44; $Y_{3}$}
\lbl[r]{90,44; $R_{3}$}

\lbl[l]{40,105; $R_{5}$}

\lbl[l]{75,94; $R_1$}

\lbl[l]{82,81; $R_{4}$}

\lbl[r]{25,44; $P_{1}$}
\lbl[l]{138,10; $P_{2}$}
\lbl[l]{171,95; $P_{3}$}
\lbl[l]{151,186; $P_{4}$}
\end{lpic}}

\end{figure}

\newpage

\begin{figure}[H]
\footnotesize
\centering

\begin{lpic}[l(2mm),r(2mm),t(2mm),b(2mm)]{am-cas-G0(0.71)}
\lbl[l]{2,191; $\boxed{a=-1,\,155/206<  b\leq 58/77 }$}

\lbl[l]{0,126; $R_{2}$}
\lbl[c]{78,54; $X_4$}
\lbl[l]{25,4; $R_6$}
\lbl[l]{162,4; $R_3$}
\lbl[l]{166,55; $X_5$}
\lbl[l]{66,153; $Z_2$}
\lbl[l]{119,4; $Y_3$}
\lbl[l]{165,48; $X_8$}
\lbl[c]{161,44; $X_{11}$}
\lbl[c]{73,44; $Z_3$}
\lbl[c]{78,190; $R_5$}
\lbl[c]{99,87; $W_4$}
\lbl[c]{146,44; $X_{14}$}
\lbl[c]{116,91; $X_{20}$}
\lbl[c]{116,81; $X_{17}$}
\lbl[c]{112,102; $\Pi_{2}$}
\lbl[c]{21,107; $X_{10}$}
\lbl[l]{20,148; $Y_{2}$}
\lbl[c]{20,102; $X_{7}$}

\lbl[l]{29,117; $X_{13}$}
\lbl[l]{74,113; $W_3$}
\lbl[r]{67,117; $X_{16}$}
\lbl[l]{74,107; $X_{19}$}

\lbl[l]{74,149; $X_{3}$}
\lbl[l]{130,152; $R_1$}
\lbl[l]{111,149; $X_{6}$}
\lbl[l]{165,104; $R_4$}
\lbl[r]{139,76; $X_2$}
\lbl[l]{123,97; $T_2$}
\lbl[l]{114,142; $X_9$}
\lbl[l]{114,138; $X_{12}$}
\lbl[l]{99,112; $W_2$}
\lbl[l]{108,119; $X_{15}$}
\lbl[l]{90,119; $X_{18}$}

\lbl[c]{31,102; $SS$}
\lbl[l]{74,102; $\Pi_{1}$}
\lbl[r]{139,29; $T_{1}$}
\lbl[r]{140,44; $W_1$}
\lbl[r]{139,124; $Y_1$}
\lbl[l]{156,102; $Q$}
\lbl[l]{123,102; $Z_1$}
\lbl[l]{134,4; $X_1$}

\end{lpic}
\caption{The graph $\Gamma$ for $a=-1$ and $155/206<  b\leq 58/77$. Here 
$P_{1}=(-b -2, -1)$, $P_{2}=(b +2, -3)$, 
$P_{3}=(b +4, 2 b -1)$, $P_{4}=(-b +4, 5)$,
$Q=(-b +1, 0)$, 
$R_{1}=(0, 2 b -1)$, 
$R_{2}=(-2 b, -b +1)$, 
$R_3=(3 b -2, -1)$, 
$R_{4}=(3 b -2, 4 b -3)$, 
$R_5=(-b, 1)$, 
$R_6=(b -2, -1)$, 
$S=(0, b +1)$, 
$T_{1}=(0, -13 b+9)$, 
$T_{2}=(13 b -10, -12 b +9)$, 
$W_{1}=(0, -26 b +19)$, 
$W_{2}=(26 b -20, -25 b +19)$, 
$W_{3}=(-b, 52 b -39)$, 
$W_4=(-51 b +38, -52 b +39)$, 
$X_{1}=(0,-1)$, 
	$X_{2}=(0, b -1)$, 
	$X_{3}=(-b, 2 b -1)$, 
	$X_4=(-b, -2 b +1)$, 
    $X_5=(3 b -2, -2 b +1)$, 
	$X_{6}=(5 b -4, 2 b -1)$, 
	$X_{7}=(-7b+4, 4 b -3)$, 
	$X_8=(3 b -2, -10 b +7)$, 
$X_{9}=(13 b -10, -6 b +5)$, 
$X_{10}=(-7b+4, 20 b -15)$, 	
	$X_{11}=(10-13 b, -26 b +19)$, 
	$X_{12}=(13 b -10, -38 b +29)$, 
	$X_{13}=(25 b -20, 52 b -39)$, 
	$X_{14}=(-77 b +58, -26 b +19)$,
	$X_{15}=(-51 b +38, -102 b +77)$, 
	$X_{16}=(153 b -116, 52 b -39)$,
	$X_{17}=(-205 b +154, 102 b -77)$, 
	$X_{18}=(103 b -78, -102 b +77)$, 
	$X_{19}=(-b, 206 b -155)$,
	$X_{20}=(-205 b +154, -206 b +155)$, 
$Y_1=(0, 7 b -5)$, $Y_{2}=(-7 b+4, -6 b+5)$, 
	$Y_3=(13 b -10, -1)$, $Z_{1}=(13b-10,0)$, $Z_2=(-13 b+9, 14 b -10)$, $Z_3=(-b, -26 b +19)$, 
	$\Pi_{1}=(-b, 0)$,	$\Pi_{2}=(b -1, 0)$ 
	 and $SS=(-3 b +1, 0)$.}\label{f:G}
~\vspace{-24cm}

\hspace{4.4cm}$\phantom{espaiespaiespaiespaiespaiespaie}$
{\begin{lpic}[l(2mm),r(2mm),t(2mm),b(2mm),figframe(0.2mm)]{am-cas-G(0.29)}

\lbl[l]{60,115; $S$}
\lbl[l]{27,82; $R_2$}

\lbl[r]{55,44; $R_6$}
\lbl[r]{90,44; $R_3$}




\lbl[r]{25,44; $P_{1}$}
\lbl[l]{138,10; $P_{2}$}
\lbl[l]{171,95; $P_{3}$}
\lbl[l]{151,186; $P_{4}$}
\end{lpic}}

\end{figure}

\newpage

\begin{figure}[H]
\footnotesize
\centering
\begin{lpic}[l(2mm),r(2mm),t(2mm),b(2mm)]{am-cas-H0(0.71)}
\lbl[l]{2,191; $\boxed{a=-1,\,58/77<  b\leq 19/25 }$}

\lbl[l]{1,127; $R_{2}$}
\lbl[l]{75,53; $X_{4}$}
\lbl[l]{26,5; $R_{6}$}
\lbl[l]{165,5; $R_{3}$}
\lbl[l]{169,55; $X_5$}
\lbl[l]{63,159; $Z_2$}
\lbl[l]{123,5; $Y_3$}
\lbl[l]{168,47; $X_8$}
\lbl[c]{159,35; $X_{11}$}
\lbl[l]{70,35; $Z_3$}
\lbl[l]{75,191; $R_5$}
\lbl[c]{84,72; $W_4$}
\lbl[r]{122,35; $X_{14}$}

\lbl[c]{113,103; $\Pi_{2}$}
\lbl[c]{19,113; $X_{10}$}
\lbl[l]{19,147; $Y_2$}
\lbl[c]{19,103; $X_7$}

\lbl[l]{34,132; $X_{13}$}
\lbl[l]{75,132; $W_3$}

\lbl[l]{75,150; $X_3$}
\lbl[l]{144,144; $R_1$}
\lbl[l]{113,150; $X_6$}
\lbl[l]{169,104; $R_4$}
\lbl[5]{136,78; $X_2$}
\lbl[l]{129,94; $T_2$}
\lbl[l]{119,142; $X_9$}
\lbl[l]{118,124; $X_{12}$}
\lbl[l]{103,108; $W_2$}

\lbl[c]{31,103; $SS$}
\lbl[l]{75,103; $\Pi_{1}$}
\lbl[r]{140,25; $T_1$}

\lbl[c]{138,35; $W_1$}
\lbl[r]{140,127; $Y_1$}
\lbl[l]{155,103; $Q$}
\lbl[l]{129,103; $Z_1$}

\lbl[l]{135,5; $X_1$}

\end{lpic}
\caption{The graph $\Gamma$ for $a=-1$ and $58/77<  b\leq 19/25$. Here $P_{1}=(-b -2, -1)$, $P_{2}=(b +2, -3)$, $P_{3}=(b +4, 2 b -1)$, $P_{4}=(-b +4, 5)$, $Q=(-b +1, 0)$, 
$R_{1}=(0, 2 b -1)$, 
$R_{2}=(-2 b, -b +1)$, 
$R_3=(3 b -2, -1)$, 
$R_{4}=(3 b -2, 4 b -3)$
$R_5=(-b, 1)$, 
$R_6=(b -2, -1)$, 
$S=(0, b +1)$, 
$T_{1}=(0, -13 b+9)$, 
$T_{2}=(13 b -10, -12 b +9)$, 
$W_{1}=(0, -26 b +19)$, 
$W_{2}=(26 b -20, -25 b +19)$,  
$W_{3}=(-b, 52 b -39)$,  
$W_4=(-51 b +38, -52 b +39)$, 
$X_{1}=(0,-1)$, 
$X_{2}=(0, b -1)$, 
$X_{3}=(-b, 2 b -1)$,  
$X_4=(-b, -2 b +1)$, 
$X_5=(3 b -2, -2 b +1)$, 
$X_{6}=(5 b -4, 2 b -1)$,
$X_{7}=(-7b+4, 4 b -3)$, 
$X_8=(3 b -2, -10 b +7)$, 
$X_{9}=(13 b -10, -6 b +5)$, 
$X_{10}=(-7b+4, 20 b -15)$, 
$X_{11}=(10-13 b, -26 b +19)$, 
$X_{12}=(13 b -10, -38 b +29)$, 
$X_{13}=(25 b -20, 52 b -39)$, 
$X_{14}=(-77 b +58, -26 b +19)$, 
$Y_1=(0, 7 b -5)$,	
$Y_{2}=(-7 b+4, -6 b+5)$, 
$Y_3=(13 b -10, -1)$, 
$Z_{1}=(13b-10,0)$, 
$Z_2=(-13 b+9, 14 b -10)$, 	
$Z_3=(-b, -26 b +19)$, 
$\Pi_{1}=(-b, 0)$,
$\Pi_{2}=(b -1, 0)$, 
and $SS=(-3 b +1, 0)$.}\label{f:H}
~\vspace{-24.1cm}

\hspace{4.2cm}$\phantom{espaiespaiespaiespaiespaiespaie}$
{\begin{lpic}[l(2mm),r(2mm),t(2mm),b(2mm),figframe(0.2mm)]{am-cas-H(0.29)}

\lbl[l]{60,115; $S$}
\lbl[l]{27,82; $R_{2}$}

\lbl[r]{55,44; $R_{6}$}
\lbl[r]{90,44; $R_{3}$}

\lbl[l]{40,105; $R_5$}

\lbl[l]{75,95; $R_1$}

\lbl[l]{85,83; $R_4$}

\lbl[r]{25,44; $P_{1}$}
\lbl[l]{138,10; $P_{2}$}
\lbl[l]{171,95; $P_{3}$}
\lbl[l]{151,186; $P_{4}$}
\end{lpic}}

\end{figure}

\newpage

\begin{figure}[H]
\footnotesize
\centering

\begin{lpic}[l(2mm),r(2mm),t(2mm),b(2mm)]{am-cas-I0(0.71)}
\lbl[l]{2,191; $\boxed{a=-1,\,19/25<  b\leq 29/38 }$}

\lbl[l]{1,125; $R_{2}$}
\lbl[l]{75,52; $X_{4}$}
\lbl[l]{28,6; $R_{6}$}
\lbl[l]{163,6; $R_{3}$}
\lbl[l]{168,53; $X_{5}$}
\lbl[l]{58,162; $Z_{2}$}
\lbl[l]{129,6; $Y_{3}$}
\lbl[l]{168,44; $X_{8}$}
\lbl[l]{152,29; $X_{11}$}
\lbl[c]{73,23; $Z_{3}$}
\lbl[c]{80,189; $R_{5}$}
\lbl[l]{75,40; $\Delta_3$}
\lbl[c]{86,32; $X_{14}$}

\lbl[c]{111,103; $\Pi_{2}$}
\lbl[c]{16,120; $X_{10}$}
\lbl[l]{18,142; $Y_2$}
\lbl[c]{16,103; $X_7$}

\lbl[l]{48,156; $X_{13}$}
\lbl[l]{66,156; $\Delta_2$}

\lbl[l]{74,150; $X_3$}
\lbl[l]{142,142; $R_1$}
\lbl[l]{112,150; $X_6$}
\lbl[l]{168,104; $R_4$}
\lbl[l]{133,76; $X_2$}
\lbl[l]{133,88; $T_2$}
\lbl[l]{122,138; $X_9$}
\lbl[l]{133,108; $X_{12}$}
\lbl[l]{115,93; $W_2$}

\lbl[c]{31,103; $SS$}
\lbl[l]{74,103; $\Pi_1$}
\lbl[l]{132,20; $T_1$}
\lbl[l]{133,32; $W_1$}
\lbl[l]{133,127; $Y_1$}
\lbl[l]{154,103; $Q$}
\lbl[l]{133,103; $Z_1$}
\lbl[r]{126,103; $\Delta_1$}

\lbl[l]{134,6; $X_1$}

\end{lpic}
\caption{The graph $\Gamma$ for $a=-1$ and $19/25<  b\leq 29/38$. Here $P_{1}=(-b -2, -1)$, $P_{2}=(b +2,-3)$, $P_{3}=(b +4, 2 b -1)$, $P_{4}=(-b +4, 5)$,  $Q=(-b +1, 0)$,
$R_{1}=(0, 2 b -1)$, 
$R_{2}=(-2 b, -b +1)$, 
$R_3=(3 b -2, -1)$, 	
$R_{4}=(3 b -2, 4 b -3)$, 
$R_5=(-b, 1)$,  
$R_6=(b -2, -1)$, 
$S=(0, b +1)$, 
$T_{1}=(0,-13 b+9)$,
$T_{2}=(13 b -10, -12 b +9)$, 	
	 $W_{1}=(0, -26 b +19)$, 
	$W_{2}=(26 b -20, -25 b +19)$,
    $X_{1}=(0,-1)$, 	
	$X_{2}=(0, b -1)$, 
	$X_{3}=(-b, 2 b -1)$, 
	$X_4=(-b, -2 b +1)$, 
    $X_5=(3 b -2, -2 b +1)$, 
	$X_{6}=(5 b -4, 2 b -1)$,
	$X_{7}=(-7b+4, 4 b -3)$, 
	$X_8=(3 b -2, -10 b +7)$, 
	$X_{9}=(13 b -10, -6 b +5)$, 
$X_{10}=(-7b+4, 20 b -15)$, 	
$X_{11}=(10-13 b, -26 b +19)$, 
	$X_{12}=(13 b -10, -38 b +29)$, 
	$X_{13}=(25 b -20, 52 b -39)$, 
	$X_{14}=(-77 b +58, -26 b +19)$,  
	$Y_1=(0, 7 b -5)$,
    $Y_{2}=(-7 b+4, -6 b+5)$, 
	$Y_3=(13 b -10, -1)$,	
	$Z_{1}=(13b-10,0)$, 
	$Z_2=(-13 b+9, 14 b -10)$, 
	$Z_3=(-b, -26 b +19)$, 
$\Pi_{1}=(-b, 0)$,
$\Pi_{2}=(b -1, 0)$, 	
$\Delta_1=(51b-39,0)$, $\Delta_2=(-b, 52 b -39)$, $\Delta_{3}=(-51 b +38
, -52 b +39)$ and $SS=(-3 b +1, 0)$.}\label{f:I}
~\vspace{-24cm}

\hspace{4.2cm}$\phantom{espaiespaiespaiespaiespaiespaie}$
{\begin{lpic}[l(2mm),r(2mm),t(2mm),b(2mm),figframe(0.2mm)]{am-cas-I(0.29)}

\lbl[l]{62,115; $S$}
\lbl[l]{27,82; $R_2$}

\lbl[r]{55,44; $R_6$}
\lbl[r]{90,44; $R_3$}

\lbl[l]{40,105; $R_5$}

\lbl[l]{75,94; $R_1$}

\lbl[l]{85,81; $R_4$}

\lbl[r]{25,44; $P_{1}$}
\lbl[l]{138,10; $P_{2}$}
\lbl[l]{171,95; $P_{3}$}
\lbl[l]{151,186; $P_{4}$}
\end{lpic}}

\end{figure}

\newpage

\begin{figure}[H]
\footnotesize
\centering

\begin{lpic}[l(2mm),r(2mm),t(2mm),b(2mm)]{am-cas-J0(0.71)}
\lbl[l]{2,191; $\boxed{a=-1,\,29/38<  b\leq 10/13}$}

\lbl[l]{1,125; $R_{2}$}
\lbl[l]{75,52; $X_4$}
\lbl[l]{28,6; $R_6$}
\lbl[l]{167,6; $R_3$}
\lbl[l]{170,54; $X_5$}
\lbl[l]{57,167; $Z_2$}
\lbl[l]{133,6; $Y_3$}
\lbl[l]{170,44; $X_8$}
\lbl[l]{148,18; $X_{11}$}
\lbl[l]{75,21; $Z_3$}
\lbl[l]{75,188; $R_5$}

\lbl[c]{113,103; $\Pi_{2}$}
\lbl[c]{16,127; $X_{10}$}
\lbl[l]{16,141; $Y_2$}
\lbl[c]{16,103; $X_7$}

\lbl[l]{75,151; $X_3$}
\lbl[l]{143,143; $R_1$}
\lbl[l]{119,151; $X_6$}
\lbl[l]{170,107; $R_4$}
\lbl[l]{144,79; $X_2$}
\lbl[l]{130,81; $T_{2}$}
\lbl[l]{128,136; $X_9$}
\lbl[l]{130,93; $X_{12}$}
\lbl[l]{125,86; $W_2$}

\lbl[c]{31,103; $SS$}
\lbl[l]{75,103; $\Pi_{1}$}
\lbl[l]{132,15; $T_1$}
\lbl[l]{132,21; $W_1$}
\lbl[l]{143,135; $Y_1$}
\lbl[l]{154,103; $Q$}
\lbl[l]{129,103; $Z_1$}

\lbl[l]{143,6; $X_1$}

\end{lpic}
\caption{The graph $\Gamma$ for $a=-1$ and $29/38<  b\leq 10/13$. Here 	$P_{1}=(-b -2, -1)$, $P_{2}=(b +2, -3)$, $P_{3}=(b +4, 2 b -1)$, $P_{4}=(-b +4, 5)$, $Q=(-b +1, 0)$,
$R_{1}=(0, 2 b -1)$, 
$R_{2}=(-2 b, -b +1)$, 
$R_3=(3 b -2, -1)$, 
$R_{4}=(3 b -2, 4 b -3)$, 
$R_5=(-b, 1)$, 
$R_6=(b -2, -1)$, 
$S=(0, b +1)$, 
$T_{1}=(0,-13 b+9)$, 
$T_{2}=(13 b -10, -12 b +9)$, 
$W_{1}=(0, -26 b +19)$, 
$W_{2}=(26 b -20, -25 b +19)$, 
$X_{1}=(0,-1)$, 
$X_{2}=(0, b -1)$, 
$X_{3}=(-b, 2 b -1)$, 
$X_4=(-b, -2 b +1)$, 
$X_5=(3 b -2, -2 b +1)$, 
$X_{6}=(5 b -4, 2 b -1)$, 
$X_{7}=(-7b+4, 4 b -3)$, 
$X_8=(3 b -2, -10 b +7)$, 
$X_{9}=(13 b -10, -6 b +5)$, 
$X_{10}=(-7b+4, 20 b -15)$, 
$X_{11}=(10-13 b, -26 b +19)$, 
$X_{12}=(13 b -10, -38 b +29)$, 
$Y_1=(0, 7 b -5)$, 
$Y_{2}=(-7 b+4, -6 b+5)$, 
$Y_3=(13 b -10, -1)$, 	
$Z_{1}=(13b-10,0)$, 
$Z_2=(-13 b+9, 14 b -10)$, 
$Z_3=(-b, -26 b +19)$, 
$\Pi_{1}=(-b, 0)$,    
$\Pi_{2}=(b -1, 0)$ and	
$SS=(-3 b +1, 0)$. }\label{f:J}
~\vspace{-23.9cm}

\hspace{4.2cm}$\phantom{espaiespaiespaiespaiespaiespaie}$
{\begin{lpic}[l(2mm),r(2mm),t(2mm),b(2mm),figframe(0.2mm)]{am-cas-J(0.29)}

\lbl[l]{62,115; $S$}
\lbl[l]{27,82; $R_{2}$}

\lbl[r]{57,44; $R_6$}
\lbl[r]{90,44; $R_3$}

\lbl[l]{40,105; $R_5$}

\lbl[l]{75,94; $R_1$}

\lbl[l]{85,81; $R_4$}

\lbl[r]{25,44; $P_{1}$}
\lbl[l]{138,10; $P_{2}$}
\lbl[l]{171,95; $P_{3}$}
\lbl[l]{151,186; $P_{4}$}
\end{lpic}}

\end{figure}

\newpage

\begin{figure}[H]
\footnotesize
\centering
\begin{lpic}[l(2mm),r(2mm),t(2mm),b(2mm)]{am-cas-K0(0.71)}
\lbl[l]{2,191; $\boxed{a=-1,\,10/13<  b\leq 11/14}$}

\lbl[l]{1,123; $R_{2}$}
\lbl[l]{77,49; $X_{4}$}
\lbl[l]{31,5; $R_{6}$}
\lbl[l]{172,5; $R_3$}
\lbl[l]{174,50; $X_5$}
\lbl[l]{57,179; $Z_2$}
\lbl[l]{151,5; $Y_3$}
\lbl[l]{174,30; $X_8$}
\lbl[l]{47,33; $Y_6$}
\lbl[l]{51,5; $Z_3$}
\lbl[l]{75,186; $R_5$}

\lbl[c]{116,101; $\Pi_2$}
\lbl[l]{75,165; $Y_5$}
\lbl[l]{10,133; $Y_2$}
\lbl[l]{10,103; $X_7$}

\lbl[l]{75,151; $X_3$}
\lbl[l]{144,143; $R_1$}
\lbl[l]{127,143; $X_6$}
\lbl[l]{175,108; $R_4$}
\lbl[l]{146,78; $X_2$}
\lbl[l]{154,85; $Y_4$}
\lbl[l]{154,130; $X_9$}

\lbl[c]{30,101; $T$}
\lbl[l]{75,101; $\Pi_1$}

\lbl[l]{134,136; $Y_1$}
\lbl[l]{167,101; $Q$}
\lbl[l]{154,101; $Z_1$}

\lbl[l]{144,5; $X_1$}

\end{lpic}
\caption{The graph $\Gamma$ for $a=-1$ and $10/13<  b\leq 11/14$. Here $P_{1}=(-b -2, -1)$, 
	$P_{2}=(b +2, -3)$, $P_{3}=(b +4, 2 b -1)$, 
	$P_{4}=(-b +4, 5)$, $Q=(-b +1, 0)$, 
$R_{1}=(0, 2 b -1)$, 
$R_{2}=(-2 b, -b +1)$, 
$R_3=(3 b -2, -1)$, 
$R_{4}=(3 b -2, 4 b -3)$, 
$R_5=(-b, 1)$,   
$R_6=(b -2, -1)$, 
$S=(0, b +1)$, 
$T=(-3 b +1, 0)$, 
$X_{1}=(0,-1)$, 
$X_{2}=(0, b -1)$, 
$X_{3}=(-b, 2 b -1)$, 
$X_4=(-b, -2 b +1)$, 
$X_5=(3 b -2, -2 b +1)$, 
$X_{6}=(5 b -4, 2 b -1)$, 
$X_{7}=(-7b+4, 4 b -3)$,  
$X_8=(3 b -2, -10 b +7)$, 
$X_{9}=(13 b -10, -6 b +5)$, 
$Y_1=(0, 7 b -5)$,
$Y_{2}=(-7 b+4, -6 b+5)$, 
$Y_3=(13 b -10, -1)$, 
$Y_4=(13 b -10, 14 b -11)$
$Y_5=(-b, 28 b -21)$,  
$Y_6=(-27 b +20, -28 b +21)$, 
$Z_1=(13 b -10, 0)$, 
$Z_2=(13 b -11, 14 b -10)$, 
$Z_3=(-27 b +20, -1)$, 
$\Pi_{1}=(-b, 0)$,  and	 
$\Pi_{2}=(b -1, 0)$.}\label{f:K}
~\vspace{-22.9cm}

\hspace{4.2cm}$\phantom{espaiespaiespaiespaiespaiespaie}$
{\begin{lpic}[l(2mm),r(2mm),t(2mm),b(2mm),figframe(0.2mm)]{am-cas-K(0.29)}

\lbl[l]{60,115; $S$}
\lbl[l]{27,82; $R_2$}

\lbl[r]{55,44; $R_6$}
\lbl[r]{90,44; $R_3$}

\lbl[l]{40,105; $R_5$}

\lbl[l]{75,93; $R_1$}

\lbl[l]{85,81; $R_4$}

\lbl[r]{25,44; $P_{1}$}
\lbl[l]{138,10; $P_{2}$}
\lbl[l]{171,95; $P_{3}$}
\lbl[l]{151,186; $P_{4}$}
\end{lpic}}

\end{figure}

\newpage

\begin{figure}[H]
\footnotesize
\centering

\begin{lpic}[l(2mm),r(2mm),t(2mm),b(2mm)]{am-cas-L0(0.71)}
\lbl[l]{2,191; $\boxed{a=-1,\,11/14<  b\leq 4/5}$}

\lbl[l]{-1,117; $R_{2}$}
\lbl[l]{77,43; $X_4$}
\lbl[l]{36,2; $R_6$}
\lbl[l]{178,2; $R_3$}
\lbl[l]{172,43; $X_5$}

\lbl[l]{170,2; $Y_3$}
\lbl[l]{172,15; $X_8$}

\lbl[l]{77,184; $R_5$}

\lbl[c]{118,99; $\Pi_2$}

\lbl[l]{4,122; $Y_2$}
\lbl[l]{5,106; $X_7$}

\lbl[l]{77,151; $X_3$}
\lbl[l]{147,145; $R_1$}
\lbl[l]{136,145; $X_6$}
\lbl[l]{181,111; $R_4$}
\lbl[l]{148,77; $X_2$}
\lbl[l]{175,104; $Y_4$}
\lbl[l]{175,119; $X_9$}

\lbl[c]{28,99; $T$}
\lbl[l]{77,99; $\Pi_{1}$}

\lbl[l]{136,140; $Y_1$}
\lbl[l]{156,99; $Q$}

\lbl[l]{147,2; $X_1$}

\end{lpic}
\caption{The graph $\Gamma$ for $a=-1$ and $11/14<  b\leq 4/5$. Here $P_{1}=(-b -2, -1)$, $P_{2}=(b +2, -3)$, $P_{3}=(b +4, 2 b -1)$, $P_{4}=(-b +4, 5)$, $Q=(-b +1, 0)$,  
$R_{1}=(0, 2 b -1)$,
$R_{2}=(-2 b, -b +1)$, 
$R_3=(3 b -2, -1)$, 
$R_{4}=(3 b -2, 4 b -3)$, 
$R_5=(-b, 1)$, 
$R_6=(b -2, -1)$, 
$S=(0, b +1)$, 
$T=(-3 b +1, 0)$, 
$X_{1}=(0,-1)$, 
$X_{2}=(0, b -1)$,
$X_{3}=(-b, 2 b -1)$, 
$X_4=(-b, -2 b +1)$, 
$X_5=(3 b -2, -2 b +1)$,
$X_{6}=(5 b -4, 2 b -1)$,  
$X_{7}=(-7b+4, 4 b -3)$, 
$X_8=(3 b -2, -10 b +7)$, 
$X_{9}=(13 b -10, -6 b +5)$, 
$Y_1=(0, 7 b -5)$, 
$Y_{2}=(-7 b+4, -6 b+5)$, 
$Y_3=(13 b -10, -1)$, 
$Y_4=(13 b -10, 14 b -11)$,	
$\Pi_{1}=(-b, 0)$ and	  
$\Pi_{2}=(b -1, 0)$.
}\label{f:L}
~\vspace{-22.2cm}

\hspace{4.2cm}$\phantom{espaiespaiespaiespaiespaiespaie}$
{\begin{lpic}[l(2mm),r(2mm),t(2mm),b(2mm),figframe(0.2mm)]{am-cas-L(0.29)}

\lbl[l]{60,115; $S$}
\lbl[l]{27,82; $R_{2}$}

\lbl[r]{55,44; $R_6$}
\lbl[r]{90,44; $R_{3}$}

\lbl[l]{40,105; $R_5$}

\lbl[l]{75,94; $R_1$}

\lbl[l]{85,81; $R_4$}

\lbl[r]{25,44; $P_{1}$}
\lbl[l]{138,10; $P_{2}$}
\lbl[l]{171,95; $P_{3}$}
\lbl[l]{151,186; $P_{4}$}
\end{lpic}}

\end{figure}

\newpage
\begin{figure}[H]
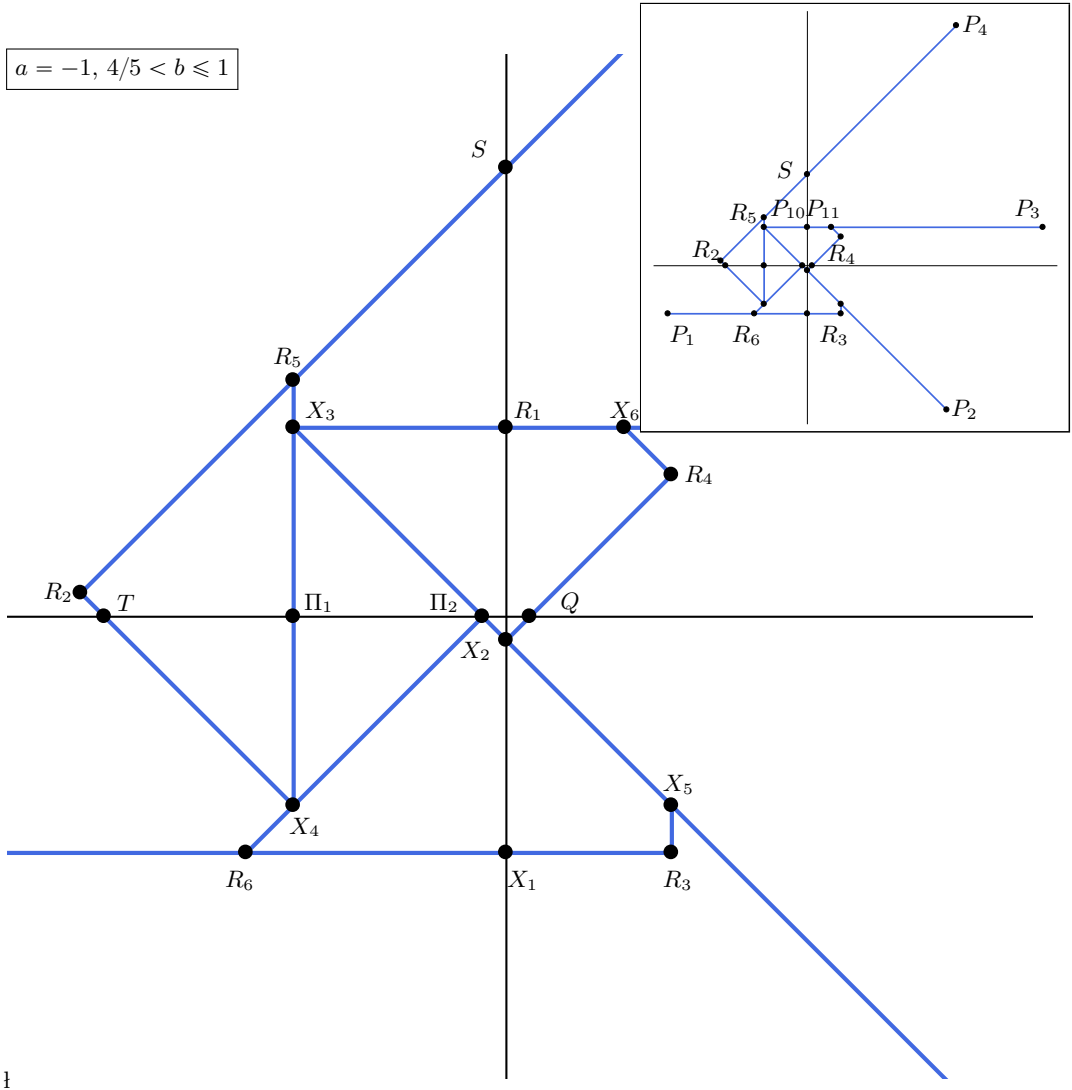

\footnotesize
\centering
\begin{lpic}[l(2mm),r(2mm),t(2mm),b(2mm)]{am-cas-M0(0.71)}
\lbl[l]{2,191; $\boxed{a=-1,\,4/5<  b\leq 1}$}

\lbl[l]{89,176; $S$}
\lbl[l]{9,93; $R_{2}$}
\lbl[l]{55,49; $X_{4}$}
\lbl[l]{43,39; $R_{6}$}
\lbl[l]{125,39; $R_{3}$}
\lbl[l]{125,57; $X_{5}$}

\lbl[l]{52,137; $R_5$}

\lbl[c]{84,91; $\Pi_{2}$}

\lbl[l]{58,127; $X_3$}
\lbl[l]{97,127; $R_1$}
\lbl[l]{115,127; $X_6$}
\lbl[l]{129,115; $R_4$}
\lbl[l]{87,82; $X_2$}
\l

\lbl[l]{20,91; $T$}
\lbl[l]{55,91; $\Pi_{1}$}

\lbl[l]{103,91; $Q$}

\lbl[l]{93,39; $X_1$}

\end{lpic}
\caption{The graph $\Gamma$ for $a=-1$ and $4/5<  b\leq 1$. Here $P_{1}=(-b -2, -1)$, 
$P_{2}=(b +2, -3)$, $P_{3}=(b +4, 2 b -1)$,$P_{4}=(-b +4, 5)$, $Q=(-b +1, 0)$, 
$R_1=(0, 2 b -1)$, 
$R_{2}=(-2 b, -b +1)$, 
$R_3=(3 b -2, -1)$, 
$R_4=(3 b -2, 4 b -3)$, 
$R_5=(-b, 1)$, 
$R_6=(b -2, -1)$, 
$S_{1}=(0, b +1)$, 
$T=(-3 b +1, 0)$, 
$X_1=(0, -1)$, 
$X_2=(0, b -1)$, 
$X_3=(-b, 2 b -1)$, 
$X_{4}=(-b, -2 b +1)$, 
$X_5=(3 b -2, -2 b +1)$, 
$X_6=(5 b -4, 2 b -1)$, 
$\Pi_1=(-b, 0)$ and
$\Pi_{2}=(b -1, 0)$.}\label{f:M}
~\vspace{-19.1cm}

\hspace{4.2cm}$\phantom{espaiespaiespaiespaiespaiespaie}$
{\begin{lpic}[l(2mm),r(2mm),t(2mm),b(2mm),figframe(0.2mm)]{am-cas-M(0.29)}

\lbl[l]{62,120; $S$}
\lbl[l]{23,83; $R_{2}$}

\lbl[r]{55,44; $R_6$}
\lbl[r]{95,44; $R_3$}

\lbl[l]{40,100; $R_5$}

\lbl[l]{75,102; $P_{11}$}
\lbl[l]{59,102; $P_{10}$}

\lbl[l]{85,81; $R_4$}

\lbl[r]{25,44; $P_{1}$}
\lbl[l]{142,10; $P_{2}$}
\lbl[l]{171,102; $P_{3}$}
\lbl[l]{147,186; $P_{4}$}
\end{lpic}}

\end{figure}

\newpage

\begin{figure}[H]
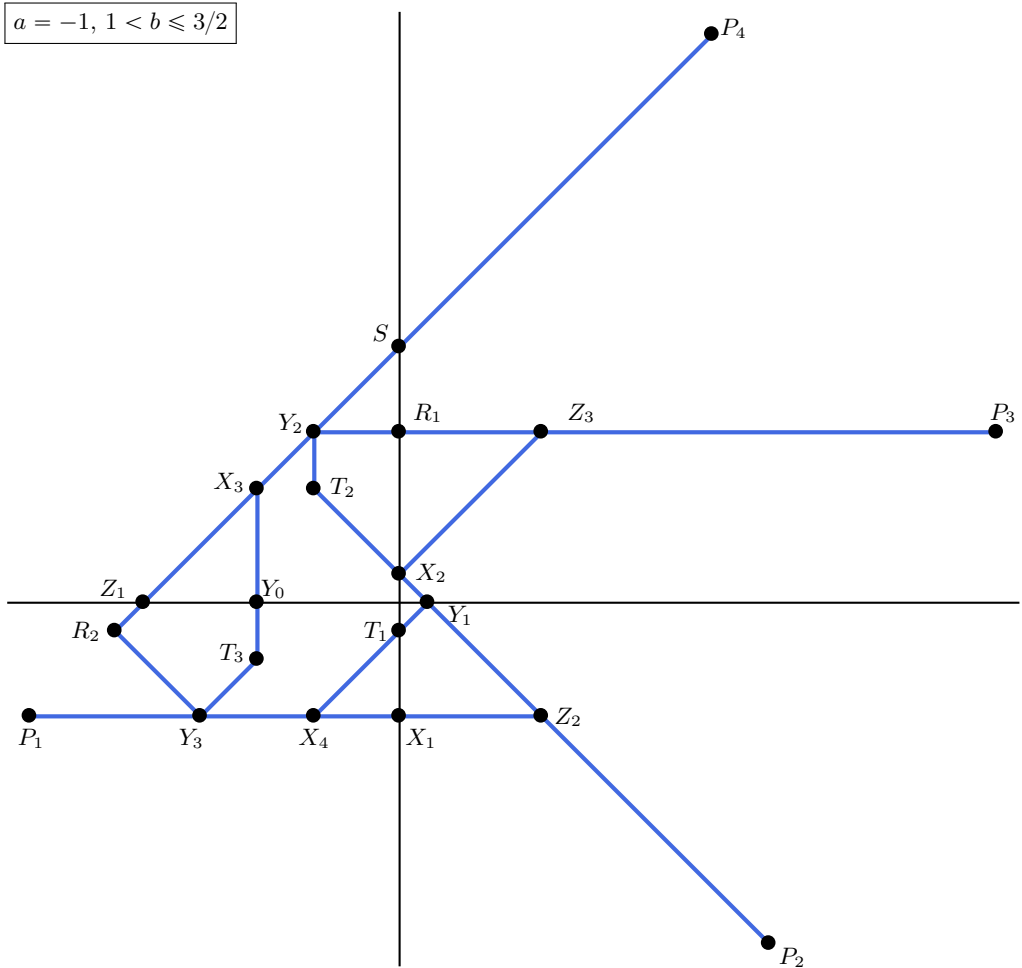

\footnotesize
\centering

\begin{lpic}[l(2mm),r(2mm),t(2mm),b(2mm)]{am-cas-21-a(0.71)}
\lbl[l]{3,185; $\boxed{a=-1,\,1< b\leq 3/2}$}

\lbl[l]{137,184; $P_{4}$}
\lbl[r]{75,127; $S$}
\lbl[r]{59,110; $Y_2$}
\lbl[r]{48,99; $X_3$}

\lbl[r]{26,79; $Z_1$}
\lbl[r]{21,71; $R_2$}

\lbl[c]{8,51; $P_{1}$}
\lbl[c]{38,51; $Y_3$}
\lbl[c]{61,51; $X_4$}
\lbl[c]{81,51; $X_1$}
\lbl[l]{106,55; $Z_2$}

\lbl[l]{64,98; $T_{2}$}

\lbl[l]{148,10; $P_{2}$}

\lbl[l]{80,82; $X_{2}$}

\lbl[l]{86,74; $Y_1$}

\lbl[l]{51,79; $Y_0$}

\lbl[c]{190,112;  $P_{3}$}
\lbl[c]{111,112; $Z_3$}
\lbl[l]{78,112; { $R_1$}}

\lbl[r]{75,71; { $T_1$}}

\lbl[r]{48,67; { $T_3$}}

\end{lpic}

\caption{The graph $\Gamma$ for $a=-1$ and $1/< b\leq 3/2$. Here $P_{1}=(
-b -2, -1)$, $P_{2}=(b +2, -3)$, $P_{3}=(b +4, 2 b -1)$, $P_4=(-b +4, 5)$
$R_1=(0, 2 b -1)$, 
$R_2=(-2 b, -b +1)$, 
$S=(0, b +1)$, 
$T_1=(0, -b +1)$,
$T_2=(b -2, 1)$,
$T_3=(-b, 2 b -3)$, 
$X_1=(0, -1)$, 
$X_2=(0, b -1)$, 
$X_3=(-b, 1)$, 
$X_4=(b -2, -1)$, 
$Y_0=(-b, 0)$, 
$Y_1=(b -1, 0)$, 
$Y_2=(b -2, 2 b -1)$, 
$Y_3=(-3 b +2, -1)$, 
$Z_1=(-b -1, 0)$, 
$Z_2=(b, -1)$ and
$Z_3=(b, 2 b -1)$.}\label{f:21a}
\end{figure}

\begin{figure}[H]
\footnotesize
\centering

\begin{lpic}[l(2mm),r(2mm),t(2mm),b(2mm)]{am-cas-21-b(0.71)}
\lbl[l]{3,185; $\boxed{a=-1,\,3/2< b\leq 8/5}$}

\lbl[r]{77,131; $S$}

\lbl[r]{67,123; $Y_2$}

\lbl[r]{45,99; $X_3$}

\lbl[r]{26,81; $Z_1$}

\lbl[r]{14,68; $R_2$}

\lbl[c]{25,53; $Y_3$}

\lbl[c]{68,53; $X_4$}

\lbl[c]{81,53; $X_1$}

\lbl[l]{101,53; $Z_2$}

\lbl[l]{91,81; $Y_1$}

\lbl[l]{80,88; $X_2$}

\lbl[c]{109,124; $Z_3$}

\lbl[l]{80,123; $R_1$}

\lbl[l]{49,81; $T_3$}

\lbl[l]{45,74; $W_1$}

\lbl[l]{32,62; $T_6$}

\lbl[r]{18,62; $W_4$}

\lbl[r]{76,67; $T_1$}

\lbl[r]{85,75; $T_4$}

\lbl[l]{94,75; $W_2$}

\lbl[c]{75,115; $T_5$}

\lbl[c]{75,123; $W_3$}

\lbl[c]{75,98; $T_2$}

\lbl[c]{8,53; $P_{1}$}

\lbl[l]{148,13; $P_{2}$}

\lbl[c]{190,124;  $P_{3}$}

\lbl[l]{130,178; $P_{4}$}
\end{lpic}

\caption{The graph $\Gamma$ for $a=-1$ and $3/2< b\leq 8/5$. Here $P_{1}=(-b -2, -1)$, $P_{2}=(b +2, -3)$, $P_{3}=(b +4, 2 b -1)$, $P_4=(-b +4, 5)$,
$R_1=(0, 2 b -1)$, 
$R_2=(-2 b, -b +1)$, 
$S=(0, b +1)$, $T_1=(0, -b +1)$,
$T_2=(b -2, 1)$, $T_3=(-b, 2 b -3)$, 
 $T_4=(-b +2, -2 b +3)$, $T_5=(b -2, -2 b +5)$,
$T_6=(-4+b, 4 b -7)$,
$W_1=(-3 b +3, 0)$, $W_2=(3 b -4, -2 b +3)$,
$W_3=(5 b -8, 2 b -1)$, $W_4=(
-7 b +8, 4 b -7)$,
$X_1=(0, -1)$, $X_2=(0, b -1)$, 
$X_3=(-b, 1)$, $X_4=(b -2, -1)$,
$Y_1=(b -1, 0)$, 
$Y_2=(b -2, 2 b -1)$, $Y_3=(-3 b +2, -1)$, 
$Z_1=(-b -1, 0)$, $Z_2=(b, -1)$ and
$Z_3=(b, 2 b -1)$.}\label{f:21b}
\end{figure}

\begin{figure}[H]
\footnotesize
\centering

\begin{lpic}[l(2mm),r(2mm),t(2mm),b(2mm)]{am-cas-21-c(0.71)}
\lbl[l]{3,185; $\boxed{a=-1,\,8/5< b\leq 7/4}$}

\lbl[r]{77,132; $S$}

\lbl[r]{69,125; $Y_2$}

\lbl[r]{44,100; $X_3$}

\lbl[r]{26,82; $Z_1$}

\lbl[r]{11,65; $R_2$}

\lbl[c]{20,55; $Y_3$}

\lbl[c]{72,55; $X_4$}

\lbl[c]{82,55; $X_1$}

\lbl[l]{109,55; $Z_2$}

\lbl[l]{91,82; $Y_1$}

\lbl[l]{81,92; $X_2$}

\lbl[c]{111,128; $Z_3$}

\lbl[l]{80,128; $R_1$}

\lbl[l]{49,85; $T_3$}

\lbl[l]{40,75; $W_1$}

\lbl[l]{32,69; $T_6$}

\lbl[r]{18,73; $W_4$}

\lbl[r]{76,67; $T_1$}

\lbl[c]{83,75; $T_4$}

\lbl[c]{103,75; $W_2$}

\lbl[r]{71,111; $ T_5$}

\lbl[l]{81,117; $ Q$}

\lbl[l]{88,128; $ W_3$}

\lbl[r]{71,98; $ T_2$}

\lbl[c]{8,55; $P_{1}$}

\lbl[l]{148,16; $P_{2}$}

\lbl[c]{189,128;  $P_{3}$}

\lbl[l]{126,176; $P_{4}$}

\end{lpic}

\caption{The graph $\Gamma$ for $a=-1$ and $8/5< b\leq 7/4$. Here $P_{1}=(-b -2, -1)$, $P_{2}=(b +2, -3)$, $P_{3}=(b +4, 2 b -1)$, $P_4=(-b +4, 5)$, $Q=(0, -3 b +7)$
$R_1=(0, 2 b -1)$,
$R_2=(-2 b, -b +1)$, 
$S=(0, b +1)$,  $T_1=(0, -b +1)$, $T_2=(b -2, 1)$, $T_3=(-b, 2 b -3)$, $T_4=(-b +2, -2 b +3)$, $T_5=(b -2, -2 b +5)$, $T_6=(-4+b, 4 b -7)$, 
$W_1=(-3 b +3, 0)$, $W_2=(3 b -4, -2 b +3)$,  $W_3=(5 b -8, 2 b -1)$, $W_4=(
3 b -8, 4 b -7).$
$X_1=(0, -1)$, $X_2=(0, b -1)$, 
$X_3=(-b, 1)$, $X_4=(b -2, -1)$, 
$Y_1=(b -1, 0)$, 
$Y_2=(b -2, 2 b -1)$, $Y_3=(-3 b +2, -1)$, 
$Z_1=(-b -1, 0)$, $Z_2=(b, -1)$ and $Z_3=(b, 2 b -1)$.}\label{f:21c}
\end{figure}

\begin{figure}[H]
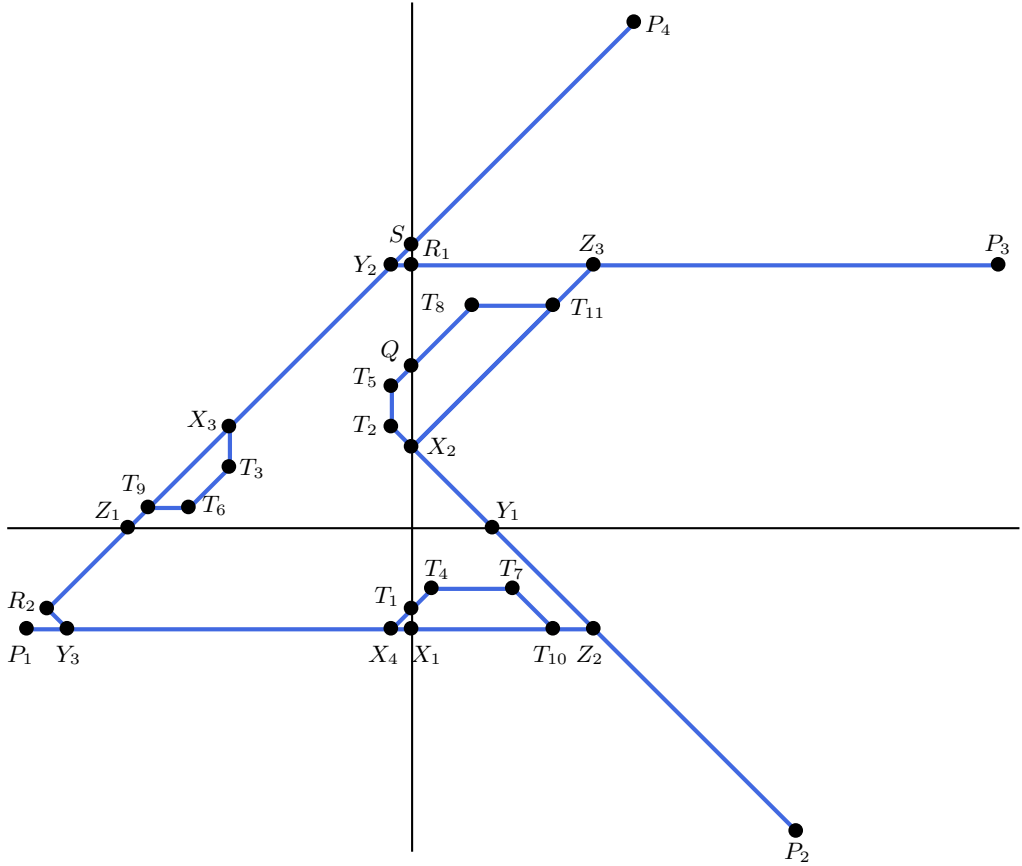

\footnotesize
\centering

\begin{lpic}[l(2mm),r(2mm),t(2mm),b(2mm)]{am-cas-21-d(0.71)}
\lbl[l]{3,185; $\boxed{a=-1,\,7/4< b< 2}$}

\lbl[r]{78,134; $S$}

\lbl[r]{73,128; $Y_2$}

\lbl[r]{43,99; $X_3$}

\lbl[r]{25,82; $Z_1$}

\lbl[r]{9,65; $R_2$}

\lbl[c]{15,55; $Y_3$}

\lbl[c]{74,55; $X_4$}

\lbl[c]{82,55; $X_1$}

\lbl[l]{110,55; $Z_2$}

\lbl[l]{95,82; $Y_1$}

\lbl[l]{82,94; $X_2$}

\lbl[c]{113,132; $Z_3$}

\lbl[c]{84,131; $R_1$}

\lbl[l]{47,90; $T_3$}

\lbl[l]{40,83; $T_6$}

\lbl[l]{25,87; $T_9$}

\lbl[r]{77,66; $T_1$}

\lbl[c]{84,71; $T_4$}

\lbl[c]{98,71; $T_7$}

\lbl[l]{102,55; $T_{10}$}

\lbl[r]{73,98; $T_2$}

\lbl[r]{73,107; $ T_5$}

\lbl[r]{77,112; $Q$}

\lbl[l]{81,121; $T_8$}

\lbl[l]{109,120; $T_{11}$}

\lbl[c]{6,55; $P_{1}$}

\lbl[l]{149,18; $P_{2}$}

\lbl[c]{189,132;  $P_3$}

\lbl[l]{123,173; $P_4$}

\end{lpic}

\caption{The graph $\Gamma$ for $a=-1$ and $7/4< b< 2$. Here $P_{1}=(
	-b -2, -1)$, $P_{2}=(b +2, -3)$, $P_{3}=(b +4, 2 b -1)$, $P_4=(-b +4, 5)$, $Q=(0, -3 b +7)$, 
	$R_1=(0, 2 b -1)$, $R_2=(-2 b, -b +1)$,
$S=(0, b +1)$, 
$T_1=(0, -b +1)$, $T_2=(b -2, 1)$,
	  $T_3=(-b, 2 b -3)$, $T_4=(-b +2, -2 b +3)$,  $T_5=(	b -2, -2 b +5)$,
	$T_6=(-4+b, 4 b -7)$,   $T_7=(-5 b +10, -2 b +3)$, $T_{8}=(-3 b +6, -6 b +13)$,  
   $T_9=(3 b -8, 4 b -7)$, $T_{10}=(-7 b +14, -1)$, $T_{11}=(-7 b +14, -6 b +13)$
$X_1=(0, -1)$, $X_2=(0, b -1)$, 
$X_3=(-b, 1)$, $X_4=(b -2, -1)$, 
$Y_1=(b -1, 0)$, 
$Y_2=(b -2, 2 b -1)$, 
$Y_3=(-3 b +2, -1)$, 
$Z_1=(-b -1, 0)$, $Z_2=(b, -1)$ and $Z_3=(b, 
	2 b -1)$.}\label{f:21d}
\end{figure}

\begin{figure}[H]
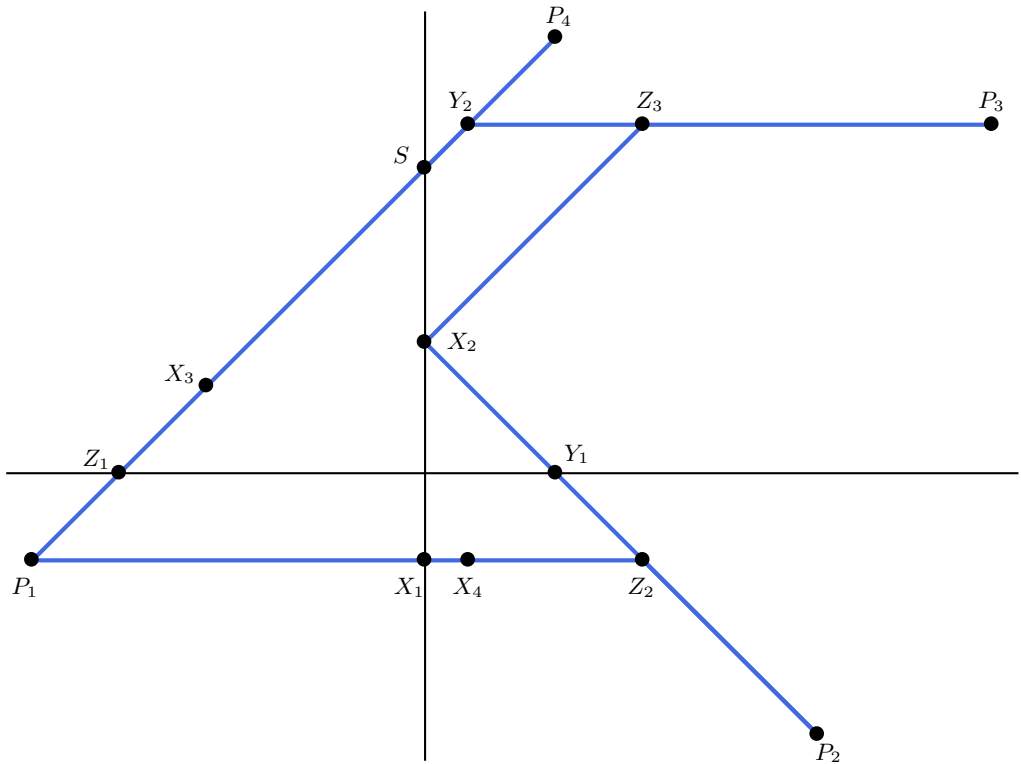

\footnotesize
\centering

\begin{lpic}[l(2mm),r(2mm),t(2mm),b(2mm)]{am-cas-22-espb(0.71)}
\lbl[l]{3,185; $\boxed{a=-1,\,2\leq b<3}$}

\lbl[r]{79,141; $S$}

\lbl[r]{23,84; $Z_1$}

\lbl[c]{7,60; $P_1$}

\lbl[c]{79,60; $X_1$}
\lbl[r]{125,60; $Z_2$}

\lbl[l]{155,29; $P_2$}

\lbl[l]{86,106; $X_2$}

\lbl[l]{108,85; $Y_1$}

\lbl[c]{188,151;  $P_3$}
\lbl[c]{124,151; $Z_3$}
\lbl[r]{91,151; { $Y_2$}}

\lbl[c]{107,167;  $P_{4}$}

\lbl[c]{36,100; $X_3$}
\lbl[c]{90,60; $X_4$}

\end{lpic}

\caption{The graph $\Gamma$ for $a=-1$ and $2\leq b<3$. Here $P_{1}=(-b -2, -1)$, $P_{2}=(b +2, -3)$,  $P_{3}=(b +4, 2 b -1)$, $P_{4}=(
-b +4, 5)$, $S=(0, b +1)$,  $X_1=(0, -1)$,
$X_2=(0, b -1)$, $X_{3}=(-b, 1)$, $X_{4}=(b -2, -1)$, $Y_{1}=(
b -1, 0)$, $Y_2=(b -2, 2 b -1)$,  
$Z_1=(-b -1, 0)$,   $Z_2=(b, -1)$  
and 
 $Z_3=(b, 2 b -1)$.}\label{ff:22}
\end{figure}

\begin{figure}[H]
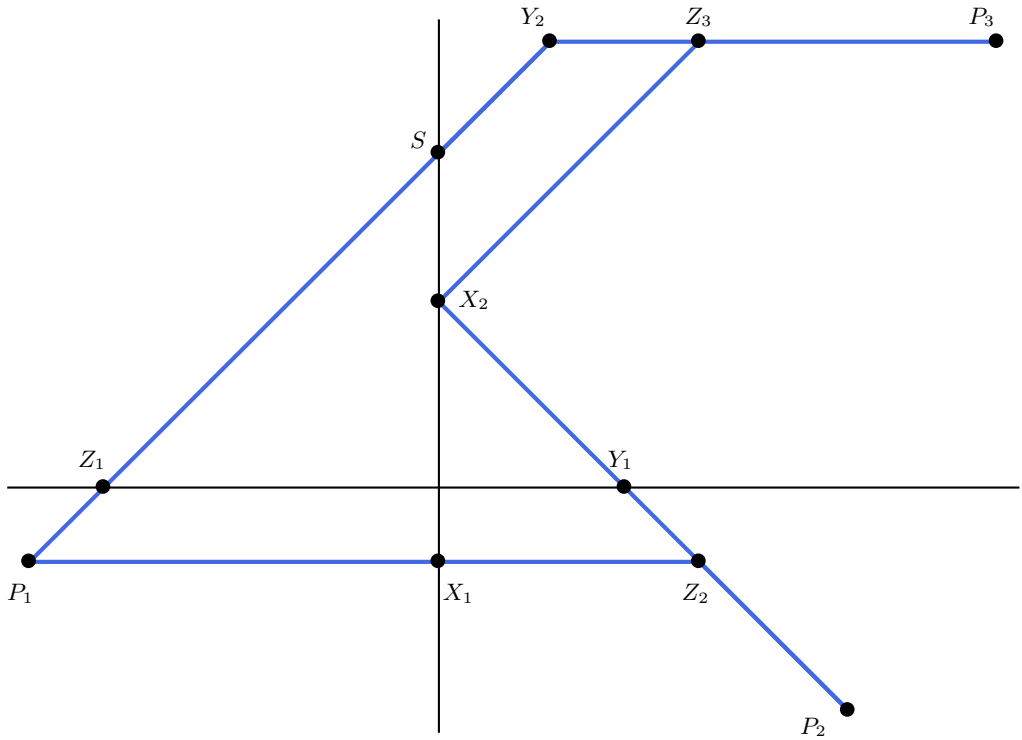

\footnotesize
\centering
\begin{lpic}[l(2mm),r(2mm),t(2mm),b(2mm)]{am-cas-23(0.71)}
\lbl[l]{3,185; $\boxed{a=-1,\,b\geq 3}$}

\lbl[r]{82,142; $S$}

\lbl[r]{22,82; $Z_1$}

\lbl[c]{6,57; $P_1$}

\lbl[c]{88,57; $X_1$}
\lbl[r]{135,57; $Z_2$}

\lbl[l]{152,32; $P_2$}

\lbl[l]{88,112; $X_2$}

\lbl[l]{116,82; $Y_1$}

\lbl[c]{186,165;  $P_3$}
\lbl[c]{133,165; $Z_3$}
\lbl[c]{102,165; $Y_2$}
\end{lpic}

\caption{The graph $\Gamma$ for $a=-1$ and $b\geq 3$. Here 
$P_1=(-b -2, -1)$,  $P_2=(b+2,-3)$, $P_3=(b +4, 2 b -1)$, $S=(0, b +1)$, 
$X_1=(0, -1)$, $X_2=(0,b-1)$, 
$Y_1=(b-1,0)$, $Y_2=(b-2, 2 b -1)$,
$Z_1=(-b-1,0)$, 
$Z_2=(b,-1)$, 
and
$Z_3=(b, 2 b -1)$.}\label{f:23}
\end{figure}

\end{document}